\documentclass[onefignum,onetabnum]{siamart190516}

\usepackage{graphicx}
\usepackage{multirow,bigstrut}
\usepackage{amsmath,amsfonts,amssymb}
\usepackage{mathrsfs}
\usepackage{color}
\usepackage{upgreek}
\usepackage{bm}
\usepackage{verbatim}
\usepackage{mathtools}
\usepackage{calligra}
\usepackage[T1]{fontenc}
\usepackage{subfigure}
\usepackage{pgfplots}
\usepackage{xcolor}
\usepackage{textgreek}

\usepackage{hyperref}
\usepackage{amsmath}
\usepackage{cleveref}

\usepackage{tikz}
  \usetikzlibrary{calc,shapes,decorations,decorations.text, mindmap,shadings,patterns,matrix,arrows,intersections,automata,backgrounds}
  
\usetikzlibrary{shapes,arrows}

\usepackage{verbatim}
\usepackage{exscale}
\usepackage{relsize}
\usepackage{enumitem}
\theoremstyle{plain}
\newtheorem{rem}[theorem]{Remark}
\newtheorem{exam}{\hspace{1mm}Example}[section]
\newtheorem{assumption}[theorem]{Assumption}
\newcommand{\cm}[1]{{\color{black}{#1}}}

\makeatletter
\tikzset{
        hatch distance/.store in=\hatchdistance,
        hatch distance=5pt,
        hatch thickness/.store in=\hatchthickness,
        hatch thickness=5pt
        }
\pgfdeclarepatternformonly[\hatchdistance,\hatchthickness]{north east hatch}
    {\pgfqpoint{-1pt}{-1pt}}
    {\pgfqpoint{\hatchdistance}{\hatchdistance}}
    {\pgfpoint{\hatchdistance-1pt}{\hatchdistance-1pt}}%
    {
        \pgfsetcolor{\tikz@pattern@color}
        \pgfsetlinewidth{\hatchthickness}
        \pgfpathmoveto{\pgfqpoint{0pt}{0pt}}
        \pgfpathlineto{\pgfqpoint{\hatchdistance}{\hatchdistance}}
        \pgfusepath{stroke}
    }
\makeatother

\newcommand{\vertiii}[1]{{\left\vert\kern-0.25ex\left\vert\kern-0.25ex\left\vert #1 
    \right\vert\kern-0.25ex\right\vert\kern-0.25ex\right\vert}}

\begin{document}
 

\headers{Unified framework for MS-GFEM}{} 
 
\title{A unified framework for multiscale spectral generalized FEMs and low-rank approximations to multiscale PDEs}

\author{Chupeng Ma\thanks 
{School of Sciences, Great Bay University, Dongguan 523000, China (\email{chupeng.ma@gbu.edu.cn}).}
}
%
\maketitle

\begin{abstract}
Multiscale partial differential equations (PDEs), featuring heterogeneous coefficients oscillating across possibly non-separated scales, pose computational challenges for standard numerical techniques. Over the past two decades, a range of specialized methods has emerged that enables the efficient solution of such problems. Two prominent approaches are numerical multiscale methods with problem-adapted coarse approximation spaces, and structured inverse methods that exploit a low-rank property of the associated Green's functions to obtain approximate matrix factorizations.

This work presents an abstract framework for the design, implementation, and analysis of the multiscale spectral generalized finite element method (MS-GFEM), a particular numerical multiscale method originally proposed in [I. Babuska and R. Lipton, Multiscale Model.\;\,Simul., 9 (2011), pp.~373--406]. MS-GFEM is a partition of unity method employing optimal local approximation spaces constructed from local spectral problems. We establish a general local approximation theory demonstrating exponential convergence with respect to local degrees of freedom under certain assumptions, with explicit dependence on key problem parameters. Our framework applies to a broad class of multiscale PDEs with $L^{\infty}$-coefficients in both continuous and discrete, finite element settings, including highly indefinite problems (convection-dominated diffusion, as well as the high-frequency Helmholtz, Maxwell and elastic wave equations with impedance boundary conditions), and higher-order problems. Notably, we prove a local convergence rate of $O(e^{-cn^{1/d}})$ for MS-GFEM for all these problems, improving upon the $O(e^{-cn^{1/(d+1)}})$ rate shown by Babuska and Lipton.


Moreover, based on the abstract local approximation theory for MS-GFEM, we establish a unified framework for showing low-rank approximations to multiscale PDEs. This framework applies to the aforementioned problems, proving that the associated Green's functions admit an $O(|\log\epsilon|^{d})$-term separable approximation on well-separated domains with error $\epsilon>0$. Our analysis improves and generalizes the result in [M. Bebendorf and W. Hackbusch, Numerische Mathematik, 95 (2003), pp.~1-28] where an $O(|\log\epsilon|^{d+1})$-term separable approximation was proved for Poisson-type problems. It provides a rigorous theoretical foundation for diverse structured inverse methods, and also clarifies the intimate connection between approximation mechanisms in such methods and MS-GFEM.

\end{abstract}


\tableofcontents
\section{Introduction}\label{sec:1}
Multiscale problems appear in a wide range of scientific and engineering applications in different fields. Representative examples include flow in highly heterogeneous porous media and the mechanical analysis of composite materials. These problems are often modeled by partial differential equations (PDEs) with coefficients and solutions encoding the underlying multiscale character. Direct numerical solution of these PDE models requires very fine mesh discretizations to resolve all relevant fine-scale features and is thus prohibitively expensive. To solve multiscale problems efficiently and accurately, suitable approximation at some level is needed.

Classical homogenization theory \cite{bensoussan2011asymptotic,jikov2012homogenization} provides a powerful framework for solving multiscale problems with a periodic microstructure. The basic idea is to approximate the original PDE model by a homogenized equation via suitable cell problems. This conceptual foundation has profoundly influenced and motivated the development of diverse multiscale methodologies in broader contexts. For instance, in settings characterized by scale separation, the heterogeneous multiscale methods (HMM) \cite{abdulle2012heterogeneous,ming2005analysis,weinan2003heterognous}. For general multiscale problems without scale separation, numerical multiscale methods are commonly devised within the variational framework. These methods hinge on approximating relevant solution spaces by low or finite dimensional spaces with localized bases. There is an extensive literature on this subject, encompassing notable methodologies such as multiscale finite element methods (MsFEM) \cite{hou1997multiscale,hou1999convergence,le2014msfem}, generalized multiscale finite element methods (GMsFEM) \cite{efendiev2013generalized,chung2018constraint}, localized orthogonal decomposition (LOD) \cite{maalqvist2014localization,hauck2023super,freese2024super}, variational multiscale methods (VMS) \cite{hughes1995multiscale,hughes2007variational}, flux norm homogenization \cite{owhadi2011localized}, game-theoretical approach \cite{owhadi2015bayesian,owhadi2017gamblets,owhadi2019operator}, and more \cite{araya2013multiscale,hetmaniuk2010special}. The multiscale spectral generalized finite element method (MS-GFEM) \cite{babuska2011optimal,ma2021novel} addressed in this work also falls into this category. These methods are primarily concerned with identifying problem-specific local basis functions exhibiting favorable approximation qualities, often through carefully designed local problems. For comprehensive surveys of numerical multiscale methods, we refer to \cite{altmann2021numerical,chungmultiscale,engquist2008asymptotic,efendiev2009multiscale,maalqvist2020numerical,ming2006numerical}.

Homogenization-based methods and coarse space based methods gain efficiency through approximations at the model level and discretization level, respectively. There exists yet another class of multiscale methods \cite{bebendorf2003existence,grasedyck2008parallel,xia2010superfast,xia2013efficient,amestoy2015improving,ho2016hierarchical,chaillat2017theory,engquist2011sweeping} based on approximate matrix factorizations, with approximation at the solver level. These methods aim at an efficient solution of large linear systems resulting from standard discretizations of multiscale PDEs. The key idea is to exploit a low-rank property of the associated matrices and their inverses. Indeed, research \cite{bebendorf2003existence,bebendorf2005efficient,chandrasekaran2010numerical} has revealed that off-diagonal blocks in certain dense intermediate matrices in matrix factorizations for elliptic type PDEs, even with rough coefficients, admit low-rank approximations. These dense matrices can then be well approximated by rank-structured matrices, such as hierarchical matrices ($\mathcal{H}$-matrices) \cite{hackbusch2015hierarchical}, $\mathcal{H}^{2}$-matrices \cite{borm2010efficient}, and other compression formats \cite{amestoy2015improving,xia2010fast}. This adaption considerably reduces memory requirements and the complexity of matrix operations. The resulting approximate factorization can work as an effective preconditioner for iterative solvers or as an accurate direct solver, depending on the accuracy of matrix approximations. The foundation of these structured inverse methods rests upon a low-rank approximation property of Green's functions or equivalently, the solution spaces of the underlying PDEs. This theoretical rank structure governs achievable matrix compression rates and associated computational costs in the methods mentioned above.

In what follows, we review related works on the MS-GFEM and low-rank approximations to elliptic PDEs underlying the structured inverse methods, which will be unified in a theoretical framework in this paper.

\subsection{MS-GFEM}
The MS-GFEM, as its name indicates, is rooted in the generalized finite element method (GFEM) \cite{babuvska1997partition,melenk1995generalized,strouboulis2001generalized}. The basic idea of the GFEM is to partition the computational domain into a series of overlapping subdomains $\{\omega_i\}_{i=1}^{M}$, and construct finite-dimensional approximation spaces tailored to each subdomain. These local approximation spaces are glued together using a partition of unity to build a global ansatz space which is then used in the Galerkin scheme. The MS-GFEM is a special GFEM with optimal local approximation spaces based on local spectral problems. In a seminal work by Babuska and Lipton \cite{babuska2011optimal}, the optimal local approximation spaces were constructed for Poisson-type equations with $L^{\infty}$-coefficients by means of the singular value decomposition of specific compact operators. Within this approach, the local space on a subdomain $\omega_i$ is spanned by selected eigenfunctions of a local eigenproblem posed on a generalized harmonic space, which, in the $H^{1}$ setting, takes the form 
\begin{equation*}
H_{a}(\omega_{i}^{\ast}) = \{u\in H^{1}(\omega_{i}^{\ast}): a_{\omega_i^{\ast}}(u,v) = 0 \quad \text{for all}\;\,v\in H_{0}^{1}(\omega_{i}^{\ast})\},  
\end{equation*}
where $\omega^{\ast}_i\supset \omega_i$ is a so-called oversampling domain, and $a_{\omega_{i}^{\ast}}(\cdot,\cdot)$ is the associated local bilinear form. It was rigorously proved that such a local space, augmented by the solution of a local Poisson problem, is capable of approximating the solution on $\omega_i$ in the energy norm with error bound $\exp(-n^{1/(d+1)-\epsilon})$, where $d$ and $n$ represent the dimensions of the domain and the local space, respectively. A notable feature of the MS-GFEM is its inherent parallelism, facilitated by the independent execution of local computations. Significantly, it needs relatively few degrees of freedom in practical computations, attributed to the exponential convergence rate. The method was later extended to other applications including elasticity problems \cite{babuvska2014machine} and parabolic problems \cite{schleuss2022optimal}. Moreover, the concept of MS-GFEM has been harnessed in the development of novel model order reduction methods \cite{buhr2018randomized,smetana2016optimal} and other multiscale methods \cite{chen2021exponential,chen2023exponentially,chen2023exponentially-2} demonstrating exponential convergence rates. Details of a practical implementation of the method can be found in \cite{chen2020random,babuvska2020multiscale}.

Recently, we introduced new optimal local approximation spaces for MS-GFEM in \cite{ma2021novel}, utilizing local spectral problems involving the partition of unity, and achieving a similar exponential convergence rate. The novelty of our approach is that the information of the partition of unity is suitably incorporated into the local spaces such that the global error is fully determined by local approximation errors. This advantageous attribute simplifies the implementation of adaptive algorithms, particularly in the discrete setting (see Remark~\ref{adaptive-remarks}), and allows the use of extremely small overlaps in practical computations. In \cite{ma2022error}, we extended the method to the FE discrete setting and proved similar theoretical results. The discrete MS-GFEM essentially serves as an efficient model reduction technique for discretized multiscale PDEs. For generalizations of our approach to other applications and practical issues of an implementation, we refer to \cite{benezech2022scalable,chupeng2023wavenumber,ma2022exponential}.

While exponential convergence rates have been proved for MS-GFEM, numerical experiments \cite{benezech2022scalable,ma2021novel} showed that it converges much faster than theoretical predictions. Therefore, the established local convergence rate $O(e^{-cn^{1/(d+1)}})$ is not optimal. Indeed, the local error estimate in the MS-GFEM is closely linked to classical approximation theory. Specifically, this error estimate relates to the optimal approximation of the generalized harmonic space $H_{a}(\omega_{i}^{\ast})$ in a compact subset $\omega_{i}\subset\omega_{i}^{\ast}$ by $n$-dimensional spaces. The quest for optimal approximations of 'well-behaved' functions within compact subsets by finite-dimensional spaces has a long-standing history dating back to Bernstein's work. In \cite{bernstein1912ordre}, it was demonstrated that any analytic function in a ellipse $E_{\rho}\supset [-1,1]$ ($\rho>1$), with sup-norm bounded by 1, can be uniformly approximated in $[-1,1]$ by polynomials of degree at most $n$ with a convergence rate $O(\rho^{-n})$. Later Kolmogrov and collaborators studied this best approximation problem for analytic functions and harmonic functions by means of the concept of the Kolmogrov $n$-width \cite{pinkus1985n}. In the context of harmonic functions, it was established \cite{nguyen1973bases} that the unit ball (in the sup-norm topology) of the space of bounded harmonic functions in $D\subset\mathbb{R}^{2}$ can be uniformly approximated in a compact subset by $n$-dimensional spaces, exhibiting a convergence rate $O(e^{-cn})$. For harmonic functions in $\mathbb{R}^{d}$ ($d\geq 3$), similar results are known, yet limited to the case of concentric balls with a convergence rate $O(e^{-cn^{1/(d-1)}})$ \cite{skiba2000asymptotics}. Nevertheless, since these results primarily pertain to error estimates in the sup-norm, and their proofs critically hinge on favorable properties of analytic or harmonic functions, it is very difficult to adapt them to our context for establishing local error bounds in the $H^{1}$-norm.

\subsection{Low-rank approximations to elliptic PDEs}
The intrinsic low-rank nature of elliptic PDEs is characterized by approximate separability of the associated Green's functions. For a differential operator with sufficiently smooth coefficients, the Green's function $G({\bm x}, {\bm y})$ is smooth (away from the singularity). This characteristic implies that $G({\bm x}, {\bm y})$ restricted to disjoint compact sets can be approximated in a separable manner: $\sum_{i=1}^{k}u_{i}({\bm x})v_{i}({\bm y})$, with exponential convergence in $k$. This property has been exploited extensively to develop efficient numerical algorithms \cite{hackbusch2015hierarchical,greengard1987fast,engquist2007fast,candes2009fast}. Its proof commonly relies on standard approximation techniques, e.g.\;Taylor expansions or polynomial interpolation \cite{borm2016approximation,candes2009fast}. Such techniques fail in the absence of the required smoothness, which happens for differential operators with non-smooth coefficients. In this case, a similar separability is highly nontrivial and it hinges on the intrinsic property of the underlying differential operators. A pioneering work on this problem was done by Hackbusch and Bebendorf \cite{bebendorf2003existence}, serving as a key theoretical instrument for demonstrating the approximability of FEM inverse matrices for second-order elliptic PDEs by $\mathcal{H}$-matrices. Their work showed that for Poisson-type equations with $L^{\infty}$-coefficients, the Green's functions have an $O(|\log \epsilon|^{d+1})$-term separable approximation (in the $L^{2}$-norm) on well-separated domains with error $\epsilon>0$. The key to the proof is the Caccioppoli inequality for generalized harmonic functions. This result was later extended to general second-order elliptic PDEs within the same framework \cite{bebendorf2005efficient}. The approximate separability established therein inevitably depends on the lower-order terms of the problems. Indeed, this approximate separability is closely linked to a low-dimensional property of the corresponding solution space, often offering an equivalent yet more straightforward approach to characterize the low-rank property. This insight was underscored in \cite{borm2010approximation}. Through arguments analogous to those in \cite{bebendorf2003existence,bebendorf2005efficient}, it was demonstrated that the solution space of a second-order elliptic PDE can be approximated by finite-dimensional spaces in the energy norm on a subset isolated from the support of the right-hand side with exponential convergence. While fundamentally it is sufficient to establish a low-rank property for a differential operator in the continuous setting, a corresponding property within a discrete setting is often more illustrative and useful. In the context of $\mathcal{H}$-matrix techniques, the discrete counterpart of the result of \cite{borm2010approximation} in the FE setting was established in \cite{faustmann2015h}, which was recently extended to Maxwell's equations with constant coefficients and low wavenumbers \cite{faustmann2022h}. Moreover, related low-rank approximation results at the matrix level for elliptic PDEs with finite difference discretizations were discussed in \cite{chandrasekaran2010numerical}. For low-rank approximations of parametric elliptic PDEs, we refer to \cite{bachmayr2017kolmogorov,bachmayr2018parametric}.

Most results on the low-rank property of elliptic PDEs focus on demonstrating the existence of low-rank approximations, with scarce investigation into their dependence on important parameters in practical applications. A notable exception is found in the work of \cite{engquist2018approximate}. Therein, the authors aimed at the wavenumber dependence of the approximate separability for the Helmholtz equation in the high-frequency limit, and derived sharp lower and upper bound estimates by means of an explicit expression of the Green's function. Finally, we note that on top of serving as a fundamental component in the development of fast, traditional algorithms, the low-rank property for elliptic PDEs has also been leveraged as a mathematical cornerstone for underpinning recent machine learning approaches based on learning Green's functions \cite{boulle2022learning,boulle2023elliptic,boulle2023learning}.

\subsection{Contributions}
In this paper, motivated by the recent works \cite{ma2021novel,ma2022exponential,chupeng2023wavenumber}, we establish a unified framework for the design, implementation, and analysis of MS-GFEM. This framework is developed in an abstract Hilbert space setting and utilizes appropriate local spaces and associated extension and restriction operators. Correspondingly, the framework has three main components: (i) construction of optimal local approximation spaces via local spectral problems involving an abstract partition of unity, (ii) efficient numerical algorithms for solving local eigenproblems based on a mixed formulation, and (iii) an abstract convergence analysis, including local approximation error estimates and a quasi-optimal global convergence analysis. The core of the framework lies in local approximation error estimates. We identify two fundamental conditions -- a Caccioppoli-type inequality and a weak approximation property -- that guarantee the exponential decay of the local error with respect to the dimension of the local subspace. Notably, the derived error bounds are explicit in important problem parameters like the wavenumber for wave problems or Péclet number for convection-diffusion. Under a G$\mathring{\rm a}$rding-type inequality and certain resolution conditions on the subdomain sizes and the number of local degrees of freedom (for indefinite problems), we prove quasi-optimal global convergence. We demonstrate the generality of the framework by applying it to various multiscale PDEs with $L^\infty$-coefficients in continuous and finite element settings. These include positive definite (Poisson, linear elasticity, $H({\rm curl})$ elliptic), highly indefinite (convection-diffusion, Helmholtz, Maxwell and elastic wave equations with large wavenumbers and impedance boundary conditions), as well as higher-order problems. Significantly, we prove the local errors for MS-GFEM decay as $O(e^{-cn^{1/d}})$, improving upon previous results \cite{babuska2011optimal,babuvska2014machine,ma2021novel,ma2022exponential,chupeng2023wavenumber} that showed a decay of $O(e^{-cn^{1/(d+1)}})$. Moreover, our result also improves corresponding convergence rates of other multiscale methods \cite{chen2021exponential,chen2023exponentially,chen2023exponentially-2} that use similar approximation techniques. The key to fitting practical applications into our framework lies in verifying the two fundamental conditions. While it is relatively standard for elliptic type problems, the verification of these two conditions, especially the weak approximation property, is highly nontrivial for Maxwell type problems. Indeed, for Maxwell's equations with rough coefficients and particularly large wavenumbers, the design and analysis of numerical methods remain tremendously challenging and existing studies are very limited and mainly focused on the positive definite case \cite{chung2019adaptive,gallistl2018numerical,henning2020computational}. Thus, for this particular application our work can be considered as pioneering.

Building upon the local convergence theory of the abstract MS-GFEM, we also establish a unified framework for deriving low-rank approximations to multiscale PDEs. Specifically, under the aforementioned fundamental conditions, we prove that the Green's function of the underlying problem admits an $O(|\log\epsilon|^{\beta})$-term separable approximation on well-separated domains with error $\epsilon>0$. An equivalent low-rank property of the associated solution space is also demonstrated without presuming the existence of the Green's function. Given the attained local convergence rates for MS-GFEM, we show $\beta=d$ for the practical applications mentioned earlier. Hence, our findings improve and generalize previous results \cite{bebendorf2003existence,bebendorf2005efficient,borm2010approximation} on the low-rank properties of elliptic PDEs. For highly indefinite problems, the approximate separability of Green's functions typically depends on parameters like wavenumber and Péclet number. However, our estimate reveals special configurations where a high separability holds regardless of these parameters. Such configurations were observed and leveraged to devise fast algorithms for high-frequency Helmholtz equations \cite{engquist2018approximate,candes2009fast,luo2014fast}. Our analysis unveils that these special configurations commonly exist for indefinite multiscale problems. The low-rank trait plays an integral role in constructing efficient numerical algorithms. By virtue of its broad applicability, our unified framework can provide a rigorous theoretical basis for existing algorithms and catalyze new developments. Moreover, the enhanced theoretical results are vital for precisely assessing the computational complexity of these algorithms.

\subsection{Outline}
The rest of this paper is organized as follows. In \cref{sec:2}, we define the abstract variational problem considered in this paper, and formulate the MS-GFEM in the abstract setting, with particular focus on the construction of optimal local approximation spaces and on efficient algorithms for solving the local eigenproblems. In \cref{sec:3}, we carry out a complete convergence analysis of the abstract MS-GFEM under certain hypotheses, including the local approximation error estimates and the quasi-optimal global convergence analysis. In \cref{sec:low-rank approximation}, based on the established local convergence theory of MS-GFEM, we prove estimates for the approximate separability of the Green's function of the underlying problem, and for an equivalent low-rank property of the corresponding solution space. In \cref{sec:verification_of_two_conditions}, we provide detailed guidelines on the verification of the two fundamental conditions, i.e., the Caccioppoli-type inequality and the weak approximation property. \Cref{sec:5} is devoted to applying the abstract framework to a broad selection of the most prominent multiscale PDEs appearing in practice. Central to this section is the verification of the two fundamental conditions in various settings. Finally, a brief conclusion along with important directions for future research is presented in \cref{sec:conclusion}.


\section{Abstract MS-GFEM}\label{sec:2}
\subsection{Problem formulation and abstract GFEM}\label{sec:2-1}
Let $\Omega\subset \mathbb{R}^{d}$ $(d=2,3)$ be a bounded Lipschitz domain, and let $\mathcal{H}(\Omega)\subset\mathcal{L}(\Omega)\subset \mathcal{H}(\Omega)^{\prime}$ be Hilbert spaces of functions defined on $\Omega$ with continuous embeddings. Here $\mathcal{H}(\Omega)^{\prime}$ denotes the dual space of $\mathcal{H}(\Omega)$ with respect to the pivot space $\mathcal{L}(\Omega)$, i.e., 
\begin{equation}
\Vert u\Vert_{\mathcal{H}^{\prime}(\Omega)} = \sup_{v\in \mathcal{H}(\Omega)} \frac{(u,v)_{\mathcal{L}(\Omega)}}{\Vert v\Vert_{\mathcal{H}(\Omega)}}.    
\end{equation}
Let $B(\cdot,\cdot):\mathcal{H}(\Omega)\times \mathcal{H}(\Omega)\rightarrow \mathbb{C}$ be a sesquilinear form that satisfies
\begin{assumption}\label{ass:2-1-0}
\hspace{2ex}
\begin{itemize}
\item[(i)] (Boundedness) There exists a constant $C_{b}>0$ such that
\begin{equation}\label{eq:2-1}
|B(u,v)|\leq C_{b}\Vert u\Vert_{\mathcal{H}(\Omega)} \Vert v\Vert_{\mathcal{H}(\Omega)}\quad \forall u,v\in \mathcal{H}(\Omega).
\end{equation}
\item[(ii)] (G$\mathring{a}$rding-type inequality) There exist constants $C_{1}>0$ and $C_{0}\geq 0$ such that
\begin{equation}\label{Garding-inequality}
C_{1}\Vert u\Vert_{\mathcal{H}(\Omega)}^{2} - C_{0} \Vert u\Vert_{\mathcal{L}(\Omega)}^{2} \leq |B(u,u)|\quad \forall u\in \mathcal{H}(\Omega).
\end{equation}
\end{itemize}
\end{assumption}
The G$\mathring{\rm a}$rding-type inequality will only be used for proving the quasi-optimal convergence of MS-GFEM. We consider the following abstract variational problem: Given $F\in \mathcal{H}(\Omega)^{\prime}$, find $u^{e}\in \mathcal{H}(\Omega)$ such that
\begin{equation}\label{eq:2-3}
B(u^{e},v) = F(v)\quad \forall v\in \mathcal{H}(\Omega).
\end{equation}
For later use, we introduce the adjoint of problem \cref{eq:2-3}: Find $w^{e}\in \mathcal{H}(\Omega)$ for given $F\in \mathcal{H}(\Omega)^{\prime}$ such that
\begin{equation}\label{eq:2-3-0}
B(v,w^{e}) = F(v)\quad \forall v\in \mathcal{H}(\Omega).
\end{equation}
We make the following assumption concerning the well-posedness of problem \cref{eq:2-3} and its adjoint \cref{eq:2-3-0}:
\begin{assumption}\label{ass:2-0}
For any $F\in \mathcal{H}(\Omega)^{\prime}$, the problem \cref{eq:2-3} and the adjoint problem \cref{eq:2-3-0} are uniquely solvable in $\mathcal{H}(\Omega)$. In addition, there exist constants $C_{\rm stab}>0$ and $C_{\rm stab}^{a}>0$ such that for any $F\in \mathcal{L}(\Omega)$,
\begin{equation}\label{stablity-estimate}
\begin{array}{lll}
{\displaystyle \Vert u^{e} \Vert_{\mathcal{H}(\Omega)} \leq C_{\rm stab}\Vert F\Vert_{\mathcal{L}(\Omega)},\quad \Vert w^{e} \Vert_{\mathcal{H}(\Omega)} \leq C^{a}_{\rm stab}\Vert F\Vert_{\mathcal{L}(\Omega)}.}
\end{array}
\end{equation}
\end{assumption}

If the form $B(\cdot,\cdot)$ is coercive on $\mathcal{H}(\Omega)$, i.e., $C_{0}= 0$ in \cref{Garding-inequality}, \cref{ass:2-0} is satisfied by the Lax--Milgram theorem. If, in addition, $B(\cdot,\cdot)$ is symmetric (Hermitian), it is often convenient to use the following energy norm for the space $\mathcal{H}(\Omega)$:
\begin{equation}\label{eq:elliptic-norm}
\Vert u\Vert_{\mathcal{H}(\Omega)}:=\sqrt{B(u,u)}\qquad \forall u\in \mathcal{H}(\Omega). 
\end{equation}
In this case, much of the analysis in the paper can be simplified. We refer to this important case as the \textit{elliptic} case in the sequel.

In the following, we will set up an abstract framework for the GFEM for solving problem \cref{eq:2-3}. To do this, we need suitable definitions of local spaces, i.e., spaces of functions defined on a subdomain of $\Omega$, and of associated operators and sesquilinear forms, which are formulated in terms of the following assumption.
\begin{assumption}\label{ass:2-1-1}
There exist Hilbert spaces $\big\{\mathcal{H}(D),\;\mathcal{H}_{0}(D): D\subset \Omega\big\}$ such that
\vspace{0.5ex}
\begin{itemize}
\item[(i)] $(\mathrm{Continuous \;\,embeddings}).$ For any subdomain $D\subset \Omega$, $\mathcal{H}_{0}(D)\subset \mathcal{H}(D)$, and $\mathcal{H}_{0}(D) = \mathcal{H}(D)$ if $D=\Omega$.

\vspace{0.5ex}
\item[(ii)] $(\mathrm{Zero\;\, extension}).$ For any subdomains $D\subset D^{\ast}$, there exists a linear operator $E_{D,D^{\ast}}:\mathcal{H}_{0}(D)\rightarrow \mathcal{H}_{0}(D^{\ast})$ such that $\big\Vert E_{D,D^{\ast}}(u) \big\Vert_{\mathcal{H}_{0}(D^{\ast})} =  \Vert u\Vert_{\mathcal{H}_{0}(D)}$ for all $u\in  \mathcal{H}_{0}(D)$. In particular, $\Vert u\Vert_{\mathcal{H}_{0}(D)} = \big\Vert E_{D,\Omega}(u) \big\Vert_{\mathcal{H}(\Omega)}$.

\vspace{0.5ex}
\item[(iii)] $(\mathrm{Restriction}).$ For any subdomains $D\subset D^{\ast}$, there exists a linear operator $R_{D^{\ast},D}:\mathcal{H}(D^{\ast})\rightarrow \mathcal{H}(D)$ that satisfies $\big\Vert R_{D^{\ast},D}(u) \big\Vert_{\mathcal{H}(D)}\leq \Vert u\Vert_{\mathcal{H}(D^{\ast})}$ for all $u\in \mathcal{H}(D^{\ast})$, $R_{D^{\ast},D}\circ E_{D,D^{\ast}}=id$, and $R_{D^{\ast},D}\circ R_{D^{\ast\ast},D^{\ast}} = R_{D^{\ast\ast},D}$ for any $D\subset D^{\ast}\subset D^{\ast\ast}$. 

If no ambiguity arises, we simply write $R_{D}$ and denote $R_{D}(u)$ by $u|_{D}$.

\vspace{0.5ex}
\item[(iv)] $(\mathrm{Local\;\, sesquilinear \;\,forms}).$ For any subdomain $D$, there is a bounded sesqui-\\linear form $B_{D}(\cdot,\cdot)$ on $\mathcal{H}(D)\times  \mathcal{H}(D)$ with $B_{\Omega}(\cdot,\cdot) = B(\cdot,\cdot)$. Moreover, for any $D\subset D^{\ast}$, $u\in \mathcal{H}(D^{\ast})$, and $v\in \mathcal{H}_{0}(D)$, it holds $B_{D}(u|_{D},\,v) = B_{D^{\ast}}\big(u, \,E_{D,D^{\ast}}(v)\big)$.
\end{itemize}
\end{assumption}
\begin{rem}
\Cref{ass:2-1-1} is often trivially satisfied. In most applications, $\mathcal{H}(D)=\{v|_{D}:v\in \mathcal{H}(\Omega)\}$, $\mathcal{H}_{0}(D)$ consists of functions in $\mathcal{H}(D)$ that vanish on $\partial D\cap \Omega$ in the sense of trace, and $B_{D}(\cdot,\cdot)$ is defined by restricting the integral in $B(\cdot,\cdot)$ to $D$.
\end{rem}

For ease of notation, we identify functions in $\mathcal{H}_{0}(D)$ with their zero extensions in $\mathcal{H}(\Omega)$ and drop the notation $E_{D,\Omega}$. Note that in the elliptic case, the form $B_{D}(\cdot,\cdot)$ defined above is symmetric and positive definite on $\mathcal{H}_{0}(D)$.



\textbf{Abstract GFEM}. Let $\{ \omega_{i} \}_{i=1}^{M}$ be a collection of overlapping subdomains of $\Omega$ such that $\cup_{i=1}^{M} \omega_{i} = \Omega$. A key ingredient of the classical GFEM is a \textit{partition of unity} subordinate to the open cover. To formulate the method in an abstract setting, we need an abstract partition of unity consisting of suitably defined operators. 
\begin{definition}\label{def:2-1-1}
Let $P_{i}: \mathcal{H}(\omega_{i})\rightarrow \mathcal{H}_{0}(\omega_{i})$, $i=1,\ldots,M$, be a set of bounded linear operators such that 
\begin{equation}
u=\sum_{i=1}^{M} P_{i}(u|_{\omega_{i}}) \qquad \text{for all}\;\;u\in \mathcal{H}(\Omega).    
\end{equation}
Then $\{P_{i}\}_{i=1}^{M}$ is called an abstract partition of unity. 
\end{definition}
\begin{rem}\label{rem:POU}
The construction of the abstract partition of unity is often standard. In a continuous setting, the operators $P_{i}$ are usually defined by $P_{i}u=\chi_{i}u$, where $\{\chi_{i}\}_{i=1}^{M}$ is a classical partition of unity subordinate to the open cover $\{\omega_i\}_{i=1}^{M}$. In a finite element setting, $P_{i}$ can be defined by $P_{i}u=I_{h}(\chi_{i}u)$, where $I_{h}$ denotes the associated finite element interpolant.
\end{rem}

The basic idea of the GFEM is to build the ansatz space from carefully designed local spaces with good approximation properties. Assume that on each subdomain $\omega_{i}$, we are given a \textit{local particular function} $u_{i}^{p}\in \mathcal{H}(\omega_{i})$ and a \textit{local approximation space} $S_{n_i}(\omega_{i}) \subset \mathcal{H}(\omega_{i})$ of dimension $n_{i}$. As in the classical GFEM, we construct the global particular function and the global approximation space by gluing the local components together with the (abstract) partition of unity: 
\begin{equation}\label{eq:2-9}
u^{p}=\sum_{i=1}^{M}P_{i}u_{i}^{p},\quad S_{n}(\Omega) =\Big\{\sum_{i=1}^{M}P_{i}\phi_{i}\,:\, \phi_{i}\in S_{n_{i}}(\omega_{i})\Big\}.
\end{equation}
By \cref{def:2-1-1}, we see that $u^{p}\in \mathcal{H}(\Omega)$ and $S_{n}(\Omega)\subset \mathcal{H}(\Omega)$. 
The abstract GFEM approximation of problem \cref{eq:2-3} is then defined by:
\begin{equation}\label{eq:2-10}
{\rm Find}\;\;u^{G}\in u^{p}+  S_{n}(\Omega)\quad \;{\rm such \;\;that}\;\quad B(u^{G},v) = F(v)\quad \forall v\in S_{n}(\Omega).
\end{equation}

Before giving the approximation theorem for the abstract GFEM, let us define a coloring constant associated with the open cover $\{\omega_{i}\}_{i=1}^{M}$, which in general is determined by the intersections between the subdomains. 
\begin{definition}\label{def:2-0-1}
Let $\zeta$ be an integer such that the set of the subdomains $\{\omega_i\}_{i=1}^{M}$ can be partitioned into $\zeta$ classes $\{\mathcal{C}_{j}, \;1\leq j\leq \zeta\}$ satisfying that 
\begin{equation*}
\{\omega_{j_{1}},\,\omega_{j_{2}}\}\subset \mathcal{C}_{j}\;\; \text{for some}\;j\; \Longleftrightarrow\; (u_{1},u_{2})_{\mathcal{H}(\Omega)} = 0\;\;\; \forall u_{1}\in \mathcal{H}_{0}(\omega_{j_{1}}),\;u_{2}\in \mathcal{H}_{0}(\omega_{j_{2}}).
\end{equation*}
Then $\zeta$ is called the coloring constant associated with the open cover $\{\omega_{i}\}_{i=1}^{M}$.
\end{definition}

The following theorem shows that the global approximation error of the abstract GFEM is fully determined by local approximation errors, thereby providing the mathematical foundation for the method. It is important to note that in contrast to the classical GFEM, the abstract method is based on locally approximating $P_{i}u^{e}$ instead of $u^{e}|_{\omega_i}$, thus incorporating the information of the partition of unity into the local approximation spaces. This seemingly minor difference is significant to the method in terms of theoretical analysis and practical algorithms. Most notably, the resulting global error estimate does not involve the partition of unity, which, in practice, greatly facilitates the design of adaptive algorithms; see \cref{adaptive-remarks}.
\begin{theorem}\label{thm:2-1}
Let $u\in \mathcal{H}(\Omega)$. For each $i=1,\ldots,M$, assume that 
\begin{equation}\label{eq:2-11}
\inf_{v_{i}\in u_{i}^{p}+ S_{n_{i}}(\omega_{i})} \big\Vert P_{i}(u|_{\omega_i} - v_{i})\big\Vert_{\mathcal{H}_{0}(\omega_i)} \leq \varepsilon_{i}.
\end{equation}
Then, with the coloring constant $\zeta$ defined in \cref{def:2-0-1}, 
\begin{equation}\label{eq:2-12}
\inf_{v\in u^{p}+ S_{n}(\Omega)}\Vert u - v\Vert_{\mathcal{H}(\Omega)} \leq \Big(\zeta \sum_{i=1}^{M}\varepsilon_{i}^{2}\Big)^{1/2}.
\end{equation}
\end{theorem}
\begin{proof}
By \cref{eq:2-11}, there exist $\varphi_{i}\in S_{n_{i}}(\omega_{i})$, $i=1,\ldots,M$, such that
\begin{equation}\label{eq:2-13}
\big\Vert P_{i}(u|_{\omega_i} - u_{i}^{p}-\varphi_{i})\big\Vert_{\mathcal{H}_{0}(\omega_i)} \leq \varepsilon_{i}.
\end{equation}
Let $\varphi = \sum_{i=1}^{M}P_{i}(\varphi_{i})\in S_{n}(\Omega)$. It follows from the Cauchy--Schwarz inequality, \cref{def:2-0-1}, and \cref{ass:2-1-1} ($ii$) that
\begin{equation}\label{eq:2-14}
\begin{array}{lll}
{\displaystyle \big\Vert u - u^{p}-\varphi\big\Vert_{\mathcal{H}(\Omega)}^{2} = \Big \Vert\sum_{i=1}^{M}P_{i}\big(u|_{\omega_i} - u_{i}^{p}-\varphi_{i}\big)\Big\Vert_{\mathcal{H}(\Omega)}^{2} }\\[3mm]
{\displaystyle \leq \zeta\sum_{i=1}^{M}\big \Vert P_{i}\big(u|_{\omega_i} - u_{i}^{p}-\varphi_{i}\big)\big\Vert_{\mathcal{H}(\Omega)}^{2} = \zeta \sum_{i=1}^{M}\big \Vert P_{i}\big(u|_{\omega_i} - u_{i}^{p}-\varphi_{i}\big)\big\Vert_{\mathcal{H}_{0}(\omega_{i})}^{2}.}
\end{array}
\end{equation}
Inserting \cref{eq:2-13} into \cref{eq:2-14} completes the proof of this theorem.
\end{proof}


\Cref{thm:2-1} clearly shows that the accuracy of the abstract GFEM depends essentially on choosing the local particular functions and local approximation spaces which are able to approximate the solution well. In the following, we will detail the construction of local approximations with exponential convergence rates. 

\subsection{Optimal local approximation spaces}\label{sec:2-2}
In this subsection, we will construct the local particular functions, and particularly, the local approximation spaces which are optimal in an appropriate sense. Exponential convergence rates for such local approximations will be established in \cref{sec:3} under certain conditions.


\textbf{Basic idea}. To construct the local approximation on a subdomain $\omega_i$, we split the solution $u^{e}$ on a slightly larger domain $\omega_{i}^{\ast}$ into two parts. The first part satisfies a local variational problem with the same right-hand side as problem \cref{eq:2-3}, and its restriction to the subdomain $\omega_i$ is defined as the local particular function. The second part belongs to a so-called \textit{generalized harmonic space} (defined below) that has a low-dimensional approximation in $\omega_i$. The local approximation space, designed to approximate the second part, is spanned by the singular vectors of a compact operator defined on the generalized harmonic space involving the partition of unity.

\begin{figure}\label{fig:2-1}
\centering
\includegraphics[scale=0.5]{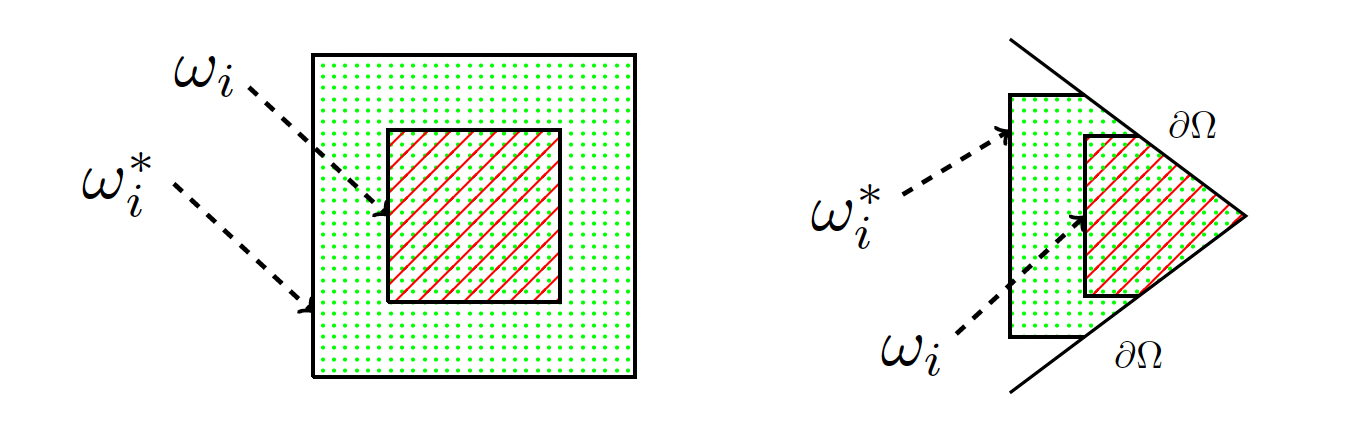}
\caption{Illustration of a subdomain $\omega_i$ that lies within the interior of $\Omega$ (left) and one that intersects the outer boundary $\partial \Omega$ (right) with associated oversampling domains $\omega_i^{\ast}$.}
\end{figure}

To start with, for each subdomain $\omega_i$, we introduce an \textit{oversampling domain} $ \omega_{i}^{\ast}$ (typically with a Lipschitz boundary) that satisfies $\omega_{i}\subset \omega_{i}^{\ast}\subset \Omega$ as illustrated in \cref{fig:2-1}. The relative distance between $\omega_i$ and $\omega_i^{\ast}$ is crucial to the convergence rate of the associated local approximation, as we will see in \cref{sec:3-1}. In order to define the local particular functions, we need to assume the existence of solutions to the following local variational problems posed on the oversampling domains:
\begin{assumption}\label{ass:2-2-1}
For each $i=1,\ldots,M$, there exists a $\psi_{i}\in \mathcal{H}(\omega^{\ast}_i)$ such that 
\begin{equation}\label{eq:2-15}
B_{\omega_{i}^{\ast}}(\psi_{i}, v) = F(v)\qquad \forall v\in \mathcal{H}_{0}(\omega_{i}^{\ast}).
\end{equation}
\end{assumption}
Now we can define the local particular functions as follows.
\begin{definition}[\textbf{Local particular functions}]\label{def:2-1}
The local particular function on $\omega_{i}$ is defined by $u_{i}^{p} = \psi_{i}|_{\omega_{i}}$, where $\psi_{i}\in \mathcal{H}(\omega^{\ast}_i)$ is a solution of \cref{eq:2-15}.
\end{definition}
\begin{rem}
Since we do not specify boundary conditions for $\psi_{i}$, in general, there exist many functions in $\mathcal{H}(\omega^{\ast}_i)$ that satisfy \cref{eq:2-15}. Hence, the local particular function on $\omega_{i}$ is not uniquely defined.  A proper selection of the local particular function depends on the specific problem under consideration. In the elliptic case, we can simply choose $\psi_{i}\in \mathcal{H}_{0}(\omega^{\ast}_i)$ and \cref{ass:2-2-1} is fulfilled. For time harmonic wave problems, homogeneous impedance boundary conditions are typically used for $\psi_{i}$.
\end{rem}

Next we introduce the generalized harmonic space for any subdomain $D\subset \Omega$:
\begin{equation}\label{generalized_harmonic_space}
\mathcal{H}_{B}(D) = \big\{u\in \mathcal{H}(D)\,:\, B_{D}(u,v) = 0\quad \forall v\in \mathcal{H}_{0}(D) \big\}.    
\end{equation}
The following results are a direct consequence of \cref{ass:2-1-1}.
\begin{proposition}\label{prop:2-1}
$\mathcal{H}_{B}(D)$ is a closed subspace of $\mathcal{H}(D)$. In addition, if $D\subset D^{\ast}$, then for any $u\in \mathcal{H}_{B}(D^{\ast})$, $u|_{D}\in \mathcal{H}_{B}(D)$.
\end{proposition}

In general, the local solution $\psi_{i}$ is not a good local approximation of the solution $u^{e}$, and hence it is necessary to consider the residual $u^{e}|_{\omega_i}-\psi_{i}|_{\omega_i}$, or equivalently, $(u^{e}|_{\omega^{\ast}_i}-\psi_{i})|_{\omega_i}$. 
Combining \cref{eq:2-3}, \cref{eq:2-15}, and \cref{ass:2-1-1} $(iv)$ leads to the key observation:
\begin{equation}\label{eq:2-16}
B_{\omega_{i}^{\ast}}(u^{e}|_{\omega_{i}^{\ast}}-\psi_{i}, v)=B(u^{e}, v) -B_{\omega_{i}^{\ast}}(\psi_{i}, v) = 0\quad \forall v\in \mathcal{H}_{0}(\omega_{i}^{\ast}),
\end{equation}
namely, $u^{e}|_{\omega_{i}^{\ast}}-\psi_{i}$ lies in the generalized harmonic space $\mathcal{H}_{B}(\omega_{i}^{\ast})$. 
It turns out that under certain conditions that are valid for many practical problems, $\mathcal{H}_{B}(\omega_{i}^{\ast})$ has a low-dimensional property, i.e., it can be approximated by finite-dimensional spaces in $\omega_i$ with exponential convergence. To identify such low-dimensional approximation spaces, it is important to specify the norm (equivalently, the inner product) for $\mathcal{H}_{B}(\omega_{i}^{\ast})$. By \cref{prop:2-1}, a natural but typically suboptimal choice is to use the inner product on $\mathcal{H}(\omega_i^{\ast})$. In practice, it is often desirable to use some "energy" inner products directly associated with the form $B_{\omega_{i}^{\ast}}(\cdot,\cdot)$. Unfortunately, these inner products, denoted by $B^{+}_{\omega_i^{\ast}}(\cdot,\cdot)$, often have a finite-dimensional kernel in $\mathcal{H}_{B}(\omega_i^{\ast})$. To fix this problem and for later use in \cref{sec:3}, we make the following assumption. 
\begin{assumption}\label{ass:2-2-3}
For any domain $\omega_i\subset D\subset \omega_i^{\ast}$, there exist a finite-dimensional space $\mathcal{K}_{D}\subset \mathcal{H}_{B}(D)$ and a semi-inner product $B^{+}_{D}(\cdot,\cdot)$ on $\mathcal{H}(D)$ such that

\vspace{1mm}
\begin{itemize}
\item[(i)] $B^{+}_{D}(v,w) = 0$ holds for any $w\in \mathcal{H}(D)$ if and only if $v\in \mathcal{K}_{D}$;

\vspace{1mm}
\item[(ii)] $\big(\mathcal{H}_{B}(D)/\mathcal{K}_{D},\,B^{+}_{D}(\cdot,\cdot)\big)$ is a Hilbert space when $D$ satisfies the cone condition.
\end{itemize}

\vspace{2mm}
\cm{When $D = \omega_i^{\ast}$, we denote the kernel by $\mathcal{K}_{i}$ and let $l_{i}:={\rm dim}(\mathcal{K}_{i})$. Moreover, we assume that $P_{i}|_{\mathcal{K}_{i}}$ is injective.}

\end{assumption}

\begin{rem}
For uniformly elliptic problems, we typically set $B^{+}_{\omega_i^{\ast}}=B_{\omega_i^{\ast}}$. For indefinite problems, these two forms, albeit closely related, are different. 
\end{rem}

\begin{rem}\label{kernel-remark}
Let $\omega_i^{\ast}$ lie in the interior of $\Omega$. For convection diffusion problems with $B^{+}_{\omega_i^{\ast}}(u,v):=\int_{\omega_i^{\ast}}A\nabla u\cdot \nabla v \,d{\bm x}$, $\mathcal{K}_{i}$ is the space of constant functions and $l_i=1$. For elasticity problems in $\mathbb{R}^{d}$ with $B^{+}_{\omega_i^{\ast}}({\bm u},{\bm v}) : = \int_{\omega_i^{\ast}}\mathcal{C} \varepsilon({\bm u}) : \varepsilon({\bm v}) \,d{\bm x}$, where $\varepsilon({\bm u}) = (\nabla {\bm u} + \nabla {\bm u}^{\mathsf{\top}})/2$, $\mathcal{K}_{i}$ is the space of rigid motions and $l_i=d(d+1)/2$. For fourth-order problems in $\mathbb{R}^{d}$ with $B^{+}_{\omega_i^{\ast}}(u,v):=\sum_{k,l}\int_{\omega_i^{\ast}}a_{kl}u_{x_{k}x_{l}}v_{x_{k}x_{l}}\,d{\bm x}$, $\mathcal{K}_i$ is the space of linear functions and $l_i=d+1$.    
\end{rem}

According to \cref{ass:2-2-3}, $B^{+}_{\omega_i^{\ast}}(\cdot,\cdot)$ induces a seminorm on $\mathcal{H}_{B}(\omega_i^{\ast})$, which is a norm when the kernel $\mathcal{K}_i=\emptyset$ (we do not exclude this possibility) or when $\mathcal{K}_{i}$ is removed. With a slight abuse of notation, we define
\begin{equation}\label{eq:2-24-3}
\Vert u\Vert_{B^{+},\omega_i^{\ast}}:=\big(B^{+}_{\omega_i^{\ast}}(u,u)\big)^{1/2}.
\end{equation}
Having identified the (semi-)norm for $\mathcal{H}_{B}(\omega_{i}^{\ast})$, we are now ready to find the low-dimensional approximation spaces. To this end, we first define the operator $\widehat{P}_{i}: \mathcal{H}_{B}(\omega_{i}^{\ast})\rightarrow \mathcal{H}_{0}(\omega_{i})$ by 
\begin{equation}\label{eq:2-18}
\widehat{P}_{i} u = P_{i}(u|_{\omega_i})\qquad \forall u\in \mathcal{H}_{B}(\omega_{i}^{\ast}),
\end{equation}
and introduce the following subspace of $\mathcal{H}_{B}(\omega_i^{\ast})$:
\begin{equation}\label{eq:2-24-1}
\mathcal{H}_{B}^{0}(\omega_i^{\ast}) = \big\{v\in \mathcal{H}_{B}(\omega_{i}^{\ast}): (\widehat{P}_{i}v, \widehat{P}_{i}w)_{\mathcal{H}_{0}(\omega_{i})} = 0\quad \forall w\in \mathcal{K}_i \big\}.
\end{equation}
By \cref{ass:2-2-3}, it is clear that $\Vert \cdot\Vert_{B^{+},\omega_i^{\ast}}$ is a norm on $\mathcal{H}^{0}_{B}(\omega_i^{\ast})$ whether $\mathcal{K}_{i}=\emptyset$ or not. Indeed, we have
\begin{lemma}\label{lem:2-2-3}
$\big(\mathcal{H}_{B}^{0}(\omega_i^{\ast}), B^{+}_{\omega_{i}^{\ast}}(\cdot,\cdot)\big)$ is a Hilbert space.
\end{lemma}


Next we consider the problem of optimal finite-dimensional approximations for the set $\widehat{P}_{i}(\mathcal{H}^{0}_{B}(\omega_{i}^{\ast}))\subset \mathcal{H}_{0}(\omega_{i})$. Given any $n$-dimensional subspace $Q_{i}(n)$ of $\mathcal{H}_{0}(\omega_{i})$, the accuracy of approximating $\widehat{P}_{i}(\mathcal{H}^{0}_{B}(\omega_{i}^{\ast}))$ by $Q_{i}(n)$ can be measured by 
\begin{equation}\label{eq:2-19}
d(Q_{i}(n),\omega_{i}) :=\sup_{u\in \mathcal{H}^{0}_{B}(\omega_{i}^{\ast})} \inf_{v\in Q_{i}(n)}\frac{\Vert \widehat{P}_{i}u-v\Vert_{\mathcal{H}_{0}(\omega_{i})}}{\Vert u \Vert_{B^{+},\omega_i^{\ast}}}.
\end{equation}
Here we note that by \cref{ass:2-1-1} ($ii$), $\Vert u\Vert_{\mathcal{H}_{0}(\omega_i)} := \Vert u\Vert_{\mathcal{H}(\Omega)}$ for all $u\in \mathcal{H}_{0}(\omega_i)$. We are interested in finding, for each $n\in \mathbb{N}$, an $n$-dimensional space $Q^{\rm opt}_{i}(n)\subset \mathcal{H}_{0}(\omega_{i})$ with the best approximation accuracy, i.e., $d(Q^{\rm opt}_{i}(n),\omega_{i}) \leq d(Q_{i}(n),\omega_{i})$ for all other $n$-dimensional subspaces $Q_{i}(n)$ of $\mathcal{H}_{0}(\omega_{i})$. The problem of finding the optimal approximation spaces can be formulated by means of the Kolmogorov $n$-width \cite{pinkus1985n} as
\begin{equation}\label{eq:2-20}
\begin{array}{lll}
{\displaystyle d_{n}(\omega_{i},\omega_{i}^{\ast}) := \inf_{Q_{i}(n)\subset \mathcal{H}_{0}(\omega_{i})}d(Q_{i}(n),\omega_{i}) }\\[4mm]
{\displaystyle \qquad =\inf_{Q_{i}(n)\subset \mathcal{H}_{0}(\omega_{i})}\sup_{u\in \mathcal{H}^{0}_{B}(\omega_{i}^{\ast})} \inf_{v\in Q_{i}(n)}\frac {\Vert \widehat{P}_{i}u-v\Vert_{\mathcal{H}_{0}(\omega_{i})}}{\Vert u \Vert_{B^{+},\omega_{i}^{\ast}}}.}
\end{array}
\end{equation}
The quantity $d_{n}(\omega_i,\omega_i^{\ast})$ characterizes the best accuracy of approximation of the set $\widehat{P}_{i}(\mathcal{H}^{0}_{B}(\omega_{i}^{\ast}))$ by $n$-dimensional subspaces of $\mathcal{H}_{0}(\omega_{i})$. See \cref{Kolmogrov-nwidth} for a general definition of the Kolmogorov $n$-width of an operator and some useful properties. To ensure the existence of a low-dimensional approximation space, it is essential that $\lim_{n\rightarrow \infty}d_{n}(\omega_{i},\omega_{i}^{\ast}) = 0$. Clearly, this convergence implies the compactness of operator $\widehat{P}_{i}|_{\mathcal{H}^{0}_{B}(\omega_{i}^{\ast})}$, or equivalently, the compactness of operator $\widehat{P}_{i}$ as $\mathcal{K}_i$ is finite-dimensional. In view of this fact, the following assumption is crucial:
\begin{assumption}\label{ass:2-2-2}
For each $i=1,\ldots,M$, the operator $\widehat{P}_{i}$ is compact.
\end{assumption}
\begin{rem}
In a discrete setting, this assumption trivially holds true. In the continuous setting, it is a consequence of the Caccioppoli-type inequality and the weak approximation property in \cref{sec:3-1}; see \cref{prop:3-1-1}. In practice, nevertheless, it is often proved by using the Caccioppoli-type inequality and certain compactness results that are much easier to verify than the weak approximation property.
\end{rem}

The Kolmogrov $n$-width of a compact operator in Hilbert spaces has been well studied. By \cref{lem:1-1}, $d_{n}(\omega_{i},\omega_{i}^{\ast})$ can be characterized in terms of the singular values and (left) singular vectors of the operator $\widehat{P}_{i}|_{\mathcal{H}^{0}_{B}(\omega_{i}^{\ast})}$ as follows.
\begin{lemma}\label{lem:2-2-1}
For each $j\in\mathbb{N}$, let $(\lambda_{i,j},\,\phi_{i,j})\in\mathbb{R}\times \mathcal{H}^{0}_{B}(\omega_{i}^{\ast})$ be the $j$-th eigenpair of the problem:
\begin{equation}\label{eq:2-21}
\big(\widehat{P}_{i}\phi_{i},\,  \widehat{P}_{i} v\big)_{\mathcal{H}_{0}(\omega_{i})}= \lambda_{i}\,B^{+}_{\omega_i^{\ast}}(\phi_{i},  v)\quad \forall v\in \mathcal{H}^{0}_{B}(\omega_{i}^{\ast}).
\end{equation}
Then, $d_{n}(\omega_{i},\omega_{i}^{\ast}) =\lambda^{1/2}_{i,n+1}$, and the optimal approximation space is given by 
\begin{equation}\label{eq:2-22}
Q^{\rm opt}_{i}(n) = {\rm span}\big\{ \widehat{P}_{i}\phi_{i,1}, \ldots, \widehat{P}_{i}\phi_{i,n}\big\}.
\end{equation}
\end{lemma}

Before proceeding, we show a close relation between the local eigenproblem \cref{eq:2-21} and the one posed on $\mathcal{H}_{B}(\omega_{i}^{\ast})$: Find $\lambda_{i}\in \mathbb{R}\cup \{+\infty\}$ and $\phi_{i}\in \mathcal{H}_{B}(\omega_{i}^{\ast})$ such that
\begin{equation}\label{eq:2-24-5}
\big(\widehat{P}_{i}\phi_{i},\,  \widehat{P}_{i} v\big)_{\mathcal{H}_{0}(\omega_i)} = \lambda_{i}\,B^{+}_{\omega_{i}^{\ast}}(\phi_{i},  v)\quad \forall v\in \mathcal{H}_{B}(\omega_{i}^{\ast}).
\end{equation}

\begin{lemma}\label{lem:2-2-2}
Let $\mathcal{K}_i$ and $l_i$ be as in \cref{ass:2-2-3}, and let $(\lambda_{i,j},\,\phi_{i,j})$ denote the $j$-th eigenpair (arranged in non-ascending order) of problem \cref{eq:2-24-5}. Then, $\lambda_{i,1}=\cdots=\lambda_{i,l_i}=+\infty$, $\mathcal{K}_{i} = {\rm span}\big\{{\phi_{i,1}},\ldots,{\phi_{i,l_i}}\big\}$, and for each $j>l_i$, $(\lambda_{i,j},\,\phi_{i,j})$ is the $(j-l_i)$-th eigenpair of problem \cref{eq:2-21}.
\end{lemma}
\begin{proof}
By \cref{ass:2-2-3} and the definition of $\mathcal{H}^{0}_{B}(\omega_{i}^{\ast})$, we have the orthogonal decomposition $\mathcal{H}_{B}(\omega_{i}^{\ast}) = \mathcal{H}^{0}_{B}(\omega_{i}^{\ast}) \oplus \mathcal{K}_i$ with respect to the scalar products $B^{+}_{\omega_{i}^{\ast}}(\cdot,\cdot)$ and $(\widehat{P}_{i}\cdot,\,  \widehat{P}_{i} \cdot)_{\mathcal{H}_{0}(\omega_{i})}$. Thus, the local eigenproblem \cref{eq:2-24-5} can be split into two subproblems: one on $\mathcal{K}_i$ with eigenvalues $+\infty$ and eigenvectors spanning $\mathcal{K}_i$, and another on $\mathcal{H}^{0}_{B}(\omega_{i}^{\ast})$ with finite eigenvalues, i.e., problem \cref{eq:2-21}.
\end{proof}

Now we are ready to define the desired optimal local approximation spaces. Thanks to \cref{lem:2-2-2}, these spaces can be defined in a unified way based on local eigenproblems \cref{eq:2-24-5} for two cases: $\mathcal{K}_{i}=\emptyset$ or $\mathcal{K}_{i}\neq\emptyset$.
\begin{definition}[\textbf{Local approximation spaces}]\label{def:2-2}
Let $\mathcal{K}_i$ and $l_i$ be as in \cref{ass:2-2-3}, and let $n_{i}>l_i$. The local approximation space on $\omega_{i}$ is defined as
\begin{equation}\label{eq:2-23}
S_{n_{i}}(\omega_i) =  {\rm span}\big\{{\phi_{i,1}}|_{\omega_i},\ldots,{\phi_{i,n_{i}}}|_{\omega_i}\big\},
\end{equation}
where $\phi_{i,j}$ denotes the $j$-th eigenvector of problem \cref{eq:2-24-5}.
\end{definition}
\begin{rem}
By \cref{lem:2-2-2}, it holds that $S_{n_{i}}(\omega_i)= \mathcal{K}_i\oplus S^{0}_{n_{i}-l_i}(\omega_i)$, where $S^{0}_{n_{i}-l_i}(\omega_i)$ is spanned by the restrictions (to $\omega_i$) of the first $n_i-l_i$ eigenvectors of problem \cref{eq:2-21}. 
\end{rem}

With the local approximation constructed above, we have the following local error estimate for the abtract GFEM.
\begin{theorem}[\textbf{Local approximation error estimates}]\label{thm:local_approximation_error}
Let the local particular functions and the local approximation spaces be as in \cref{def:2-1,def:2-2}, and let $u^{e}$ be the solution of problem \cref{eq:2-3}. Then, for each $i=1,\ldots,M$,
\begin{equation}\label{eq:2-24}
\inf_{\varphi\in u_{i}^{p}  + S_{n_i}(\omega_i)}\big\Vert {P}_{i}(u^{e}|_{\omega_i}-\varphi)\big\Vert_{\mathcal{H}_{0}(\omega_{i})}\leq d_{n_i-l_i}(\omega_{i},\omega_{i}^{\ast}) \,\big\Vert u^{e}|_{\omega_{i}^{\ast}} - \psi_{i}\big\Vert_{B^{+},\omega^{\ast}_i},
\end{equation}
where $l_i\in\mathbb{N}$ is the dimension of the space $\mathcal{K}_i$ given in \cref{ass:2-2-3}.
\end{theorem}
\begin{proof}
Let $S^{0}_{n_{i}-l_i}(\omega_i)$ denote the space spanned by the restrictions (to $\omega_i$) of the first $n_i-l_i$ eigenvectors of problem \cref{eq:2-21}.
Recalling that $u^{e}|_{\omega_{i}^{\ast}} - \psi_{i}\in \mathcal{H}_{B}(\omega_{i}^{\ast})$, we can find a $c_{i}\in \mathcal{K}_i$ such that $u^{e}|_{\omega_{i}^{\ast}} - \psi_{i}-c_{i}\in \mathcal{H}^{0}_{B}(\omega_{i}^{\ast})$. Then, by the definition of the $n$-width $d_{n}(\omega_{i},\omega_{i}^{\ast})$ and \cref{lem:2-2-1}, we see that
\begin{equation}\label{eq:2-24-0}
\inf_{\varphi\in S^{0}_{n_i-l_i}(\omega_i)}\big\Vert {P}_{i}(u^{e}|_{\omega_i}-\psi_{i}-c_{i}-\varphi)\big\Vert_{\mathcal{H}_{0}(\omega_{i})}\leq d_{n_i-l_i}(\omega_{i},\omega_{i}^{\ast}) \,\big\Vert u^{e}|_{\omega_{i}^{\ast}} - \psi_{i}\big\Vert_{B^{+},\omega^{\ast}_i},
\end{equation}
where we have used the fact that $B^{+}_{\omega_i^{\ast}}(v,c_i) = 0$ for any $v\in \mathcal{H}(\omega_i^{\ast})$. Then, estimate \cref{eq:2-24} follows from \cref{eq:2-24-0} and the decomposition $S_{n_{i}}(\omega_i)= \mathcal{K}_{i}\oplus S^{0}_{n_{i}-l_i}(\omega_i)$.
\end{proof}

We conclude this subsection by briefly discussing the important elliptic case. In this case, by using the energy norm \cref{eq:elliptic-norm}, we have $\Vert u\Vert_{\mathcal{H}_{0}(\omega_i)} = \big(B_{\omega_{i}}(u,u)\big)^{1/2}$, and in general, $B^{+}_{\omega_{i}^{\ast}}(\cdot,\cdot) = B_{\omega_{i}^{\ast}}(\cdot,\cdot)$. Hence, the local eigenproblem \cref{eq:2-24-5} takes the form:
\begin{equation}\label{eq:2-24-2-0}
B_{\omega_i}\big(\widehat{P}_{i}\phi_{i},\,  \widehat{P}_{i} v\big)= \lambda_{i}B_{\omega_{i}^{\ast}}(\phi_{i},  v)\quad \forall v\in \mathcal{H}_{B}(\omega_{i}^{\ast}).
\end{equation}
Furthermore, let $\psi_{i}\in \mathcal{H}_{0}(\omega_i^{\ast})$. It follows that $B_{\omega_i^{\ast}}( u^{e}|_{\omega_i^{\ast}} - \psi_{i}, \psi_i) =0$, and thus $\big\Vert u^{e}|_{\omega_i^{\ast}} - \psi_{i}\big\Vert_{B,\omega^{\ast}_i} \leq    \big\Vert u^{e}|_{\omega_i^{\ast}}\big\Vert_{B,\omega^{\ast}_i}$. Substituting this inequality into \cref{eq:2-24} gives the following simplified local error estimate:
\begin{equation}\label{local_error_estimate_elliptic}
\inf_{\varphi\in u_{i}^{p}  + S_{n_i}(\omega_i)}\big\Vert {P}_{i}(u^{e}|_{\omega_i}-\varphi)\big\Vert_{B,\omega_{i}}\leq d_{n_i-l_i}(\omega_{i},\omega_{i}^{\ast}) \,\big\Vert u^{e}|_{\omega_{i}^{\ast}} \big\Vert_{B,\omega^{\ast}_i}.
\end{equation}

\subsection{Efficient solution of local eigenproblems}\label{sec:2-3}
In the continuous setting, the local eigenproblems \cref{eq:2-24-5} are posed on infinite-dimensional spaces, which need to be solved numerically by the FEM or other standard numerical methods in practical applications. The resulting discretized local eigenproblems are often similar in form to the local eigenproblems in a discrete setting. In what follows, we develop efficient solution techniques for both cases in a unified way. 

In a continuous setting, let $\mathcal{V}_{h,0}(\omega^{\ast})\subset \mathcal{V}_{h}(\omega^{\ast})$ denote one-parameter families of finite-dimensional subspaces of $\mathcal{H}_{0}(\omega^{\ast})$ and $\mathcal{H}(\omega^{\ast})$, respectively, where $h>0$ often denotes the mesh size of a discretization. Here we omit the subscript $i$ for simplicity. The discrete analogue of the generalized harmonic space $\mathcal{H}_{B}(\omega^{\ast})$ is defined by
\begin{equation}\label{eq:2-25}
\mathcal{V}_{h,B}(\omega^{\ast})=\big\{u_{h}\in \mathcal{V}_{h}(\omega^{\ast}):  B_{\omega^{\ast}}(u_{h},v_{h}) = 0\quad \forall v_{h}\in \mathcal{V}_{h,0}(\omega^{\ast}) \big\}. 
\end{equation}
The finite-dimensional approximation of the (continuous) local eigenproblem \cref{eq:2-24-5} reads: Find $\lambda_{h}\in \mathbb{R}\cup \{+\infty\}$ and $\phi_{h}\in \mathcal{V}_{h,B}(\omega^{\ast})$ with $\phi_{h}\neq 0$ such that
\begin{equation}\label{eq:2-26}
\big(\widehat{P}\phi_{h},\,  \widehat{P} v_{h}\big)_{\mathcal{H}_{0}(\omega)}= \lambda_{h}B^{+}_{\omega^{\ast}}\big(\phi_{h},  v_{h}\big)\quad \forall v_{h}\in \mathcal{V}_{h,B}(\omega^{\ast}),
\end{equation}
or equivalently, with $\mu_{h}=\lambda_{h}^{-1}$,
\begin{equation}\label{eq:2-27}
 B^{+}_{\omega^{\ast}}\big(\phi_{h},  v_{h}\big)=\mu_{h}\big(\widehat{P}\phi_{h},\,  \widehat{P} v_{h}\big)_{\mathcal{H}_{0}(\omega)}\quad \forall v_{h}\in \mathcal{V}_{h,B}(\omega^{\ast}).
\end{equation}
In a discrete setting, with $\mathcal{H}(\omega^{\ast}) = \mathcal{V}_{h}(\omega^{\ast})$ and $\mathcal{H}_{0}(\omega^{\ast}) = \mathcal{V}_{h,0}(\omega^{\ast})$, the local eigenproblems \cref{eq:2-24-5} can also be rewritten as \cref{eq:2-27}.

The presence of the constraint in $\mathcal{V}_{h,B}(\omega^{\ast})$ makes it challenging to solve the discretized local eigenproblem \cref{eq:2-26}. A straightforward, yet inefficient solution is to first compute a basis for $\mathcal{V}_{h,B}(\omega^{\ast})$, and then solve the resulting matrix eigenvalue problem. In general, it requires solving $m$ ($m\in\mathbb{N}$) local boundary value problems to get $m$ basis functions for $\mathcal{V}_{h,B}(\omega^{\ast})$. This step is thus rather expensive especially when the dimension of $\mathcal{V}_{h,B}(\omega^{\ast})$ is large. Several attempts \cite{babuvska2020multiscale,chen2020random,ma2021novel} have been made to circumvent this difficulty in the context of scalar elliptic problems, based on approximating the space $\mathcal{V}_{h,B}(\omega^{\ast})$ through some special techniques, such as random sampling, local Steklov eigenfunctions, and harmonic extensions of special boundary functions. A common drawback of these techniques is that an additional approximation error is introduced to the method that is usually difficult to quantify. Moreover, the approximations of $\mathcal{V}_{h,B}(\omega^{\ast})$ generally need to be enriched or reconstructed when implementing an adaptive enrichment of the local approximation spaces. 

Motivated by \cite{ma2022error} where problem \cref{eq:2-26} was studied in a scalar elliptic PDEs setting, we propose an efficient and accurate method for solving \cref{eq:2-26} in the abstract setting based on mixed formulation as follows.

\textbf{Mixed formulation}. Let us consider the following eigenvalue problem: Find $\mu_{h}\in \mathbb{R}\cup\{+\infty\}$, $\phi_{h}\in \mathcal{V}_{h}(\omega^{\ast})$, and $p_{h}\in \mathcal{V}_{h,0}(\omega^{\ast})$ with $\big(\phi_{h}, p_{h})\neq (0,0)$ such that
\begin{equation}\label{eq:2-28}
\begin{aligned}
B^{+}_{\omega^{\ast}}\big(\phi_{h},  v_{h}\big) + B_{\omega^{\ast}}(v_h,p_h) =&\, \mu_{h} \big(\widehat{P}\phi_{h},\,  \widehat{P} v_{h}\big)_{\mathcal{H}_{0}(\omega)}\quad\, \forall v_h\in \mathcal{V}_{h}(\omega^{\ast}),\\
B_{\omega^{\ast}}(\phi_{h},\xi_h) =&\;0\qquad\qquad\qquad\qquad \quad\;\,\forall \xi_{h}\in \mathcal{V}_{h,0}(\omega^{\ast}).
\end{aligned}
\end{equation}
Here $p_{h}$ is a Lagrange multiplier introduced to relax the constraint in the test space. It is easy to see that while the augmented problem \cref{eq:2-28} has more eigenvectors (corresponding to $\mu_{h} = +\infty$) than problem \cref{eq:2-27}, the eigenvectors corresponding to finite eigenvalues are the same. Therefore, problems \cref{eq:2-27,eq:2-28} are equivalent for building the local approximation spaces.

Now let us look at the discrete system \cref{eq:2-28} at the matrix level. Let $\big\{\varphi_{1},..,\varphi_{n_1}\big\}$ be a basis for $\mathcal{V}_{h,0}(\omega^{\ast})$ and $\big\{\varphi_{1},\ldots,\varphi_{n_1}\big\}\cup \big\{\varphi_{n_1+1},\ldots,\varphi_{n_1+n_2}\big\}$ a basis for $\mathcal{V}_{h}(\omega^{\ast})$. Then problem \cref{eq:2-28} can be formulated as the following matrix eigenvalue problem: Find $\mu_{h}\in\mathbb{R}\cup\{+\infty\}$, ${\bm \phi} \in\mathbb{C}^{n_{1}+n_{2}}$, and ${\bm p}\in \mathbb{C}^{n_1}$ such that
\begin{equation}\label{eq:2-28-0}
{\left( \begin{array}{ccc}
{\bf B}^{+} & {\bf B} \\
{\bf B}^{\ast} & {\bf 0} 
\end{array} 
\right )}
{\left(\begin{array}{c}
{\bm \phi} \\
{\bm p}
\end{array} 
\right )} = \mu_{h}
{\left( \begin{array}{ccc}
{\bf P} & {\bf 0} \\
{\bf 0} & {\bf 0} 
\end{array} 
\right )}
{\left( \begin{array}{c}
{\bm \phi} \\
{\bm p}
\end{array} 
\right )},
\end{equation}
where ${\bf B}^{\ast}$ denotes the conjugate-transpose of ${\bf B}$. Here the matrices ${\bf B}\in \mathbb{C}^{(n_1+n_2)\times n_{2}}$, and ${\bf B}^{+}, \,{\bf P}\in \mathbb{C}^{(n_{1}+n_{2})\times(n_{1}+n_{2})}$ are given by $B_{i,j} = B_{\omega^{\ast}}(\varphi_{i},\varphi_{j})$, $B^{+}_{i,j} = B^{+}_{\omega_i^{\ast}}(\varphi_{i},\varphi_{j})$, and $P_{i,j} = \big(\widehat{P}\varphi_{i},\widehat{P}\varphi_{j}\big)_{\mathcal{H}_{0}(\omega)}$, respectively. We have the following result concerning the invertibility of the matrix on the left-hand side of \cref{eq:2-28-0}.
\begin{lemma}\label{lem:2-3-1}
Suppose that the form $B^{+}_{\omega^{\ast}}(\cdot,\cdot)$ is coercive on $\mathcal{V}_{h,B}(\omega^{\ast})$, and that there exists a constant $\beta(h)$ such that for any $p_{h}\in \mathcal{V}_{h,0}(\omega^{\ast})$, the problem 
\begin{equation}
{\rm Find}\;\,\phi_{h}\in \mathcal{V}_{h}(\omega^{\ast}) \quad{\rm such \;\,that}\quad B_{\omega^{\ast}}(\phi_{h},v_{h}) = (p_{h}, v_{h})_{\mathcal{H}(\omega^{\ast})}\quad \forall v_{h}\in \mathcal{V}_{h}(\omega^{\ast})
\end{equation}
is uniquely solvable with the estimate $\Vert \phi_{h}\Vert_{\mathcal{H}(\omega^{\ast})}\leq \beta(h) \Vert p_{h}\Vert_{\mathcal{H}(\omega^{\ast})}$. Then the matrix on the left-hand side of \cref{eq:2-28-0} is invertible.
\end{lemma}
The second assumption in \cref{lem:2-3-1} implies the discrete inf-sup condition, which, combining with the coercivity condition, gives the desired result (see, e.g., \cite{boffi2013mixed}). The invertibility of the matrix ensures that problem \cref{eq:2-28-0} can be effectively solved by standard eigenvalue solvers. Compared with previous solution techniques, our method introduces no additional approximation errors and is typically easier to implement. Moreover, in the important elliptic case, we have a fast algorithm for solving problem \cref{eq:2-28-0} by exploiting its special structure.


\textbf{The elliptic case}. Following along the same lines as above, the finite-dimensional approximation of eigenproblem \cref{eq:2-24-2-0} can be written in matrix notation as follows: Find $\lambda_{h}^{-1}\in\mathbb{R}\cup\{+\infty\}$, ${\bm \phi} = ({\bm \phi}_{1}, {\bm \phi}_{2})\in\mathbb{C}^{n_{1}+n_{2}}$, and ${\bm p}\in \mathbb{C}^{n_1}$ such that
\begin{equation}\label{eq:2-31}
{\left( \begin{array}{ccc}
{\bf B}_{11} & {\bf B}_{12} & {\bf B}_{11}\\
{\bf B}_{21} & {\bf B}_{22} & {\bf B}_{21}\\
{\bf B}_{11} & {\bf B}_{12} & {\bf 0}
\end{array} 
\right )}
{\left( \begin{array}{c}
{\bm \phi}_{1} \\
{\bm \phi}_{2} \\
{\bm p}
\end{array} 
\right )} = \lambda_{h}^{-1}
{\left( \begin{array}{ccc}
{\bf P}_{11} & {\bf P}_{12} & {\bf 0}\\
{\bf P}_{21} & {\bf P}_{22} & {\bf 0}\\
{\bf 0} & {\bf 0} & {\bf 0}
\end{array} 
\right )}
{\left( \begin{array}{c}
{\bm \phi}_{1} \\
{\bm \phi}_{2} \\
{\bm p}
\end{array} 
\right )},
\end{equation}
where ${\bf B}_{ij}=\big(B_{\omega^{\ast}}(\varphi_{k}, \varphi_{l})\big)_{k\in I_{i}, l\in I_{j}}$ and ${\bf P}_{ij}=\big(B_{\omega}\big(\widehat{P}\varphi_{k},\,  \widehat{P} \varphi_{l}\big)\big)_{k\in I_{i}, l\in I_{j}}$. Here, $I_{1} = \{1,\ldots,n_{1}\}$ and $I_{2} = \{n_{1}+1,\ldots,n_{1}+n_{2}\}$ are index sets. When using a finite element discretization, $\{\varphi_{k}\}_{k\in I_{2}}$ often represent the set of finite element basis functions supported on the boundary of $\omega^{\ast}$, which typically vanish in $\omega$, i.e., $\varphi_{k}|_{\omega} = 0$ for all $k\in I_{2}$. In this case, we see that ${\bf P}_{12}={\bf P}_{21} = {\bf P}_{22} = {\bf 0}$. Taking this observation into account and using a simple matrix manipulation, problem \cref{eq:2-31} can be reduced to 
\begin{equation}\label{eq:2-32}
\mathbf{P}_{11}{\bm \phi}_{1} = \lambda_{h} \mathbf{B}_{11}{\bm p},
\end{equation}
where $ {\bm \phi}_{1}$ is computed from ${\bm p}$ by solving
\begin{equation}\label{eq:2-33}
{\bf B}^{+}{\left( \begin{array}{c}
{\bm \phi}_{1} \\
{\bm \phi}_{2} 
\end{array} 
\right )} = 
{\left( \begin{array}{c}
{\bf 0} \\
-\mathbf{B}_{21}{\bm p}
\end{array} 
\right )}, 
\quad {\rm with}\;\,{\bf B}^{+}={\left( \begin{array}{cc}
{\bf B}_{11} & {\bf B}_{12} \\
{\bf B}_{21} & {\bf B}_{22} 
\end{array} 
\right )}.
\end{equation}

The reduced problem \cref{eq:2-32} is less than half the size of problem \cref{eq:2-31} and can be seen as a generalized eigenvalue problem concerning ${\bm p}$ where the matrix-vector product on the left-hand side requires solving the linear system \cref{eq:2-33}. A fast solver for this problem thus hinges on the efficient solution of \cref{eq:2-33}. In practical applications, the "Neumann" matrix ${\bf B}^{+}$ often has a kernel of fixed dimension and is thus singular. Indeed, recall \cref{ass:2-2-3} and let $\mathcal{K} \subset \mathcal{V}_{h}(\omega^{\ast})$ without loss of generality. Then, $\mathcal{K}$ is the kernel of ${\bf B}^{+}$, and the eigenvectors of problem \cref{eq:2-31} corresponding to finite eigenvalues are orthogonal to $\mathcal{K}$ with respect to the scalar product $B_{\omega}(\widehat{P}\cdot, \widehat{P}\cdot)$, i.e., in function notation, $B_{\omega}(\widehat{P}\phi_{h}, \widehat{P}u_{j}) = 0$, $j=1,\ldots,l$, where $\{u_{1},\cdots, u_{l}\}$ is a basis of $\mathcal{K}$; see \cref{lem:2-2-2}. Based on this observation, we can impose the above conditions on the eigenvectors of problem \cref{eq:2-31} to ensure the unique solvability of the linear system \cref{eq:2-33}. To do this, we introduce another Lagrange multiplier ${\bm q}\in \mathbb{C}^{l}$ and rewrite the linear system \cref{eq:2-33} as
\begin{equation}\label{eq:2-35}
\widetilde{\bf B}^{+}{\left( \begin{array}{c}
{\bm \phi}_{1} \\
{\bm \phi}_{2} \\
{\bm q}
\end{array} 
\right )} = 
{\left( \begin{array}{c}
{\bf 0} \\
-\mathbf{B}_{21}{\bm p}\\
{\bf 0}
\end{array} 
\right )}, 
\quad {\rm with}\;\,\widetilde{\bf B}^{+}={\left( \begin{array}{ccc}
{\bf B}_{11} & {\bf B}_{12}&  {\bf M}_{1}\\
{\bf B}_{21} & {\bf B}_{22} & {\bf M}_{2} \\[1mm]
{\bf M}^{\ast}_{1} & {\bf M}_{2}^{\ast} & {\bf 0}
\end{array} 
\right )},
\end{equation}
where the matrices ${\bf M}_{i}\in \mathbb{C}^{n_{i}\times l} \;(i=1,2)$ are given by $\big(B_{\omega}\big(\widehat{P}\varphi_{k}, \widehat{P}u_{j}\big)\big)_{k\in I_i, 1\leq j\leq l}$. Clearly, the matrix $\widetilde{\bf B}^{+}$ is symmetric and invertible by \cref{ass:2-2-3} (i). 

An efficient way of solving the linear system \cref{eq:2-35} for different right-hand sides is based on the LU factorization. More precisely, we can carry out the $LDL^{T}$ factorization: $\widetilde{\bf B}^{+}={\bf L}{\bf D}{\bf L}^{T}$, where ${\bf L}$ is a lower triangular matrix and $\mathbf{D}$ is a diagonal matrix. Then, ${\bm \phi}_{1}$ can be computed by solving a lower and an upper triangular systems:
\begin{equation}\label{eq:2-34}
{\left( \begin{array}{c}
{\bm \phi}_{1} \\
{\bm \phi}_{2} \\
{\bm q}
\end{array} 
\right )} = \mathbf{L}^{-T}({\bf L}{\bf D})^{-1}
{\left( \begin{array}{c}
{\bf 0} \\
-\mathbf{B}_{21}{\bm p}\\
{\bf 0}
\end{array} 
\right )}.
\end{equation}
Finally, the basis $\{u_{1},\ldots,u_{l}\}$ of the kernel $\mathcal{K}$ should be added manually to the local approximation spaces. By combining the reduction technique and the LU factorization, our method is more efficient for solving the augmented problem \cref{eq:2-28-0} than the direct solution, especially in the 3D case.



\section{Convergence theory}\label{sec:3}
The purpose of this section is to perform a rigorous convergence analysis for the abstract MS-GFEM presented in the preceding section. We first establish exponential upper bounds on the local approximation errors, and then prove the quasi-optimal global convergence for the method.
\subsection{Exponential convergence for local approximations}\label{sec:3-1}
\subsubsection{Two fundamental assumptions}\label{sec:fundamental-assumptions}
\Cref{thm:local_approximation_error} shows that the local approximation errors of the abstract MS-GFEM are bounded by the $n$-widths $d_{n}(\omega_i,\omega_i^{\ast})$. Therefore, it is essential to establish decay rates for the $n$-widths. To do this in the abstract setting, we make two assumptions concerning the generalized harmonic spaces: a \textit{Caccioppoli-type inequality} and a \textit{weak approximation property}. They will play a fundamental role in the proof of exponential decay rates for the $n$-widths, and will be verified for many practical applications in \cref{sec:5}.

To formulate the two fundamental assumptions, let us introduce some notation. Let $\{\mathcal{L}(D)\}_{D\subset\Omega}$ be a family of local Hilbert spaces associated with $\mathcal{L}(\Omega)$. We assume that (i) for any subdomain $D\subset\Omega$, $\mathcal{H}(D)\subset\mathcal{L}(D)$, and (ii) the restriction operator is well defined on $\mathcal{L}(D)$, i.e., for any $D\subset D^{\ast}$ and $u\in \mathcal{L}(D^{\ast})$, $u|_{D}\in \mathcal{L}(D)$ with $\Vert u|_{D}\Vert_{\mathcal{L}(D)}\leq \Vert u\Vert_{\mathcal{L}(D^{\ast})}$. Moreover, we recall \cref{ass:2-2-3} and denote by $\Vert u \Vert_{B^{+},D}=\big({B}^{+}_{D}(u,u)\big)^{1/2}$ as before.

\begin{assumption}[Caccioppoli-type inequality]\label{ass:3-1-1}
There exist constants $C^{\rm I}_{\rm cac}>0$ and $C^{\rm II}_{\rm cac}\geq 0$, such that for any subdomains $D\subset D^{\ast}$ with $\delta := {\rm dist} ({D},\partial D^{\ast}\setminus \partial \Omega)>0$, 
\begin{equation}\label{eq:3-1}
\big\Vert u|_{D} \big\Vert_{B^{+},D}\leq C^{\rm I}_{\rm cac} \delta^{-1}\Vert u\Vert_{\mathcal{L}({D}^{\ast}\setminus {D})} + C^{\rm II}_{\rm cac} \Vert u\Vert_{\mathcal{L}({D}^{\ast})}\quad \text{for all}\;\; u\in \mathcal{H}_{B}({D}^{\ast}).
\end{equation}
\end{assumption}
\begin{rem}
In general, for positive definite problems, $C^{\rm II}_{\rm cac}=0$. 
\end{rem}
\cm{
\begin{rem}\label{rem:cac_assumption}
In a finite element discrete setting, the Caccioppoli-type inequality is often proved under a technical assumption: $h_{K}\lesssim \delta$ for each element $K$ with $K\cap (D^{\ast}\setminus D)\neq \emptyset$, where $h_{K}$ denotes the diameter of $K$. 
see \cref{lem:cac-5-2} and \cref{rem:h_technical_assp}.
\end{rem}
}


Inequality \cref{eq:3-1} is a generalization of the classical Caccioppoli inequality for harmonic functions. As will be shown in \cref{sec:5}, it holds for various PDEs in both the continuous and FE discrete settings. Roughly speaking, this inequality gives an estimate of a strong norm ($\mathcal{H}_{B}$-norm) of a generalized harmonic function in terms of its weak norm ($\mathcal{L}$-norm) on a larger domain. In this sense, the restriction operator can be viewed as a \textit{smoothing} operator on the generalized harmonic spaces. This smoothing property is the key to achieving exponential convergence rates for MS-GFEM for multiscale problems with low regularity solutions. It is important to note that the first term on the right hand side of inequality \cref{eq:3-1} only involves the $\mathcal{L}$-norm of $u$ on the strip $D^{\ast}\setminus D$. This observation is key to the proof of sharper decay rates for the $n$-widths than those previously established.

\begin{assumption}[Weak approximation property]\label{ass:3-1-2}
Let $D\subset D^{\ast}\subset {D}^{\ast\ast}$ be subdomains of $\Omega$ such that $\delta:={\rm dist} ({D}^{\ast},\partial D^{\ast\ast}\setminus \partial \Omega)>0$ and that ${\rm dist} ({\bm x},\partial D^{\ast}\setminus \partial \Omega)\leq \delta$ for all ${\bm x}\in \partial D\setminus \partial \Omega$, and let $a,b\geq 0$. For each $m\in \mathbb{N}$, there exists an $m$-dimensional space $Q_{m}(D^{\ast\ast})\subset \mathcal{L}({D}^{\ast\ast})$ such that for all $u\in \mathcal{H}_{B}(D^{\ast\ast})$,
 \begin{equation}\label{eq:3-2}
 \begin{array}{lll}
 {\displaystyle \inf_{v\in Q_{m}(D^{\ast\ast})} \big(a\big\Vert (u-v)|_{{D}^{\ast}\setminus D}\big\Vert_{\mathcal{L}({D}^{\ast}\setminus D)} + b\big\Vert (u-v)|_{{D}^{\ast}}\big\Vert_{\mathcal{L}({D}^{\ast})}\big) }\\[3ex]
 {\displaystyle \quad \leq C_{\rm wa} \big(a\big|\mathsf{V}_{\delta}({D}^{\ast}\setminus D)\big|^{\alpha}+b|{D}^{\ast\ast}|^{\alpha}\big) m^{-\alpha}\,\Vert u \Vert_{B^{+},D^{\ast\ast}},}
 \end{array}
 \end{equation}
 where $\mathsf{V}_{\delta}({D}^{\ast}\setminus D):= \big\{{\bm x}\in {D}^{\ast\ast}: {\rm dist}({\bm x}, {D}^{\ast}\setminus D)\leq \delta \big\}$ is the $\delta$-vicinity of the set ${D}^{\ast}\setminus D$, and $C_{\rm wa}>0$ and $\alpha\in (0,1)$ are constants independent of the domains and $m$.
 
The cases when the estimate \cref{eq:3-2} holds for all $(a>0,\,b=0)$, all $(a=0,\,b>0)$, and all $(a>0,\,b>0)$ are referred to as case a, case b, and case ab, respectively.
\end{assumption}

\begin{rem}
In fact, to prove exponential decay rates for the $n$-widths, it suffices that estimate \cref{eq:3-2} holds for $m\geq C_{1} \big|\mathsf{V}_{\delta}({D}^{\ast}\setminus D)\big| \delta^{-\alpha}$ for cases $a$ and $ab$, and $m\geq C_{1} |{D}^{\ast\ast}| \delta^{-\alpha}$ for case $b$, where $C_{1}>0$ depends only on $d$. Moreover, \cref{ass:3-1-2} is used in the proof with $D$, $D^{\ast}$, and $D^{\ast\ast}$ being nested domains between $\omega_i$ and $\omega_i^{\ast}$. Since $\omega_i$ and $\omega_i^{\ast}$ are typically regular domains, such as cubes, balls, or polyhedra, it is often sufficient for \cref{ass:3-1-2} to hold true for such domains.     
\end{rem}

\begin{rem}
Estimate \cref{eq:3-2} is a local approximation result since the functions are only approximated in compact subsets of $D^{\ast\ast}$. As such, it holds for very general domains, which enables us to prove local exponential convergence for general domains $\omega_i$ and $\omega^{\ast}_i$. While in many applications the result of global approximation holds for the domain $D^{\ast\ast}$ with certain regularity, it fails for Maxwell-type problems where $\mathcal{H}_{B}(D^{\ast\ast})$ is only compactly embedded into $\mathcal{L}(D^{\ast})$ for a compact subset $D^{\ast}$ of $D^{\ast\ast}$.
\end{rem}


\Cref{ass:3-1-2} states that there exists an $m$-dimensional approximation space such that some power of $m^{-1}$ is gained in estimating the weak norm ($\mathcal{L}$-norm) of the approximation error in terms of the strong norm ($\mathcal{H}_{B}$-norm) of a generalized harmonic function being approximated. Note that the aforementioned weak and strong norms in the weak approximation property exactly match with those in the Caccioppoli-type inequality, thereby making an iteration argument possible by combining them. In practice, this property is often verified by means of spectral estimates for certain operators (e.g., the Laplacian) or by error estimates for finite element interpolants.

Before proceeding, we remark that \cref{ass:2-2-2}, i.e., the compactness of the operators $\widehat{P}_{i}$, is a consequence of the two fundamental assumptions. Indeed, we have the following result. The proof is given in \cref{sec:appendix-1-2}.
\begin{proposition}\label{prop:3-1-1}
Let \cref{ass:3-1-1} and \ref{ass:3-1-2} be satisfied. Then, for any subdomains $D\subset D^{\ast}$ of $\Omega$ with $\delta := {\rm dist} ({D},\partial D^{\ast}\setminus \partial \Omega)>0$, the restriction operator $R:\mathcal{H}_{B}(D^{\ast})\rightarrow \mathcal{H}_{B}(D)$ (i.e., $R(u) = u|_{D}$) is compact.  
\end{proposition}

Now we are in a position to give the central result of this paper. To do this, let us fix some notation. For brevity, we omit the subdomain index $i$. Let $\mathcal{H}^{0}_{B}(\omega)$ be defined similarly to $\mathcal{H}^{0}_{B}(\omega^{\ast})$ (see \cref{eq:2-24-1}), so that $\big(\mathcal{H}^{0}_{B}(\omega), B^{+}_{\omega}(\cdot,\cdot)\big)$ is a Hilbert space, and let $P_{B}$ denote the restriction of the partition of unity operator to $\mathcal{H}^{0}_{B}(\omega)$, associated with the operator norm $\Vert P_{B}\Vert$. To simplify the presentation of the result, we assume that $\omega$ and $\omega^{\ast}$ are (truncated) concentric cubes of side length $H$ and $H^{\ast}$ ($H^{\ast}>H$), respectively, and let $\delta^{\ast}:=H^{\ast}-H$. This assumption is not essential, and an extension of our result to general domains is straightforward; see \cref{rem:general_domains} below. \cm{Moreover, we recall the constants $C^{\rm I}_{\rm cac}$, $C^{\rm II}_{\rm cac}$, $C_{\rm wa}$, and $\alpha$ in \cref{ass:3-1-1} and \ref{ass:3-1-2}, and define
\begin{equation}\label{eq:Theta_and_sigma}
\Theta = 4(3d)^{\alpha} C_{\rm wa}C^{\rm I}_{\rm cac} \big(H^{\ast}/\delta^{\ast}\big)^{1-\alpha}(H^{\ast})^{d\alpha -1}, \quad \sigma^{\ast} =  C_{\rm cac}^{\rm II} \delta^{\ast}/(4C_{\rm cac}^{\rm I}).  
\end{equation}
}

\begin{theorem}[\textbf{Exponential convergence rates for local approximations}]\label{thm:3-1}
\cm{Assume that $\omega$ and $\omega^{\ast}$ are (truncated) concentric cubes. Let $\Theta$ and $\sigma^{\ast}$ be defined by \cref{eq:Theta_and_sigma}, and let \cref{ass:3-1-1} be satisfied. 
\begin{itemize}
\item[$(\rm i)$] If case $a$ of \cref{ass:3-1-2} holds and $C^{\rm II}_{\rm cac}=0$, then for $n>n_{0}:=(e\Theta)^{1/\alpha}$,
\begin{equation}\label{eq:3-3}
 d_{n}(\omega, \omega^{\ast}) \leq \Vert P_{B}\Vert \,e^{-cn^{\alpha}}\quad \text{with}\quad c=\big(e\Theta +1\big)^{-1};
\end{equation}
\item[$(\rm ii)$]If case $b$ of \cref{ass:3-1-2} holds, then for all $n>n_{0}:=(e\widetilde{\Theta})^{1/\alpha}$,
\begin{equation}\label{eq:3-4}
d_{n}(\omega, \omega^{\ast}) \leq \Vert P_{B}\Vert \,e^{\sigma^{\ast}} e^{-cn^{\alpha/(\alpha+1)}}\quad \text{with}\quad  c=\big(e\widetilde{\Theta} +1\big)^{-1/(1+\alpha)},
\end{equation}
where $\widetilde{\Theta} = (3d)^{-\alpha}(H^{\ast}/\delta^{\ast})^{\alpha}\,\Theta$;
\item[$(\rm iii)$] If case $ab$ of \cref{ass:3-1-2} holds, then \cref{eq:3-4} holds. Moreover, for all $n>n_{0}$ with $n_0:= (3d)^{1/(\alpha-1)}(2e\Theta)^{1/\alpha} \big(\sigma^{\ast} (H^{\ast}/\delta^{\ast})^{\alpha} \big)^{\frac{1}{\alpha(1-\alpha)}}$,    
\begin{equation}\label{eq:3-5}
d_{n}(\omega, \omega^{\ast}) \leq \Vert P_{B}\Vert \,e^{-cn^{\alpha}} \quad \text{with}\quad c=\big(2e\Theta +1\big)^{-1}.
\end{equation}
\end{itemize}
}
\end{theorem}


\begin{rem}
In almost all practical applications (see \cref{sec:5}), $\alpha = 1/d$. Hence \cref{thm:3-1} gives a sharper decay rate for the $n$-width, compared with the rate $\mathcal{O}(e^{-cn^{1/(d+1)}})$ established in previous works \cite{babuska2011optimal,babuvska2014machine,ma2021novel,ma2022error,chupeng2023wavenumber}. \cm{Note that with $\alpha = 1/d$, the exponents $c$ above depend on the oversampling parameter $H^{\ast}/\delta^{\ast}$, but not on the actual size of the domains $\omega$ and $\omega^{\ast}$.}
\end{rem}

\begin{rem}
We are particularly interested in the dependence of the decay rates on the constant $C^{\rm II}_{\rm cac}$, which is often related to some important problem parameters. For example, it can be shown that for wave propagation problems, $C^{\rm II}_{\rm cac} = O(k)$, where $k$ denotes the wavenumber, and that for convection-diffusion problems, $C^{\rm II}_{\rm cac} = O(\Vert {\bm b}\Vert_{L^{\infty}(\omega^{\ast})})$ with ${\bm b}$ denoting the velocity field; see \cref{sec:5}. \cm{We note that the exponents $c$ in estimates \cref{eq:3-4} and \cref{eq:3-5} are independent of $C^{\rm II}_{\rm cac}$. Moreover, if $H^{\ast}\sim H\sim (C^{\rm II}_{\rm cac})^{-1}$, then $\sigma^{\ast}\lesssim 1$ and thus the constant $n_0$ in case (iii) is also independent of $C^{\rm II}_{\rm cac}$. The exponential convergence result in this case is very similar to case (i) with $C^{\rm II}_{\rm cac}=0$.}
\end{rem}


\cm{
\begin{rem}\label{rem:general_domains}
For general domains $\omega$ and $\omega^{\ast}$, the exponential convergence results also hold true if we replace $H^{\ast}$ and $\delta^{\ast}$ above with $C|\omega^{\ast}|^{1/d}$ and $2\,{\rm dist}(\omega,\partial\omega^{\ast}\setminus\partial \Omega)$, respectively. Here the constant $C$ depends on $S(\omega)/|\omega|^{(d-1)/d}$ and $S(\omega^{\ast})/|\omega^{\ast}|^{(d-1)/d}$, where $S(\omega)$ denotes the surface area of $\omega$. The proof is a simple adaptation of that of \cref{thm:3-1}; see \cref{rem:general-domains}.
\end{rem}

\begin{rem}
In a finite element discrete setting, the technical assumption $h_{K}\lesssim \delta$ for the Caccioppoli-type inequality (see \cref{rem:cac_assumption}) leads to a condition on the mesh-size for \cref{thm:3-1}: $h_{K}\lesssim c\delta^{\ast}n^{-r}$ for each element $K$ with $K\cap (\omega^{\ast}\setminus \omega)\neq \emptyset$, where $r=\alpha$ in cases (i) and (iii) and $r=\alpha/(\alpha+1)$ in case (ii). See \cref{rem:proof_hK_ass} for the derivation of this condition. Nevertheless, this condition is not essential in practice as the mesh size is typically small enough. 
\end{rem}
}


\subsubsection{Proof of the main result}\label{main-proof}
Before giving the proof, let us briefly describe the main idea and illustrate the central role of the oversampling technique and the resulting smoothing effect. The proof is divided into two steps. First, we combine the two fundamental assumptions to prove algebraic decay rates for the $n$-width of the restriction from $\mathcal{H}_{B}({D}^{\ast})$ into $\mathcal{H}_{B}({D})$ for general subdomains $D\subset D^{\ast}$ with ${\rm dist} ({D},\partial D^{\ast}\setminus \partial \Omega)>0$. Next, we iterate the obtained estimates for a family of nested domains between $\omega$ and $\omega^{\ast}$ to get the desired exponential decay rates. 

It is interesting to remark that the proof of \cref{thm:3-1} surprisingly resembles the convergence proof of multigrid methods in a very different context \cite[Chapter V]{braess2007finite}, which is also based on a combination of a smoothing property and an approximation property, and also uses two norms for specifying the errors. There, the advantage of multigrid methods over classical iterative methods hinges on the approximation property associated with a coarse-grid correction. Here, nevertheless, the key to the novel exponential convergence lies in the smoothing property. Indeed, similar weak approximation estimates as \cref{ass:3-1-2} serve as the mathematical foundation of many standard numerical methods (e.g., classical FEMs). \cm{However, these methods may exhibit very slow convergence for low regularity solutions to PDEs with rough coefficients \cite{babuvska2000can}}. The lack of regularity is overcome in our method by means of the smoothing property resulting from oversampling, which, combining with the weak approximation property, leads to the exponential decay rates under very low regularity assumptions.


Next we prove \cref{thm:3-1}. As the first step, we prove the following auxiliary result, making clear the fundamental importance of \cref{ass:3-1-1} and \ref{ass:3-1-2}. \cm{We recall \cref{ass:2-2-3} concerning the kernel $\mathcal{K}\subset \mathcal{H}_{B}(D^{\ast})$ of $B^{+}_{D^{\ast}}(\cdot,\cdot)$.}

\begin{lemma}\label{lem:3-1}
 Let $C^{\rm I}_{\rm cac}$, $C^{\rm II}_{\rm cac}$, $C_{\rm wa}$, and $\alpha$ be the constants as in \cref{ass:3-1-1} and \ref{ass:3-1-2}, and let \cref{ass:3-1-1} be satisfied. Suppose that $D^{\ast}$ satisfies the cone condition, and that ${D}\subset{D}^{\ast}$ with $\delta := {\rm dist} ({D},\partial D^{\ast}\setminus \partial \Omega)>0$. Then, there exist $m$-dimensional spaces $S_{m}(D^{\ast})\subset \mathcal{H}_{B}(D^{\ast})$ \cm{$(m>{\rm dim}\,(\mathcal{K}))$} such that for any $u\in \mathcal{H}_{B}(D^{\ast})$,

\vspace{2mm}
\begin{itemize}
\item[$(\rm i)$] if case $a$ of \cref{ass:3-1-2} holds and $C^{\rm II}_{\rm cac}=0$, then 
\begin{equation}\label{eq:3-6-0}
\inf_{v\in S_{m}(D^{\ast})} \Vert  u-v\Vert_{B^{+},D} \leq 
2C_{\rm wa}{C^{\rm I}_{\rm cac}}m^{-\alpha} {\delta}^{-1}|\mathsf{V}_{\delta/2}({D}^{\ast}\setminus D)\big|^{\alpha} \Vert u\Vert_{B^{+},D^{\ast}},
\end{equation}
 where $\mathsf{V}_{\delta/2}({D}^{\ast}\setminus D):= \big\{{\bm x}\in D^{\ast}: {\rm dist}({\bm x}, {D}^{\ast}\setminus D)\leq \delta/2 \big\}$;

\vspace{2mm}
\item[$(\rm ii)$] if case $b$ of \cref{ass:3-1-2} holds, then 
\begin{equation}\label{eq:3-6-1}
 \inf_{v\in S_{m}(D^{\ast})} \Vert u-v \Vert_{B^{+},D} \leq 
C_{\rm wa}m^{-\alpha} \big(2{C^{\rm I}_{\rm cac}}{\delta}^{-1} + C^{\rm II}_{\rm cac}\big)|{D}^{\ast}|^{\alpha}\Vert u \Vert_{B^{+},D^{\ast}};
\end{equation}

\item[$(\rm iii)$] if case $ab$ of \cref{ass:3-1-2} holds, then 
\begin{equation}\label{eq:3-6-2}
\begin{array}{ll}
 {\displaystyle \inf_{v\in S_{m}(D^{\ast})} \Vert  u-v\Vert_{B^{+},D} }  \\[3mm]
{\displaystyle \leq C_{\rm wa}m^{-\alpha}\Big(2{C^{\rm I}_{\rm cac}}{\delta}^{-1}\big|\mathsf{V}_{\delta/2}({D}^{\ast}\setminus D)\big|^{\alpha} + C^{\rm II}_{\rm cac}|{D}^{\ast}|^{\alpha}\Big)\Vert u\Vert_{B^{+},D^{\ast}}}. 
\end{array}
\end{equation}
\end{itemize}
\end{lemma}
\begin{proof}
Let $\widetilde{D}^{\ast}$ be an intermediate set between $D$ and $D^{\ast}$ such that ${\rm dist}(\widetilde{D}^{\ast}, \partial D^{\ast}\setminus \partial \Omega) = \delta /2$ and that ${\rm dist}({\bm x}, \partial \widetilde{D}^{\ast}\setminus\partial \Omega) = \delta/2$ for all ${\bm x}\in \partial D\setminus \partial \Omega$. We first consider case (i). Denote by $D_{c} = \widetilde{D}^{\ast}\setminus{D}$ and $\mathcal{K}_{c}=\mathcal{K}|_{D_c}$, where $\mathcal{K}$ is the kernel of $B^{+}_{D^{\ast}}(\cdot,\cdot)$; see \cref{ass:2-2-3}. Furthermore, we define the operator $R_{c}:\mathcal{H}_{B}({D}^{\ast})/\mathcal{K}\rightarrow \mathcal{L}(D_{c})/\mathcal{K}_{c}$ by $R_{c} u = u|_{D_{c}}$. By the regularity assumption on $D^{\ast}$, $\big(\mathcal{H}_{B}({D}^{\ast})/\mathcal{K}, B^{+}_{D^{\ast}}(\cdot,\cdot)\big)$ is a Hilbert space. Hence, we can define the following $m$-width of operator $R_{c}$:
\begin{equation}\label{eq:3-7}
d_{m}(R_{c}):=\inf_{S(m)\subset \mathcal{L}(D_{c})/\mathcal{K}_{c}}\sup_{u\in \mathcal{H}_{B}({D}^{\ast})/\mathcal{K}} \inf_{v\in S(m)}\frac {\big\Vert u|_{D_{c}}-v\big\Vert_{\mathcal{L}(D_{c})/\mathcal{K}_{c}}}{\Vert u \Vert_{B^{+},{D}^{\ast}}}.   
\end{equation}
We claim that the operator $R_{c}$ is compact. Indeed, \cref{ass:3-1-2} (case $a$) shows that there exist $m$-dimensional spaces $Q_{m}(D^{\ast})\subset\mathcal{L}(D^{\ast})$ such that for all $u\in \mathcal{H}_{B}({D^{\ast}})$,
\begin{equation}\label{eq:3-7-0}
\inf_{v\in Q_{m}(D^{\ast})}\big\Vert u|_{D_c}-v|_{D_c}\big\Vert_{\mathcal{L}({D}_{c})}  \leq C_{\rm wa}  |\mathsf{V}_{\delta/2}(D_{c})|^{\alpha} m^{-\alpha} \Vert u\Vert_{B^{+},{D}^{\ast}}.
\end{equation}
Since $\Vert u \Vert_{B^{+},{D}^{\ast}}$ vanishes for all $u\in \mathcal{K}$, estimate \cref{eq:3-7-0} implies that $\mathcal{K}_{c}\subset Q_{m}(D^{\ast})|_{D_c}$. Denoting by $\widehat{Q}_{m}(D_c) = \big( Q_{m}(D^{\ast})|_{D_c}\big)/\mathcal{K}_c\subset \mathcal{L}({D}_{c})/\mathcal{K}_{c}$, estimate \cref{eq:3-7-0} further yields that for all $u\in \mathcal{H}_{B}({D^{\ast}})/\mathcal{K}$,
\begin{equation}\label{eq:3-8}
\inf_{v\in \widehat{Q}_{m}(D_c)}\big\Vert u|_{D_c}-v\big\Vert_{\mathcal{L}({D}_{c})/\mathcal{K}_c}  \leq C_{\rm wa}  |\mathsf{V}_{\delta/2}(D_{c})|^{\alpha} m^{-\alpha} \Vert u \Vert_{B^{+},{D}^{\ast}}.
\end{equation}
Without loss of generality, we assume that ${\rm dim}(\mathcal{K}_c) = {\rm dim}(\mathcal{K}) = l$, and hence ${\rm dim}\big(\widehat{Q}_{m}(D_c)\big)\leq m-l$. We then deduce from \cref{eq:3-8} that
\begin{equation}\label{eq:3-9}
d_{m-l}(R_{c})\leq C_{\rm wa} |\mathsf{V}_{\delta/2}(D_{c})|^{\alpha}m^{-\alpha} \;\rightarrow 0 \;\quad {\rm as}\;\;m\rightarrow \infty,
\end{equation}
which implies that operator $R_{c}$ is compact. Let $\widehat{S}_{m-l}(D^{\ast}) = {\rm span}\{u_{1},\ldots,  u_{m-l}\}\subset \mathcal{H}_{B}({D}^{\ast})/\mathcal{K}$ be the space spanned by the first $m-l$ right singular vectors of the operator $R_{c}$. It follows from \cref{lem:1-1} and \cref{eq:3-9} that
\begin{equation}\label{eq:3-10}
\sup_{u\in \mathcal{H}_{B}({D}^{\ast})/\mathcal{K}} \inf_{v\in \widehat{S}_{m-l}(D^{\ast})}\frac {\big\Vert u|_{D_{c}}-v|_{D_{c}}\big\Vert_{\mathcal{L}({D_{c}})/\mathcal{K}_{c}}}{\Vert u \Vert_{B^{+},{D}^{\ast}}} \leq C_{\rm wa} |\mathsf{V}_{\delta/2}(D_{c})|^{\alpha}m^{-\alpha}.
\end{equation}
Let $S_{m}(D^{\ast}) = \widehat{S}_{m-l}(D^{\ast}) \oplus \mathcal{K}$. Estimate \cref{eq:3-10} implies that for all $u\in \mathcal{H}_{B}({D}^{\ast})$,
\begin{equation}\label{eq:3-11}
\inf_{v\in S_{m}(D^{\ast})}\big\Vert u|_{D_{c}}-v|_{D_{c}}\big\Vert_{\mathcal{L}({D_{c}})}\leq C_{\rm wa} |\mathsf{V}_{\delta/2}(D_{c})|^{\alpha}m^{-\alpha} \Vert u \Vert_{B^{+},{D}^{\ast}}.    
\end{equation}
Combining \cref{eq:3-11} and the Caccioppoli-type inequality (\cref{ass:3-1-1}), and noting that $\mathsf{V}_{\delta/2}(\widetilde{D}^{\ast}\setminus D)\subset \mathsf{V}_{\delta/2}(D^{\ast}\setminus D)$, lead us to \cref{eq:3-6-0}. Case $({\rm ii})$ can be proved by following the same lines as above. To prove case $(\rm iii)$, we first equip the space $\mathcal{L}(\widetilde{D}^{\ast})$ with the norm $\vertiii{\cdot}_{\mathcal{L}(\widetilde{D}^{\ast})}$ defined by 
\begin{equation}\label{eq:3-11-1}
 \vertiii{u}_{\mathcal{L}(\widetilde{D}^{\ast})}:=\Big( \big(2C^{\rm I}_{\rm cac}/{\delta}\big)^{2}\Vert u\Vert^{2}_{\mathcal{L}(\widetilde{D}^{\ast}\setminus D)} + (C^{\rm II}_{\rm cac})^{2} \Vert u\Vert^{2}_{\mathcal{L}(\widetilde{D}^{\ast})}\Big)^{1/2}.
\end{equation}
Clearly, $\big(\mathcal{L}(\widetilde{D}^{\ast}), \vertiii{\cdot}_{\mathcal{L}(\widetilde{D}^{\ast})}\big)$ is a Hilbert space. Next we consider the $n$-width of the restriction from $\mathcal{H}_{B}(D^{\ast})/\mathcal{K}$ into $\mathcal{L}(\widetilde{D}^{\ast})/(\mathcal{K}|_{\widetilde{D}^{\ast}})$, defined similarly to \cref{eq:3-7}. Estimate \cref{eq:3-6-2} then follows from a similar argument as above by combining \cref{ass:3-1-1} and \ref{ass:3-1-2}. 
\end{proof}

Before proceeding to the second step of the proof, we recall that $\omega$ and $\omega^{\ast}$ are assumed to be (truncated) concentric cubes of side lengths $H$ and $H^{\ast}$, respectively. Let $\delta^{\ast} = H^{\ast}-H$. Given $N\in\mathbb{N}$, let $\{\omega^{j}\}_{j=1}^{N+1}$ be the nested family of (truncated) cubes of side lengths $\{H^{\ast}-\delta^{\ast} (j-1)/N\}_{j=1}^{N+1}$ such that $\omega=\omega^{N+1}\subset \omega^{N}\subset\ldots\subset\omega^{1} = \omega^{\ast}$. For simplicity, we assume that \cref{ass:2-2-3} holds for all $\omega^{k}$ $(k=1,\cdots,N+1)$ with the same kernel $\mathcal{K}$ of dimension $l$.

\begin{lemma}\label{lem:3-1-0}
Let $m,N\in\mathbb{N}$ with $m>l$, and let \cref{ass:3-1-1} be satisfied. Then
\begin{equation}\label{eq:3-15}
d_{mN}(\omega,\omega^{\ast})\leq \big\Vert P_{B}\big \Vert \big[\xi(N,m)\big]^{N},
\end{equation}
where 
\begin{equation}\label{xi-case-a}
    \xi(N,m) = 4(3d)^{\alpha} C_{\rm wa}C^{\rm I}_{\rm cac} m^{-\alpha}N^{1-\alpha}\big(H^{\ast}/\delta^{\ast}\big)^{1-\alpha} (H^{\ast})^{d\alpha-1}
\end{equation}
if case $a$ of \cref{ass:3-1-2} holds and $C^{\rm II}_{\rm cac}=0$, 
\begin{equation}\label{xi-case-b}
    \xi(N,m) =  4(1+\sigma^{\ast}/N)C_{\rm wa}C^{\rm I}_{\rm cac} m^{-\alpha}N\big(H^{\ast}/\delta^{\ast}\big) (H^{\ast})^{d\alpha -1} 
\end{equation}
with $\sigma^{\ast}=C_{\rm cac}^{\rm II}\delta^{\ast}/(4C_{\rm cac}^{\rm I})$ if case $b$ of \cref{ass:3-1-2} holds, and 
\begin{equation}\label{xi-case-ab}
    \xi(N,m) = C_{\rm wa}m^{-\alpha} \Big(4(3d)^{\alpha}C^{\rm I}_{\rm cac} \big(N/\delta^{\ast}\big)^{1-\alpha} (H^{\ast})^{(d-1)\alpha} + C^{\rm II}_{\rm cac} (H^{\ast})^{\alpha d}\Big)
\end{equation}
if case $ab$ of \cref{ass:3-1-2} holds.
\end{lemma}
\begin{proof}
Estimate \cref{eq:3-15} is essentially a direct consequence of the multiplicativity property \cref{eq:1-3}. Indeed, in the case $\mathcal{K}=\emptyset$, let ${R}_{B}^{k}$ $(k=1,\cdots,N)$ denote the restriction operators from $\big(\mathcal{H}_{B}(\omega^{k}), \Vert\cdot\Vert_{B^{+},\omega^{k}}\big)$ into $\big(\mathcal{H}_{B}(\omega^{k+1}), \Vert\cdot\Vert_{B^{+},\omega^{k+1}}\big)$, and let $d_{m}({R}_{B}^{k})$ be the $n$-widths of operators ${R}^{k}_{B}$. Since
\begin{equation}\label{eq:3-14}
\widehat{P} = P_{B}{R}^{N}_{B}\cdots{R}^{1}_{B},
\end{equation}
a simple application of the multiplicativity property \cref{eq:1-3} gives 
\begin{equation}\label{eq:3-15-0}
 d_{mN}(\omega,\omega^{\ast})\leq \big\Vert P_{B} \big\Vert \prod_{k=1}^{N}d_{m}({R}_{B}^{k}).
\end{equation}
Combining \cref{eq:3-15-0} and \cref{lem:3-1}, and noting that $\delta:={\rm dist}(\omega^{k},\omega^{k+1}) = \delta^{\ast}/(2N)$ and $\big|\mathsf{V}_{\delta/2}(\omega^{k}\setminus\omega^{k+1})\big|\leq 3d\delta (H^{\ast})^{d-1}$ $(k=1,\cdots, N)$, complete the proof of this lemma. The proof in the case $\mathcal{K}\neq\emptyset$ is based on an iteration argument. Arguing similarly as in the proof of \cref{eq:3-10} with $D=\omega^{2}$ and $D^{\ast}=\omega^{1}$ (recall that $\Vert\cdot\Vert_{B^{+},\omega^{1}}$ is a norm on $\mathcal{H}^{0}_{B}(\omega^{1})$), we deduce that there exists an $(m-l)$-dimensional space $\widetilde{Q}_{m-l}(\omega^{1})\subset \mathcal{H}_{B}(\omega^{1})$ such that 
\begin{equation}\label{eq:3-15-1}
\inf_{v\in \widetilde{Q}_{m-l}(\omega^{1})} \big\Vert (u-v)|_{\omega^{2}}\big\Vert_{B^{+},\omega^{2}} \leq  \xi(N,m)\,\Vert u\Vert_{B^{+},\omega^{1}} \quad \text{for all}\;\; u\in \mathcal{H}^{0}_{B}(\omega^{1}),
\end{equation}
where $\xi(N,m)$ is given above. Similarly, for each $k=2,\cdots,N$, it follows from \cref{lem:3-1} that there exists an $m$-dimensional space $\widetilde{Q}_{m}(\omega^{k})\subset \mathcal{H}_{B}(\omega^{k})$ such that
\begin{equation}\label{eq:3-15-2}
\inf_{v\in \widetilde{Q}_{m}(\omega^{k})} \big\Vert(u-v)|_{\omega^{k+1}}\big\Vert_{B^{+},\omega^{k+1}} \leq  \xi(N,m)  \,\Vert u\Vert_{B^{+},\omega^{k}}\quad \text{for all}\;\; u\in \mathcal{H}_{B}(\omega^{k}).
\end{equation}
Let $u\in \mathcal{H}^{0}_{B}(\omega^{1})$. Combining \cref{eq:3-15-1} and \cref{eq:3-15-2} and recalling \cref{prop:2-1}, we see that there is a $v^{1}_{u}\in \widetilde{Q}_{m-l}(\omega^{1})|_{\omega^{N+1}}+\widetilde{Q}_{m}(\omega^{2})|_{\omega^{N+1}}+\cdots+\widetilde{Q}_{m}(\omega^{N})|_{\omega^{N+1}}$ satisfying
\begin{equation}\label{eq:3-15-3}
\big\Vert u|_{\omega^{N+1}}-v_{u}^{1}\big\Vert_{B^{+},\omega^{N+1}} \leq  \big[\xi(N,m)\big]^{N}\,\Vert u\Vert_{B^{+},\omega^{1}}.
\end{equation}
Furthermore, let $n = mN$ and denote by
\begin{equation}\label{eq:3-15-4}
\widetilde{Q}_{n}(\omega)= \mathcal{K}+\widetilde{Q}_{m-l}(\omega^{1})|_{\omega^{N+1}}+\widetilde{Q}_{m}(\omega^{2})|_{\omega^{N+1}}+\cdots+\widetilde{Q}_{m}(\omega^{N})|_{\omega^{N+1}}.
\end{equation}
We can find a $v_{u}=v^{1}_{u}+v_{u}^{2}\in \widetilde{Q}_{n}(\omega)$ with $v_{u}^{2}\in \mathcal{K}$, such that $u|_{\omega}-v_{u}\in \mathcal{H}_{B}^{0}(\omega)$ and
\begin{equation}\label{eq:3-15-5}
\big\Vert u|_{\omega}-v_{u}\big\Vert_{B^{+},\omega} \leq  \big[\xi(N,m)\big]^{N}\,\Vert u\Vert_{B^{+},\omega^{\ast}}.
\end{equation}
Using \cref{eq:3-15-5} and keeping in mind that $\Vert \cdot \Vert_{B^{+},\omega}$ is a norm on $\mathcal{H}_{B}^{0}(\omega)$, we obtain
\begin{equation}\label{eq:3-15-6}
\begin{array}{cc}
{\displaystyle \big\Vert P(u|_{\omega}) - P(v_{u})\big\Vert_{\mathcal{H}_{0}(\omega)} \leq \big\Vert P_{B}\big\Vert \,\big\Vert u|_{\omega}-v_{u}\big\Vert_{B^{+},\omega} }\\[3mm]
{\displaystyle \leq \big\Vert P_{B}\big\Vert \big[\xi(N,m)\big]^{N}\,\Vert u\Vert_{B^{+},\omega^{\ast}}.}
\end{array}
\end{equation}
Let $Q_{n}(\omega) = P\big(\widetilde{Q}_{n}(\omega)\big)\subset \mathcal{H}_{0}(\omega)$. It follows from \cref{eq:3-15-6} that
\begin{equation}
\sup_{u\in \mathcal{H}^{0}_{B}(\omega^{\ast})}\inf_{v\in Q_{n}(\omega)}\frac{\Vert P(u|_{\omega}) -v\Vert_{\mathcal{H}_{0}(\omega)}}{\Vert u\Vert_{B^{+},\omega^{\ast}}}\leq \big\Vert P_{B} \big\Vert \big[\xi(N,m)\big]^{N},
\end{equation}
which gives the desired estimate \cref{eq:3-15} by recalling the definition of the $n$-width.
\end{proof}
\cm{
\begin{rem}\label{rem:general-domains}
If $\omega$ and $\omega^{\ast}$ are general domains, we can construct the intermediate domains iteratively as follows:
\begin{equation}\label{eq:inter_domains}
    \omega^{j} = \bigcup_{{\bm x}\in \omega^{j+1}} B({\bm x}, \delta^{\ast}/N)\cap \Omega, \qquad j=N,N-1,\ldots,2,
\end{equation}
where $\delta^{\ast}= {\rm dist}(\mathcal{\omega},\partial \mathcal{\omega}^{\ast} \setminus\partial \Omega)$ and $B({\bm x}, \delta^{\ast}/N)$ denotes the ball centered at ${\bm x}$ with radius $\delta^{\ast}/N$. It is clear that $\{\omega^{j}\}_{j=2}^{N}$ all satisfy the cone condition. To prove \cref{lem:3-1-0} in this case, it suffices to modify the estimates of $\big|\mathsf{V}_{\delta/2}(\omega^{k}\setminus\omega^{k+1})\big|$ and $|\omega^k|$ as follows:
\begin{equation}\label{eq:estimate_general_domains}
\big|\mathsf{V}_{\delta/2}(\omega^{k}\setminus\omega^{k+1})\big| \lesssim S(\omega^{k}) \delta \lesssim |\omega^{k}|^{(d-1)/d} \delta  \leq  |\omega^{\ast}|^{(d-1)/d} \delta,\quad |\omega^k|\leq |\omega^{\ast}|,
\end{equation}
where $S(\omega^{k})$ denotes the surface area of $\omega^{k}$. Note that the hidden constant in \cref{eq:estimate_general_domains} depends on $S(\omega^{k})/|\omega^{k}|^{(d-1)/d}$, which in turn depends on $S(\omega^{\ast})/|\omega^{\ast}|^{(d-1)/d}$ $(k=1)$ or $S(\omega)/|\omega|^{(d-1)/d}$ $(k\geq 2)$ by the construction \cref{eq:inter_domains}. Using \cref{eq:estimate_general_domains} and following along the same lines as above, we can prove similar results as \cref{xi-case-a,xi-case-b,xi-case-ab}, with $|\omega^{\ast}|^{1/d}$ and $2\,{\rm dist}(\omega,\partial\omega^{\ast}\setminus\partial \Omega)$ in place of $H^{\ast}$ and $\delta^{\ast}$, respectively.
\end{rem}
}

Now we are in a position to prove \cref{thm:3-1}.
\begin{proof}[Proof of \cref{thm:3-1}]

We first consider case (i). Recalling $\Theta$ defined by \cref{eq:Theta_and_sigma},  
for each $N\in \mathbb{N}$, we let $m\in\mathbb{N}$ satisfy
\begin{equation}\label{eq:3-20}
\big(e\Theta N^{1-\alpha}\big)^{1/\alpha} \leq m <\big(e\Theta N^{1-\alpha}\big)^{1/\alpha} + 1.
\end{equation}
\cm{Note that in practical applications, condition \cref{eq:3-20} yields that $m>l:={\rm dim}\, (\mathcal{K})$, and thus the condition in \cref{lem:3-1-0} is satisfied.} It follows from \cref{eq:3-20} that $\xi(N,m)$ given by \cref{xi-case-a} satisfies $\xi(N,m)\leq e^{-1}$,
and thus estimate \cref{eq:3-15} becomes
\begin{equation}\label{eq:3-22}
 d_{mN}(\omega,\omega^{\ast})\leq \Vert P_{B}\Vert \,e^{-N}.
\end{equation}
Let $n=Nm$. Using \cref{eq:3-20} again, we see that 
\begin{equation}\label{eq:3-23}
n\leq \big(e\Theta +1\big)^{1/\alpha} N^{1/\alpha} \quad \Longrightarrow \quad N\geq \big(e\Theta +1\big)^{-1} n^{\alpha}.
\end{equation}
Denoting by $c=\big(e\Theta +1\big)^{-1}$ and inserting \cref{eq:3-23} into \cref{eq:3-22}, we arrive at
\begin{equation}\label{eq:3-24}
d_{n}(\omega,\omega^{\ast})\leq \Vert P_{B} \Vert\,e^{-cn^{\alpha}}.
\end{equation}
Furthermore, since $N\geq 1$, we deduce from \cref{eq:3-20} that $n\geq  N \big(e\Theta N^{1-\alpha}\big)^{1/\alpha} \geq n_{0}: = \big(e\Theta\big)^{1/\alpha}$. Hence, \cref{thm:3-1}$(\rm i)$ is proved. Case $(\rm ii)$ can be proved similarly by using the inequality $\big(1+\sigma^{\ast}/N\big)^{N}\leq e^{\sigma^{\ast}}$. The details are omitted.

We proceed to prove case (iii). Clearly, estimate \cref{eq:3-4} is satisfied if case $ab$ of \cref{ass:3-1-2} holds. 
To prove \cref{eq:3-5}, we first choose $N\geq 1$ such that
\begin{equation}\label{eq:3-27}
4(3d)^{\alpha} C_{\rm wa}C^{\rm I}_{\rm cac}\big(N/\delta^{\ast}\big)^{1-\alpha} (H^{\ast})^{(d-1)\alpha}\geq C_{\rm wa}C^{\rm II}_{\rm cac} (H^{\ast})^{\alpha d}, 
\end{equation}
and then let $m$ satisfy
\begin{equation}\label{eq:case(iii)-m}
\big(2e\Theta N^{1-\alpha}\big)^{1/\alpha} \leq m <\big(2e\Theta N^{1-\alpha}\big)^{1/\alpha} + 1.
\end{equation}
It follows from \cref{eq:3-27} and \cref{eq:case(iii)-m} that $\xi(N,m)\leq e^{-1}$ with $\xi(N,m)$ given by \cref{xi-case-ab}. Inserting this result into \cref{eq:3-15} yields $d_{n}(\omega,\omega^{\ast})\leq \Vert P_{B}\Vert \,e^{-N}$ with $n=mN$. Proceeding as above, we can show that 
\begin{equation}
d_{n}(\omega, \omega^{\ast}) \leq \Vert P_{B}\Vert \,e^{-cn^{\alpha}} \quad \text{with}\quad c:=\big(2e\Theta +1\big)^{-1}
\end{equation}
holds for all $n>n_{0}:= (3d)^{1/(\alpha-1)}(2e\Theta)^{1/\alpha} \big(\sigma^{\ast} (H^{\ast}/\delta^{\ast})^{\alpha} \big)^{\frac{1}{\alpha(1-\alpha)}}$.
Hence, estimate \cref{eq:3-5} is proved and the proof of \cref{thm:3-1} is complete.
\end{proof}

\cm{
\begin{rem}\label{rem:proof_hK_ass}
Since we apply the Caccioppoli-type inequality in the proof with $\delta\sim \delta^{\ast}/N$, the assumption $h_{K}\lesssim \delta$ in a finite element discrete setting (see \cref{rem:cac_assumption}) becomes $h_{K}\lesssim \delta^{\ast}/N$. Then, \cref{eq:3-23} implies that $h_{K}\lesssim c\delta^{\ast}n^{-\alpha}$ in case (i). Other cases follow similarly.
\end{rem}
}

Finally, we note that by slightly modifying the proof of \cref{lem:3-1} and following the remaining lines above, it can be shown that there exists an $n$-dimensional space $Q_{n}(\omega)\subset \mathcal{H}_{B}(\omega)$ such that for any $u\in \mathcal{H}_{B}(\omega^{\ast})$, 
\begin{equation}
\inf_{v\in Q_{n}(\omega)} \Vert u-v\Vert_{\mathcal{L}(\omega)}\leq \texttt{err}_{n}\Vert u\Vert_{\mathcal{L}(\omega^{\ast})},   
\end{equation}
where $\texttt{err}_{n}$ satisfies the same bounds as $d_{n}(\omega, \omega^{\ast})$ without involving the norm $\Vert P_{B}\Vert$.

\subsubsection{Extension to higher order problems}\label{sec:framework-higher-order} 
The abstract approximation theory presented in \cref{sec:fundamental-assumptions} typically applies to second-order problems. To render it applicable to higher order problem, we need to generalize the Caccioppoli-type inequality and the weak approximation property. To do this, let us introduce some new local Hilbert spaces and associated operators. For any subdomain $D\subset\Omega$, let $\mathcal{L}_{i}(D)$, $i=1,\cdots, L$, be Hilbert spaces, and let $A_{D,i}:\mathcal{L}(D)\rightarrow \mathcal{L}_{i}(D)$ be bounded linear operators. In addition, we assume that there exists a constant $c_{0}>0$ such that
\begin{equation}\label{hp:norm-equiv}
\Vert u\Vert_{\mathcal{L}(D)} \leq c_{0}\sum_{i=1}^{L} \Vert A_{D,i} (u) \Vert_{\mathcal{L}_{i}(D)} \quad \text{for all}\;\; u\in \mathcal{L}({D}).
\end{equation}
Now we can give the desired generalized Caccioppoli-type inequality: There exists $C^{\rm I}_{\rm cac}>0$ such that for any subdomains $D\subset D^{\ast}$ with $\delta := {\rm dist} ({D},\partial D^{\ast}\setminus \partial \Omega)>0$, and for any $u\in \mathcal{H}_{B}({D}^{\ast})$,
\begin{equation}\label{hp:cac-inequality}
\big\Vert u|_{D} \big\Vert_{B^{+},D}\leq C^{\rm I}_{\rm cac}\sum_{i=1}^{L} {\delta}^{-i} \big\Vert A_{D_c,i}(u|_{D_c})\big\Vert_{\mathcal{L}_{i}(D_c)} \quad \text{with}\;\; D_{c} = D^{\ast}\setminus D.
\end{equation}
We also generalize the weak approximation property as follows. Let $D\subset D^{\ast}\subset {D}^{\ast\ast}$ be subdomains of $\Omega$ satisfying that $\delta:={\rm dist} ({D}^{\ast},\partial D^{\ast\ast}\setminus \partial \Omega)>0$ and that ${\rm dist} ({\bm x},\partial D^{\ast}\setminus \partial \Omega)\leq \delta$ for all ${\bm x}\in \partial D\setminus \partial \Omega$. For each $m\in \mathbb{N}$, there exists an $m$-dimensional space $Q_{m}(D^{\ast\ast})\subset \mathcal{L}({D}^{\ast\ast})$ such that for any $u\in \mathcal{H}_{B}(D^{\ast\ast})$ and any $a_i>0$, $i=1,\cdots,L$,
 \begin{equation}\label{hp:weak-approx}
 \begin{array}{lll}
 {\displaystyle \inf_{v\in Q_{m}(D^{\ast\ast})} \Big[ \sum_{i=1}^{L}a_i\big\Vert A_{D_c,i}(u|_{D_c} - v|_{D_c})\big\Vert_{\mathcal{L}_{i}(D_c)} \Big] }\\[3ex]
 {\displaystyle \quad \leq C_{\rm wa} \Big[\sum_{i=1}^{L}a_i \big|\mathsf{V}_{\delta}(D_{c})\big|^{i\alpha}m^{-i\alpha}\Big]\,\Vert u\Vert_{B^{+},D^{\ast\ast}},}
 \end{array}
 \end{equation}
 where $D_{c} = D^{\ast}\setminus D$, $\mathsf{V}_{\delta}({D}_{c})= \big\{{\bm x}\in {D}^{\ast\ast}: {\rm dist}({\bm x}, {D}_{c})\leq \delta \big\}$, and $C_{\rm wa}$ and $\alpha$ are positive constants independent of $a_i$, $D^{\ast\ast}$, $D^{\ast}$, $D$, and $m$.

 \begin{rem}
When $L=1$, $\mathcal{L}_{1}(D) = \mathcal{L}(D)$, and $A_{D,1} = id$, the generalized Caccioppoli-type inequality reduces to the original inequality \cref{eq:3-1} with $C_{\rm cac}^{\rm II} = 0$, and the generalized weak approximation property reduces to case $a$ of \cref{ass:3-1-2}.
     
 \end{rem}


Based on the generalized Caccioppoli-type inequality and weak approximation property, we have a similar exponential bound for the $n$-widths as \cref{thm:3-1}.
\begin{theorem}\label{thm:3-2}
Let the generalized Caccioppoli-type inequality and weak approximation property be satisfied. Then, there exist $n_{0}>0$ and $c>0$ such that for all $n>n_{0}$,
\begin{equation}
 d_{n}(\omega, \omega^{\ast}) \leq \Vert P_{B}\Vert \,e^{-cn^{\alpha}}.
\end{equation}   
\end{theorem}
\begin{proof}
The proof is very similar to that of \cref{thm:3-1}, and hence we only give the sketch here. The first step is to establish a similar result as \cref{lem:3-1}(i) with estimate \cref{eq:3-6-0} replaced by
\begin{equation*}
\inf_{v\in S_{m}(D^{\ast})} \Vert u-v\Vert_{B^{+},D} \leq 
\sqrt{L}C_{\rm wa}{C^{\rm I}_{\rm cac}} \Big[\sum_{i=1}^{L} (\delta/2)^{-i}m^{-i\alpha} |\mathsf{V}_{\delta/2}({D}^{\ast}\setminus D)\big|^{i\alpha}\Big] \Vert u\Vert_{B^{+},D^{\ast}}.
\end{equation*}
To do this, we use the norm $\vertiii{u}_{\mathcal{L}({D}_{c})}:=\Big(\sum_{i=1}^{L} (\delta/2)^{-2i} \big\Vert A_{D_c,i}(u)\big\Vert^{2}_{\mathcal{L}_{i}(D_c)} \Big)^{1/2}$ for the space $\mathcal{L}(D_c)$ so that it is a Hilbert space (\cref{hp:norm-equiv} implies that the two norms $\vertiii{\cdot}_{\mathcal{L}({D}_{c})}$ and $\Vert\cdot\Vert_{\mathcal{L}({D}_{c})}$) are equivalent), and use the same arguments as for \cref{lem:3-1}. With this auxiliary result, we then prove \cref{lem:3-1-0} accordingly by following along the same lines, and the desired result follows by choosing $m\sim N^{1/\alpha-1}$ as in \cref{eq:3-20}.
\end{proof}

Representative applications of the abstract framework developed above are polyharmonic problems; see \cref{sec:fourth-order-problems} for its application to a biharmonic-type problem. As before, the generalized Caccioppoli-type inequality and the approximation theory can be extended to the case $C^{\rm II}_{\rm cac}\neq 0$, which arises in some applications, e.g., biharmonic wave problems. We omit this extension to avoid overloading the paper.


%

\subsection{Quasi-optimal global convergence}\label{sec:3-2}
This subsection is devoted to proving quasi-optimal convergence for the abstract MS-GFEM. We distinguish three cases: the symmetric and coercive case (the elliptic case), the nonsymmetric and coercive case, and the indefinite case. For the sake of convenience, we introduce and use the following notation in this subsection:
\begin{equation}\label{eq:3-28-1}
\begin{array}{lll}
{\displaystyle  d_{\mathtt{max}} = \max_{1\leq i\leq M}d_{n_i-l_i}(\omega_i,\omega_{i}^{\ast}),\quad  H^{\ast}_{\mathtt{max}} = \max_{1\leq i\leq M} {\rm diam}\,(\omega_{i}^{\ast}).}
\end{array}
\end{equation}
Moreover, we let $\zeta^{\ast}>0$ be the smallest integer such that 
\begin{equation}\label{eq:3-28-0}
\sum_{i=1}^{M}\big\Vert u|_{\omega_i^{\ast}}\big\Vert^{2}_{\mathcal{L}(\omega_i^{\ast})}\leq \zeta^{\ast} \Vert u\Vert^{2}_{\mathcal{L}(\Omega)},\quad {\rm and}\quad \sum_{i=1}^{M}\big\Vert v|_{\omega_i^{\ast}}\big\Vert^{2}_{B^{+},\omega_i^{\ast}}\leq \zeta^{\ast} \Vert v\Vert^{2}_{\mathcal{H}(\Omega)}
\end{equation}
hold for all $u\in \mathcal{L}(\Omega)$ and $v\in \mathcal{H}(\Omega)$. In general, $\zeta^{\ast}$ is given by the maximal number of overlaps of the $\omega_i^{\ast}$'s; see \cref{def:2-1-1}. 


Before proving the quasi-optimal convergence, let us derive global approximation error estimates for the method. As before, we assume that in the elliptic case, the energy norm defined by \cref{eq:elliptic-norm} is used, and that for $i=1,\cdots,M$, $B^{+}_{\omega_i^{\ast}}=B_{\omega_i^{\ast}}$, and $\psi_{i}\in \mathcal{H}_{0}(\omega_i^{\ast})$. Recall that $\zeta$ denotes the coloring constant of the subdomains $\{ \omega_{i}\}_{i=1}^{M}$.

\begin{lemma}\label{lem:3-2}
Let the global particular function $u^{p}$ and the global approximation space $S_{n}(\Omega)$ be defined by \cref{eq:2-9}, and let $u^{{e}}$ be the solution of problem \cref{eq:2-3}. Then, 
\begin{equation}\label{eq:3-29}
\inf_{\varphi \in u^{p} + S_{n}(\Omega)}\Vert u^{{e}}-\varphi \Vert_{\mathcal{H}(\Omega)}\leq d_{\mathtt{max}} \Big(\zeta \sum_{i=1}^{M}\big\Vert u^{e}|_{\omega_i^{\ast}} - \psi_{i}\big\Vert^{2}_{B^{+},\omega^{\ast}_i}\Big)^{1/2},
\end{equation}
where $\psi_{i}$ are given by \cref{ass:2-2-1}. Moreover, in the elliptic case,
\begin{equation}\label{eq:3-29-0}
\inf_{\varphi \in u^{p} + S_{n}(\Omega)}\Vert u^{{e}}-\varphi \Vert_{\mathcal{H}(\Omega)}\leq  \sqrt{\zeta \zeta^{\ast}}\, {d}_{\mathtt{max}}\Vert u^{{e}}\Vert_{\mathcal{H}(\Omega)}.
\end{equation}
\end{lemma}

\begin{proof}
Estimate \cref{eq:3-29} follows from \cref{thm:2-1,thm:local_approximation_error}. Similarly, \cref{eq:3-29-0} is an easy consequence of \cref{thm:2-1}, \cref{local_error_estimate_elliptic}, and \cref{eq:3-28-0}.
\end{proof}

To prove the global convergence of the method, we first recall \cref{ass:2-1-0} and note that combining \cref{eq:2-3,eq:2-10} yields the following Galerkin orthogonality:
\begin{equation}\label{eq:3-31}
B(u^{e}-u^{G},v) = 0\quad \forall v\in S_{n}(\Omega).
\end{equation}
The following lemma gives global error estimates for the method in coercive cases (i.e., $C_{0} = 0$ in \cref{ass:2-1-0}), which follows from C\'{e}a's Lemma and \cref{lem:3-2}.
\begin{theorem}\label{quasi-optimal-coercive}
Let $u^{{e}}$ be the solution of problem \cref{eq:2-3} and $u^{G}$ be the abstract GFEM approximation defined by \cref{eq:2-10}. Then, in the elliptic case,
\begin{equation}\label{eq:3-32}
\Vert u^{{e}}-u^{G} \Vert_{\mathcal{H}(\Omega)} \leq \sqrt{\zeta \zeta^{\ast}}\, {d}_{\mathtt{max}}\Vert u^{{e}}\Vert_{\mathcal{H}(\Omega)}, 
\end{equation} 
and in general coercive cases,
\begin{equation}\label{eq:3-32-0}
\Vert u^{{e}}-u^{G} \Vert_{\mathcal{H}(\Omega)} \leq \big(C_{b}/C_{1}\big)  d_{\mathtt{max}} \Big(\zeta \sum_{i=1}^{M}\big\Vert u^{e}|_{\omega_i^{\ast}} - \psi_{i}\big\Vert^{2}_{B^{+},\omega^{\ast}_i}\Big)^{1/2},
\end{equation} 
where $C_{b}$ and $C_{1}$ are the constants as in \cref{ass:2-1-0}.
\end{theorem}
\begin{rem}\label{adaptive-remarks}
Combining \cref{eq:3-32}, \cref{lem:2-2-1,lem:2-2-2}, we see that in the elliptic case,
\begin{equation}\label{explicit-error-bound}
\Vert u^{{e}}-u^{G} \Vert_{\mathcal{H}(\Omega)}/\Vert u^{{e}}\Vert_{\mathcal{H}(\Omega)} \leq  \sqrt{\zeta \zeta^{\ast}}\Big(\max_{1\leq i\leq M}\lambda^{1/2}_{n_i+1}\Big).
\end{equation}
Therefore, in a discrete setting, we can preselect a suitable eigenvalue threshold to achieve "any" desired accuracy. In this sense, our method is inherently adaptive. 
\end{rem}


For indefinite problems, however, some resolution conditions are required to ensure the quasi-optimal convergence of the method. To find such conditions, let us introduce the solution operator $\widehat{S}: \mathcal{L}(\Omega)\rightarrow \mathcal{H}(\Omega)$ for the problem:
\begin{equation}\label{eq:3-33}
{\rm Find} \;\;\hat{u}\in \mathcal{H}(\Omega)\quad {\rm such\;\,that}\quad B(v,\hat{u}) = (v,f)_{\mathcal{L}(\Omega)}\quad \forall v\in \mathcal{H}(\Omega),
\end{equation}
i.e., $\widehat{S}(f) := \hat{u}$, and define 
\begin{equation}\label{eq:3-34}
\eta (S_{n}(\Omega)):=\sup_{f\in \mathcal{L}(\Omega)}\inf_{\varphi\in S_{n}(\Omega)}\frac{\big\Vert \widehat{S}(f)-\varphi \big\Vert_{\mathcal{H}(\Omega)}}{\Vert f \Vert_{\mathcal{L}(\Omega)}}.
\end{equation}
By \cref{ass:2-0}, problem \cref{eq:3-33} is uniquely solvable and thus $\widehat{S}$ is well defined. Using a standard duality argument (see, e.g., \cite[Theorem 3.11]{chupeng2023wavenumber}) and recalling \cref{ass:2-1-0} with constants $C_{b}$, $C_{1}$, and $C_{0}$, we can prove the following result.
\begin{theorem}\label{thm:3-3}
Let $u^{{e}}$ be the solution of problem \cref{eq:2-3}. Assuming that
\begin{equation}\label{eq:3-35}
\eta (S_{n}(\Omega)) \leq \Big(\frac{C_1}{2C_0C_{b}^{2}}\Big)^{1/2},
\end{equation}
then problem \cref{eq:2-10} has a unique solution $u^{G}\in u^{p}+S_{n}(\Omega)$ satisfying
\begin{equation}\label{eq:3-36}
\Vert u^{{e}}-u^{G}\Vert_{\mathcal{H}(\Omega)} \leq \big(2C_{b}/C_{1}\big)\inf_{\varphi\in u^{p} + S_{n}(\Omega)}\Vert u^{{e}}-\varphi \Vert_{\mathcal{H}(\Omega)}.
\end{equation}
\end{theorem}

It remains to estimate the quantity $\eta (S_{n}(\Omega))$. To do this, we need additional assumptions. The first assumption is concerning the link between problem \cref{eq:3-33} and its adjoint. As will be shown in \cref{sec:5}, it holds naturally for convection-diffusion problems and time harmonic wave problems.
\begin{assumption}\label{ass:3-2-1}
Let $\widehat{S}$ be the solution operator of problem \cref{eq:3-33}. For each $f\in \mathcal{L}(\Omega)$, there exists $g\in \mathcal{L}(\Omega)$ such that $\widehat{S}({f})$ or $\overline{\widehat{S}({f})}$ is the solution of the problem:
\begin{equation}\label{eq:3-43}
{\rm Find} \;\;\hat{u}\in \mathcal{H}(\Omega)\quad {\rm such\;\,that}\quad B(\hat{u},v) = (g,v)_{\mathcal{L}(\Omega)}\quad \forall v\in \mathcal{H}(\Omega).
\end{equation}
In addition, there exists a constant $C_{\rm ad}>0$ independent of $f$ such that
\begin{equation}\label{eq:3-44}
\Vert g\Vert_{\mathcal{L}(\Omega)}\leq C_{\rm ad}\Vert f\Vert_{\mathcal{L}(\Omega)}.
\end{equation}
\end{assumption}
Given $g_{i}\in \mathcal{L}(\omega_i^{\ast})$, $i=1,\cdots,M$, we consider the following local problems:
\begin{equation}\label{eq:3-45}
{\rm Find} \;\;u_{i}\in \mathcal{H}_{0}(\omega_{i}^{\ast}) \quad {\rm such\;\,that} \quad B_{\omega_{i}^{\ast}}(u_{i},v) = (g_i,v)_{\mathcal{L}(\omega_i^{\ast})}\quad \forall v\in \mathcal{H}_{0}(\omega_{i}^{\ast}).
\end{equation}
Our second assumption is on the stability of the local problems \cref{eq:3-45}. Recall that $H^{\ast}_{\mathtt{max}} = \max_{i} {\rm diam}\,(\omega_{i}^{\ast})$.
\begin{assumption}\label{ass:3-2-2}
There exists a constant $C_{s}>0$ depending only on the shapes of the oversampling domains $\{\omega_{i}^{\ast}\}$, such that if 
\begin{equation}\label{eq:3-45-0}
H^{\ast}_{\mathtt{max}} \leq (C_{1}/C_{0})^{1/2}C^{-1}_{s},
\end{equation}
the problems \cref{eq:3-45} are uniquely solvable for all $g_{i}\in \mathcal{L}(\omega_i^{\ast})$ with the estimates:
\begin{equation}\label{eq:3-45-1}
\Vert u_{i}\Vert_{\mathcal{H}_{0}(\omega_i^{\ast})} \leq \big(C_{s} H^{\ast}_{\mathtt{max}} /C_{1}) \Vert g_{i}\Vert_{\mathcal{L}(\omega_i^{\ast})}, \quad i=1,\cdots,M. 
\end{equation}
\end{assumption}

For each $i=1,\ldots,M$, we define the operator
\begin{equation}\label{eq:3-47-0}
\widehat{P}_{i,0}:\mathcal{H}_{0}(\omega_i^{\ast})\rightarrow \mathcal{H}_{0}(\omega_i)\qquad \widehat{P}_{i,0}v = {P}_{i}(v|_{\omega_i})\quad \forall v\in \mathcal{H}_{0}(\omega_i^{\ast}),
\end{equation}
where ${P}_{i}:\mathcal{H}(\omega_i)\rightarrow \mathcal{H}_{0}(\omega_i)$ is the partition of unity operator. Now we are ready to estimate $\eta (S_{n}(\Omega))$.

\begin{lemma}\label{lem:3-2-2}
Let \cref{ass:3-2-1} and \ref{ass:3-2-2} be satisfied. Supposing that
\begin{equation}\label{eq:3-47}
H^{\ast}_{\mathtt{max}} \leq (C_{1}/C_{0})^{1/2}C^{-1}_{s},
\end{equation}
then,
\begin{equation}\label{eq:3-48}
\eta (S_{n}(\Omega))\leq\sqrt{\zeta \zeta^{\ast}}C_{\rm ad}\Big(\sqrt{2} d_{\mathtt{max}} \big(C_{\rm stab}+ \Lambda \big)  + \big(\max_{1\leq i\leq M}\big\Vert \widehat{P}_{i,0} \big\Vert \big)\Lambda \Big),
\end{equation}
where $\Lambda = C_{s} H^{\ast}_{\mathtt{max}} /C_{1}$ and $\big\Vert \widehat{P}_{i,0} \big\Vert$ denotes the operator norm of $\widehat{P}_{i,0}$.
\end{lemma}
\begin{proof}
Let $f\in\mathcal{L}(\Omega)$ be fixed. By \cref{ass:3-2-1}, $\widehat{S}(f)$ or $\overline{\widehat{S}(f)}$ is the solution of \cref{eq:3-43} for some $g\in \mathcal{L}(\Omega)$ with 
\begin{equation}\label{eq:3-48-0}
\Vert g\Vert_{\mathcal{L}(\Omega)}\leq C_{\rm ad}\Vert f\Vert_{\mathcal{L}(\Omega)}.
\end{equation}
We may suppose that $\hat{u} = \widehat{S}(f)$ is the solution of \cref{eq:3-43}; the other case follows by an identical argument. It follows from \cref{ass:2-0} and \cref{eq:3-48-0} that
\begin{equation}\label{eq:3-49}
\Vert \hat{u}\Vert_{\mathcal{H}(\Omega)}\leq C_{\rm stab}\Vert g\Vert_{\mathcal{L}(\Omega)}\leq C_{\rm stab} C_{\rm ad}\Vert f\Vert_{\mathcal{L}(\Omega)}.
\end{equation}
Next we consider the following local problems on the oversampling domains:
\begin{equation}\label{eq:3-50}
{\rm Find} \;\;\hat{\psi}_{i}\in \mathcal{H}_{0}(\omega_{i}^{\ast}) \quad {\rm such\;\,that} \quad B_{\omega_{i}^{\ast}}(\hat{\psi}_{i},v) = (g|_{\omega_i^{\ast}},v)_{\mathcal{L}(\omega_i^{\ast})}\quad \forall v\in \mathcal{H}_{0}(\omega_{i}^{\ast}).
\end{equation}
Let $\Lambda = C_{s} H^{\ast}_{\mathtt{max}}/C_{1}$. By means of \cref{ass:3-2-2}, we see that under the condition \cref{eq:3-47}, the problems \cref{eq:3-50} are uniquely solvable with the estimates:
\begin{equation}\label{eq:3-53}
\Vert \hat{\psi}_{i} \Vert_{\mathcal{H}_{0}(\omega_{i}^{\ast})} \leq \Lambda \big\Vert g|_{\omega_i^{\ast}}\big\Vert_{\mathcal{L}(\omega_{i}^{\ast})}.
\end{equation}
Combining \cref{eq:3-50}, \cref{ass:2-1-1} ($iv$), and \cref{eq:3-43}, 
we see that $\hat{u}|_{\omega_{i}^{\ast}}-\hat{\psi}_{i} \in \mathcal{H}_{B}(\omega_{i}^{\ast})$. Then, it follows from \cref{thm:local_approximation_error} that there exists a $\varphi_{i}\in S_{n_i}(\omega_i)$ such that
\begin{equation}\label{eq:3-54}
\begin{array}{lll}
{\displaystyle \big \Vert P_{i}(\hat{u}|_{\omega_i}-\hat{\psi}_{i}|_{\omega_i}-\varphi_{i})\big\Vert_{\mathcal{H}_{0}(\omega_{i})}\leq d_{n_i-l_i}(\omega_i,\omega_{i}^{\ast})\,\big \Vert \hat{u}|_{\omega_i^{\ast}}-\hat{\psi}_{i}\big \Vert_{B^{+},\omega_{i}^{\ast}}. }
\end{array}
\end{equation}
For simplicity and without loss of generality, let us assume $\Vert \hat{\psi}_{i} \Vert_{B^{+},\omega_i^{\ast}}\leq \Vert \hat{\psi}_{i} \Vert_{\mathcal{H}_{0}(\omega_{i}^{\ast})}$. Then, combining \cref{eq:3-53,eq:3-54} gives
\begin{equation}\label{eq:3-54-0}
\big\Vert P_{i}(\hat{u}|_{\omega_i}-\hat{\psi}_{i}|_{\omega_i}-\varphi_{i})\big\Vert_{\mathcal{H}_{0}(\omega_{i})}\leq d_{n_i-l_i}(\omega_i,\omega_{i}^{\ast})\,\big(\big\Vert \hat{u}|_{\omega_i^{\ast}}\big\Vert_{B^{+},\omega_{i}^{\ast}} + \Lambda \big\Vert g|_{\omega_i^{\ast}}\big\Vert_{\mathcal{L}(\omega_{i}^{\ast})}\big).
\end{equation}
Let $\hat{u}^{p} = \sum^{M}_{i=1}P_{i}(\hat{\psi}_{i}|_{\omega_i})$ and $\varphi = \sum^{M}_{i=1}P_{i}\varphi_{i}$. By means of \cref{eq:3-54-0} and \cref{eq:3-28-0}, we can argue similarly as in the proof of \cref{lem:3-2} to obtain
\begin{equation}\label{eq:3-55}
 \Vert \hat{u}-\hat{u}^{p}-\varphi\Vert_{\mathcal{H}(\Omega)}\leq \sqrt{2\zeta \zeta^{\ast}} d_{\mathtt{max}} \big(\Vert \hat{u}\Vert_{\mathcal{H}(\Omega)} +\Lambda \Vert g\Vert_{\mathcal{L}(\Omega)}\big).
\end{equation}
Inserting \cref{eq:3-48-0,eq:3-49} into \cref{eq:3-55} and using a triangle inequality further give
\begin{equation}\label{eq:3-56}
\Vert \hat{u}-\varphi\Vert_{\mathcal{H}(\Omega)} \leq \Vert\hat{u}^{p} \Vert_{\mathcal{H}(\Omega)} +\sqrt{2\zeta \zeta^{\ast}} d_{\mathtt{max}} C_{\rm ad}(C_{\rm stab} + \Lambda ) \Vert f\Vert_{\mathcal{L}(\Omega)}.
\end{equation}
It remains to bound $\Vert\hat{u}^{p} \Vert_{\mathcal{H}(\Omega)}$. Keeping the definition of $\hat{u}^{p}$ in mind and using \cref{eq:3-53} and \cref{eq:3-28-0}, it can be bounded as follows.
\begin{equation}\label{eq:3-57}
\begin{array}{lll}
{\displaystyle \Vert\hat{u}^{p} \Vert_{\mathcal{H}(\Omega)}^{2}\leq \zeta\sum_{i=1}^{M}\big\Vert \widehat{P}_{i,0}\big\Vert^{2} \big\Vert\hat{\psi}_{i}\big\Vert_{\mathcal{H}_{0}(\omega^{\ast}_i)}^{2}\leq \zeta\zeta^{\ast}\Lambda^{2}\Big(\max_{1\leq i\leq M}\big\Vert \widehat{P}_{i,0}\big\Vert^{2}\Big)\big\Vert g\big\Vert_{\mathcal{L}(\Omega)}^{2}.}
\end{array}
\end{equation}
Inserting \cref{eq:3-57} into \cref{eq:3-56} and using \cref{eq:3-48-0} give the desired estimate \cref{eq:3-48}. 
\end{proof}
\begin{rem}
In many applications with 'generous overlap', we have $\big\Vert \widehat{P}_{i,0} \big\Vert = O({\rm diam}(\omega_i^{\ast})/{\rm diam}(\omega_i))$.
\end{rem}

In view of \cref{thm:3-3} and \cref{lem:3-2-2}, it is clear that some resolution conditions on $d_{\mathtt{max}}$ and $H^{\ast}_{\mathtt{max}}$ are required to ensure the quasi-optimal convergence of the method. To identify such conditions, we define $\Xi := \big(\sqrt{\zeta \zeta^{\ast}}C_{\rm ad}\big)^{-1}\big(C_1/(2C_0C_{b}^{2})\big)^{1/2}$. Moreover, we assume that $d_{\mathtt{max}}\leq 1\leq\max_{1\leq i\leq M}\big\Vert \widehat{P}_{i,0}\big\Vert$ without loss of generality. Combining \cref{thm:3-3} and \cref{lem:3-2-2}, we obtain the desired quasi-optimal convergence for the method under the resolution conditions specified below.

\begin{corollary}\label{cor:3-1}
Let $u^{{e}}$ be the solution of problem \cref{eq:2-3} and $u^{G}$ be the abstract GFEM approximation. Supposing that 
\begin{equation}\label{eq:3-59}
d_{\mathtt{max}} \leq \big(4\sqrt{2}C_{\rm stab}\big)^{-1}\,\Xi,\quad  H^{\ast}_{\mathtt{max}} \leq C_{1}\Big(2C_{s}\max_{1\leq i\leq M}\big\Vert \widehat{P}_{i,0}\big\Vert  \Big)^{-1} \,\Xi,
\end{equation}
then
\begin{equation}\label{eq:3-60}
\Vert u^{{e}}-u^{G}\Vert_{\mathcal{H}(\Omega)} \leq \big(2C_{b}/C_{1}\big)\inf_{\varphi\in u^{p} + S_{n}(\Omega)}\Vert u^{{e}}-\varphi \Vert_{\mathcal{H}(\Omega)}.
\end{equation}
\end{corollary}

\section{Low-rank approximations of Green's functions and solution operators}\label{sec:low-rank approximation}
The abstract local approximation theory developed in \cref{sec:3} shows a low-dimension property of the generalized harmonic space, which is closely related to low-rank approximations of the Green's function and solution operator of the underlying problem. It is known that Green's functions of elliptic PDEs can be approximated by separable functions on well-separated domains with exponential convergence \cite{bebendorf2003existence,bebendorf2005efficient}. This analytic property provides a theoretical foundation for many efficient solvers for elliptic PDEs based on low-rank approximations at the discrete level. In the following, we provide a unified framework for establishing approximate separability of Green's functions by means of the abstract approximation theory of MS-GFEM. To do this, we recall the generalized harmonic spaces $\mathcal{H}_{B}(D)$ \cref{generalized_harmonic_space} and define an abstract Green's function associated with problem \cref{eq:2-3}. 

\begin{definition}[Abstract Green's function]\label{Green's functions}
Let $\{ g_{\bm x}\}$ be a family of elements parameterized by ${\bm x}\in \Omega$ in a normed vector space $\mathcal{B}(\Omega)$, such that for any subdomains $D_{1}$ and $D_{2}$ of $\Omega$ with ${\rm dist}(D_{1},D_{2})>0$, $g_{\bm x}\in \mathcal{H}_{B}(D_{2})$ for all ${\bm x}\in D_{1}$. Then, $\{g_{\bm x}\}$ is called an abstract Green's function.    
\end{definition} 
\begin{rem}
Let $G({\bm x},{\bm y}):\Omega\times\Omega\rightarrow \mathbb{R}\cup \{\infty\}$ be the usual Green's functions. Then the abstract Green's functions are typically defined by $\{g_{\bm x}\}:=\{G({\bm x},\cdot)\}$.    
\end{rem}

The following theorem gives upper-bound estimates on approximate separability of the abstract Green's function, which is an abstract analogue of the pioneering result \cite[Theorem 2.8]{bebendorf2003existence}. It is based on \cref{thm:3-1} and holds true as long as the two fundamental assumptions are valid. Importantly, for various practical applications, our result shows that only $O(|\log \epsilon|^{d})$ terms are needed for a separable approximation of the Green's functions with error $\epsilon>0$, as compared with previous results with $O(|\log \epsilon|^{d+1})$ terms. Moreover, our estimates are explicit in important problem parameters (e.g., the wavenumber).


\begin{theorem}[Separable approximations of Green's functions]\label{thm:separable_appro}\\
Let $\{ g_{\bm x}\}$ be an abstract Green's function, and let $D_{1}$ and $D_{2}$ be subdomains of $\Omega$ with ${\rm dist}(D_{1}, D_{2})\geq \rho \,{\rm diam}(D_2)>0$. Let $C^{\rm I}_{\rm cac}$, $C^{\rm II}_{\rm cac}$, and $\alpha$ be the constants as in \cref{ass:3-1-1} and \ref{ass:3-1-2}. Suppose that \cref{ass:3-1-1} holds. Then, there is a separable approximation $g_{n}({\bm x}) = \sum_{i=1}^{n}u_{i}({\bm x})v_{i}$ with $u_{i}({\bm x}):D_{1}\rightarrow \mathbb{C}$ and $v_{i}\in \mathcal{H}_{B}(D_{2})$, such that for all ${\bm x}\in D_{1}$,
\begin{equation}\label{seprable-approximations}
\begin{array}{ll}
\qquad \;\Vert g_{\bm x}  - g_{n}({\bm x})\Vert_{B^{+}, D_{2}}\leq  \textsf{err}_{n} \,\Vert g_{\bm x}\Vert_{B^{+}, D_{2}^{\prime}},\\[1.5ex]
 \text{and} \quad \Vert g_{\bm x}  - g_{n}({\bm x})\Vert_{\mathcal{L}(D_{2})}\leq  \textsf{err}_{n}\, \Vert g_{\bm x}\Vert_{\mathcal{L}(D_{2}^{\prime})}
\end{array}
\end{equation}
with $D_{2}^{\prime} = \{{\bm x}\in \Omega: {\rm dist}({\bm x}, D_{2}) \leq \frac{1}{2} \,{\rm dist}(D_{1},\,D_{2})\}$, where 

\begin{itemize}[nolistsep]
    \item[(i)] $\texttt{err}_{n} = e^{-cn^{\alpha}}$ if case $a$ of \cref{ass:3-1-2} holds and $C^{\rm II}_{\rm cac}=0$;\vspace{1mm}
    
    \item[(ii)] $\texttt{err}_{n} =  e^{\sigma^{\ast}} e^{-cn^{\alpha/(\alpha +1)}}$ with $\sigma^{\ast} = C_{\rm cac}^{\rm II} {\rm dist}(D_{1}, D_{2})/(4C_{\rm cac}^{\rm I})$ if case $b$ of \cref{ass:3-1-2} holds;

    \vspace{1mm}
    \item[(iii)] $\texttt{err}_{n} = e^{-cn^{\alpha}}$ if case $ab$ of \cref{ass:3-1-2} holds.
\end{itemize}
\vspace{1mm}
The constant $c$ above depends on $C^{\rm I}_{\rm cac}$, $C_{\rm wa}$, $\alpha$, $d$, and $\rho$, but not on $C_{\rm cac}^{\rm II}$.
\end{theorem}
\begin{proof}
The proof is based on a simple application of \cref{thm:3-1}. Let us consider the first estimate of \cref{seprable-approximations}. Let $D_{2}^{\prime} = \{{\bm x}\in \Omega: {\rm dist}({\bm x}, D_{2}) \leq \frac{1}{2} \,{\rm dist}(D_{1},\,D_{2})\}$. Then $D_{2}\subset D_{2}^{\prime}$ with ${\rm dist}(D_{2}, \partial D_{2}^{\prime}) = \frac{1}{2} \,{\rm dist}(D_{1},\,D_{2}) \geq \frac{\rho}{2} {\rm diam}(D_2)>0$ and ${\rm dist}(D_{1},\,D^{\prime}_{2})>0$. By \cref{Green's functions}, we see that $g_{\bm x}\in \mathcal{H}_{B}(D_{2}^{\prime})$ for all ${\bm x}\in D_{1}$. Applying \cref{thm:3-1} shows that there exists an $n$-dimensional space $Q_{n}(D_2) \subset \mathcal{H}_{B}(D_2)$ such that
\begin{equation}
\inf_{v\in Q_{n}(D_2)} \Vert {\bm g}_{x}  - v\Vert_{B^{+},D_{2}}\leq \texttt{err}_{n}  \Vert g_{\bm x}\Vert_{B^{+}, D_{2}^{\prime}}.
\end{equation}
Let $\{v_{1},\cdots, v_{n} \}$ be a basis for $Q_{n}(D_2)$. Then the above estimate implies that for all ${\bm x}\in D_{1}$, there is a $g_{n}({\bm x}) = \sum_{i=1}^{n}u_{i}({\bm x}) v_{i}$ with $u_{i}({\bm x})\in \mathbb{C}$ such that
\begin{equation}
  \Vert g_{\bm x}  - g_{n}({\bm x})\Vert_{B^{+}, D_{2}}\leq  \texttt{err}_{n} \,\Vert g_{\bm x}\Vert_{B^{+}, D_{2}^{\prime}} . 
\end{equation}
The second estimate follows from the same argument as above and the remark at the end of \cref{main-proof}.
\end{proof}
\begin{rem}
When $C_{\rm cac}^{\rm II} = 0$, which is typically the case for positive definite problems, the abstract Green's function is highly separable. In the general case, it is highly separable if ${\rm dist}(D_{1}, D_{2})=O((C_{\rm cac}^{\rm II})^{-1})$ (in particular, ${\rm dist}(D_{1}, D_{2})=O(k^{-1})$ for time-harmonic wave problems with $k$ denoting the wavenumber). See \cite{engquist2018approximate} for sharp estimates on the wavenumber-dependence of the approximate separability of the Green's function of the high-frequency Helmholtz equation.
\end{rem}

A theoretical flaw of \cref{thm:separable_appro} in practical applications is that it requires the existence of Green's functions. Indeed, the high separability of a Green's function is equivalent to a low-rank property of the corresponding solution operator. Let $S:\mathcal{H}(\Omega)^{\prime}\rightarrow \mathcal{H}(\Omega)$ be the solution operator associated with problem \cref{eq:2-3}, and let the hypotheses of \cref{thm:separable_appro} be satisfied. The following theorem shows that the solution $S(F)$ can be well approximated by a low-dimensional space in a subdomain disjoint from the support of the right-hand side $F$. This result is an abstract analogue of \cite[Theorem 1]{borm2010approximation} and is often of more use in practice as it does not assume the existence of Green's functions. 
\begin{theorem}
Let $D_{1}$ and $D_{2}$ be subdomains of $\Omega$ such that ${\rm dist}(D_{1}, D_{2})\geq \rho \,{\rm diam}(D_2)>0$, and let \cref{ass:3-1-1} and \ref{ass:3-1-2} be satisfied. Then, there exists an $n$-dimensional space $Q_{n}(D_{2})\subset \mathcal{H}_{B}(D_{2})$, such that for all right-hand sides $F\in \mathcal{H}(\Omega)^{\prime}$ satisfying that $F(v)=0$ for all $v\in \mathcal{H}_{0}(D)$ with $D\cap D_{1} = \emptyset$,
\begin{equation}
\begin{array}{ll}
\displaystyle \qquad \;\inf_{v\in Q_{n}(D_{2})}\Vert S(F)  - v\Vert_{B^{+}, D_{2}}\leq  C\textsf{err}_{n} \,\Vert F\Vert_{\mathcal{H}(\Omega)^{\prime}},\\[2.5ex]
\displaystyle \text{and} \quad \inf_{v\in Q_{n}(D_{2})}\Vert S(F)  - v\Vert_{\mathcal{L}(D_{2})}\leq  C\textsf{err}_{n} \,\Vert F\Vert_{\mathcal{H}(\Omega)^{\prime}},
\end{array}
\end{equation}
where $C$ depends on the stability constant of problem \cref{eq:2-3}, and $\textsf{err}_{n}$ is the same error bound as in \cref{thm:separable_appro}.
\begin{proof}
Use the assumption on $F$ and proceed along the same lines as in the proof of \cref{thm:separable_appro}.    
\end{proof}
\end{theorem}

The low-rank property of Green's functions and solution operators implies that off-diagonal submatrices appearing during matrix factorizations for the corresponding discretized problems have low-rank approximations. Such dense matrices can then be effectively compressed by means of various rank-structured matrices, thereby greatly reducing memory requirements and the complexity of matrix operations.

\section{Verification of the two fundamental conditions}\label{sec:verification_of_two_conditions}
The verification of the Caccioppoli-type inequality (\cref{ass:3-1-1}) and the weak appoximation property (\cref{ass:3-1-2}) is essential to the application of the abstract framework, which is discussed in this section.

\subsection{Verification of the Caccioppoli-type inequality} We distinguish two cases: the continuous setting and the discrete (FE) setting. First we consider the continuous setting and recall the classical Caccioppoli inequality for Poisson-type problems. Let $D^{\ast}$ denote a subdomain of $\Omega$. 
\begin{lemma}\label{lem:5-1}
Let $a\in L^{\infty}(D^{\ast})$ with $a({\bm x})>0$ for $a.e.\;{\bm x}\in D^{\ast}$, and let $u\in H^{1}(D^{\ast})$ satisfy $\int_{D^{\ast}}a\nabla u \cdot\nabla \varphi d{\bm x} = 0$ for all $\varphi\in H_{0}^{1}(D^{\ast})$. Then, for any $\eta\in C_{0}^{1}(D^{\ast})$,
\begin{equation}\label{eq:5-2}
    \int_{D^{\ast}}a\eta^{2}|\nabla u|^{2} d{\bm x} \leq 4\int_{D^{\ast}} a u^{2} |\nabla \eta|^{2} d{\bm x}.
\end{equation}
\end{lemma}

Choosing a cut-off function as $\eta$ in \cref{lem:5-1} gives the Caccioppoli inequality of the form \cref{eq:3-1}. This inequality can be extended to other (typically second-order) PDEs, such as linear elasticity equations, the Helmholtz equation, and Maxwell's equations. In the following, we present a unified framework for proving (sharper) Caccioppoli inequalities for various PDEs, based on an abstract \textit{identity} that can be proved under a few easily verifiable conditions. Apart from yielding Caccioppoli inequalities, this identity also provides a novel, equivalent formulation of the local eigenproblems \cref{eq:2-24-5} in practical applications. To give this abstract identity, let us recall the local spaces $\mathcal{H}(D^{\ast})$ and $\mathcal{H}_{0}(D^{\ast})$ given by \cref{ass:2-1-0}. We also need another local Hilbert space $\mathbf{L}(D^{\ast})$. 

\begin{definition}\label{def:5-1}
Let $\mathcal{H}(D^{\ast})$, $\mathcal{H}_{0}(D^{\ast})$ and $\mathbf{L}(D^{\ast})$ be given. We call $(\partial, \eta, \widehat{\partial \eta})$ an $\eta$-derivation triple if $\partial: \mathcal{H}(D^{\ast})\rightarrow \mathbf{L}(D^{\ast})$, $\eta: \mathbf{L}(D^{\ast})\rightarrow \mathbf{L}(D^{\ast})$,  $\mathcal{H}(D^{\ast})\rightarrow \mathcal{H}_{0}(D^{\ast})$, and $\widehat{\partial \eta}: \mathcal{H}(D^{\ast})\rightarrow \mathbf{L}(D^{\ast})$ are linear operators, and satisfy

\vspace{2mm}
\begin{itemize}
    \item  commutativity: $\widehat{\partial \eta} (\eta(u)) = \eta(\widehat{\partial \eta} (u)) $ for all $u\in \mathbf{L}(D^{\ast})$;
\vspace{2mm}

    \item Leibniz product rule: $\partial (\eta(v)) = \widehat{\partial \eta} (v) + \eta(\partial (v))$ for all $v\in \mathcal{H}(D^{\ast})$. 
\end{itemize}
\end{definition}
Before proceeding, we give some examples of the $\eta$-derivation triple. Let $\eta\in C_{0}^{1}(D^{\ast})$.

\begin{exam}\label{exam:5-1}
Let $\mathcal{H}(D^{\ast}):= H^{1}(D^{\ast})$, $\mathcal{H}_{0}(D^{\ast}):= H^{1}_{0}(D^{\ast})$, and $\mathbf{L}(D^{\ast}):= (L^{2}(D^{\ast}))^{d}$. For $u\in H^{1}(D^{\ast})$, we define $\partial (u) = \nabla u$, $\eta(u) = \eta u$, and $\widehat{\partial \eta}(u) = (\nabla \eta) u$.      
\end{exam}

\begin{exam}\label{exam:5-2}
Let $\mathcal{H}(D^{\ast}):= (H^{1}(D^{\ast}))^{d}$, $\mathcal{H}_{0}(D^{\ast}):= (H^{1}_{0}(D^{\ast}))^{d}$, and $\mathbf{L}(D^{\ast}):= (L^{2}(D^{\ast}))^{d\times d}$. For ${\bf u}\in (H^{1}(D^{\ast}))^{3}$, we define $\partial ({\bf u}) = \big(\nabla {\bf u} + (\nabla {\bf u})^{T}\big)/2$, $\eta({\bf u}) = \eta {\bf u}$, and $\widehat{\partial \eta}({\bf u}) = \big({\bf u}\cdot(\nabla \eta)^{T} + (\nabla \eta)\cdot {\bf u}^{T}\big)/2$.      
\end{exam}

\begin{exam}\label{exam:5-3}
Let $\mathcal{H}(D^{\ast}):= {\bf H}({\rm curl};D^{\ast})$, $\mathcal{H}_{0}(D^{\ast}):= {\bf H}_{0}({\rm curl};D^{\ast})$, and $\mathbf{L}(D^{\ast})\\:= (L^{2}(D^{\ast}))^{d}$. For ${\bf u}\in {\bf H}({\rm curl};D^{\ast})$, we define $\partial ({\bf u}) = \nabla \times {\bf u}$, $\eta({\bf u}) = \eta {\bf u}$, and $\widehat{\partial \eta}({\bf u}) = (\nabla \eta)\times {\bf u}$.      
\end{exam}

\begin{exam}\label{exam:5-4}
Let $\mathcal{H}(D^{\ast}):= {\bf H}({\rm div};D^{\ast})$, $\mathcal{H}_{0}(D^{\ast}):= {\bf H}_{0}({\rm div};D^{\ast})$, and $\mathbf{L}(D^{\ast})\\:= L^{2}(D^{\ast})$. For ${\bf u}\in \mathbf{H}({\rm div};D^{\ast})$, we define $\partial ({\bf u}) = \nabla \cdot {\bf u}$, $\eta({\bf u}) = \eta {\bf u}$, and $\widehat{\partial \eta}({\bf u}) = (\nabla \eta)\cdot {\bf u}$.      
\end{exam}
Note that in the examples above, the spaces $\mathcal{H}(D^{\ast})$ can be replaced by their FE subspaces. Next we prove the desired identity. For ease of notation, the brackets are omitted in writing the action of operators on a vector. 
\begin{proposition}[Abstract Caccioppoli identity]\label{prop:5-1}
Given $\mathcal{H}(D^{\ast})$, $\mathcal{H}_{0}(D^{\ast})$ and $\mathbf{L}(D^{\ast})$, let the $\eta$-derivation triple $(\partial, \eta, \widehat{\partial \eta})$ be as in \cref{def:5-1}. Let $b_{\mathbf{L}}(\cdot,\cdot):\mathbf{L}(D^{\ast})\times\mathbf{L}(D^{\ast})\rightarrow\mathbb{C}$ be a Hermitian sesquilinear form that satisfies
\begin{equation}\label{eq:5-3}
    b_{\mathbf{L}}(u, \eta v) = b_{\mathbf{L}}(\eta u, v)\quad \text{for all}\;\; u,v\in \mathbf{L}(D^{\ast}).
\end{equation}
Then for any $u,v\in \mathcal{H}(D^{\ast})$, 
\begin{equation}\label{eq:5-4}
    b_{\mathbf{L}}\big(\partial (\eta u), \partial (\eta v)\big) = b_{\mathbf{L}}(\widehat{\partial \eta} u , \widehat{\partial \eta} v) + \frac{1}{2}\Big( b_{\mathbf{L}}\big(\partial u, \partial (\eta^{2} v)\big) +  \overline{b_{\mathbf{L}}\big(\partial v, \partial (\eta^{2} u)\big) } \Big). 
\end{equation}
\end{proposition}
\begin{proof}
We prove the result by some simple algebraic manipulations based on the given conditions. By the Leiniz product rule, we see that for any $u,v\in \mathcal{H}(D^{\ast})$,
\begin{equation}\label{eq:5-5}
\begin{array}{lll}
{\displaystyle    b_{\mathbf{L}}\big(\partial (\eta u), \partial (\eta v)\big) =  b_{\mathbf{L}}\big(\widehat{\partial \eta} u, \partial (\eta v)\big) + b_{\mathbf{L}}\big(\eta \partial u, \partial (\eta v)\big) }\\[2mm]
{\displaystyle \quad =  b_{\mathbf{L}}\big(\widehat{\partial \eta} u, \widehat{\partial \eta} v\big) + b_{\mathbf{L}}\big(\widehat{\partial \eta} u,\eta \partial v\big) + b_{\mathbf{L}}\big(\eta \partial u, \partial (\eta v)\big). }
   \end{array}
\end{equation}
Using \cref{eq:5-3}, the Leibniz product rule, and the commutative property, we have
\begin{equation}\label{eq:5-6}
\begin{array}{lll}
{\displaystyle  b_{\mathbf{L}}\big(\eta \partial u, \partial (\eta v)\big) = b_{\mathbf{L}}\big( \partial u, \eta \partial (\eta v)\big)  =  b_{\mathbf{L}}\big( \partial u, \partial (\eta^{2} v)\big) - b_{\mathbf{L}}\big( \partial u, \widehat{\partial \eta}  (\eta v)\big) }\\[2mm]
  {\displaystyle \qquad \qquad \qquad \quad =  b_{\mathbf{L}}\big( \partial u, \partial (\eta^{2} v)\big) - b_{\mathbf{L}}\big( \eta \partial u, \widehat{\partial \eta} v\big).}
\end{array}
\end{equation}
Inserting \cref{eq:5-6} into \cref{eq:5-5} gives 
\begin{equation}\label{eq:5-7}
\begin{array}{cc}
{\displaystyle   b_{\mathbf{L}}\big(\partial (\eta u), \partial (\eta v)\big) = b_{\mathbf{L}}\big(\widehat{\partial \eta} u, \widehat{\partial \eta} v\big) + b_{\mathbf{L}}\big( \partial u, \partial (\eta^{2} v)\big)}\\[2mm]
{\displaystyle \qquad \qquad \qquad \quad  +\,b_{\mathbf{L}}\big(\widehat{\partial \eta} u,\eta \partial v\big)  - b_{\mathbf{L}}\big( \eta \partial u, \widehat{\partial \eta} v\big). }
\end{array}
\end{equation}
Using the Hermitian property of the form $b_{\mathbf{L}}(\cdot,\cdot)$, we can exchange $u$ and $v$ in \cref{eq:5-7} and then take the complex conjugate of the resulting equation to obtain
\begin{equation}\label{eq:5-8}
\begin{array}{cc}
{\displaystyle   b_{\mathbf{L}}\big(\partial (\eta u), \partial (\eta v)\big) = b_{\mathbf{L}}\big(\widehat{\partial \eta} u, \widehat{\partial \eta} v\big) + \overline{b_{\mathbf{L}}\big( \partial v, \partial (\eta^{2} u)\big)}}\\[2mm]
{\displaystyle \qquad \qquad \qquad \quad  +\,b_{\mathbf{L}}\big( \eta \partial u, \widehat{\partial \eta} v\big) - b_{\mathbf{L}}\big(\widehat{\partial \eta} u,\eta \partial v\big).}
\end{array}
\end{equation}
Adding \cref{eq:5-7,eq:5-8} together yields the desired identity \cref{eq:5-4}.
\end{proof}

Now we give a simple example to illustrate the application of the abstract identity. Consider the setting of Example~\ref{exam:5-1}, and let $a({\bm x})$ be as in \cref{lem:5-1}. We define $b_{\bf L}(\cdot,\cdot):(L^{2}(D^{\ast}))^{d}\times (L^{2}(D^{\ast}))^{d}\rightarrow \mathbb{R}$ as $b_{\bf L}({\bf u},{\bf v}) = \int_{D^{\ast}}a({\bm x}){\bf u}\cdot {\bf v}d{\bm x}$. For any $a$-harmonic functions $u$ and $v$, since $\eta^{2}u,\eta^{2}v\in H_{0}^{1}(D^{\ast})$, we see that $ b_{\mathbf{L}}\big(\partial u, \partial (\eta^{2} v)\big)= b_{\mathbf{L}}\big(\partial v, \partial (\eta^{2} u)\big) =0$, and hence we can deduce from the identity \cref{eq:5-4} that
\begin{equation}\label{eq:5-10}
 \int_{D^{\ast}}a({\bm x})\nabla (\eta u)\cdot \nabla(\eta v)\,d{\bm x} = \int_{D^{\ast}}a({\bm x})|\nabla \eta|^{2}uv\,d{\bm x}.   
\end{equation}
Taking $u=v$ in \cref{eq:5-10} and choosing $\eta$ a cut-off function, we get the classical Caccioppoli inequality (with a sharper constant). Similar identities as \cref{eq:5-10} hold for various other PDEs, and they can be used to rewrite the local eigenproblems \cref{eq:2-24-5} as novel, simpler ones; see \cref{sec:5}.

Although the abstract identity also holds at the discrete level (with $\mathcal{H}(D^{\ast})$ being FE subspaces), it does not give discrete Caccioppoli inequalities in a straightforward way as at the continuous level. The reason is that $\eta^{2} u$ ($u\in \mathcal{H}(D^{\ast})$) generally does not belong to the corresponding FE subspace of $\mathcal{H}_{0}(D^{\ast})$. In practice, the proof of discrete Caccioppoli inequalities requires a more sophisticated analysis, which often hinges on certain FE super-approximation results. A widely used super-approximation result for Lagrange finite elements is given as follows (see \cite[Theorem 2.1]{demlow2011local}).
\begin{lemma}\label{lem:5-2}
Let $\eta\in C^{\infty}(\Omega)$ satisfy $|\eta|_{W^{j,\infty}(\Omega)}\leq C\delta^{-j}$ for each $j\in \mathbb{N}$, and let $U_{h}\subset H^{1}(\Omega)$ be a standard Lagrange finite element subspace. Then for each $u_{h}\in U_{h}$ and each element $K$ with $h_{K} :={\rm diam}(K)\leq \delta$, 
\begin{align}
{\displaystyle \Vert \eta^{2}u_{h}- I_{h}(\eta^{2}u_{h})\Vert_{H^{1}(K)}\leq C\big(\frac{h_{K}}{\delta}\Vert \nabla (\eta u_{h})\Vert_{L^{2}(K)}+\frac{h_{K}}{\delta^{2}}\Vert u_{h}\Vert_{L^{2}(K)}\big),\label{eq:5-13}}\\
{\displaystyle \Vert \eta^{2}u_{h}- I_{h}(\eta^{2}u_{h})\Vert_{L^{2}(K)}\leq C\big(\frac{h^{2}_{K}}{\delta}\Vert \nabla(\eta u_{h})\Vert_{L^{2}(K)}+\frac{h^{2}_{K}}{\delta^{2}}\Vert u_{h}\Vert_{L^{2}(K)}\big),\label{eq:5-14}}
\end{align}
where $I_{h}$ is the standard Lagrange interpolation operator.
\end{lemma}
A similar super-approximation result as \cref{lem:5-2} also holds for Nedel\'{e}c elements. Proofs of discrete Caccioppoli inequalities by using the super-approximation results and the abstract identity can be found in \cref{sec:5}.

\subsection{Verification of the weak approximation property} Compared with the Caccioppoli-type inequality, the verification of the weak approximation property is often more complicated, particularly for Maxwell-type problems, and it lies at the core of the theory. In what follows, we verify the property for an important case when $\mathcal{H}_{B}(D^{\ast})$ is a subspace of $H^{1}(D^{\ast})$ and $\mathcal{L}(D^{\ast}) = L^{2}(D^{\ast})$. 

\begin{lemma}\label{lem:5-3}
Let $D\subset D^{\ast}\subset {D}^{\ast\ast}$ be subdomains of $\Omega$ with $\delta:={\rm dist} ({D}^{\ast},\partial D^{\ast\ast}\setminus \partial \Omega)>0$, and let $\mathsf{V}_{\delta}({D}^{\ast}\setminus D):= \big\{{\bm x}\in {D}^{\ast\ast}: {\rm dist}({\bm x}, {D}^{\ast}\setminus D)\leq \delta \big\}$. Then, there exist constants $C_{1}, C_{2}>0$ depending only on $d$, such that the following results hold.

\vspace{2mm}
\begin{itemize}
    \item[$(\rm i)$] For each integer $m\geq C_{1}\big|\mathsf{V}_{\delta}({D}^{\ast}\setminus D)\big|\delta^{-d}$, there exists an $m$-dimensional space $Q_{m}(D^{\ast\ast})\subset L^{2}(D^{\ast\ast})$ such that for all $u\in H^{1}(D^{\ast\ast})$,
    \begin{equation}\label{eq:5-11}
        \inf_{v\in Q_{m}(D^{\ast\ast})}\Vert u-v\Vert_{L^{2}(D^{\ast}\setminus D)} \leq C_{2} \big|\mathsf{V}_{\delta}({D}^{\ast}\setminus D)\big|^{1/d} m^{-1/d}\Vert \nabla u \Vert_{L^{2}(D^{\ast\ast})}.
    \end{equation}

        \item[$(\rm ii)$] For each integer $m\geq C_{1}\big|D^{\ast\ast}\big|\delta^{-d}$, there exists an $m$-dimensional space $Q_{m}(D^{\ast\ast})\subset L^{2}(D^{\ast\ast})$ such that for all $u\in H^{1}(D^{\ast\ast})$,
    \begin{equation}\label{eq:5-12}
        \inf_{v\in Q_{m}(D^{\ast\ast})}\Vert u-v\Vert_{L^{2}(D^{\ast})} \leq C_{2} \big|D^{\ast\ast}\big|^{1/d} m^{-1/d}\Vert\nabla u\Vert_{L^{2}(D^{\ast\ast})}.
    \end{equation}

        \item[$(\rm iii)$] For each integer $m\geq C_{1}\big|D^{\ast\ast}\big|\delta^{-d}$ and any $a,b>0$, there exists an $m$-dimensional space $Q_{m}(D^{\ast\ast})\subset L^{2}(D^{\ast\ast})$ such that for all $u\in H^{1}(D^{\ast\ast})$,
    \begin{equation}\label{eq:5-20}
    \begin{array}{ll}
     {\displaystyle \inf_{v\in Q_{m}(D^{\ast\ast})}\big(a\Vert u-v\Vert_{L^{2}(D^{\ast}\setminus D)} + b\Vert u-v\Vert_{L^{2}(D^{\ast})}\big) }\\[4mm]
     {\displaystyle \leq C_{2} \big(a\big|\mathsf{V}_{\delta}({D}^{\ast}\setminus D)\big|^{1/d} + b\big|D^{\ast\ast}\big|^{1/d}\big)m^{-1/d}\Vert\nabla u\Vert_{L^{2}(D^{\ast\ast})}.}
    \end{array}
    \end{equation}
\end{itemize}
\end{lemma}
\begin{proof}
We begin by fixing a quasi-uniform family of triangulations $\{ \mathcal{T}_{H}\}_{H>0}$ of $\Omega$ with $\max_{T\in \mathcal{T}_{H}} h_{T} = H \lesssim \min_{T\in \mathcal{T}_{H}} h_{T}$, constructed by successively refining an arbitrary initial mesh. Given any subdomains ${D}^{\ast}\subset{D}^{\ast\ast}$ of $\Omega$ with $\delta = {\rm dist}({D}^{\ast},\,\partial {D}^{\ast\ast}\setminus\partial \Omega)>0$, we can select a triangulation $\mathcal{T}_{H}$ constructed above with $0<H\leq\delta$. Now we consider case $(\rm i)$. Let $\widetilde{\mathcal{T}}_{H}$ denote the collection of elements in $\mathcal{T}_{H}$ that intersect $D^{\ast}\setminus D$, i.e., $\widetilde{\mathcal{T}}_{H}  = \big\{T\in \mathcal{T}_{H}:T\cap (D^{\ast}\setminus D)\neq \emptyset \big\}$, and let $\widetilde{D}_{H}$ denote the domain made of the elements in $\widetilde{\mathcal{T}}_{H}$. Since $0<H\leq\delta$, we see that $(D^{\ast}\setminus D)\subset \widetilde{D}_{H} \subset \mathsf{V}_{\delta}({D}^{\ast}\setminus D)$. Let $m$ denote the number of elements in $\widetilde{\mathcal{T}}_{H}$. Now we define the desired approximation space $Q_{m}(D^{\ast\ast})\subset L^{2}(D^{\ast\ast})$ as the space of $\widetilde{\mathcal{T}}_{H}$-piecewise constant functions, i.e.,
\begin{equation}\label{eq:5-22}
Q_{m}({D}^{\ast\ast}) = {\rm span}\big\{\mathbf{1}_{T}: T\in \widetilde{\mathcal{T}}_{H}\big\},
\end{equation}
where $\mathbf{1}_{T}$ denotes the characteristic function of the mesh element $T$. It is well-known that the following approximation property holds:
\begin{equation}\label{eq:5-23}
 \inf_{v\in Q_{m}({D}^{\ast\ast})} \Vert u-v\Vert_{L^{2}(\widetilde{D}_{H})} \leq C_{d,0}H\Vert \nabla u\Vert_{L^{2}(\widetilde{D}_{H})} \quad \text{for all}\;\; u\in H^{1}(\widetilde{D}_{H}).
\end{equation}
Combining \cref{eq:5-23} and the relation $(D^{\ast}\setminus D)\subset \widetilde{D}_{H} \subset \mathsf{V}({D}^{\ast}\setminus D,\,\delta)$, we further have
\begin{equation}\label{eq:5-24}
 \inf_{v\in Q_{m}({D}^{\ast\ast})} \Vert u-v\Vert_{L^{2}(D^{\ast}\setminus D)} \leq C_{d,0}H\Vert \nabla u\Vert_{L^{2}(D^{\ast\ast})} \quad \text{for all}\;\; u\in H^{1}(D^{\ast\ast}).
\end{equation}
Note that by the quasi uniformity of the mesh, $mH^{d} \simeq |\widetilde{D}_{H}|$. Combining this relation and \cref{eq:5-24} yields that for all $u\in H^{1}(D^{\ast\ast})$,
\begin{equation}\label{eq:5-25}
\inf_{v\in Q_{m}({D}^{\ast\ast})} \Vert u-v\Vert_{L^{2}(D^{\ast}\setminus D)} \leq C_{2} m^{-1/d} \big|\mathsf{V}_{\delta}({D}^{\ast}\setminus D)\big|^{1/d} \Vert \nabla u\Vert_{L^{2}(D^{\ast\ast})}.
\end{equation}
To guarantee the existence of an $H$ satisfying $mH^{d}\simeq C_{d,2}|\widetilde{D}_{H}|$, $0<H\leq \delta$, and that $\mathcal{T}_{H}\in \{ \mathcal{T}_{H}\}$, it is sufficient that $m\geq C_{1}\big|\mathsf{V}_{\delta}({D}^{\ast}\setminus D)\big|\delta^{-d}$ for some $C_{1}>0$ depending only on $d$. Hence, case $(\rm i)$ is proved, and case $(\rm ii)$ follows from a similar argument.

Case $(\rm iii)$ can be proved by using cases $(\rm i)$ and $(\rm ii)$. Let $Q_{m}^{\rm i}(D^{\ast\ast})$ and $Q_{m}^{\rm ii}(D^{\ast\ast})$ be the approximation spaces in cases $(\rm i)$ and $(\rm ii)$, respectively. Denoting by
\begin{equation}\label{eq:5-26}
Q_{m}^{\rm i}(D^{\ast}\setminus D)=\big\{v\mathbf{1}_{D^{\ast}\setminus D}: v\in Q_{m}^{\rm i}(D^{\ast\ast})\big\},\;\; Q_{m}^{\rm ii}(D)=\big\{v\mathbf{1}_{D}: v\in Q_{m}^{\rm ii}(D^{\ast\ast})\big\},   
\end{equation}
it follows from \cref{eq:5-12,eq:5-20} that for all $u\in H^{1}(D^{\ast\ast})$,
\begin{equation}\label{eq:5-27}
\begin{array}{ll}
{\displaystyle  \inf_{v\in Q_{m}^{\rm i}(D^{\ast}\setminus D)}\Vert u-v\Vert_{L^{2}(D^{\ast}\setminus D)} \leq C_{2} \big|\mathsf{V}_{\delta}({D}^{\ast}\setminus D)\big|^{1/d} m^{-1/d}\Vert\nabla u\Vert_{L^{2}(D^{\ast\ast})},}\\[3mm]
{\displaystyle \inf_{v\in Q_{m}^{\rm ii}(D)}\Vert u-v\Vert_{L^{2}(D)} \leq C_{2} \big|D^{\ast\ast}\big|^{1/d} m^{-1/d}\Vert\nabla u\Vert_{L^{2}(D^{\ast\ast})}.}
\end{array}
\end{equation}
Now let us define the approximation space $Q_{2m}(D^{\ast\ast}) := Q_{m}^{\rm i}(D^{\ast}\setminus D)+ Q_{m}^{\rm ii}(D)$. Noting that functions in $Q_{m}^{\rm i}(D^{\ast}\setminus D)$ and $Q_{m}^{\rm ii}(D)$ vanish in $D$ and $D^{\ast}\setminus D$, respectively, we see that for any $a,b>0$ and any $u\in H^{1}(D^{\ast\ast})$,
   \begin{equation}\label{eq:5-28}
    \begin{array}{ll}
     {\displaystyle \inf_{v\in Q_{2m}(D^{\ast\ast})}\big(a\Vert u-v\Vert_{L^{2}(D^{\ast}\setminus D)} + b\Vert u-v\Vert_{L^{2}(D^{\ast})}\big) }\\[4mm]
     {\displaystyle \leq  (a+b)\inf_{v\in Q_{m}^{\rm i}(D^{\ast}\setminus D)}\Vert u-v\Vert_{L^{2}(D^{\ast}\setminus D)} + b \inf_{v\in Q_{m}^{\rm ii}(D)}\Vert u-v\Vert_{L^{2}(D)}, }
    \end{array}
    \end{equation}
which, together with \cref{eq:5-27}, yields \cref{eq:5-20}.    
\end{proof}

\Cref{lem:5-3} gives the weak approximation property for second-order scalar elliptic-type problems in both the continuous and FE discrete settings. Analogous results can be proved for the settings of linear elasticity problems and fourth-order problems using similar techniques, as will be shown in the following section.  

\section{Applications}\label{sec:5}
In this section, we will apply the abstract theory established in previous sections to various PDE problems with $L^{\infty}$-coefficients. To simplify the presentation, we assume that $\Omega\subset\mathbb{R}^{d}$, $d=2,3$, is a bounded Lipschitz domain with a polygonal boundary $\Gamma=\partial \Omega$. Let $\{ \omega_{i} \}_{i=1}^{M}$ be a collection of subdomains of $\Omega$ that satisfy $\cup_{i=1}^{M} \omega_{i} = \Omega$ and a pointwise overlap condition:
\begin{equation}\label{eq:5-29}
\exists \,\zeta\in \mathbb{N}\quad\forall {\bm x}\in\Omega \quad {\rm card}\{i\;|\;{\bm x}\in \omega_{i} \}\leq\zeta.
\end{equation}
In practice, the coloring constant in \cref{def:2-0-1} is given by $\zeta$ above. We introduce a partition of unity $\{ \chi_{i} \}_{i=1}^{M}$ subordinate to the open cover $\{ \omega_{i} \}_{i=1}^{M}$ such that
\begin{equation}\label{eq:5-30}
\begin{array}{lll}
{\displaystyle {\rm supp} \,(\chi_i)\subset \overline{\omega_i}, \quad 0\leq \chi_{i}\leq 1,\quad \sum_{i=1}^{M}\chi_{i} =1 \;\;\text{on} \;\,\Omega,}\\[4mm]
{\displaystyle \chi_{i}\in C^{1}(\overline{\omega_{i}}),\;\;\Vert\nabla \chi_{i} \Vert_{L^{\infty}(\omega_i)} \leq \frac{C_{\chi}}{\mathrm{diam}\,(\omega_{i})}.}
\end{array}
\end{equation}
In practical applications below, the partition of unity in \cref{def:2-1-1} is defined as in \cref{rem:POU}. \cm{Since we focus on fitting practical applications into our framework and on proving exponential convergence rates, we will always assume that \cref{ass:2-0} and \ref{ass:2-2-1} hold true in this section. \Cref{ass:2-0} indeed holds true for a wide class of PDEs, whereas certain conditions on the smallness of the subdomains are often needed to verify \cref{ass:2-2-1} for indefinite problems; see \cite[Remark 2.6]{chupeng2023wavenumber}}.

\cm{We also note that in order to apply the established abstract framework, we only consider applications with (part of) zero Dirichlet boundary conditions in this section. However, it is straightforward for our method to treat nonzero Dirichlet boundary conditions -- we simply impose the nonzero boundary conditions for the local particular functions on boundary subdomains, which does not change the local and global error estimates. We refer to \cite{babuvska2020multiscale,ma2021novel,ma2022error} for this treatment for scalar elliptic problems}.

\subsection{Convection-diffusion problems}\label{sec:convection-diffusion}
We start with the following convection-diffusion equation in a heterogeneous medium: Find $u:\Omega\rightarrow \mathbb{R}$ such that
\begin{equation}\label{eq:5-1-1}
\left\{
\begin{array}{lll}
{\displaystyle -{\rm div}(A\nabla u) -{\bm b}\cdot\nabla u= f\,\quad {\rm in}\;\, \Omega }\\[2mm]
{\displaystyle  \;\,\qquad \quad \quad \qquad \qquad u = 0\quad \,{\rm on}\;\,\Gamma_{\mathsf{D}}}\\[2mm]
{\displaystyle \qquad \qquad \quad \,\;A\nabla u\cdot {\bm n} = g  \quad \,{\rm on}\;\,\Gamma_{\mathsf{N}},}
\end{array}
\right.
\end{equation}
where ${\bm n}$ denotes the outward unit normal to $\Gamma$, $\Gamma_{\mathsf{D}}\cap\Gamma_{\mathsf{N}} = \emptyset$, and $\Gamma = \overline{\Gamma_{\mathsf{D}}}\cup\overline{\Gamma_{\mathsf{N}}}$. Here, the diffusion coefficient $A \in L^{\infty}(\Omega,\mathbb{R}_{\rm sym}^{d\times d})$ is uniformly elliptic, i.e., there exists $0< a_{\rm min} < a_{\rm max}<\infty$ such that $a_{\rm min} |{\bm \xi}|^{2} \leq A({\bm x}){\bm \xi}\cdot{\bm \xi} \leq a_{\rm max}  |{\bm \xi}|^{2}$ for all ${\bm \xi}\in \mathbb{R}^{d} $ and ${\bm x} \in \Omega$; ${\bm b}\in {\bm L}^{\infty}(\Omega)$ stands for the velocity field; $f\in H^{1}(\Omega)^{\prime}$, $g\in H^{-1/2}(\Gamma_{\mathsf{N}})$. See \cite{abdulle2014discontinuous,bonizzoni2022super,calo2016multiscale,le2017numerical,maalqvist2011multiscale,park2004multiscale,zhao2022constraint} for multiscale methods for convection-diffusion problems.


For any subdomain $D\subset \Omega$, we define $H_{\mathsf{D}}^{1}(D):=\big\{v\in H^{1}(D): v=0\;\text{on}\;\Gamma_{\mathsf{D}} \big\}$, $H_{\mathsf{D},0}^{1}(D):=\big\{v\in H_{\mathsf{D}}^{1}(D): v=0\;\text{on}\; \partial D\cap \Omega \big\}$, and set 
\begin{equation}\label{convection-diffusion:local-forms}
\begin{array}{cc}
{\displaystyle B_{D}(u,v):=\int_{D}\big(A\nabla u\cdot\nabla v + ({\bm b}\cdot\nabla u) v\big)d{\bm x},}\\[3mm]
{\displaystyle  B^{+}_{D}(u, v) := \int_{D}A\nabla u\cdot\nabla v \,d{\bm x}, \quad \Vert u\Vert_{B^{+},D}:= \big(B^{+}_{D}(u, u)\big)^{1/2}.}
\end{array}   
\end{equation}
When $D=\Omega$, we simply write $B(\cdot,\cdot)$, $B^{+}(\cdot,\cdot)$, and $\Vert \cdot\Vert_{B^{+}}$. Setting $\mathcal{H}(\Omega)=H^{1}_{\mathsf{D}}(\Omega)$, $\Vert\cdot\Vert_{\mathcal{H}(\Omega)}= \Vert \cdot\Vert_{B^{+}}$, and $\mathcal{L}(\Omega)=L^{2}(\Omega)$, it is clear that the bilinear form $B(\cdot,\cdot)$ satisfies \cref{ass:2-1-0}. Moreover, \cref{ass:2-1-1} holds true with $\mathcal{H}(D)= H^{1}_{\mathsf{D}}(D)$, $\mathcal{H}_{0}(D)= H^{1}_{\mathsf{D},0}(D)$, and with $B_{D}(\cdot,\cdot)$ defined by \cref{convection-diffusion:local-forms}. The weak formulation of problem \cref{eq:5-1-1} is to find $u\in H^{1}_{\mathsf{D}}(\Omega)$ such that 
\begin{equation}\label{eq:5-1-2}
B(u,v) = F(v):=\langle f,v\rangle_{\Omega} + \langle g,v\rangle_{\Gamma_\mathsf{N}}\quad \text{for all}\;\; v\in H^{1}_{\mathsf{D}}(\Omega),
\end{equation}
where $\langle \cdot,\cdot\rangle_{\Omega}$ and $\langle \cdot,\cdot\rangle_{\Gamma_\mathsf{N}}$ denote the corresponding duality pairings. With the setting above, the generalized harmonic space on a subdomain $D\subset \Omega$ takes the form
\begin{equation}\label{eq:5-1-4}
 H_{B}(D)=\big\{u\in H_{\mathsf{D}}^{1}(D): B_{D}(u,v) = 0 \;\;\text{for all}\;\,v\in H_{\mathsf{D},0}^{1}(D)\}.
\end{equation}
We see that $\big(H_{B}(D), B_{D}^{+}(\cdot,\cdot)\big)$ is a Hilbert space when $\partial D\cap \Gamma_{\mathsf{D}} \neq \emptyset$, and that $\big(H_{B}(D)/\mathbb{R}, B_{D}^{+}(\cdot,\cdot)\big)$ is a Hilbert space
when $\partial D\cap \Gamma_{\mathsf{D}} = \emptyset$ and $D$ satisfies the cone condition. Hence, \cref{ass:2-2-3} holds true with $B^{+}_{D}(\cdot,\cdot)$ defined by \cref{convection-diffusion:local-forms}. Next we give the Caccioppoli-type inequality in this setting. 
\begin{lemma}[Caccioppoli-type inequality]\label{lem:cac-5-1}
Let $D\subset D^{\ast}$ be subdomains of $\Omega$ with $\delta:={\rm dist}(D,\partial D^{\ast}\setminus\partial \Omega)>0$, and let $\eta \in C^{1}(\overline{D^{\ast}})$ satisfy $\eta({\bm x}) = 0$ on $\partial D^{\ast} \cap \Omega$. Then, 
\begin{equation}\label{cac:5-1-1}
B^{+}_{D^{\ast}}(\eta u,\eta v) = \int_{D^{\ast}}\big((A\nabla \eta \cdot \nabla \eta)uv -\frac{1}{2}\eta^{2}{\bm b}\cdot\nabla (uv)\big) \,d{\bm x}\quad \text{for all}\;\; u, \,v\in H_{B}(D^{\ast}).
\end{equation}
In addition, for any $u\in H_{B}(D^{\ast})$,
\begin{equation}\label{cac:5-1-2}
\Vert u\Vert^{2}_{B^{+},D} \leq (2a_{\rm max} +1)\delta^{-2} \Vert u\Vert^{2}_{L^{2}(D^{\ast}\setminus D)} + (a_{\rm min}^{-1}+1)\Vert{\bm b}\Vert^{2}_{L^{\infty}(D^{\ast})}\Vert u \Vert^{2}_{L^{2}(D^{\ast})},
\end{equation}
where $a_{\rm max}$ and $a_{\rm min}$ are the spectral upper and lower bounds of $A$.
\end{lemma}
\begin{rem}
If $\nabla\cdot{\bm b}\in L^{\infty}(D^{\ast})$, we further have
\begin{equation*}
\Vert u\Vert^{2}_{B^{+},D} \leq \big(a_{\rm max}\delta^{-2} + \Vert{\bm b}\Vert_{L^{\infty}(D^{\ast})}\delta^{-1}\big) \Vert u\Vert^{2}_{L^{2}(D^{\ast}\setminus D)} + \frac{1}{2}\Vert\nabla\cdot{\bm b}\Vert_{L^{\infty}(D^{\ast})}\Vert u \Vert^{2}_{L^{2}(D^{\ast})}.
\end{equation*}
\end{rem}

\begin{proof}
We use \cref{prop:5-1} to prove the identity \cref{cac:5-1-1}. Consider the setting of Example~\ref{exam:5-1} and define $b_{\bf L}(\cdot,\cdot):{\bm L}^{2}(D^{\ast})\times {\bm L}^{2}(D^{\ast})\rightarrow \mathbb{R}$ as $b_{\bf L}({\bf u},{\bf v}) = \int_{D^{\ast}}A{\bf u}\cdot {\bf v}d{\bm x}$. Then, identity \cref{cac:5-1-1} follows from the abstract identity \cref{eq:5-4} and the fact that 
\begin{equation}
b_{\bf L}(\nabla u, \nabla (\eta^{2}v)) = - \int_{D^{\ast}}({\bm b}\cdot \nabla u)\eta^{2}v\,d{\bm x}\quad\; \text{for all}\;\,u, \,v\in H_{B}(D^{\ast}).
\end{equation}
To prove the inequality \cref{cac:5-1-2}, we first let $u=v$ in \cref{cac:5-1-1} and observe that
\begin{equation}
B^{+}_{D^{\ast}}(\eta u,\eta u) = \int_{D^{\ast}}(A\nabla \eta \cdot \nabla \eta)u^{2}d{\bm x}+\int_{D^{\ast}}\eta u{\bm b}\cdot\big(u\nabla \eta-\nabla (\eta u)\big)\,d{\bm x}.   
\end{equation}
Using the H\"{o}lder inequality and the uniform ellipticity of the coefficient $A$ gives
\begin{equation}\label{cac:5-1-3}
\Vert \eta u\Vert^{2}_{B^{+},D^{\ast}}\leq (2a_{\rm max} +1)\Vert u\nabla \eta\Vert^{2}_{L^{2}(D^{\ast})} +  (a_{\rm min}^{-1}+1)\Vert{\bm b}\Vert^{2}_{L^{\infty}(D^{\ast})}\Vert \eta u \Vert^{2}_{L^{2}(D^{\ast})},  
\end{equation}
and \cref{cac:5-1-2} follows by choosing a cut-off function $\eta\in C^{1}(\overline{D^{\ast}})$ such that $0\leq \eta\leq 1$ in $D^{\ast}$, $\eta=1$ in $D$, and $|\nabla \eta|\leq 1/\delta$.
\end{proof}

\textbf{Local eigenproblems}. Based on the Caccioppoli-type inequality \cref{cac:5-1-2} and the compact embeddings $H^{1}(\omega_i^{\ast})\subset L^{2}(\omega_i^{\ast})$, it is clear that the operators $\widehat{P}_{i}: H_{B}(\omega_i^{\ast})\rightarrow H_{\mathsf{D},0}^{1}(\omega_i)$ ($\widehat{P}_{i}u=\chi_{i}u$) are compact. Hence, \cref{ass:2-2-2} holds true. Noting that $H_{\mathsf{D},0}^{1}(\omega_i)$ are equipped with the scalar products $B^{+}_{\omega_i}(\cdot, \cdot)$, the abstract local eigenproblems \cref{eq:2-24-5} take the form: Find $\lambda_{i}\in \mathbb{R}\cup \{+\infty\}$ and $\phi_{i}\in H_{B}(\omega_{i}^{\ast})$ such that
\begin{equation}\label{eigen:5-1}
B_{\omega_i}^{+}(\chi_{i}\phi_{i},\chi_i v) = B_{\omega_i^{\ast}}^{+}(\phi_{i},\,  v)\quad \forall v\in H_{B}(\omega_{i}^{\ast}).
\end{equation}
If the velocity field ${\bm b}$ satisfies $\nabla\cdot{\bm b}\in L^{2}(\omega_i^{\ast})$ and ${\bm n}\cdot {\bm b} = 0$ on $\Gamma_{\mathsf{N}}\cap \partial \omega_i^{\ast}$, then it follows from identity \cref{cac:5-1-1} and an integration by parts that \cref{eigen:5-1} can be rewritten as
\begin{equation}\label{neweigen:5-1}
\big(Q_{\bm b}\phi_{i},\,  v\big)_{L^{2}(\omega_i^{\ast})} = \lambda_{i}B_{\omega_i^{\ast}}^{+}(\phi_{i},\,  v)\quad \forall v\in H_{B}(\omega_{i}^{\ast}),
\end{equation}
where $Q_{\bm b}:=A\nabla \chi_i \cdot \nabla \chi_{i} +\frac{1}{2}\nabla\cdot(\chi_{i}^{2}{\bm b})$, i.e., a Laplacian-type eigenvalue problem with weighted $L^{2}$ inner product.

Next we verify the exponential convergence of the local approximations. By \cref{lem:cac-5-1}, the Caccioppoli-type inequaility (\cref{ass:3-1-1}) holds true with $C_{\rm cac}^{\rm I} = (2a_{\rm max} +1)^{1/2}$ and $C_{\rm cac}^{\rm II} = (a_{\rm min} +1)^{1/2}\Vert{\bm b}\Vert_{L^{\infty}(\omega_i^{\ast})}$. The weak approximation property (\cref{ass:3-1-2}) in this setting with $\alpha = 1/d$ is an easy consequence of \cref{lem:5-3}.
\begin{theorem}[Exponential local convergence]\label{upperbound:5-1}
For each $i=1,\cdots,M$, \cref{thm:3-1} holds true with $C_{\rm cac}^{\rm I} = (2a_{\rm max} +1)^{1/2}$, $C_{\rm cac}^{\rm II} = (a^{-1}_{\rm min} +1)^{1/2}\Vert{\bm b}\Vert_{L^{\infty}(\omega_i^{\ast})}$, and $\alpha = 1/d$.
\end{theorem}

Next we prove the quasi-optimal global convergence of the method. To do this, we assume that the velocity field ${\bm b}$ satisfies
\begin{equation}\label{divergence-velocity}
\nabla\cdot {\bm b}\in L^{\infty}(\Omega) \quad \text{and}\quad {\bm n}\cdot {\bm b} = 0 \;\;{\rm on}\;\;\Gamma_{\mathsf{N}}.    
\end{equation}
In the incompressible case, i.e., $\nabla\cdot{\bm b} = 0$, $B(\cdot,\cdot)$ is coercive on $H_{\mathsf{D}}^{1}(\Omega)$, and the quasi-optimal convergence follows from \cref{quasi-optimal-coercive} with $C_{1} = 1$ and $C_{b} = 1+ C_{\Omega} a^{-1}_{\rm min} \Vert {\bm b}\Vert_{L^{\infty}(\Omega)}$. In the general case, we need to verify \cref{ass:3-2-1} and \ref{ass:3-2-2}.

\begin{lemma}\label{adjoint-stability-5-1}
Let the velocity field ${\bm b}$ satisfy \cref{divergence-velocity}. Then, \cref{ass:3-2-1} holds true with $C_{\rm ad} = 1+C_{\Omega}C^{a}_{\rm stab}a_{\rm min}^{-1/2}(\Vert \nabla\cdot{\bm b}\Vert_{L^{\infty}(\Omega)} + \Vert {\bm b} \Vert_{L^{\infty}(\Omega)})$.
\end{lemma}
\begin{proof}
Let $f\in L^{2}(\Omega)$, and let $\hat{u}\in H^{1}_{\mathsf{D}}(\Omega)$ such that
\begin{equation}\label{stability-estimate-5-1-1}
B(v,\hat{u}) = (v,f)_{L^{2}(\Omega)}\quad \forall v\in H^{1}_{\mathsf{D}}(\Omega).
\end{equation}
Using the assumption on ${\bm b}$ and an integration by parts, we see that $\hat{u}$ satisfies $B(\hat{u},v) = (g,v)_{L^{2}(\Omega)}$ for all $v\in H^{1}_{\mathsf{D}}(\Omega)$, with $g\in L^{2}(\Omega)$ given by $g = f+(\nabla\cdot{\bm b}) \hat{u} + 2{\bm b}\cdot \nabla \hat{u}$. Using this fact and the stability estimate $\Vert \hat{u}\Vert_{B^{+}}\leq C^{a}_{\rm stab} \Vert f\Vert_{L^{2}(\Omega)}$ (\cref{ass:2-1-0}) yields $\Vert g\Vert_{L^{2}(\Omega)}\leq C_{\rm ad} \Vert f\Vert_{L^{2}(\Omega)}$.
\end{proof}

In order to verify \cref{ass:3-2-2}, we introduce a uniform Poincar\'{e} constant $C_{P}$ (see \cite[ Corollary A.15]{toselli2004domain}) that satisfies
\begin{equation}\label{uniform-poincare}
\Vert u \Vert_{L^{2}(\omega_{i}^{\ast})}\leq C_{P}\,{\rm diam}(\omega_{i}^{\ast})\, \Vert \nabla u\Vert_{L^{2}(\omega_{i}^{\ast})},\quad \forall u\in H_{\mathsf{D},0}^{1}(\omega^{\ast}_i),\quad \forall i=1,\cdots,M.
\end{equation}

\begin{lemma}\label{local-stablity-5-1}
\Cref{ass:3-2-2} holds true with $C_{s}  = 2C_{P}a_{\rm min}^{-1/2}$.    
\end{lemma}
\begin{proof}
Given $g_{i}\in L^{2}(\omega_i^{\ast})$, $i=1,\cdots,M$, let $u_{i}\in H^{1}_{\mathsf{D},0}(\omega_i^{\ast})$ such that
\begin{equation}\label{poincare-5-1-1}
B_{\omega_{i}^{\ast}}(u_{i},v) = (g_i,v)_{L^{2}(\omega_i^{\ast})}\quad \forall v\in  H^{1}_{\mathsf{D},0}(\omega_i^{\ast}).
\end{equation}
It follows from \cref{ass:2-1-0} and inequality \cref{uniform-poincare} that for any $v\in H^{1}_{\mathsf{D},0}(\omega_i^{\ast})$,
\begin{equation*}
|B(v,v)|\geq C_{1}\Vert v\Vert^{2}_{B^{+},\omega_i^{\ast}}-C_{0}\Vert v\Vert^{2}_{L^{2}(\omega_i^{\ast})} \geq \big(C_{1}-C_{0}a^{-1}_{\rm min}C_{P}^{2}[{\rm diam}(\omega_{i}^{\ast})]^{2}\big) \Vert v\Vert^{2}_{B^{+},\omega_i^{\ast}}.  
\end{equation*}
Hence, if $H^{\ast}_{\mathtt{max}}:=\max_{i} {\rm diam}(\omega_{i}^{\ast}) \leq (C_{1}/C_{0})^{1/2}a_{\rm min}^{1/2}C^{-1}_{P}/\sqrt{2}$, we get $|B(v,v)|\geq \frac{1}{2}C_{1} \Vert v\Vert^{2}_{B^{+},\omega_i^{\ast}}$, and thus problems \cref{poincare-5-1-1} are uniquely solvable. Combining \cref{uniform-poincare,poincare-5-1-1} gives $\frac{1}{2}C_{1} \Vert u_i\Vert^{2}_{B^{+},\omega_i^{\ast}} \leq |B(u_i,u_i)|\leq C_{P} a_{\rm min}^{-1/2} \,{\rm diam}(\omega_{i}^{\ast})\Vert u_i\Vert_{B^{+},\omega_i^{\ast}}\Vert g_i \Vert_{L^{2}(\omega_i^{\ast})}$. Therefore, we have $\Vert u_i\Vert_{B^{+},\omega_i^{\ast}} \leq C_{s}H^{\ast}_{\mathtt{max}}/C_{1} \Vert g_i \Vert_{L^{2}(\omega_i^{\ast})}$ with $C_{s}  = 2C_{P}a_{\rm min}^{-1/2}$.   
\end{proof}
Combining \cref{adjoint-stability-5-1,local-stablity-5-1} leads to the quasi-optimal convergence of the method.
\begin{theorem}[Quasi-optimal convergence]\label{quasi-optimal-convergence-5-1}
\hspace{-2mm}Let the velocity field ${\bm b}$ satisfy \cref{divergence-velocity}. Then, \cref{cor:3-1} holds true with
$d_{\mathtt{max}} = O(C^{-1}_{\rm stab}\widetilde{\Xi}^{-1})$ and $H^{\ast}_{\mathtt{max}} = O(\widetilde{\Xi}^{-1})$, where
$\Xi = O\big(C^{a}_{\rm stab}\Vert\nabla\cdot{\bm b}\Vert^{1/2}_{L^{\infty}(\Omega)} (\Vert \nabla\cdot{\bm b}\Vert_{L^{\infty}(\Omega)} + \Vert {\bm b} \Vert_{L^{\infty}(\Omega)}) (1 + \Vert {\bm b} \Vert_{L^{\infty}(\Omega)})\big)$.   
\end{theorem}
\begin{proof}
Use \cref{adjoint-stability-5-1,local-stablity-5-1}, and the fact that $C_{1} = 1$, $C_{0} = \Vert \nabla\cdot{\bm b}\Vert_{L^{\infty}(\Omega)}$, and $C_{b} = 1+ C_{\Omega} a^{-1}_{\rm min} \Vert {\bm b}\Vert_{L^{\infty}(\Omega)}$.
\end{proof}
\begin{rem}
The quasi-optimal convergence result above is pessimistic in that very stringent resolution conditions are needed. This is mainly due to using the global approximation space as the trial and test spaces and applying the standard perturbation argument. In practice, to obtain discrete stability, a better choice is to use enriched trial and test spaces, or to work within the Petrov-Galerkin framework and use the Babuska theory \cite{babuska1972survey}. We will investigate this extension in forthcoming works.    
\end{rem}

\subsection{Convection-diffusion problems (discrete setting)} In this subsection, we apply our theory to discretized convection-diffusion problems in the FE setting. For ease of notation, we assume that $\Gamma_{\mathsf{N}} = \emptyset$. Let $\{\mathcal{T}_{h}\}$ be a family of shape-regular triangulations of $\Omega$ consisting of triangles ($d=2$) or tetrahedrons ($d=3$). The mesh size $h$ is assumed small enough to resolve the fine-scale details of the coefficient $A$ and to make the discrete problem (defined below) well-posed. Let $V_{h}\subset H^{1}(\Omega)$ be (conforming) finite element spaces consisting of piecewise polynomial functions, and denote by $V_{h,0}=V_{h} \cap H^{1}_{0}(\Omega)$. Let the bilinear forms $B_{D}(\cdot,\cdot)$ and $B_{D}^{+}(\cdot,\cdot)$ be as in \cref{convection-diffusion:local-forms}. The standard finite element approximation of problem \cref{eq:5-1-1} (with $\Gamma_{\mathsf{N}} = \emptyset$) is defined by: Find $u_{h}\in V_{h,0}$, such that
\begin{equation}\label{eq:5-2-1}
B(u_{h}, v) = F(v):=\langle f, v \rangle_{\Omega}\quad\;\; \text{for all}\;\; v\in V_{h,0}.  
\end{equation}
Since $V_{h,0}$ is a subspace of $H^{1}_{0}(\Omega)$, \cref{ass:2-0} holds true with $\mathcal{H}(\Omega) = V_{h,0}$, $\Vert \cdot\Vert_{\mathcal{H}(\Omega)} = \Vert\cdot\Vert_{B^{+}}$, and $\mathcal{L}(\Omega) = L^{2}(\Omega)$. For any subdomain $D\subset\Omega$, define
\begin{equation}\label{eq:5-2-2}
 V_{h}(D)=\big\{v|_{D}:v\in V_{h,0}\big\}, \quad  V_{h,0}(D) = \big\{v\in V_{h}(D):{\rm supp}(v)\subset \overline{D}\big\}.
\end{equation}
We can easily see that \cref{ass:2-1-1} holds true with $\mathcal{H}(D) = V_{h}(D)$ and $\mathcal{H}_{0}(D) =V_{h,0}(D)$. On a subdomain $D\subset \Omega$, the generalized harmonic space is defined by
\begin{equation}
 V_{h,B}(D)=\big\{u\in V_{h}(D): B_{D}(u,v) = 0 \;\;\text{for all}\;\,v\in V_{h,0}(D)\}.
\end{equation} 
As in the continuous setting, we see that \cref{ass:2-2-3} holds true. 

\textbf{Local eigenproblems}. For each $i=1,\cdots,M$, let the operator $\widehat{P}_{i}:  V_{h,B}(\omega_i^{\ast})\rightarrow V_{h,0}(\omega_i)$ be defined by $\widehat{P_i}u = I_{h}(\chi_{i} u)$, where $I_{h}:C(\overline{\Omega})\rightarrow V_{h}$ is the standard nodal interpolation operator. Since $V_{h,0}(\omega_i)$ are finite-dimensional spaces, the operators $\widehat{P}_{i}$ are compact, and thus \cref{ass:2-2-2} holds true. The local eigenproblems \cref{eq:2-24-5} in this setting take the form: Find $\lambda_{i}\in \mathbb{R}\cup \{+\infty\}$ and $\phi_{i}\in V_{h,B}(\omega_{i}^{\ast})$ such that
\begin{equation}\label{eigen:5-2}
B^{+}_{\omega^{\ast}_i}(I_{h}(\chi_{i}\phi_{i}),\,  I_{h}(\chi_i v)\big)= \lambda_{i} B^{+}_{\omega_i^{\ast}}(\phi_{i},v)\quad\forall v\in V_{h,B}(\omega_{i}^{\ast}).
\end{equation}


Next we verify the two fundamental conditions in \cref{sec:3-1}. Since $V_{h,B}(D)\subset H^{1}(D)$, the weak approximation property (\cref{ass:3-1-2}) holds true with $\alpha = 1/d$ due to \cref{lem:5-3}. The discrete Caccioppoli-type inequality, however, requires a careful analysis (noting that $V_{h,B}(D)\nsubseteq H_{B}(D)$), where the super-approximation results in \cref{lem:5-2} play a key role.
\begin{lemma}[Caccioppoli-type inequality]\label{lem:cac-5-2}
Let $D\subset D^{\ast}$ be subdomains of $\Omega$ with $\delta:={\rm dist}\big(D, \, \partial D^{\ast}\setminus\partial \Omega\big)>0$, and let $h_{K}\leq \delta/3$ for each $K\in \mathcal{T}_{h}$ with $K\cap (D^{\ast}\setminus D)\neq \emptyset$. Then, for any $u\in V_{h,B}(D^{\ast})$,
\begin{equation}\label{cac:5-2-1}
\Vert u\Vert^{2}_{B^{+},D}\leq C\big(a^{2}_{\rm max}a^{-1}_{\rm min} +1 \big)\delta^{-2}\Vert u \Vert^{2}_{L^{2}(D^{\ast}\setminus D)}+\, C(a_{\rm min}^{-1}+1)\Vert{\bm b}\Vert^{2}_{L^{\infty}(D^{\ast})}\Vert u \Vert^{2}_{L^{2}(D^{\ast})},
\end{equation}
where $C$ depends only on $d$ and the shape regularity of the mesh.
\end{lemma}
\begin{proof}
Let $D_{e}$ and $D^{\ast}_{e}$ be the union of elements in $\mathcal{T}_{h}$ that intersect $D$ and that are contained in $D^{\ast}$, respectively. By the assumption $\delta:={\rm dist}\,\big(D, \, \partial D^{\ast}\setminus\partial \Omega\big)\geq 3\max_{K\cap (D^{\ast}\setminus D)\neq \emptyset}h_{K}$, it is clear that $D\subset D_{e}\subset D^{\ast}_{e}\subset D^{\ast}$. Moreover, we have ${\rm dist}\,\big(D_{e}, \, \partial D^{\ast}_{e}\setminus\partial \Omega\big) \geq \frac{1}{3}\delta$. Let $\eta\in C^{\infty}(\overline{{D}_{e}^{\ast}})$ be a cut-off function such that $\eta \equiv 1$ in $D_{e}$, $\eta = 0$ on $\partial {D}_{e}^{\ast}\setminus\partial \Omega$, and $|\eta|_{W^{j,\infty}(D_{e}^{\ast})}\leq C\delta^{-j},\;\forall j\in\mathbb{N}$. Applying the abstract identity \cref{eq:5-4} as in \cref{lem:cac-5-1} yields that for any $u\in V_{h}(D^{\ast}_{e})$,
\begin{equation}\label{cac:5-2-4}
\begin{array}{ll}
{\displaystyle \Vert \eta u\Vert^{2}_{B^{+},D_{e}^{\ast}} = \int_{D_{e}^{\ast}}(A\nabla \eta \cdot \nabla \eta) u^{2}\,d{\bm x} + \int_{D_{e}^{\ast}}A\nabla u\cdot\nabla (\eta^{2}u)\,d{\bm x} }\\[3mm]
{\displaystyle = \int_{D_{e}^{\ast}}(A\nabla \eta \cdot \nabla \eta) u^{2}\,d{\bm x} - \int_{D_{e}^{\ast}}({\bm b}\cdot \nabla u)\eta^{2}u\,d{\bm x} + B_{D_{e}^{\ast}}(u,\eta^{2}u). }
\end{array}
\end{equation}
Using a similar argument as in the proof of \cref{cac:5-1-3} and properties of the cut-off function, the first two terms on the right-hand side of \cref{cac:5-2-4} can be bounded as follows.
\begin{equation}\label{cac:5-2-5}
\begin{array}{cc}
{\displaystyle \Big| \int_{D_{e}^{\ast}}(A\nabla \eta \cdot \nabla \eta) u^{2}\,d{\bm x} - \int_{D_{e}^{\ast}}({\bm b}\cdot \nabla u)\eta^{2}u\,d{\bm x}\Big| \leq \frac{1}{4}\Vert \eta u\Vert^{2}_{B^{+},D_{e}^{\ast}} }\\[2ex]
{\displaystyle  +\,C\big(a_{\rm max} +\frac{1}{2}\big)\delta^{-2}\Vert u\Vert^{2}_{L^{2}(D_{e}^{\ast}\setminus D_{e})} + C\big(a_{\rm min}^{-1}+\frac12\big)\Vert{\bm b}\Vert^{2}_{L^{\infty}(D_{e}^{\ast})}\Vert u\Vert^{2}_{L^{2}(D_{e}^{\ast})}. }
\end{array}
\end{equation}
It remains to bound the last term of \cref{cac:5-2-4}, i.e., $B_{D_{e}^{\ast}}(u,\eta^{2}u)$. Since $\eta^{2}u\notin V_{h,0}(D_{e}^{\ast})$, this term does not vanish for $u\in V_{h,B}(D^{\ast}_{e})$, as contrasted with the continuous setting. Here a crucial observation is that $B_{D_{e}^{\ast}}(u, I_{h}(\eta^{2}u)) = 0$, where $I_{h}$ is the standard nodal interpolation operator. It follows that
\begin{equation}\label{cac:5-2-6}
B_{D_{e}^{\ast}}(u,\eta^{2}u) = \big(A\nabla u, \nabla (\eta^{2}u - I_{h}(\eta^{2}u))\big)_{L^{2}(D_{e}^{\ast})} + \big({\bm b}\cdot \nabla u, \eta^{2}u- I_{h}(\eta^{2}u)\big)_{L^{2}(D_{e}^{\ast})}.
\end{equation}
The estimate of the two terms of \cref{cac:5-2-6} is where the super-approximation results come into play. Applying \cref{eq:5-13} and an inverse estimate (local to each element), and noting that $\eta^{2}u- I_{h}(\eta^{2}u)$ vanishes on $D_{e}$, we obtain
\begin{equation}\label{cac:5-2-7}
\begin{array}{lll}
{\displaystyle \big(A\nabla u, \nabla (\eta^{2}u - I_{h}(\eta^{2}u))\big)_{L^{2}(D_{e}^{\ast})}}\\[2mm]
{\displaystyle \leq \sum_{K\subset D_{e}^{\ast}\setminus D_{e}}Ca_{\rm max} h_{K}\Vert \nabla u\Vert_{L^{2}(K)} \big(\delta^{-1} \Vert\nabla (\eta u)\Vert_{L^{2}(K)} + \delta^{-2} \Vert u\Vert_{L^{2}(K)}\big)}\\[4mm]
{\displaystyle \leq Ca_{\rm max}^{2}/(a_{\rm min}\delta^{2})\sum_{K\subset D_{e}^{\ast}\setminus D_{e}}  \Vert u\Vert^{2}_{L^{2}(K)} + \frac{a_{\rm min}}{4}\sum_{K \subset D_{e}^{\ast}} \Vert\nabla (\eta u)\Vert^{2}_{L^{2}(K)}}\\[3mm]
{\displaystyle \leq Ca_{\rm max}^{2}/(a_{\rm min}\delta^{2})\Vert u\Vert^{2}_{L^{2}(D_{e}^{\ast}\setminus D_{e})} + \frac{1}{4}\Vert \eta u\Vert^{2}_{B^{+},D_{e}^{\ast}}.}
\end{array}
\end{equation}
Similarly, we can prove
\begin{equation}\label{cac:5-2-8}
\begin{array}{lll}
{\displaystyle \big({\bm b}\cdot \nabla u, \eta^{2}u- I_{h}(\eta^{2}u)\big)_{L^{2}(D_{e}^{\ast})}}\\[2mm]
{\displaystyle \leq C\delta^{-2}\Vert u\Vert^{2}_{L^{2}(D_{e}^{\ast}\setminus D_{e})} + C(a_{\rm min}^{-1}+1)\Vert{\bm b}\Vert^{2}_{L^{\infty}(D_{e}^{\ast})}\Vert u \Vert^{2}_{L^{2}(D_{e}^{\ast})} + \frac{1}{4}\Vert \eta u\Vert^{2}_{B^{+},D_{e}^{\ast}}.}
\end{array}
\end{equation}
Combining \cref{cac:5-2-4,cac:5-2-5,cac:5-2-6,cac:5-2-7,cac:5-2-8} and recalling that $D\subset D_{e}\subset D^{\ast}_{e}\subset D^{\ast}$ yield \cref{cac:5-2-1}.   
\end{proof}

\cm{\begin{rem}\label{rem:h_technical_assp}
In \cite[Lemma 3.12]{angleitner2023exponential}, based on a suitable decomposition of the set of elements into two groups, the following weaker Caccioppoli-type inequality was proved without the condition $h_{K}\lesssim \delta$:
\begin{equation}
 |u|_{H^{1}(D)}\lesssim \delta^{-1}\Vert u\Vert_{L^{2}(D^{\ast})}\quad \text{for all}\;\; u\in V_{h,B}(D^{\ast}).   
\end{equation}
\end{rem}}

With the weak approximation property and the Caccioppoli-type inequality, we see that the local approximations have exponential convergence rates.
\begin{theorem}[Exponential local convergence]\label{upperbound:5-2}
For each $i=1,\cdots,M$, \cref{thm:3-1} holds true with $C_{\rm cac}^{\rm I} = C\big(a^{2}_{\rm max}a^{-1}_{\rm min} +1 \big)^{1/2}$, $C_{\rm cac}^{\rm II} = C(a^{-1}_{\rm min} +1)^{1/2}\Vert{\bm b}\Vert_{L^{\infty}(\omega_i^{\ast})}$, and $\alpha = 1/d$, where $C>0$ only depends on $d$ and the shape regularity of the mesh.
\end{theorem}

Finally, we note that the quasi-optimal convergence of the method in the discrete setting follows from the same verification procedure as in the continuous setting, and the result remains the same. 

\subsection{Linear elasticity}\label{sec:linear-elasticity}
We consider the equilibrium equations for a linear heterogeneous elastic material: Find ${\bm u}=(u_{1},\ldots,u_{d})^{\top}:\Omega\rightarrow \mathbb{R}^{d}$ such that
\begin{equation}\label{eq:5-3-1}
\left\{
\begin{array}{lll}
 -{\rm div}({\bm \sigma}({\bm u})) = {\bm f}\,\quad &{\rm in}\;\, \Omega \\[2mm]
  \quad \quad\;\;\; {\bm \sigma}({\bm u}) = {\bm C}:{\bm \epsilon}({\bm u}) \,\quad &{\rm in}\;\, \Omega  \\[2mm]
  \;\,\qquad \quad \quad\,  {\bm u} = 0\quad \,&{\rm on}\;\,\Gamma_{\mathsf{D}}\\[2mm]
  \qquad \;\; {\bm \sigma}({\bm u}){{\bm n}} = {\bm g}  \quad \,&{\rm on}\;\,\Gamma_{\mathsf{N}},
\end{array}
\right.
\end{equation}
where ${\bm \sigma}$ is the Cauchy stress tensor, and ${\bm \epsilon}({\bm u}):=\big(\nabla {\bm u} + \nabla {\bm u}^{\top}\big)/2$ is the strain tensor. We assume that the fourth-order elasticity tensor ${\bm C} = (c_{ijkl})$ satisfies (i)  for each $i,j,k,l=1,\ldots,d$, $c_{ijkl}\in L^{\infty}(\Omega)$, (ii) $({\bm C}{\bm \tau}):{\bm \mu} = ({\bm C}{\bm \mu}):{\bm \tau}$ for any ${\bm \tau}, {\bm \mu}\in \mathbb{R}^{d\times d}_{\rm sym}$, and (iii) there exists $0<C_{\rm min}< C_{\rm max}< +\infty$ such that for any ${\bm \tau}\in \mathbb{R}^{d\times d}_{\rm sym}$ and ${\bm x}\in \Omega$, $C_{\rm min}|{\bm \tau}|^{2}\leq {\bm C}{\bm \tau} : {\bm \tau}\leq C_{\rm max} |{\bm \tau}|^{2}$, with $|{\bm \tau}|^{2}= \sum_{ij}\tau_{ij}^{2}$. Finally, we assume that ${\bm f}\in {\bm H}^{1}(\Omega)^{\prime}$ and ${\bm g}\in {\bm H}^{-1/2}(\Gamma_{\mathsf{N}})$. We refer to \cite{abdulle2006analysis,chung2014generalized,chung2019mixed,henning2016multiscale,harder2016hybrid,xia2015high} for multiscale methods for linear elasticity problems.

For any subdomain $D\subset \Omega$, we define $ {\bm H}_{\mathsf{D}}^{1}(D)=\big\{{\bm v}\in {\bm H}^{1}(D): {\bm v} = {\bf 0}\;\; {\rm on} \;\;\Gamma_{\mathsf{D}} \big\}$,  ${\bm H}_{\mathsf{D},0}^{1}(D)=\big\{{\bm v}\in {\bm H}_{\mathsf{D}}^{1}(D): {\bm v} = {\bf 0}\;\; {\rm on} \;\;\partial D\cap\Omega \big\}$, and set
\begin{equation}\label{eq:5-3-5}
\begin{array}{ll}
 \displaystyle B_{D}({\bm u}, {\bm v})=\int_{D}({\bm C}{\bm \epsilon}({\bm u})):{\bm \epsilon}({\bm v}) \,d{\bm x}, \qquad \Vert {\bm u}\Vert_{B,D} := \big(B_{D}({\bm u},{\bm u}) \big)^{1/2}. 
 \end{array}
\end{equation}
Similarly as before, we simply write $B(\cdot,\cdot)$ and $\Vert\cdot\Vert_{B}$ when $D=\Omega$. $({\bm H}_{\mathsf{D}}^{1}(\Omega), \Vert\cdot\Vert_{B})$ is a Hilbert space by Korn's inequalities \cite[Theorem 2.5]{oleinik1992mathematical}. Therefore, setting $\mathcal{H}(\Omega) = {\bm H}_{\mathsf{D}}^{1}(\Omega)$, $\Vert\cdot\Vert_{\mathcal{H}(\Omega)} = \Vert\cdot\Vert_{B}$, and $\mathcal{L}(\Omega) = {\bm L}^{2}(\Omega)$, we see that \cref{ass:2-1-0} is trivially satisfied with $C_{b}=C_{1} =1$ and $C_{0}=0$. Also, \cref{ass:2-1-1} holds true with $\mathcal{H}(D) =  {\bm H}_{\mathsf{D}}^{1}(D)$, $\mathcal{H}_{0}(D) =  {\bm H}_{\mathsf{D},0}^{1}(D)$, and with the local bilinear form $B_{D}(\cdot,\cdot)$ defined by \cref{eq:5-3-5}. The weak formulation of problem \cref{eq:5-3-1} is to find ${\bm u}\in {\bm H}^{1}_{\mathsf{D}}(\Omega)$ such that 
\begin{equation}\label{eq:5-3-4}
B({\bm u},{\bm v})= F({\bm v}):=\langle {\bm f}, {\bm v}\rangle_{\Omega} + \langle {\bm g},{\bm v}\rangle_{\Gamma_\mathsf{N}} \quad \text{for all}\;\; {\bm v}\in {\bm H}^{1}_{\mathsf{D}}(\Omega).
\end{equation}
Since $B(\cdot,\cdot)$ is symmetric and coercive on ${\bm H}^{1}_{\mathsf{D}}(\Omega)$, problem \cref{eq:5-3-4} is uniquely solvable. 

In this setting, the generalized harmonic space on $D\subset \Omega$ is given by
\begin{equation}
 {\bm H}_{B}(D)=\big\{{\bm u}\in {\bm H}_{\mathsf{D}}^{1}(D): B_{D}({\bm u}, {\bm v}) = 0 \quad \text{for all} \;\,{\bm v}\in {\bm H}_{\mathsf{D},0}^{1}(D) \big\}.   
\end{equation}
Let $\mathcal{K}$ denote the space of rigid body modes, i.e., $\mathcal{K}=\big\{{\bm a}+{\bm b}\times {\bm x}: {\bm a},{\bm b}\in\mathbb{R}^{d}\big\}$. By Korn's inequalities, we see that $\big({\bm H}_{B}(D)/\mathcal{K}, B_{D}(\cdot, \cdot) \big)$ is a Hilbert space if $\partial D\cap\Gamma_{\mathsf{D}}=\emptyset$ and $D$ satisfies the cone condition, and that $\big({\bm H}_{B}(D), B_{D}(\cdot, \cdot) \big)$ is a Hilbert space if $\partial D\cap\Gamma_{\mathsf{D}}\neq\emptyset$. Hence, \cref{ass:2-2-3} holds true with $B^{+}_{D}(\cdot,\cdot) = B_{D}(\cdot,\cdot)$. Next, we give the Caccioppoli-type inequality in this setting.
\begin{lemma}[Caccioppoli-type inequality]\label{lem:cac-5-3}
Let $D\subset D^{\ast}$, $\delta>0$, and $\eta\in C^{1}(\overline{D^{\ast}})$ be as in \cref{lem:cac-5-1}. Then, for any ${\bm u}, \,{\bm v}\in {\bm H}_{B}(D^{\ast})$,
\begin{equation}\label{cac:5-3-1}
B_{D^{\ast}}(\eta{\bm u}, \eta{\bm v}) = \int_{D^{\ast}} \big({\bm C}({\bm u}(\nabla \eta)^{\top} + (\nabla \eta) {\bm u}^{\top})\big): ({\bm v}(\nabla \eta)^{\top} + (\nabla \eta) {\bm v}^{\top}) \,d{\bm x}.
\end{equation}
In addition, for any ${\bm u}\in {\bm H}_{B}(D^{\ast})$,
\begin{equation}\label{cac:5-3-2}
\Vert {\bm u} \Vert_{B,D} \leq 2C^{1/2}_{\rm max}\delta^{-1}\Vert {\bm u}\Vert_{L^{2}(D^{\ast}\setminus D)}.
\end{equation}
\end{lemma}

\begin{proof}
As before, we use the abstract identity \cref{eq:5-4} to prove \cref{cac:5-3-1}. Consider the setting of Example~\ref{exam:5-2} and define $b_{\bf L}(\cdot,\cdot):L^{2}(D^{\ast},\mathbb{R}^{d\times d})\times L^{2}(D^{\ast},\mathbb{R}^{d\times d})\rightarrow \mathbb{R}$ as $b_{\bf L}({\bm \tau},{\bm \mu}) = \int_{D^{\ast}}({\bm C}{\bm \tau}): {\bm \mu}\,d{\bm x}$. Then, identity \cref{cac:5-3-1} follows from \cref{prop:5-1} and the fact that for any ${\bm u}, \,{\bm v}\in {\bm H}_{B}(D^{\ast})$, $b_{\mathbf{L}}\big({\bm \epsilon}({\bm u}), {\bm \epsilon}(\eta^{2}{\bm v})\big) =b_{\mathbf{L}}\big({\bm \epsilon}({\bm v}), {\bm \epsilon}(\eta^{2}{\bm u})\big) =0$. Inequality \cref{cac:5-3-2} can be derived from \cref{cac:5-3-1} by choosing a cut-off function $\eta$ as before.
\end{proof}

\textbf{Local eigenproblems}. By \cref{lem:cac-5-3} and the compact embeddings ${\bm H}^{1}(\omega_i^{\ast})\subset {\bm L}^{2}(\omega_i^{\ast})$, we see that the operators $\widehat{P}_{i}:{\bm H}_{B}(\omega_{i}^{\ast})\rightarrow  {\bm H}_{\mathsf{D},0}^{1}(\omega_i)$ defined by $\widehat{P}_{i}{\bm u} = \chi_{i} {\bm u}$ are compact. Hence, \cref{ass:2-2-3} holds true. The abstract local eigenproblems \cref{eq:2-24-5} take the form: Find $\lambda_{i}\in \mathbb{R}\cup \{+\infty\}$ and ${\bm \phi}_{i}\in {\bm H}_{B}(\omega_{i}^{\ast})$ such that
\begin{equation}\label{eigen:5-3}
B_{\omega^{\ast}_i}(\chi_{i}{\bm \phi}_{i},\,  \chi_{i}{\bm v}) = \lambda_{i}\,B_{\omega^{\ast}_i}({\bm \phi}_{i},\, {\bm v}) \quad  \forall {\bm v}\in {\bm H}_{B}(\omega_{i}^{\ast}).
\end{equation}
Using the identity \cref{cac:5-3-1}, the eigenproblems above can be rewritten as
\begin{equation}\label{neweigen:5-3}
\int_{\omega^{\ast}_i} {\bm C}\Theta({\bm \phi}_{i}): \Theta({\bm v})\,d{\bm x} = \lambda_{i}\,B_{\omega^{\ast}_i}({\bm \phi}_{i},\, {\bm v}) \quad \text{with}\;\; \Theta({\bm u}) := {\bm u}(\nabla \chi_{i})^{\top} + (\nabla \chi_{i}) {\bm u}^{\top}.
\end{equation}
Next we verify the exponential local convergence. Having proved the Caccioppoli-type inequality, we only need to show the weak approximation property, which is stated in the following lemma. 
\begin{lemma}\label{lem:wa-5-3-3}
Let $D\subset D^{\ast}\subset {D}^{\ast\ast}$ be subdomains of $\Omega$ with $\delta:={\rm dist} ({D}^{\ast},\partial D^{\ast\ast}\setminus \partial \Omega)>0$, and let $\mathsf{V}_{\delta}({D}^{\ast}\setminus D):= \big\{{\bm x}\in {D}^{\ast\ast}: {\rm dist}({\bm x}, {D}^{\ast}\setminus D)\leq \delta \big\}$. Then, for each integer $m\geq C_{1}\big|\mathsf{V}_{\delta}({D}^{\ast}\setminus D)\big|\delta^{-d}$, there exists an $m$-dimensional space $Q_{m}(D^{\ast\ast})\subset {\bm L}^{2}(D^{\ast\ast})$ such that for all ${\bm u}\in {\bm H}^{1}(D^{\ast\ast})$,
    \begin{equation}\label{wa:5-3-1}
        \inf_{{\bm v}\in Q_{m}(D^{\ast\ast})}\Vert{\bm u}-{\bm v}\Vert_{L^{2}(D^{\ast}\setminus D)} \leq C \big|\mathsf{V}_{\delta}({D}^{\ast}\setminus D)\big|^{1/d} m^{-1/d}\Vert{\bm \varepsilon} ({\bm u})\Vert_{L^{2}(D^{\ast\ast})},
    \end{equation}
where $C>0$ depends only on $d$.


\end{lemma}


To prove \cref{lem:wa-5-3-3}, we need the following auxiliary approximation result which may be of independent interest. The key point is that the constant depends only on the spatial dimension and the chunkiness parameter of the domain.
\begin{lemma}\label{lem:wa-5-3-2}
Suppose that $D\subset\mathbb{R}^{d}$ is a bounded domain of diameter $r$ and is star-shaped with respect to a ball $Q_{r_1}$ of radius $r_{1}$. Let $\mathcal{K}=\big\{{\bm a}+{\bm b}\times {\bm x}: {\bm a},{\bm b}\in\mathbb{R}^{d}\big\}$. Then for any ${\bm u}\in {\bm H}^{1}(D)$,
\begin{equation}\label{wa:5-3-2-1}
  \inf_{{\bm v}\in \mathcal{K}}\Vert {\bm u} -{\bm v}\Vert_{L^{2}(D)}\leq Cr\Vert {\bm \varepsilon} ({\bm u})\Vert_{L^{2}(D)},
\end{equation}
where $C$ depends only on $d$ and $r/r_{1}$.
\end{lemma}
\begin{proof}
Let ${\bm u}\in {\bm H}^{1}(D)$. It is easy to see that there exists a ${\bm v}\in \mathcal{K}$ such that ${\bm u}^{0} = (u^{0}_{1},\ldots,u^{0}_{d})^{\top}:={\bm u} - {\bm v}$ satisfies $\int_{D} {\bm u}^{0} \,d{\bm x} = {\bm 0}$, and $\int_{Q_{r_1}}\Big(\frac{\partial u^{0}_{i}}{\partial x_j} - \frac{\partial u^{0}_{j}}{\partial x_i}\Big) \,d{\bm x} = 0$, $i,j=1,\cdots,d$. Using Friedrichs' inequality and a nonstandard Korn-type inequality (see \cref{lem:wa-5-3-1}) gives that
\begin{equation}\label{wa:5-3-2-3}
\begin{array}{cc}
\displaystyle \Vert {\bm u}^{0}\Vert^{2}_{L^{2}(D)}\leq C r^{2}\Vert \nabla {\bm u}^{0}\Vert^{2}_{L^{2}(D)} \\[3mm]
\displaystyle \leq Cr^{2}\Big(C_{1}\Big(\frac{r}{r_1} \Big)^{d+1} \Vert {\bm \varepsilon} ({\bm u}^{0})\Vert^{2}_{L^{2}(D)} + C_{2}\Big(\frac{r}{r_1} \Big)^{d} \Vert \nabla {\bm u}^{0} \Vert^{2}_{L^{2}(Q_{r_1})}\Big),
\end{array}
\end{equation}
where $C_{1}$ and $C_{2}$ depend only on $d$. Furthermore, we can use properties of ${\bm u}^{0}$ again and a Korn's inequality for a sphere \cite{payne1961korn} to get
\begin{equation}\label{wa:5-3-2-4}
     \Vert \nabla {\bm u}^{0} \Vert_{L^{2}(Q_{r_1})} \leq C \Vert {\bm \varepsilon}({\bm u}^{0})\Vert_{{\bm L}^{2}(Q_{r_1})} \leq C \Vert {\bm \varepsilon}({\bm u}^{0})\Vert_{{\bm L}^{2}(D)},
\end{equation}
where $C>0$ depends only on $d$. Substituting \cref{wa:5-3-2-4} into \cref{wa:5-3-2-3} and noting that ${\bm \varepsilon}({\bm u}^{0})={\bm \varepsilon}({\bm u})$ complete the proof of inequality \cref{wa:5-3-2-1}.
\end{proof}

With \cref{lem:wa-5-3-2} in hand, the weak approximation property above can be proved similarly as \cref{lem:5-3} (i). The only difference is that we modify the definition of $Q_{m}(D^{\ast\ast})$ in \cref{eq:5-22} as $Q_{m}({D}^{\ast\ast}) = {\rm span}\big\{{\bm v}|_{T}: {\bm v}\in \mathcal{K},\;T\in \widetilde{\mathcal{T}}_{H}\big\}$ with $\mathcal{K}=\big\{{\bm a}+{\bm b}\times {\bm x}: {\bm a},{\bm b}\in\mathbb{R}^{d}\big\}$, and use a similar approximation result as \cref{eq:5-23} resulting from inequality \cref{wa:5-3-2-1}. The details are omitted. Combining the Caccioppoli inequality and the weak approximation property, we obtain the following result.
\begin{theorem}[Exponential local convergence]\label{upperbound:5-3}
For each $i=1,\cdots,M$, \cref{thm:3-1} holds true with $C_{\rm cac}^{\rm I} = 2C_{\rm max}^{1/2}$, $C_{\rm cac}^{\rm II} = 0$, and $\alpha = 1/d$.
\end{theorem}
Since we are in the elliptic case, the quasi-optimal convergence of the method follows directly from \cref{quasi-optimal-coercive}. 


\textbf{Discrete setting}. We end this subsection by briefly discussing the application of our theory to discretized linear elasticity problems with conforming FEs. In this setting, the weak approximation property is a direct consequence of \cref{lem:wa-5-3-3} due to the conformity of FE spaces. The discrete Caccioppoli inequality can be proved by following the same lines as in the proof of \cref{lem:cac-5-2} and by using Korn's inequalities. The details are omitted here. Therefore, \cref{thm:3-1} holds true in this setting.


\subsection{\texorpdfstring{$\bm{H}$}{$H$}$({\rm curl})$ elliptic problems}\label{sec:Hcurl-elliptic}
In this subsection, we consider the following curl-curl problem: Find ${\bm u}:\Omega\rightarrow \mathbb{C}^{d}$ such that
\begin{equation}\label{eq:5-4-1}
\left\{
\begin{array}{lll}
 \nabla \times (\nu \nabla\times {\bm u}) + \kappa {\bm u}= {\bm f}\,\quad &{\rm in}\;\, \Omega \\[2mm]
  \qquad \qquad\quad\quad  {\bm n}\times {\bm u} = {\bm 0} \,\quad &{\rm on}\;\, \Gamma,  
\end{array}
\right.
\end{equation}
with real-valued coefficient $\nu>0$ and complex-valued coefficient $\kappa$. Various equations of Maxwell type can be written as or reduced to this curl-curl problem \cref{eq:5-4-1}. For example, it represents the time-harmonic Maxwell's equations in an absorbing medium, by setting $\nu=\mu^{-1}$, $\kappa=-\textit{\textomega}^{2}\varepsilon + {\rm i}\textit{\textomega}\sigma$, and ${\bm f} = -{\rm i}\textit{\textomega}{\bm J}$, with the magnetic permeability $\mu$, the electric permittivity $\varepsilon$, the conductivity $\sigma$, and the current density ${\bm J}$. Another example is semi-discretized (in time) Maxwell's equations with an implicit time stepping scheme. In this case, $\nu,\kappa>0$, and ${\bm f}\in {\bm H}_{0}({\rm curl};\Omega)^{\prime}$. We refer to \cite{chung2019adaptive,gallistl2018numerical,henning2020computational} for multiscale methods for $H({\rm curl})$ elliptic problems.

We make the following assumption on the data of the problem: $\nu,\,\kappa\in L^{\infty}(\Omega)$, ${\bm f}\in{\bm H}_{0}({\rm curl};\Omega)^{\prime}$, and there exists $0<\nu_{\rm min}<\nu_{\rm max}<\infty$ such that $\nu_{\rm min}\leq \nu({\bm x})\leq \nu_{\rm max}$ for $a.e \; {\bm x}\in \Omega$, and exists a $C_{1}>0$ such that for any subdomain $D\subset \Omega$,
\begin{equation}\label{eq:5-4-2}
|B_{D}({\bm v}, {\bm v})|\geq C_{1} \Vert {\bm v}\Vert^{2}_{{\bm H}({\rm curl};D)} \quad \text{for all}\;\;{\bm v}\in {\bm H}({\rm curl};D),
\end{equation}
where the sesquilinear form $B_{D}:{\bm H}({\rm curl};D)\times {\bm H}({\rm curl};D)\rightarrow \mathbb{C}$ is defined as
\begin{equation}\label{eq:5-4-3}
    B_{D}({\bm v}, {\bm w}):=(\nu\nabla\times{\bm v}, \nabla\times{\bm w})_{L^{2}(D)} + (\kappa{\bm v}, {\bm w})_{L^{2}(D)}.
\end{equation}
Extension of the results below to matrix-valued coefficients is straightforward. 

Setting $B(\cdot,\cdot):=B_{\Omega}(\cdot,\cdot)$, the weak formulation of problem \cref{eq:5-4-1} is defined by: Find ${\bm u}\in {\bm H}_{0}({\rm curl};\Omega)$ such that 
\begin{equation}\label{eq:5-4-4}
 B({\bm u}, {\bm v}) = \langle {\bm f}, {\bm v} \rangle_{\Omega}\quad \quad \text{for all}\;\;{\bm v}\in {\bm H}_{0}({\rm curl};\Omega). 
\end{equation}
By assumption, the sesquilinear form $B(\cdot,\cdot)$ is bounded and coercive on ${\bm H}_{0}({\rm curl};\Omega)$. Therefore, \cref{ass:2-1-0} holds true and problem \cref{eq:5-4-4} is uniquely solvable. For any subdomain $D\subset \Omega$, we define 
\begin{equation}
{\bm H}_{\Gamma}({\rm curl};D) = \{{\bm v}\in {\bm H}({\rm curl};D):{\bm v} = {\bm 0}\;\;\text{on}\;\;\partial D\cap \Gamma\}.    
\end{equation}
Then, \cref{ass:2-1-1} holds true with $\mathcal{H}(D) =  {\bm H}_{\Gamma}({\rm curl};D)$, $\mathcal{H}_{0}(D) =  {\bm H}_{0}({\rm curl};D)$, and with the local bilinear form $B_{D}(\cdot,\cdot)$ defined by \cref{eq:5-4-3}.

While there is no essential difficulty in treating the general case, we restrict ourselves to the case $\kappa\geq \kappa_{\rm min}>0$ in the rest of this subsection to simplify the presentation. Moreover, we assume that all the functions involved are real-valued. In this case, the bilinear form $B_{D}(\cdot,\cdot)$ is symmetric and coercive on ${\bm H}({\rm curl};D)$, and we use $\Vert {\bm v}\Vert_{B,D} := \big(B_{D}({\bm v}, {\bm v})\big)^{1/2}$ as the norm for ${\bm H}({\rm curl};D)$.

For any subdomain $D\subset \Omega$, we define the generalized harmonic space
\begin{equation}
 {\bm H}_{B}({\rm curl};D)=\big\{{\bm u}\in {\bm H}_{\Gamma}({\rm curl};D): B_{D}({\bm u}, {\bm v}) = 0 \quad \text{for all} \;\,{\bm v}\in {\bm H}_{0}({\rm curl};D)\big\}.   
\end{equation}
Clearly, $\big({\bm H}_{B}({\rm curl};D),\Vert \cdot \Vert_{B,D}\big)$ is a Hilbert space, and hence \cref{ass:2-2-3} holds true with $B^{+}_{D}(\cdot,\cdot)=B_{D}(\cdot,\cdot)$ and $\mathcal{K} = \emptyset$. An important property of ${\bm H}_{B}({\rm curl};D)$ is that 
\begin{equation}\label{divergence-free-property}
\nabla \cdot(\kappa {\bm v}) = 0 \quad \text{for all}\;\; {\bm v}\in {\bm H}_{B}({\rm curl};D).     
\end{equation}
Next we prove the Caccioppoli inequality in this setting, again based on the abstract identity \cref{eq:5-4}.

\begin{lemma}[Caccioppoli-type inequality]\label{lem:cac-5-4}
Let $D\subset D^{\ast}$, $\delta>0$, and $\eta\in C^{1}(\overline{D^{\ast}})$ be as in \cref{lem:cac-5-1}. Then, for any ${\bm u}, \,{\bm v}\in {\bm H}_{B}({\rm curl};D^{\ast})$,
\begin{equation}\label{cac:5-4-1}
B_{D^{\ast}}(\eta {\bm u},\eta{\bm v}) = \int_{D^{\ast}}\nu\nabla \eta\times {\bm u}\cdot \nabla \eta\times {\bm v} \,d{\bm x}.
\end{equation}
In addition, for any ${\bm u}\in {\bm H}_{B}({\rm curl};D^{\ast})$,
\begin{equation}\label{cac:5-4-2}
\Vert {\bm u} \Vert_{B,D} \leq \nu_{\rm max}\delta^{-1}\Vert {\bm u} \Vert_{L^{2}(D^{\ast}\setminus D)}.
\end{equation}
\end{lemma}
\begin{proof}
We consider the setting of Example~\ref{exam:5-3} and define $b_{\bf L}(\cdot,\cdot): {\bm L}^{2}(D^{\ast})\times {\bm L}^{2}(D^{\ast})\rightarrow \mathbb{R}$ as $b_{\bf L}({\bm u},{\bm v}) = \int_{D^{\ast}}\nu {\bm u}\cdot{\bm v}\,d{\bm x}$. Then, the identity \cref{eq:5-4-2} is a consequence of \cref{prop:5-1} and the following fact: for all ${\bm u},{\bm v}\in {\bm H}_{B}({\rm curl};D^{\ast})$,
\begin{equation}
b_{\mathbf{L}}\big(\nabla\times{\bm u}, \nabla\times(\eta^{2}{\bm v})\big) = b_{\mathbf{L}}\big(\nabla\times{\bm v}, \nabla\times(\eta^{2}{\bm u})\big) = - (\kappa\eta {\bm u},\eta{\bm v})_{L^{2}(D^{\ast})}.
\end{equation} 
Inequality \cref{cac:5-4-2} follows immediately from \cref{cac:5-4-1} by choosing a cut-off function $\eta$.
\end{proof}

In the following, we verify the compactness of operators $\widehat{P}_{i}:{\bm H}_{B}({\rm curl};\omega_i^{\ast})\rightarrow {\bm H}_{0}({\rm curl};\omega_i)$ defined by $\widehat{P}_{i}{\bm u} = \chi_{i}{\bm u}$. Here it is crucial that ${\rm dist}(\omega_i,\partial \omega_{i}^{\ast}\setminus\partial \Omega)>0$.

\begin{lemma}\label{compactness-5-1}
\Cref{ass:2-2-2} holds true if ${\rm dist}(\omega_i,\partial \omega_{i}^{\ast}\setminus\partial \Omega)>0$ ($1\leq i\leq M$).   
\end{lemma}
\begin{proof}
Let $\widehat{\omega}_{i}$ be an intermediate set such that $\omega_{i}\subset \widehat{\omega}_{i}\subset \omega_{i}^{\ast}$ and ${\rm dist}(\omega_{i},\partial \widehat{\omega}_{i} \setminus\partial \Omega)={\rm dist}(\widehat{\omega}_{i},\partial \omega_{i}^{\ast}\setminus\partial \Omega)>0$. We first show that ${\bm H}_{B}({\rm curl};\omega_{i}^{\ast})$ is compactly embedded into ${\bm L}^{2}(\widehat{\omega}_{i})$. Let $\eta\in C^{1}(\overline{\omega_i^{\ast}})$ satisfy $\eta =0 $ on $\partial \omega_i^{\ast}\setminus\partial \Omega$ and $\eta = 1$ in $\widehat{\omega}_{i}$, and define $S_{\eta}(\omega_i^{\ast}) = \{\eta {\bm v}:{\bm v}\in  {\bm H}_{B}({\rm curl};\omega_{i}^{\ast})\}$. Using the divergence-free property \cref{divergence-free-property} yields
\begin{equation}\label{eq:compact_emb_1}
S_{\eta}(\omega_i^{\ast})\subset \{{\bm v}\in {\bm H}({\rm curl};\omega_{i}^{\ast}) : \nabla\cdot (\kappa {\bm v})\in L^{2}(\omega_{i}^{\ast}),\;\; {\bm n}\times {\bm v} = 0\;\;{\rm on}\;\, \partial \omega_i^{\ast} \}.  
\end{equation}
On the other hand, it is known that the second set in \cref{eq:compact_emb_1} is compactly embedded into ${\bm L}^{2}(\omega_i^{\ast})$ \cite{bauer2016maxwell}. Hence, recalling that $\eta =1$ in $\widehat{\omega}_{i}$, we have the compact embedding ${\bm H}_{B}({\rm curl};\omega_{i}^{\ast})\subset {\bm L}^{2}(\widehat{\omega}_{i})$. Combining this compactness result and inequality \cref{cac:5-4-2} (on $\widehat{\omega}_{i}$ and $\omega_{i}$), we further deduce that the inclusion ${\bm H}_{B}({\rm curl};\omega_{i}^{\ast})\subset {\bm H}_{B}({\rm curl};\omega_{i})$ is compact. Therefore, \cref{ass:2-2-2} is verified.
\end{proof}

\textbf{Local eigenproblems}. In this setting, the local eigenproblems are defined by: Find $\lambda_{i}\in \mathbb{R}$ and ${\bm \phi}_{i}\in {\bm H}_{B}({\rm curl};\omega_{i}^{\ast})$ such that
\begin{equation}\label{eigen:5-4}
B_{\omega^{\ast}_i}(\chi_{i}{\bm \phi}_{i},\,  \chi_{i}{\bm v}) = \lambda_{i}\,B_{\omega^{\ast}_i}({\bm \phi}_{i},\, {\bm v}) \quad  \forall {\bm v}\in {\bm H}_{B}({\rm curl};\omega_{i}^{\ast}).
\end{equation}
Again, it follows from the identity \cref{cac:5-4-1} that problem \cref{eigen:5-4} is equivalent to
\begin{equation}
(\nu \nabla\chi_{i}\times {\bm \phi}_{i}, \nabla\chi_{i}\times {\bm v})_{L^{2}(\omega^{\ast}_i)}= \lambda_{i}\,B_{\omega^{\ast}_i}({\bm \phi}_{i},\, {\bm v}) \quad  \forall {\bm v}\in {\bm H}_{B}({\rm curl};\omega_{i}^{\ast}).
\end{equation}

Now we pass to the verification of the exponential local convergence. Since the Caccioppoli inequality has been established in \cref{lem:cac-5-4}, we only need to prove the weak approximation property. Moreover, noting that in this setting $C_{\rm cac}^{\rm II}=0$, only case $a$ of \cref{ass:3-1-2} is required for \cref{thm:3-1}. Compared to elliptic problems, verifying the weak approximation property for Maxwell type problems is considerably more complicated. This complication arises from two aspects. First, as previously mentioned, the compact embedding ${\bm H}_{B}({\rm curl};D)\subset {\bm L}^{2}(\widehat{D})$ only holds for $\widehat{D} \Subset D$. As such, functions in ${\bm H}_{B}({\rm curl};D)$ can only be approximated in the interior of $D$ in the ${\bm L}^{2}$-norm by finite-dimensional spaces. Secondly, due to the general $L^{\infty}$-coefficient $\kappa$, the regularity of functions in ${\bm H}_{B}({\rm curl};D)$ is typically much lower than $H^{1}$, making it difficult to use well developed approximation techniques for elliptic problems. 
\begin{figure}
    \centering
    \includegraphics[scale=0.5]{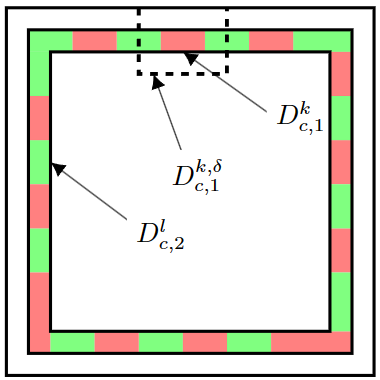}
    \caption{Illustration of the decomposition of $D_{c} = D^{\ast}\setminus D$ (the colored region). Here $D_{c} = D_{c,1}\cup D_{c,2}$, where $D_{c,1}$ and $D_{c,2}$ are the unions of the red and green colored regions, respectively. }
    \label{fig:domain_decomposition}
\end{figure}
\begin{proof}[Proof of the weak approximation property]
To overcome the two difficulties, we will use suitable Helmholtz decompositions, the obtained exponential convergence results for scalar elliptic problems, and a variational characterization of eigenvalues. For the sake of clarity, the proof is split into four steps. 

\textit{Step 1}. Let $D\subset D^{\ast}\subset {D}^{\ast\ast}\subset \Omega$ satisfy $\delta:={\rm dist} ({D}^{\ast},\partial D^{\ast\ast}\setminus \partial \Omega)>0$ and that ${\rm dist} ({\bm x},\partial D^{\ast}\setminus \partial \Omega)\leq \delta$ for all ${\bm x}\in \partial D\setminus \partial \Omega$. To fix ideas, we assume that $D\subset D^{\ast}\subset {D}^{\ast\ast}$ are nested concentric cubes or truncated cubes in the case of boundary subdomains. Denoting by $D_{c} = D^{\ast}\setminus D$, we introduce an operator $R_{c}:{\bm H}_{B}({\rm curl};D^{\ast\ast})\rightarrow {\bm L}^{2}(D^{\ast\ast})$ defined by $R_{c}{\bm u} = {\bm 1}_{D_c}{\bm u}$, where ${\bm 1}_{D_c}$ denotes the characteristic function of $D_{c}$, and consider the $n$-width of operator $R_{c}$:
\begin{equation}\label{wa:5-4-1-1}
d_{n}(R_{c}) = \inf_{Q(n)\subset {\bm L}^{2}(D^{\ast\ast})} \sup_{{\bm u}\in {\bm H}_{B}({\rm curl};D^{\ast\ast})}\inf_{{\bm v}\in Q(n)}\frac{\Vert R_{c}{\bm u} - {\bm v}\Vert_{L^{2}(D^{\ast\ast})} }{\Vert {\bm u}\Vert_{B, D^{\ast\ast}}}.    
\end{equation}
To estimate $d_{n}(R_{c})$, we partition the strip-shaped domain $D_{c}$ into several disjoint sets, denoted by $D_{c,1},\cdots, D_{c,r}$, such that $D_{c,i}\cap D_{c,j}= \emptyset$ for $i\neq j$ and $\cup_{1\leq i\leq r} D_{c,i} = D_{c}$. Moreover, for each $i$, the set $D_{c,i}$ is the union of a finite number of disjoint sets, denoted by $D^{1}_{c,i},\cdots, D^{m_i}_{c,i}$, such that for $1\leq k,l\leq m_i$,
\begin{equation}\label{wa:5-4-1-3}
 2\delta \leq {\rm diam}(D^{k}_{c,i})\leq 4\delta,\quad \text{and}\;\; {\rm dist}(D^{k}_{c,i},D^{l}_{c,i} ) \geq 2\delta  \quad \text{if}\;\; k\neq l.
\end{equation}
See \cref{fig:domain_decomposition} for an illustration of the two-dimensional case with $r=2$. Next, for each $i=1,\cdots, r$, we define the operator $R_{c,i}:{\bm H}_{B}({\rm curl};D^{\ast\ast})\rightarrow {\bm L}^{2}(D^{\ast\ast})$ by $R_{c,i}{\bm u} = {\bm 1}_{D_{c,i}}{\bm u}$. In view of the construction of the sets $D_{c,i}$, it is evident that 
\begin{equation}\label{wa:5-4-1-4}
R_{c} = R_{c,1}+\cdots+R_{c,r}. 
\end{equation}
For each $i=1,\cdots, r$, let $d_{n}(R_{c,i})$ denote the $n$-width of operator $R_{c,i}$. It follows from the additivity property \cref{additivity-property} that 
\begin{equation}\label{wa:5-4-1-7}
d_{nr}(R_{c}) \leq \sum_{i=1}^{r} d_{n}(R_{c,i}).  
\end{equation}
Hence, to verify the weak approximation property, it suffices to prove that
\begin{equation}\label{wa:5-4-1-2}
d_{n}(R_{c,i}) \leq C \big|\mathsf{V}({D}^{\ast}\setminus D,\,\delta)\big|^{1/d} n^{-1/d},   \quad i=1,\cdots,r,
\end{equation}
where $\mathsf{V}({D}^{\ast}\setminus D,\,\delta):= \big\{{\bm x}\in {D}^{\ast\ast}: {\rm dist}({\bm x}, {D}^{\ast}\setminus D)\leq \delta \big\}$. 


\textit{Step 2}. Let $1\leq i \leq r$ be fixed. The proof of \cref{compactness-5-1} shows that the operator $R_{c,i}$ is compact. Therefore, it follows from \cref{lem:1-1} that $d_{n}(R_{c,i}) = \lambda_{i,n+1}^{1/2}$, where $\lambda_{i,n+1}$ is the $(n+1)$-th eigenvalue of $R^{\ast}_{c,i}R_{c,i}: {\bm H}_{B}({\rm curl};D^{\ast\ast})\rightarrow {\bm H}_{B}({\rm curl};D^{\ast\ast})$.
To estimate $d_{n}(R_{c,i})$, we consider the distribution function of eigenvalues of $R^{\ast}_{c,i}R_{c,i}$. Let $N_{i}(\lambda, D^{\ast\ast}) = \#\{j:\lambda_{i,j} > \lambda\}$. By Glazman Lemma (\cref{Glazmann Lemma}), we see that
\begin{equation}\label{wa:5-4-2-2}
\begin{array}{ll}
N_{i}(\lambda, D^{\ast\ast}) = {\rm max\, dim}\big\{\mathcal{L}\subset {\bm H}_{B}({\rm curl};D^{\ast\ast}):   \\[2mm]
\qquad \Vert R_{c,i} {\bm u} \Vert^{2}_{L^{2}(D^{\ast\ast})} > \lambda \Vert {\bm u} \Vert^{2}_{B, D^{\ast\ast}}, \; {\bm u} \in \mathcal{L}\setminus {\bm 0} \big\}.  
\end{array}  
\end{equation}
Recall that $D_{c,i} = \cup_{k=1}^{m_i}D^{k}_{c,i}$. For each $k=1,\cdots, m_i$, let 
\begin{equation}\label{wa:5-4-2-3}
D^{k,\delta}_{c,i} = \{{\bm x}\in \Omega: {\rm dist}({\bm x}, D^{k}_{c,i})\leq \delta \},\quad \text{and let }\;\;D^{0,\delta}_{c,i} = D^{\ast\ast}\setminus \cup_{k=1}^{m_i}D^{k,\delta}_{c,i};
\end{equation}
see \cref{fig:domain_decomposition}. It follows that $D^{\ast\ast} = \cup_{k=0}^{m_i} D^{k,\delta}_{c,i}$. Moreover, we see from \cref{wa:5-4-1-3} that $D^{k,\delta}_{c,i} \cap D^{l,\delta}_{c,i} = \emptyset$ for $k\neq l$. To proceed, we consider the following space:
\begin{equation}\label{wa:5-4-2-4}
\widetilde{{\bm H}}_{B}({\rm curl};D^{\ast\ast}) := \big\{{\bm u}\in {\bm L}^{2}(D^{\ast\ast}): {\bm u}|_{D^{k,\delta}_{c,i}}\in {\bm H}_{B}({\rm curl};D^{k,\delta}_{c,i}),\;\; k=0,\cdots,m_i\big\}.     
\end{equation}
Clearly, ${\bm H}_{B}({\rm curl};D^{\ast\ast})\subset \widetilde{{\bm H}}_{B}({\rm curl};D^{\ast\ast})$. Combining this relation with \cref{wa:5-4-2-2} gives
\begin{equation}\label{wa:5-4-2-5}
\begin{array}{ll}
 N_{i}(\lambda, D^{\ast\ast}) \leq \widetilde{N}_{i}(\lambda, D^{\ast\ast}) :=  {\rm max\, dim}\big\{\mathcal{L}\subset \widetilde{\bm H}_{B}({\rm curl};D^{\ast\ast}):   \\[2mm]
\qquad \Vert R_{c,i} {\bm u} \Vert^{2}_{L^{2}(D^{\ast\ast})} > \lambda \Vert {\bm u} \Vert^{2}_{B, D^{\ast\ast}}, \; {\bm u} \in \mathcal{L}\setminus {\bm 0} \big\}.  
\end{array}  
\end{equation}
On the other hand, since $D^{k,\delta}_{c,i}$ are disjoint sets and $D^{\ast\ast} = \cup_{k=0}^{m_i} D^{k,\delta}_{c,i}$, we deduce that
\begin{equation}\label{wa:5-4-2-6}
 \widetilde{N}_{i}(\lambda, D^{\ast\ast}) =  \sum_{k=0}^{m_i} N_{i}(\lambda, D^{k,\delta}_{c,i})
\end{equation}
where 
\begin{equation}
\begin{array}{ll}
N_{i}(\lambda, D^{k,\delta}_{c,i}) = {\rm max\, dim}\big\{\mathcal{L}\subset {\bm H}_{B}({\rm curl};D^{k,\delta}_{c,i}):   \\[2mm]
\qquad \Vert R_{c,i} {\bm u} \Vert^{2}_{L^{2}(D^{k,\delta}_{c,i})} > \lambda \Vert {\bm u} \Vert^{2}_{B, D^{k,\delta}_{c,i}}, \; {\bm u} \in \mathcal{L}\setminus {\bm 0} \big\}.  
\end{array}  
\end{equation}
By Glazman Lemma, $N_{i}(\lambda, D^{k,\delta}_{c,i})$ are the distribution functions of eigenvalues of the problems:
\begin{equation}\label{wa:5-4-2-7}
({\bm u}, {\bm v})_{L^{2}(D^{k,\delta}_{c,i}\cap D_{c,i})} = \lambda_{i}^{k} B_{D^{k,\delta}_{c,i}}({\bm u}, {\bm v}) \quad \text{for all}\;\;{\bm v}\in {\bm H}_{B}({\rm curl};D^{k,\delta}_{c,i}).       
\end{equation}
Since $D^{0,\delta}_{c,i}\cap D_{c,i} = \emptyset$, we have $N_{i}(\lambda, D^{0,\delta}_{c,i})=0$. Combining \cref{wa:5-4-2-5,wa:5-4-2-6} gives
\begin{equation}\label{wa:5-4-2-8}
N_{i}(\lambda, D^{\ast\ast}) \leq \sum_{k=1}^{m_i} N_{i}(\lambda, D^{k,\delta}_{c,i}).    
\end{equation}
Recalling that for $k=1,\cdots,m_i$, $D^{k,\delta}_{c,i}\cap D_{c,i} = D_{c,i}^{k}$, and using \cref{lem:1-1} again, we see that the $(n+1)$-th eigenvalue of problem \cref{wa:5-4-2-7} has the following characterization:
\begin{equation}\label{wa:5-4-2-9}
 \sqrt{\lambda_{i}^{k,n+1}} = d_{n}(D^{k,\delta}_{c,i}, D^{k}_{c,i}):=   \inf_{Q(n)\subset {\bm L}^{2}(D^{k}_{c,i})} \sup_{{\bm u}\in {\bm H}_{B}({\rm curl};D^{k,\delta}_{c,i})}\inf_{{\bm v}\in Q(n)}\frac{\Vert {\bm u} - {\bm v}\Vert_{{\bm L}^{2}(D^{k}_{c,i})} }{\Vert {\bm u}\Vert_{B, D^{k,\delta}_{c,i}}}.
\end{equation}
To prove \cref{wa:5-4-1-2}, it suffices to prove that
\begin{equation}\label{wa:5-4-2-10}
d_{n}(D^{k,\delta}_{c,i}, D^{k}_{c,i}) \leq C|D^{k,\delta}_{c,i}|^{1/d}n^{-1/d},\quad k=1,\cdots, m_i.   
\end{equation}
Indeed, using \cref{wa:5-4-2-9,wa:5-4-2-10} yields that $N_{i}(\lambda, D^{k,\delta}_{c,i})\leq C|D^{k,\delta}_{c,i}|\lambda^{-d/2}$ ($1\leq k\leq m_i$). Inserting these estimates into \cref{wa:5-4-2-8} gives  
\begin{equation}\label{wa:5-4-2-11}
N_{i}(\lambda, D^{\ast\ast})\leq C\big|\mathsf{V}({D}^{\ast}\setminus D,\,\delta) \big|\lambda^{-d/2}.
\end{equation}
Recalling that $d_{n}(R_{c,i}) = \lambda_{i,n+1}^{1/2}$ and $N_{i}(\lambda, D^{\ast\ast}) = \#\{j:\lambda_{i,j} > \lambda\}$, we see that the two estimates \cref{wa:5-4-2-11,wa:5-4-1-2} are equivalent. Hence, the verification of the weak approximation property is reduced to proving the estimates \cref{wa:5-4-2-10}.

\textit{Step 3}. To avoid cumbersome notations in the proof of \cref{wa:5-4-2-10}, we write $D$ and $D^{\delta}$ in place of $D^{k}_{c,i}$ and $D^{k,\delta}_{c,i}$. Let $D\subset D^{\delta}$ be (Lipschitz) subdomains of $\Omega$ with ${\rm diam}(D^{\delta})\leq 6\delta $ and ${\rm dist}(D,D^{\delta})=\delta$. Denoting by 
\begin{equation}\label{wa:5-4-3-1}
  X(\kappa;D^{\delta}) = \big\{{\bm u}\in  {\bm H}_{0}({\rm curl}; D^{\delta}): \nabla\cdot (\kappa {\bm u}) = 0\quad \text{in}\;\,D^{\delta}\big\},  
\end{equation}
we see that any ${\bm u}\in {\bm H}_{B}({\rm curl};D^{\delta})$ admits the Helmholtz decomposition: 
\begin{equation}\label{eq:H-decomposition}
{\bm u} = {\bm u}_{0} + \nabla \phi, 
\end{equation}
where ${\bm u}_{0}\in X(\kappa;D^{\delta})$, and $\phi \in H^{1}(D^{\delta})$ satisfies (noting that $\nabla \cdot(\kappa {\bm u})=0$)
\begin{equation}\label{wa:5-4-3-2}
 \nabla \cdot (\kappa \nabla \phi)= 0\,\quad {\rm in}\;\, D^{\delta}; \quad 
  {\bm n}\times \nabla \phi = {\bm n}\times {\bm u} \,\quad {\rm on}\;\, \partial D^{\delta}.
\end{equation}
We first consider the approximation of ${\bm u}_{0}$. It is known that $X(\kappa;D^{\delta})$ is compactly embedded into ${\bm L}^{2}(D^{\delta})$ and that $\Vert \nabla\times{\bm u}\Vert_{L^{2}(D^{\delta})} $ is a norm on $X(\kappa;D^{\delta})$. Therefore, using a usual argument shows that there exists $S_{n}(D^{\delta})\subset {\bm L}^{2}(D^{\delta})$ such that for any ${\bm u}\in X(\kappa;D^{\delta})$,
\begin{equation}\label{wa:5-4-3-3}
    \inf_{{\bm v}\in S_{n}(D^{\delta})}\Vert {\bm u} - {\bm v}\Vert_{L^{2}(D^{\delta})} \leq \widetilde{\lambda}^{1/2}_{n+1} \Vert \nabla \times {\bm u}\Vert_{L^{2}(D^{\delta})},
\end{equation}
where $\widetilde{\lambda}_{n+1}$ is the $(n+1)$-th eigenvalue of the problem
\begin{equation}\label{wa:5-4-3-4}
    ({\bm u}, {\bm v})_{L^{2}(D^{\delta})} = \widetilde{\lambda}  (\nabla\times{\bm u}, \nabla\times{\bm v})_{L^{2}(D^{\delta})}\quad \text{for all}\;\;{\bm v}\in  X(\kappa;D^{\delta}).
\end{equation}
Weyl asymptotics for Maxwell operators with heterogeneous coefficients give the following asymptotic estimate \cite{birman2007weyl}:
\begin{equation}\label{wa:5-4-3-5}
\widetilde{\lambda}^{1/2}_{n+1} = C|D^{\delta} |^{1/d} n^{-1/d}(1+ o(1)). 
\end{equation}
When $D^{\delta}$ is a convex domain, we can derive a non-asymptotic approximation error estimate as follows. Let $\Pi_{1}:  X(\kappa;D^{\delta})\rightarrow  X(1;D^{\delta})$ be defined as $\Pi_{1}{\bm u} = {\bm u} - \nabla \psi$, where $\psi\in H_{0}^{1}(D^{\delta})$ satisfies $(\nabla \psi-{\bm u},\nabla \xi)_{L^{2}(D^{\delta})} =0 $ for all $\xi\in H_{0}^{1}(D^{\delta})$, and let $\Pi_{\kappa}: X(1;D^{\delta})\rightarrow  X(\kappa;D^{\delta})$ be defined as $\Pi_{\kappa}{\bm u} = {\bm u} - \nabla \varphi$, where $\varphi\in H_{0}^{1}(D^{\delta})$ satisfies $(\kappa\nabla \varphi-\kappa{\bm u},\nabla \xi)_{L^{2}(D^{\delta})} =0 $ for all $\xi\in H_{0}^{1}(D^{\delta})$. It is easy to verify the following properties of operators $\Pi_{1}$ and $\Pi_{\kappa}$:
\begin{equation}\label{wa:5-4-3-6}
\begin{array}{cc}
\nabla \times \Pi_{1}{\bm u} = \nabla \times {\bm u},\quad  \Vert \Pi_{1}{\bm u} \Vert_{L^{2}(D^{\delta})}\leq \Vert {\bm u}\Vert_{L^{2}(D^{\delta})}; \\[2mm]
\nabla \times \Pi_{\kappa}{\bm u} = \nabla \times {\bm u},\quad  \Vert \kappa^{1/2}\Pi_{\kappa}{\bm u} \Vert_{L^{2}(D^{\delta})}\leq \Vert \kappa^{1/2}{\bm u}\Vert_{L^{2}(D^{\delta})}; \\[2mm]
\Pi_{1}\Pi_{\kappa} =\Pi_{\kappa}\Pi_{1}= {\rm Id}.
\end{array}
\end{equation}
Since the inclusion $X(1;D^{\delta})\subset {\bm H}^{1}(D^{\delta})$ holds for a convex domain $D^{\delta}$, there exists $\widetilde{\bm S}_{n}(D^{\delta})\subset {\bm L}^{2}(D^{\delta})$ such that for any ${\bm u}\in X(\kappa;D^{\delta})$,
\begin{equation}\label{wa:5-4-3-7}
     \inf_{{\bm v}\in \widetilde{\bm S}_{n}(D^{\delta})}\Vert \Pi_{1}{\bm u} - {\bm v}\Vert_{L^{2}(D^{\delta})} \leq C |D^{\delta} |^{1/d} n^{-1/d} \Vert \nabla (\Pi_{1}{\bm u})\Vert_{L^{2}(D^{\delta})}.
\end{equation}
Moreover, using a Friedrichs-type inequality \cite{saranen1982inequality} on convex domains and the fact that functions in $X(1;D^{\delta})$ are divergence-free gives $\Vert \nabla (\Pi_{1}{\bm u})\Vert_{L^{2}(D^{\delta})} \leq \Vert \nabla \times (\Pi_{1}{\bm u})\Vert_{L^{2}(D^{\delta})}$. Combining this estimate with \cref{wa:5-4-3-6,wa:5-4-3-7} yields that for any ${\bm u}\in X(\kappa;D^{\delta})$,
\begin{equation}\label{wa:5-4-3-8}
    \inf_{{\bm v}\in \widetilde{\bm S}_{n}(D^{\delta})}\Vert \Pi_{1}{\bm u} - {\bm v}\Vert_{L^{2}(D^{\delta})} \leq C |D^{\delta} |^{1/d} n^{-1/d} \Vert \nabla \times {\bm u}\Vert_{L^{2}(D^{\delta})}.   
\end{equation}
Next, we extend $\Pi_{\kappa}$ to all of ${\bm L}^{2}(D^{\delta})$ such that $\Vert \kappa^{1/2}\Pi_{\kappa}{\bm u} \Vert_{L^{2}(D^{\delta})}\leq \Vert \kappa^{1/2}{\bm u}\Vert_{L^{2}(D^{\delta})}$ holds, and let ${\bm S}_{n}(D^{\delta})= \Pi_{\kappa}\big(\widetilde{\bm S}_{n}(D^{\delta})\big)$. It follows from \cref{wa:5-4-3-6,wa:5-4-3-8} that for any ${\bm u}\in X(\kappa;D^{\delta})$,
\begin{equation}\label{wa:5-4-3-9}
\begin{array}{ll}
{\displaystyle \inf_{{\bm v}\in {\bm S}_{n}(D^{\delta})}\Vert \kappa^{1/2}({\bm u} - {\bm v})\Vert_{L^{2}(D^{\delta})} = \inf_{{\bm v}\in \widetilde{\bm S}_{n}(D^{\delta})}\Vert \kappa^{1/2}(\Pi_{\kappa}\Pi_{1}{\bm u} - \Pi_{\kappa}{\bm v})\Vert_{L^{2}(D^{\delta})} }  \\[3mm]
{\displaystyle  \leq \inf_{{\bm v}\in \widetilde{\bm S}_{n}(D^{\delta})}\Vert \kappa^{1/2}(\Pi_{1}{\bm u} - {\bm v})\Vert_{L^{2}(D^{\delta})} \leq C\kappa_{\rm max}^{1/2}  |D^{\delta} |^{1/d} n^{-1/d} \Vert \nabla \times {\bm u}\Vert_{L^{2}(D^{\delta})}. }
\end{array}
\end{equation}
Hence, we see that for any ${\bm u}\in X(\kappa;D^{\delta})$,
\begin{equation}\label{wa:5-4-3-10}
\inf_{{\bm v}\in {\bm S}_{n}(D^{\delta})}\Vert{\bm u} - {\bm v}\Vert_{L^{2}(D^{\delta})}   \leq  C(\kappa_{\rm max}/\kappa_{\rm min})^{1/2}  |D^{\delta} |^{1/d} n^{-1/d} \Vert \nabla \times {\bm u}\Vert_{L^{2}(D^{\delta})}. 
\end{equation}

\textit{Step 4}. Finally we approximate the second part in the decomposition \cref{eq:H-decomposition}, and we adopt the notation used in the previous step. First, let us show that $\phi$ defined by \cref{wa:5-4-3-2} is uniquely determined up to an additive constant. Using the relation $\nabla_{\partial D^{\delta}} \phi = ({\bm n}\times \nabla \phi)\times {\bm n}$ on $\partial D^{\delta}$, where $\nabla_{\partial D^{\delta}}$ denotes the surface gradient, we see that the boundary condition prescribed for $\phi$ is equivalent to $\nabla_{\partial D^{\delta}} \phi = ({\bm n} \times {\bm u})\times {\bm n} = \gamma_{T}({\bm u}) \in {\bm H}^{-1/2}({\rm curl};\partial D^{\delta})$. Hence, we conclude that $\phi\in H^{1/2}(\partial D^{\delta})$. Classical results for elliptic problems show that $\phi\in H^{1}(D^{\delta})$ with the following estimate:
\begin{equation}
\begin{array}{cc}
\Vert \kappa^{1/2}\nabla \phi\Vert_{L^{2}(D^{\delta})} \leq C\Vert \phi\Vert_{H^{1/2}(\partial D^{\delta}) }\\[3mm] 
\leq C\Vert \gamma_{T}({\bm u}) \Vert_{{\bm H}^{-1/2}({\rm curl};\partial D^{\delta})} \leq C\Vert {\bm u}\Vert_{{\bm H}({\rm curl};D^{\delta})},   
\end{array} 
\end{equation}
where the last inequality follows from the boundedness of the tangential trace operator $\gamma_{T}$ \cite[Theorem 3.31]{monk2003finite}. Using a standard scaling argument gives a refined estimate for $\phi$ as follows:
\begin{equation}\label{eq:refined-estimate}
\Vert \kappa^{1/2}\nabla \phi\Vert_{L^{2}(D^{\delta})} \leq C({\rm diam}(D^{\delta})^{-1} \Vert{\bm u} \Vert_{L^{2}(D^{\delta})}+ \Vert \nabla\times {\bm u}\Vert_{L^{2}(D^{\delta})} \big),
\end{equation}
where $C$ depends only on the shape of $D^{\delta}$, but not on its size. Finally, since ${\bm n}\times {\bm u} = {\bm 0}$ on $\partial D^{\delta}\cap \partial \Omega$, $\phi$ is a constant on $\partial D^{\delta}\cap \partial \Omega$. Therefore, noting that $\phi$ is uniquely up to a constant, we can choose a particular $\phi$ such that $\phi = 0$ on $\partial D^{\delta}\cap \partial \Omega$.

We now proceed to construct an approximation space for $\nabla \phi$. Let $H^{1}_{\Gamma}(D^{\delta}) = \{u\in H^{1}(D^{\delta}): u=0 \;\; \text{on}\;\; \partial D^{\delta}\cap \partial \Omega\}$, and define
\begin{equation}
H_{B}(D^{\delta}) = \big\{u\in H^{1}_{\Gamma}(D^{\delta}): (\kappa \nabla u,\nabla v)_{L^{2}(D^{\delta})} = 0 \;\;\text{for all}\;\; v\in H_{0}^{1}(D^{\delta}) \big\}.    
\end{equation}
We see that $\phi\in H_{B}(D^{\delta})$. Assume that $\partial D^{\delta}\cap \partial \Omega=\emptyset$ and consider the $n$-width
\begin{equation}
 d_{n}(D, D^{\delta}):= \inf_{Q(n)\subset H_{B}(D)/\mathbb{R}}\sup_{u\in H_{B}(D^{\delta})/\mathbb{R}}\inf_{v\in Q(n)}\frac{\Vert \kappa^{1/2}(\nabla u-\nabla v)\Vert_{L^{2}(D)}}{\Vert \kappa^{1/2}\nabla u\Vert_{L^{2}(D^{\delta})}}.   
\end{equation}
A similar $n$-width can be defined without involving the quotient spaces when $\partial D^{\delta}\cap \partial \Omega\neq \emptyset$. Recalling that ${\rm diam}(D^{\delta})\leq 6\delta$ and ${\rm dist}(D, \partial D^{\delta}\setminus\partial \Omega) = \delta$, we can apply \cref{thm:3-1} in the setting of scalar elliptic problems (see \cref{sec:convection-diffusion}) to deduce that there exist $n_{0} = (2e\Theta)^{d}$ and $b_{0} = (2e\Theta+1)^{-1}$, such that for any $n>n_{0}$,
\begin{equation}
 d_{n}(D, D^{\delta}) \leq e^{-b_{0}n^{1/d}},   
\end{equation}
where $\Theta $ is independent of $\delta$, $D$, and $D^{\delta}$. Denote by $Q^{\rm opt}(n)$ the optimal approximation space of $d_{n}(D, D^{\delta})$, and define
\begin{equation}
S_{n}(D^{\delta}) = \big\{{\bm 1}_{D}\nabla \varphi: \varphi\in Q^{\rm opt}(n) \} \subset {\bm L}^{2}(D^{\delta}).
\end{equation}
It follows that for any $n>n_{0}$ and any $u\in H_{B}(D^{\delta})$,
\begin{equation}\label{eq:appro-error}
    \inf_{{\bm v}\in S_{n}(D^{\delta})}\Vert \kappa^{1/2}(\nabla u - {\bm v})\Vert_{L^{2}(D)}\leq e^{-b_{0}n^{1/d}} \Vert \kappa^{1/2}\nabla u\Vert_{L^{2}(D^{\delta})}. 
\end{equation}
Since $\phi\in H_{B}(D^{\delta})$, we can combine \cref{eq:refined-estimate,eq:appro-error} to conclude 
\begin{equation}\label{eq:appro-error-phi}
 \inf_{{\bm v}\in S_{n}(D^{\delta})}\Vert \kappa^{1/2}(\nabla \phi - {\bm v})\Vert_{L^{2}(D)}\leq Ce^{-b_{0}n^{1/d}} \big({\rm diam}(D^{\delta})^{-1} \Vert{\bm u} \Vert_{L^{2}(D^{\delta})}+ \Vert \nabla\times {\bm u}\Vert_{L^{2}(D^{\delta})} \big).   
\end{equation}
Let $n\geq \big(b_{0}^{-2}\ln(\delta^{-1})\big)^{d}$. It follows that $e^{-b_{0}n^{1/d}}\leq n^{-1/d} \delta^{2} $. Inserting this estimate into \cref{eq:appro-error-phi} and using the fact that ${\rm diam}(D^{\delta}) \geq \delta$ and $|D^{\delta}|^{1/d}\geq \delta$ lead us to
\begin{equation}\label{eq:appro-error-new}
 \inf_{{\bm v}\in S_{n}(D^{\delta})}\Vert \kappa^{1/2}(\nabla \phi - {\bm v})\Vert_{L^{2}(D)}\leq C|D^{\delta} |^{1/d} n^{-1/d} \Vert {\bm u}\Vert_{{\bm H}({\rm curl};D^{\delta})},   
\end{equation}
where we have assumed that $\delta \leq 1$ without loss of generality. Define $\Psi_{2n}(D^{\delta}) = {\bm S}_{n}(D^{\delta}) + S_{n}(D^{\delta})$. It follows from \cref{wa:5-4-3-10}, \cref{eq:appro-error-new}, and the decomposition \cref{eq:H-decomposition} that for any ${\bm u} \in {\bm H}_{B}({\rm curl};D^{\delta})$,
\begin{equation}\label{eq:final-error}
   \inf_{{\bm v}\in \Psi_{2n}(D^{\delta})}\Vert{\bm u} - {\bm v}\Vert_{L^{2}(D)}   \leq C|D^{\delta} |^{1/d} n^{-1/d} \Vert  {\bm u}\Vert_{B, D^{\delta}},  
\end{equation}
where $C$ may depend on $\kappa$, $\nu$, and the shape of $D^{\delta}$. In view of the construction of the subdomains $D^{\delta}$ (i.e. $D^{k,\delta}_{c,i}$), we make a moderate assumption that the constant $C$ in \cref{eq:final-error} can be bounded above for all subdomains, and thus conclude the proof of \cref{wa:5-4-2-10}. Hence, the weak approximation property in this setting is verified.
\end{proof}

Combining the Caccioppoli inequality and the weak approximation property leads us to the exponential local convergence in this setting.
\begin{theorem}[Exponential local convergence]\label{upperbound:5-4}
For each $i=1,\cdots,M$, \cref{thm:3-1} holds true with $C_{\rm cac}^{\rm I} = \nu^{1/2}_{\rm max}$, $C_{\rm cac}^{\rm II} = 0$, and $\alpha = 1/d$.
\end{theorem}
The quasi-optimal convergence of the method follows from \cref{quasi-optimal-coercive}.

We conclude this subsection by briefly discussing the application of our theory to discretized ${\bm H}({\rm curl})$ elliptic problems with N\'{e}d\'{e}lec elements. It is known that the discrete Caccioppoli inequality holds true in this setting \cite{faustmann2022h}. The weak approximation property, however, does not follows from the continuous analysis as contrasted with previous applications. Indeed, it requires a more sophisticated analysis based on previous techniques and properties of N\'{e}d\'{e}lec elements. We will address this problem in a forthcoming work.

\subsection{High-frequency time-harmonic wave problems}
In this subsection, we apply our abstract theory to time-harmonic wave problems with high wavenumber and heterogeneous coefficients, including Helmholtz equation, elastic wave equations, and Maxwell's equations. The focus of the analysis is on wavenumber explicit convergence rates. For simplicity, we impose the impedance boundary conditions on the whole boundary. Mixed boundary conditions can also be considered in a straightforward way and make no difference to our results.

\subsubsection{Helmholtz equation}
Given $k>0$, we consider the following heterogeneous Helmholtz equation: Find $u:\Omega\rightarrow\mathbb{C}$ such that
\begin{equation}\label{eq:5.5.1-1}
\left\{
\begin{array}{lll}
{\displaystyle -{\rm div}(A\nabla u) - k^{2}V^{2}u= f\,\quad {\rm in}\;\, \Omega }\\[2mm]
{\displaystyle \qquad \,\;A\nabla u\cdot {\bm n} - {\rm i}k\beta u=g  \quad \,{\rm on}\;\, \Gamma.}
\end{array}
\right.
\end{equation}
We make the following assumptions on the data of the problem. $A \in L^{\infty}(\Omega,\mathbb{R}_{\rm sym}^{d\times d})$ satisfies the uniform ellipticity condition as in \cref{sec:convection-diffusion}. $V \in L^{\infty}(\Omega)$, $f\in H^{1}(\Omega)^{\prime}$, $g\in H^{-1/2}(\Gamma)$, and $\beta\in L^{\infty}(\Gamma)$. Moreover, $\beta>0$ on $\Gamma$. For multiscale methods for Helmholtz problems, see, e.g., \cite{cao2002multiscale,chaumont2020multiscale,chen2023exponentially,freese2021super,fu2021wavelet,hauck2022multi,ohlberger2018new,peterseim2017eliminating}.

For any subdomain $D\subset\Omega$, we define the local sesquilinear forms $B_{D}(\cdot,\cdot)$, $B^{+}_{D}(\cdot,\cdot)$, $B^{+}_{D,k}(\cdot,\cdot):H^{1}(D)\times H^{1}(D)\rightarrow \mathbb{C}$ as 
\begin{equation}\label{Helmholtz:local-forms}
\begin{array}{cc}
{\displaystyle B_{D}(u,v)=\int_{D} \big(A\nabla u\cdot {\nabla \overline{v}}-k^{2}V^{2}u\overline{v}\big) \,d{\bm x} - {\rm i}k\int_{\partial D\cap \Gamma} \beta u\overline{v}\,d{\bm s},}\\[3mm]
{\displaystyle  B^{+}_{D}(u, v) = \int_{D}A\nabla u\cdot\nabla \overline{v} \,d{\bm x}, \;\; B^{+}_{D,k}(u, v) = \int_{D}\big(A\nabla u\cdot\nabla \overline{v} + k^{2}V^{2}u\overline{v}\big)\,d{\bm x},}
\end{array}   
\end{equation}
and set
\begin{equation}
\Vert u\Vert_{{B}^{+},D}:= \big(B^{+}_{D}(u, u)\big)^{1/2}, \quad  \Vert u\Vert_{{B}^{+},D,k}:= \big(B^{+}_{D,k}(u, u)\big)^{1/2}.  
\end{equation}
As before, we drop the subscript $D$ in the symbols above when $D=\Omega$. We see that \cref{ass:2-1-0} holds true with $\mathcal{H}(\Omega) = H^{1}(\Omega)$, $\Vert\cdot\Vert_{\mathcal{H}(\Omega)} = \Vert\cdot\Vert_{B^{+},k}$, and $\mathcal{L}(\Omega) = L^{2}(\Omega)$. The weak formulation of problem \cref{eq:5.5.1-1} is to find $u\in H^{1}(\Omega)$ such that 
\begin{equation}\label{eq:5.5.1-2}
B(u,v) = F(v):=\langle f,v\rangle_{\Omega} + \langle g,v\rangle_{\Gamma} \quad \text{for all} \;\;v\in H^{1}(\Omega).
\end{equation}
Define $H^{1}_{0,I}(D)=\{v\in H^{1}(D): v=0\;\;\text{on}\;\partial D\cap \Omega\}$. Then, \cref{ass:2-1-1} holds true with $\mathcal{H}(\Omega)=H^{1}(D)$, $\mathcal{H}_{0}(D)=H_{0,I}^{1}(D)$, and with $B_{D}(\cdot,\cdot)$ defined by \cref{Helmholtz:local-forms}.

Now we define the generalized harmonic space on $D\subset \Omega$ in this setting: 
\begin{equation}
H_{B}(D)=\big\{u\in H^{1}(D): B_{D}(u,v) = 0 \;\;\text{for all}\;\,v\in H_{0,I}^{1}(D)\}.
\end{equation}
Using a general result on equivalent norms for Sobolev spaces, it can be shown that $\big(H_{B}(D), {B}^{+}_{D}(\cdot,\cdot)\big)$ is a Hilbert space for subdomain $D$ that satisfies the cone condition. Therefore, \cref{ass:2-2-3} holds true with $\mathcal{K} = \emptyset$. Next we give the Caccioppoli inequality in this setting. Proceeding along the same lines as in the proof of \cref{lem:cac-5-1}, we can prove
\begin{lemma}[Caccioppoli-type inequality]\label{lem:cac-5-5-1}
Let $D\subset D^{\ast}$, $\delta>0$, and $\eta\in C^{1}(\overline{D^{\ast}})$ be as in \cref{lem:cac-5-1}. Then, for any $u, \,v\in H_{B}(D^{\ast})$,
\begin{equation}
{B}^{+}_{D^{\ast}}(\eta u, \eta v) = \int_{D^{\ast}}\big((A\nabla \eta \cdot \nabla \eta)u\overline{v} + k^{2}V^{2}\eta^{2} u\overline{v}\big) \,d{\bm x}.
\end{equation}
In addition, for any $u\in H_{B}(D^{\ast})$,
\begin{equation}
\Vert u\Vert_{{B}^{+},D} \leq a^{1/2}_{\rm max} \delta^{-1} \Vert u\Vert_{L^{2}(D^{\ast}\setminus D)} + k\Vert V\Vert_{L^{\infty}(D^{\ast})}\Vert u \Vert_{L^{2}(D^{\ast})},
\end{equation}
where $a_{\rm max}$ is the spectral upper of $A$.
\end{lemma}

\textbf{Local eigenproblems}. As in \cref{sec:convection-diffusion}, \cref{ass:2-2-2} holds true due to \cref{lem:cac-5-5-1} and the compact embedding $H^{1}(\omega_i^{\ast})\subset L^{2}(\omega_i^{\ast})$. With the setting above, the local eigenproblems are defined by: Find $\lambda_{i}\in \mathbb{R}$ and $\phi_{i}\in H_{B}(\omega_{i}^{\ast})$ such that
\begin{equation}\label{eigen:5-5-1}
{B}^{+}_{\omega^{\ast}_i,k}(\chi_{i}\phi_{i}, \chi_i v) = \lambda_{i} {B}^{+}_{\omega^{\ast}_i}(\phi_{i}, v)\quad \forall v\in H_{B}(\omega_{i}^{\ast}).
\end{equation} 
As before, there is an equivalent formulation of the above eigenproblems:
\begin{equation}\label{eigen:5-5-1-new}
\big(Q_{k}\phi_{i},\,  v\big)_{L^{2}(\omega_i^{\ast})}= \lambda_{i} {B}^{+}_{\omega^{\ast}_i}(\phi_{i}, v)\quad \forall v\in H_{B}(\omega_{i}^{\ast}),
\end{equation} 
where $Q_{k}:= A\nabla \chi_{i} \cdot \nabla \chi_{i} + 2k^{2}V^{2}\chi_{i}^{2}$.

Combining the Caccioppoli-type inequality (\cref{lem:cac-5-5-1}) and the weak approximation property (\cref{lem:5-3}) leads us to the exponential local convergence. 
\begin{theorem}[Exponential local convergence]\label{upperbound:5-5-1}
For each $i=1,\cdots,M$, \cref{thm:3-1} holds true with $C_{\rm cac}^{\rm I} = a^{1/2}_{\rm max}$, $C_{\rm cac}^{\rm II} = k\Vert V\Vert_{L^{\infty}(\omega_i^{\ast})}$, and $\alpha = 1/d$.
\end{theorem}
Next we show the quasi-optimal convergence of the method by verifying Assumptions~\ref{ass:3-2-1} and \ref{ass:3-2-2}. 
\begin{lemma}\label{adjoint-property-Helmholtz}
\Cref{ass:3-2-1} holds true with $C_{s}  = 1$.    
\end{lemma}
\begin{proof}
Let $f\in L^{2}(\Omega)$, and let $\hat{u}\in H^{1}(\Omega)$ such that
\begin{equation}
B(v,\hat{u}) = (v,f)_{L^{2}(\Omega)}\quad \forall v\in H^{1}(\Omega).
\end{equation} 
It is easy to observe that
\begin{equation}
B(\overline{\hat{u}},v) = (\overline{f},v)_{L^{2}(\Omega)}\quad \forall v\in H^{1}(\Omega).    
\end{equation}
Therefore, \cref{ass:3-2-1} holds true with $C_{s}  = 1$.
\end{proof}
\begin{lemma}\label{local-stablity-Helmholtz}
\Cref{ass:3-2-2} holds true with $C_{s}  = 2C_{P}a_{\rm min}^{-1/2}$, where $C_{P}$ is given by \cref{uniform-poincare}.   
\end{lemma}
The proof is identical to that of \cref{local-stablity-5-1}. Now we have the following quasi-optimal convergence result.
\begin{theorem}[Quasi-optimal convergence]\label{quasi-optimal-convergence-Helmholtz}
\Cref{cor:3-1} holds true with
$d_{\mathtt{max}} = O(C^{-1}_{\rm stab}k^{-1})$ and $H^{\ast}_{\mathtt{max}} = O(k^{-1})$.
\end{theorem}
\begin{proof}
Use \cref{adjoint-property-Helmholtz,local-stablity-Helmholtz}, and the fact that $C_{1} = O(1)$, $C_{0} = O(k^{2})$, and $C_{b} = O(1)$ in this setting.
\end{proof}

Finally, we note that, as in previous applications, our theory is also applicable to discretized Helmholtz equations in the FE setting with very similar theoretical results. The key point is the discrete Caccioppoli inequality, which can be proved using the same techniques as in the proof of \cref{lem:cac-5-2}. The details are omitted here.

\subsubsection{Elastic wave equations}
Given $\textit{\textomega}>0$, we consider the following heterogeneous elastic wave equations: Find ${\bm u}:\Omega\rightarrow \mathbb{C}^{d}$ such that
\begin{equation}\label{eq:5.5.2-1}
\left\{
\begin{array}{lll}
 -{\rm div}({\bm \sigma}({\bm u})) - \textit{\textomega}^{2}\rho {\bm u}  = {\bm f}\,\quad &{\rm in}\;\, \Omega \\[2mm]
  \qquad \;\; {\bm \sigma}({\bm u}){{\bm n}} - {\rm i}\textit{\textomega} A{\bm u} = {\bm g}  \quad \,&{\rm on}\;\,\Gamma,
\end{array}
\right.
\end{equation}
with the material density $\rho$, the stress tensor ${\bm \sigma}({\bm u}) = {\bm C}:{\bm \epsilon}({\bm u})$, and the strain tensor ${\bm \epsilon}({\bm u})=\big(\nabla {\bm u} + \nabla {\bm u}^{\top}\big)/2$. $\rho\in L^{\infty}(\Omega)$, ${\bm f}\in {\bm H}^{1}(\Omega)^{\prime}$, and ${\bm g}\in {\bm H}^{-1/2}(\Gamma)$. $A\in L^{\infty}(\Gamma,\mathbb{R}_{\rm sym}^{d\times d})$ is positive definite for all ${\bm x}\in \Gamma$. The elasticity tensor ${\bm C}$ satisfies the same conditions as in \cref{sec:linear-elasticity}. We refer to \cite{fu2019efficient,chung2014generalized-2,brown2023multiscale} for multiscale methods for frequency-domain elastic wave problems.

For any subdomain $D\subset \Omega$, we define the sesquilinear forms $B_{D}(\cdot,\cdot)$, $B^{+}_{D}(\cdot,\cdot)$, $B^{+}_{D,\textit{\textomega}}(\cdot,\cdot): {\bm H}^{1}(D)\times {\bm H}^{1}(D)\rightarrow \mathbb{C}$ as
\begin{equation}\label{sesquilinear-forms-elastic-wave}
\begin{array}{cc}
{\displaystyle  B_{D}({\bm u}, {\bm v}):= \int_{D} \big( {\bm C}{\bm \epsilon}({\bm u}):{\bm \epsilon}(\overline{\bm v}) - \textit{\textomega}^{2} \rho {\bm u}\cdot \overline{\bm v} \big) \,d{\bm x} - {\rm i} \textit{\textomega} \int_{\partial D\cap \Gamma} A{\bm u}\cdot \overline{\bm v} \, d{\bm s},}\\[3mm]
{\displaystyle B^{+}_{D}({\bm u}, {\bm v}):= \int_{D}{\bm C}{\bm \epsilon}({\bm u}):{\bm \epsilon}(\overline{\bm v}) \,d{\bm x}, \;\; B^{+}_{D,\textit{\textomega}}({\bm u}, {\bm v}):= \int_{D} \big( {\bm C}{\bm \epsilon}({\bm u}):{\bm \epsilon}(\overline{\bm v}) + \textit{\textomega}^{2} \rho {\bm u}\cdot \overline{\bm v} \big) \,d{\bm x}, }
\end{array}
\end{equation}
and set
\begin{equation}
\Vert {\bm u}\Vert_{B^{+}, D} := \big(B^{+}_{D}({\bm u}, {\bm u})\big)^{1/2},\quad \Vert {\bm u}\Vert_{B^{+}, D,\textit{\textomega}} := \big(B^{+}_{D,\textit{\textomega}}({\bm u}, {\bm u})\big)^{1/2}.
\end{equation}
We see that \cref{ass:2-0} holds true with $\mathcal{H}(\Omega)= {\bm H}^{1}(\Omega)$, $\Vert\cdot\Vert_{\mathcal{H}(\Omega)} = \Vert \cdot\Vert_{B^{+}, \textit{\textomega}}$, and $\mathcal{L}(\Omega) = {\bm L}^{2}(\Omega)$. The weak formulation of problem \cref{eq:5.5.2-1} is defined by: Find ${\bm u}\in {\bm H}^{1}(\Omega)$ such that
\begin{equation}\label{eq:5.5.2-3}
 B({\bm u}, {\bm v}) = F({\bm v}):=   \langle {\bm f},{\bm v}\rangle_{\Omega} + \langle {\bm g},{\bm v}\rangle_{\Gamma} \quad \text{for all} \;\;{\bm v}\in {\bm H}^{1}(\Omega).
\end{equation}
As in the scalar Helmholtz case, we let ${\bm H}_{0,I}^{1}(D)=\{{\bm v}\in {\bm H}^{1}(D): {\bm v} = {\bm 0}\;\;\text{on}\;\;\partial D\cap \Omega \}$. Then, \cref{ass:2-1-1} holds true with $\mathcal{H}(D) = {\bm H}^{1}(D)$, $\mathcal{H}_{0}(D) =  {\bm H}_{0,I}^{1}(D)$, and with the local sesquilinear form $B_{D}(\cdot,\cdot)$ defined by \cref{sesquilinear-forms-elastic-wave}. The generalized harmonic space on a subdomain $D\subset \Omega$ is then given by 
\begin{equation}
 {\bm H}_{B}(D)=\big\{{\bm u}\in {\bm H}^{1}(D): B_{D}({\bm u}, {\bm v}) = 0 \quad \text{for all} \;\,{\bm v}\in {\bm H}_{0,I}^{1}(D) \big\}.   
\end{equation}
As before, we see that $\big({\bm H}_{B}(D),  B^{+}_{D}(\cdot,\cdot) \big)$ is a Hilbert space for subdomain $D$ that satisfies the cone condition. Hence, \cref{ass:2-2-3} holds true with $\mathcal{K}=\emptyset$. Similar to \cref{lem:cac-5-5-1}, we have the following Caccioppoli inequality in this setting.
\begin{lemma}[Caccioppoli-type inequality]\label{lem:cac-5-5-2}
Let $D\subset D^{\ast}$, $\delta>0$, and $\eta\in C^{1}(\overline{D^{\ast}})$ be as in \cref{lem:cac-5-1}. Then, for any ${\bm u}, \,{\bm v}\in {\bm H}_{B}(D^{\ast})$,
\begin{equation}
{B}^{+}_{D^{\ast}}(\eta {\bm u}, \eta {\bm v}) = \int_{D^{\ast}} \big({\bm C}\Theta_{\eta}({\bm u}):\Theta_{\eta}(\overline{\bm v}) + \textit{\textomega}^{2}\rho \eta^{2}{\bm u}\cdot\overline{\bm v}\big) \,d{\bm x},
\end{equation}
where $\Theta_{\eta}({\bm u}):= {\bm u}(\nabla \eta)^{\top} + (\nabla \eta) {\bm u}^{\top}$. In addition, for any ${\bm u}\in {\bm H}_{B}(D^{\ast})$,
\begin{equation}
\Vert{\bm u}\Vert_{{B}^{+},D} \leq 2C^{1/2}_{\rm max}\delta^{-1} \Vert {\bm u}\Vert_{L^{2}(D^{\ast}\setminus D)} + \textit{\textomega}\Vert \rho\Vert^{1/2}_{L^{\infty}(D^{\ast})}\Vert {\bm u} \Vert_{L^{2}(D^{\ast})},
\end{equation}
where $C_{\rm max}$ is the spectral upper of ${\bm C}$.
\end{lemma}

\textbf{Local eigenproblems}. Clearly, \cref{ass:2-2-2} holds true as in the linear elasticity case. In this setting, the local eigenproblems are defined by: Find $\lambda_{i}\in \mathbb{R}$ and ${\bm \phi}_{i}\in {\bm H}_{B}(\omega_{i}^{\ast})$ such that
\begin{equation}\label{eigen:5-5-2}
{B}^{+}_{\omega^{\ast}_i,\textit{\textomega}}(\chi_{i}{\bm \phi}_{i}, \chi_i {\bm v}) = \lambda_{i} {B}^{+}_{\omega^{\ast}_i}({\bm \phi}_{i}, {\bm v})\quad \forall {\bm v}\in {\bm H}_{B}(\omega_{i}^{\ast}).
\end{equation} 
Similarly as before, problem \cref{eigen:5-5-2} is equivalent to
\begin{equation}\label{eigen:5-5-2-new}
\int_{\omega^{\ast}_{i}} \big({\bm C}\Theta_{\chi_{i}}({\bm \phi}_{i}):\Theta_{\chi_i}(\overline{\bm 
 v}) + 2\textit{\textomega}^{2}\rho \chi_{i}^{2}{\bm \phi}_{i}\cdot\overline{\bm v}\big) \,d{\bm x}= \lambda_{i} {B}^{+}_{\omega^{\ast}_i}({\bm \phi}_{i}, {\bm v})\quad \forall {\bm v}\in {\bm H}_{B}(\omega_{i}^{\ast}),
\end{equation} 
where $\Theta_{\chi_i}({\bm u}):= {\bm u}(\nabla \chi_{i})^{\top} + (\nabla \chi_{i}) {\bm u}^{\top}$.

With the Caccioppoli inequality (\cref{lem:cac-5-5-2}) and the weak approximation property (\cref{lem:wa-5-3-3}), we get the exponential local convergence in this setting. The result is very similar to that in the scalar Helmholtz case.
\begin{theorem}[Exponential local convergence]\label{upperbound:5-5-2}
For each $i=1,\cdots,M$, \cref{thm:3-1} holds true with $C_{\rm cac}^{\rm I} = C^{1/2}_{\rm max}$, $C_{\rm cac}^{\rm II} = \textit{\textomega}\Vert \rho\Vert^{1/2}_{L^{\infty}(\omega_{i}^{\ast})}$, and $\alpha = 1/d$.
\end{theorem}
Finally, we get the quasi-optimal convergence of the method in this setting under similar conditions as in the scalar Helmholtz case. The proof is exactly the same as that of \cref{quasi-optimal-convergence-Helmholtz}.
\begin{theorem}[Quasi-optimal convergence]\label{quasi-optimal-convergence-elasticwave}
\Cref{cor:3-1} holds true with
$d_{\mathtt{max}} = O(C^{-1}_{\rm stab} \textit{\textomega}^{-1})$ and $H^{\ast}_{\mathtt{max}} = O(\textit{\textomega}^{-1})$.
\end{theorem}

\subsubsection{Maxwell's equations}
Given $\textit{\textomega}>0$, we consider the following heterogeneous Maxwell's equations: Find ${\bm u}:\Omega\rightarrow \mathbb{C}^{d}$ such that
\begin{equation}\label{eq:5-5-3}
\left\{
\begin{array}{lll}
 \;\,\nabla \times (\mu^{-1} \nabla\times {\bm u}) - \textit{\textomega}^{2}\varepsilon {\bm u}= {\bm f}\,\quad &{\rm in}\;\, \Omega \\[2mm]
 (\mu^{-1}\nabla\times {\bm u})\times {\bm n} - {\rm i}\textit{\textomega} \lambda {\bm u}_{T} = {\bm g} \,\quad &{\rm on}\;\, \Gamma,  
\end{array}
\right.
\end{equation}
where ${\bm u}_{T} = ({\bm n}\times {\bm u})\times {\bm n}$ on $\Gamma$. We make the following assumptions on the data of the problem. $\mu,\varepsilon\in L^{\infty}(\Omega)$, and there exist constants $\mu_{\rm min}$, $\mu_{\rm max}$, $\varepsilon_{\rm min}$, $\varepsilon_{\rm max}$, such that $0<\mu_{\rm min}\leq \mu({\bm x})\leq \mu_{\rm max}$ and $0<\varepsilon_{\rm min}\leq \varepsilon({\bm x})\leq\varepsilon_{\rm max}$ for $a.e \; {\bm x}\in \Omega$. ${\bm f}\in {\bm H}({\rm curl};\Omega)^{\prime}$, ${\bm g}\in {\bm L}_{t}^{2}(\Gamma)$. $\lambda:\Gamma\rightarrow \mathbb{R}$ is Lipschitz continuous and $\lambda({\bm x})\geq \lambda_{\rm min}>0$ for all ${\bm x}\in \Gamma$. The case $\lambda=1$ is known as the Silver--M\"{u}ller boundary condition. We refer to \cite{cao2010multiscale,ciarlet2017approximation,henning2016new,henning2020computational,verfurth2019heterogeneous} for multiscale methods for time-harmonic Maxwell's equations.

For any subdomain $D\subset \Omega$, we introduce the spaces
\begin{equation}
\begin{split}
{\bm H}^{\rm imp}({\rm curl};D) &= \{{\bm u}\in {\bm H}({\rm curl};D): {\bm u}\times {\bm n} \in {\bm L}^{2}_{t}(\partial D\cap \Gamma) \}, \\
{\bm H}_{0,I}^{\rm imp}({\rm curl};D) &= \{{\bm u}\in {\bm H}^{\rm imp}({\rm curl};D): {\bm u}\times {\bm n} = {\bm 0}\quad \text{on}\;\;\partial D\cap \Omega \},  
\end{split}
\end{equation}
and define the sesquilinear form $B_{D}(\cdot,\cdot): {\bm H}^{\rm imp}({\rm curl};D)\times {\bm H}^{\rm imp}({\rm curl};D)\rightarrow \mathbb{C}$:
\begin{equation}\label{sesquilinear-forms-Maxwell}
B_{D}({\bm u}, {\bm v}):=(\mu^{-1}\nabla\times{\bm u}, \nabla\times{\bm v})_{{\bm L}^{2}(D)} -  \textit{\textomega}^{2}(\varepsilon{\bm u}, {\bm v})_{{\bm L}^{2}(D)} - {\rm i}\textit{\textomega}(\lambda {\bm u}_{T}, {\bm v}_{T})_{{\bm L}^{2}(\partial D\cap \Gamma)},
\end{equation}
and the scalar product $B^{+}_{D}(\cdot,\cdot): {\bm H}^{\rm imp}({\rm curl};D)\times {\bm H}^{\rm imp}({\rm curl};D)\rightarrow \mathbb{C}$:
\begin{equation}
B^{+}_{D}({\bm u}, {\bm v})=(\mu^{-1}\nabla\times{\bm u}, \nabla\times{\bm v})_{{\bm L}^{2}(D)} +  \textit{\textomega}^{2}(\varepsilon{\bm u}, {\bm v})_{{\bm L}^{2}(D)} + \textit{\textomega}(\lambda {\bm u}_{T}, {\bm v}_{T})_{{\bm L}^{2}(\partial D\cap \Gamma)}.
\end{equation}
Moreover, we set $\Vert {\bm u}\Vert_{B^{+},D} = \big(B^{+}_{D}({\bm u}, {\bm u})\big)^{1/2}$. The space ${\bm H}^{\rm imp}({\rm curl};D)$ equipped with the scalar product $B^{+}_{D}(\cdot, \cdot)$ is a Hilbert space \cite[Theorem 4.1]{monk2003finite}. Setting $B(\cdot,\cdot)=B_{\Omega}(\cdot, \cdot)$ and $\Vert \cdot\Vert_{B^{+}}=\Vert \cdot\Vert_{B^{+},\Omega}$, \cref{ass:2-1-0} holds true with $\mathcal{H}(\Omega) = {\bm H}^{\rm imp}({\rm curl};\Omega)$, $\Vert \cdot\Vert_{\mathcal{H}(\Omega)}= \Vert \cdot\Vert_{B^{+}}$, and $\mathcal{L}(\Omega) ={\bm L}^{2}(\Omega)$. The weak formulation of problem \cref{eq:5-5-3} is defined by: Find ${\bm u}\in {\bm H}^{\rm imp}({\rm curl};\Omega)$, such that
\begin{equation}
    B({\bm u}, {\bm v}) = F({\bm v}):= \langle {\bm f}, {\bm v} \rangle_{\Omega} + ({\bm g}, {\bm v}_{T})_{{\bm L}^{2}_{t}(\Gamma)}\quad \text{for all}\;\; {\bm v}\in {\bm H}^{\rm imp}({\rm curl};\Omega).
\end{equation}
Similarly as before, \cref{ass:2-1-1} holds true with $\mathcal{H}(D) = {\bm H}^{\rm imp}({\rm curl};D)$, $\mathcal{H}_{0}(D) = {\bm H}_{0,I}^{\rm imp}({\rm curl};D)$, and with the local sesquilinear form $B_{D}(\cdot,\cdot)$ defined by \cref{sesquilinear-forms-Maxwell}.

With the setting above, the generalized harmonic space on $D\subset \Omega$ is defined by
\begin{equation}\label{gharmonic-5-5-3}
 {\bm H}_{B}^{\rm imp}({\rm curl};D) = \{{\bm u}\in {\bm H}^{\rm imp}({\rm curl};D): B_{D}({\bm u}, {\bm v}) = 0\;\;\forall {\bm v}\in {\bm H}_{0,I}^{\rm imp}({\rm curl};D)\}.     
\end{equation}
Since ${\bm H}_{B}^{\rm imp}({\rm curl};D)$ is a closed subspace of ${\bm H}^{\rm imp}({\rm curl};D)$, $\big({\bm H}_{B}^{\rm imp}({\rm curl};D),B^{+}_{D}(\cdot, \cdot)\big)$ is a Hilbert space. Hence, \cref{ass:2-2-3} holds true with $\mathcal{K} = \emptyset$. We also note that for any ${\bm u}\in  {\bm H}_{B}^{\rm imp}({\rm curl};D)$, $\nabla\cdot(\varepsilon {\bm u}) = 0$. Next we give the Caccioppoli inequality in this setting.
\begin{lemma}[Caccioppoli-type inequality]\label{lem:cac-5-5-3}
Let $D\subset D^{\ast}$, $\delta>0$, and $\eta\in C^{1}(\overline{D^{\ast}})$ be as in \cref{lem:cac-5-1}. Then, for any ${\bm u}, \,{\bm v}\in {\bm H}^{\rm imp}_{B}({\rm curl};D^{\ast})$,
\begin{equation}\label{cac:5-5-3-1}
\big(\mu^{-1} \nabla\times (\eta{\bm u}),\nabla\times (\eta{\bm v})\big)_{L^{2}(D^{\ast})} = \int_{D^{\ast}}\big(\mu^{-1}\nabla \eta\times {\bm u}\cdot \nabla \eta\times \overline{\bm v} + \textit{\textomega}^{2}\varepsilon (\eta {\bm u})\cdot (\eta \overline{\bm v}) \big) \,d{\bm x}.
\end{equation}
In addition, for any ${\bm u}\in {\bm H}^{\rm imp}_{B}({\rm curl};D^{\ast})$,
\begin{equation}\label{cac:5-5-3-2}
\Vert {\bm u} \Vert_{B^{+},D} \leq \sqrt{3}\mu^{-1/2}_{\rm min}\delta^{-1}\Vert {\bm u} \Vert_{L^{2}(D^{\ast}\setminus D)} + \sqrt{3}\textit{\textomega} \,\varepsilon^{1/2}_{\rm max} \Vert {\bm u} \Vert_{L^{2}(D^{\ast})}.
\end{equation}
\end{lemma}
\begin{proof}
The identity \cref{cac:5-5-3-1} follows from the abstract identity \cref{eq:5-4} and a similar argument as in \cref{lem:cac-5-4}. To prove \cref{cac:5-5-3-2}, we choose a cut-off function $\eta\in C^{1}(\overline{D}^{\ast})$ such that $\eta = 1$ in $\overline{D}$, $\eta = 0$ on $\partial D^{\ast}\cap \Omega$, and $|\nabla \eta|\leq \delta^{-1}$. Since ${\bm u}\in {\bm H}^{\rm imp}_{B}({\rm curl};D^{\ast})$, we see that $B_{D^{\ast}}({\bm u}, \eta^{2}{\bm u}) = 0$. Taking the imaginary part of this equation and using the Cauchy--Schwarz inequality give
\begin{equation}\label{cac:5-5-3-3}
\textit{\textomega} \Vert\lambda^{1/2}\eta {\bm u}_{T}\Vert^{2}_{L^{2}(\partial D^{\ast}\cap \Gamma)}\leq \frac{1}{4}\Vert \mu^{-1/2}\nabla\times (\eta {\bm u})\Vert_{L^{2}(D^{\ast})}^{2} + \Vert \mu^{-1/2}\nabla\eta\times {\bm u}\Vert_{L^{2}(D^{\ast})}^{2}.    
\end{equation}
Combining \cref{cac:5-5-3-1,cac:5-5-3-3} and using the properties of $\eta$ yield \cref{cac:5-5-3-2}.
\end{proof}


Next we verify \cref{ass:2-2-2} concerning the compactness of operators $\widehat{P}_{i}$.
\begin{lemma}\label{lem:compact-5-3-1}
\Cref{ass:2-2-2} holds true if ${\rm dist}(\omega_i,\partial \omega_{i}^{\ast}\setminus\partial \Omega)>0$ ($1\leq i\leq M$).      
\end{lemma}
\begin{proof}
The case $\partial \omega_i^{\ast}\cap \Gamma = \emptyset$ has been proved in \cref{compactness-5-1}, and hence, we assume that $\partial \omega_i^{\ast}\cap \Gamma \neq \emptyset$. Let $\{{\bm u}_{n}\}_{n=1}^{\infty}$ be a bounded sequence in ${\bm H}^{\rm imp}_{B}({\rm curl};\omega_{i}^{\ast})$. For each $n$, we let $p_{n}\in H^{1}(\omega_i^{\ast})$ be defined by
\begin{equation}
\nabla\cdot(\varepsilon \nabla p_{n}) = 0\;\; {\rm in}\; \omega_i^{\ast},\quad  \varepsilon\nabla p_{n}\cdot{\bm n} = 0\;\;\text{on}\;\partial \omega_i^{\ast}\cap \Omega,\; \nabla_{T} p_{n} = {\bm u}_{n,T} \;\;\text{on}\;\,\partial \omega_i^{\ast}\cap \Gamma,  
\end{equation}
where $\nabla_{T}$ denotes the tangential gradient. Clearly, $p_{n}$ is uniquely determined up to an additive constant with the estimate $\Vert p_{n}\Vert_{H^{1}(\omega_i^{\ast})}\leq C\Vert {\bm u}_{n,T}\Vert_{{\bm H}^{-1/2}(\partial \omega_i^{\ast}\cap\Gamma)}$. But, noting that $\{{\bm u}_{n,T}\}_{n=1}^{\infty}$ is a bounded sequence in ${\bm L}^{2}_{t}(\partial \omega_i^{\ast}\cap\Gamma)$ and that ${\bm H}_{t}^{-1/2}(\partial \omega_i^{\ast}\cap\Gamma)$ is compactly embedded into ${\bm L}^{2}_{t}(\partial \omega_i^{\ast}\cap\Gamma)$, there is a function $p\in H^{1}(\omega_i^{\ast})$ and a subsequence, still denoted by $\{p_{n}\}_{n=1}^{\infty}$ such that 
\begin{equation}\label{compact-5-5-3-1}
    p_{n}\rightarrow p\;\,\; \text{strongly in}\;\, H^{1}(\omega_i^{\ast})\;\,{\rm as}\;\,n\rightarrow \infty.
\end{equation}
For each $n$, let ${\bm u}_{n}^{0} = {\bm u}_{n}-\nabla p_{n}$. Then, recalling that $\nabla\cdot(\varepsilon {\bm u}_{n})=0$, $\{{\bm u}_{n}^{0}\}_{n=1}^{\infty}$ is a bounded sequence in $X_{0}:=\{{\bm u}\in {\bm H}({\rm curl};\omega_i^{\ast}): \nabla\cdot(\varepsilon {\bm u}) = 0 \;\,\text{in}\;\,\omega_i^{\ast},\;\,{\bm n}\times {\bm u} = {\bm 0}\;\;\text{on}\;\,\partial \omega_{i}^{\ast}\cap \Gamma\}$. Since it has been proved in \cref{compactness-5-1} that $X_{0}$ is compactly embedded into ${\bm L}^{2}(\omega_i)$, there is a function ${\bm u}^{0}\in X_{0}$, and a subsequence, still denoted by $\{{\bm u}_{n}^{0}\}_{n=1}^{\infty}$ such that
\begin{equation}\label{compact-5-5-3-2}
 {\bm u}_{n}^{0} \rightarrow {\bm u}^{0}\;\,\; \text{strongly in}\;\, {\bm L}^{2}(\omega_i)\;\,{\rm as}\;\,n\rightarrow \infty.
\end{equation}
Combining \cref{compact-5-5-3-1,compact-5-5-3-2} yields that ${\bm H}^{\rm imp}_{B}({\rm curl};\omega_{i}^{\ast})$ is compactly embedded into ${\bm L}^{2}(\omega_{i})$, which leads to the compactness of operators $\widehat{P}_{i}$ by introducing an intermediate set $\widehat{\omega}_{i}$ and using inequality \cref{cac:5-5-3-2} as in \cref{compactness-5-1}.
\end{proof}

\textbf{Local eigenproblems}. The local eigenproblems in this setting are defined by: Find $\lambda_{i}\in \mathbb{R}$ and ${\bm \phi}_{i}\in {\bm H}^{\rm imp}_{B}({\rm curl};\omega_{i}^{\ast})$ such that
\begin{equation}\label{eigen:5-5-3}
B^{+}_{\omega^{\ast}_i}(\chi_{i}{\bm \phi}_{i},\,  \chi_{i}{\bm v}) = \lambda_{i}\,B^{+}_{\omega^{\ast}_i}({\bm \phi}_{i},\, {\bm v}) \quad  \forall {\bm v}\in {\bm H}^{\rm imp}_{B}({\rm curl};\omega_{i}^{\ast}).
\end{equation}
As before, problem \cref{eigen:5-5-3} has an equivalent formulation:
\begin{equation}
\begin{array}{cc}
\big(\mu^{-1} \nabla\chi_{i}\times {\bm \phi}_{i}, \nabla\chi_{i}\times {\bm v}\big)_{L^{2}(\omega^{\ast}_i)} + 2\textit{\textomega}^{2}\big(\varepsilon \chi_{i} {\bm \phi}_{i}, \chi_{i} {\bm v}\big)_{L^{2}(\omega^{\ast}_i)}  \\[3mm]
+ \,\textit{\textomega} \big(\lambda\chi_{i} {\bm \phi}_{i}, \chi_{i} {\bm v}\big)_{L^{2}(\partial\omega^{\ast}_i\cap \Gamma)} =\lambda_{i}\,B^{+}_{\omega^{\ast}_i}({\bm \phi}_{i},\, {\bm v}) \quad  \forall {\bm v}\in {\bm H}^{\rm imp}_{B}({\rm curl};\omega_{i}^{\ast}).
\end{array}
\end{equation}

We now pass to the verification of the exponential local convergence. Since the Caccioppoli inequality has been proved in \cref{lem:cac-5-5-3}, we only need to prove the weak approximation property (\cref{ass:3-1-2}) in this setting.

\begin{proof}[Proof of the weak approximation property]
Let $\omega_{i}\subset D\subset D^{\ast}\subset {D}^{\ast\ast}\subset \omega_i^{\ast}$ with $\delta:={\rm dist} ({D}^{\ast},\partial D^{\ast\ast}\setminus \Gamma)>0$ and ${\rm dist} ({\bm x},\partial D^{\ast}\setminus \Gamma)\leq \delta$ for all ${\bm x}\in \partial D\setminus \Gamma$. We first assume that $\partial \omega_i^{\ast}\cap \Gamma = \emptyset$. In this case, case $a$ of \cref{ass:3-1-2} was verified in \cref{sec:Hcurl-elliptic}. The verification of case $b$ is much simpler than case $a$ without resort to the domain decomposition techniques. Let ${\bm u}\in {\bm H}^{\rm imp}_{B}({\rm curl};D^{\ast\ast})$. Then, the following Helmholtz decomposition holds: ${\bm u} = {\bm u}^{0} + \nabla p$, where ${\bm u}^{0} \in X(\varepsilon; D^{\ast\ast})$ with $X(\varepsilon; D^{\ast\ast})$ defined similarly to \cref{wa:5-4-3-1}, and $p\in H^{1}(D^{\ast\ast})$ satisfying $\nabla\cdot(\varepsilon \nabla p) =0$ in $D^{\ast\ast}$. The approximation of ${\bm u}^{0}$ follows the same lines as in step 3 of the verification of \cref{ass:3-1-2} for ${\bm H}({\rm curl})$ elliptic problems, and the approximation of the generalized harmonic function $p$ in $D^{\ast}$ follows from \cref{lem:3-1} (case (i)) in the scalar elliptic setting. In this way, case $b$ is verified. Case $c$ of \cref{ass:3-1-2} follows by combining case $a$ and case $b$ as in the proof of \cref{lem:5-3}. The details are omitted.

Next, we assume that $\partial \omega_i^{\ast}\cap \Gamma \neq \emptyset$. The impedance boundary condition greatly complicates the proof of the weak approximation property, as compared with the condition ${\bm n}\times {\bm u} ={\bm 0}$. Let ${\bm u}\in {\bm H}^{\rm imp}_{B}({\rm curl};D^{\ast\ast})$, and set $\Gamma^{\ast\ast}=\partial D^{\ast\ast}\cap \Gamma$ for convenience. A simple but suboptimal strategy is to use the Helmholtz decomposition as in \cref{lem:compact-5-3-1} and then apply previous results. More precisely, we decompose ${\bm u}$ as ${\bm u} = {\bm u}^{0}+\nabla p$, with ${\bm n}\times {\bm u}^{0} = {\bm 0}$ on $\Gamma^{\ast\ast}$, and $p|_{\Gamma^{\ast\ast}} \in H^{1}(\Gamma^{\ast\ast})$. ${\bm u}^{0}$ can be approximated (in the ${\bm L}^{2}$-norm) as in \cref{sec:Hcurl-elliptic} with a convergence rate $O(n^{-1/d})$. However, the convergence rate for approximating $\nabla p$ is only $O(n^{-1/{(2d-2)}})$($\sim$ the approximation numbers of the compact embedding $H^{1}(\Gamma^{\ast\ast}) \subset H^{1/2}(\Gamma^{\ast\ast})$). The key to recovering the optimal convergence rate (i.e., $O(n^{-1/d})$) of the weak approximation, is to exploit a hidden boundary regularity of functions in ${\bm H}^{\rm imp}_{B}({\rm curl};D^{\ast\ast})$. To do this, we first notice that ${\bm u}\in {\bm H}^{\rm imp}_{B}({\rm curl};D^{\ast\ast})$ satisfies
\begin{equation}\label{wa:5-5-3-1}
\nabla \times (\mu^{-1} \nabla\times {\bm u}) = \textit{\textomega}^{2}\varepsilon {\bm u}\;\;{\rm in}\;\, D^{\ast\ast},\;\; {\rm and}\;\;  (\mu^{-1}\nabla\times {\bm u})\times {\bm n} = {\rm i}\textit{\textomega} \lambda {\bm u}_{T}\;\;{\rm on} \;\,\Gamma^{\ast\ast}.
\end{equation}
Consider the problem of finding $\psi\in H^{1}(D^{\ast\ast})$ such that
\begin{equation}\label{wa:5-5-3-2}
\begin{array}{cc}
\nabla \cdot(\varepsilon \nabla \psi) = 0 \;\;{\rm in}\;\, D^{\ast\ast}, \quad \nabla_{T} \psi = {\bm 0} \;\; {\rm on}\;\, \partial D^{\ast\ast}\setminus \Gamma^{\ast\ast},\\[3mm] 
{\rm i}\textit{\textomega} \lambda \nabla_{T} \psi = (\mu^{-1}\nabla\times {\bm u})\times {\bm n} \;\; {\rm on}\;\, \Gamma^{\ast\ast},
\end{array}
\end{equation}
where $\nabla_{T} \psi = ({\bm n}\times \nabla \psi)\times {\bm n}$ denotes the tangential gradient. Using \cref{wa:5-5-3-1}, we deduce that $\mu^{-1}\nabla\times {\bm u} \in {\bm H}({\rm curl};D^{\ast\ast})$ and thus $(\mu^{-1}\nabla\times {\bm u})\times {\bm n} \in {\bm H}^{-1/2}_{t}(\Gamma^{\ast\ast})$. Therefore, problem \cref{wa:5-5-3-2} is uniquely solvable (up to an additive constant) in $H^{1}(D^{\ast\ast})$. Let ${\bm u}^{0}={\bm u}-\nabla \psi$. Since ${\bm u}\in {\bm H}^{\rm imp}_{B}({\rm curl};D^{\ast\ast})$ and $\psi$ satisfies \cref{wa:5-5-3-2}, we see that 
\begin{equation}\label{wa:5-5-3-3}
\begin{array}{cc}
 \nabla\times {\bm u}^{0} = \nabla \times {\bm u}\;\; \text{in}\;\,D^{\ast\ast},\quad \nabla\cdot(\varepsilon {\bm u}^{0}) = 0 \quad \text{in}\;\,D^{\ast\ast}, \\[3mm]
 {\bm n}\times {\bm u}^{0} = {\bm 0}\;\;\text{on}\;\,\Gamma^{\ast\ast},\quad {\bm n}\times {\bm u}^{0} = {\bm n}\times {\bm u}\;\;\text{on}\;\,\partial D^{\ast\ast}\setminus \Gamma^{\ast\ast}.
\end{array}
\end{equation}
Since ${\bm u}^{0}$ satisfies ${\bm n}\times {\bm u}^{0} = {\bm 0}$ on $\Gamma^{\ast\ast}$, it follows from the weak approximation result proved in \cref{sec:Hcurl-elliptic} and a similar argument as above for interior subdomains that there exists an $n$-dimensional space ${\bm S}^{0}_{n}(D^{\ast\ast})$ such that for any $a,b\geq 0$,
\begin{equation}\label{appro_u0}
 \begin{array}{lll}
 {\displaystyle \inf_{{\bm v}\in {\bm S}^{0}_{n}(D^{\ast\ast})} \big(a\Vert {\bm u}^{0}-{\bm v}\Vert_{L^{2}({D}^{\ast}\setminus D)} + b\Vert {\bm u}^{0}-{\bm v}\Vert_{L^{2}({D}^{\ast})}\big) }\\[3ex]
 {\displaystyle \quad \leq C \big(a\big|\mathsf{V}({D}^{\ast}\setminus D,\,\delta)\big|^{1/d}+b|{D}^{\ast\ast}|^{1/d}\big) n^{-1/d} \Vert {\bm u}^{0} \Vert_{{\bm H}({\rm curl};{D^{\ast\ast}})}.}
 \end{array}
 \end{equation}

It remains to approximate $\nabla \psi$. We first assume that $\Gamma^{\ast\ast}$ lies within one of the faces of $\Gamma$ (recalling that $\Omega$ is a polyhedral domain). The general case will be reduced to this case later.  Since $\Gamma^{\ast\ast}$ is smooth, the space $H^{3/2}(\Gamma^{\ast\ast})$ can be defined in a usual way using the Gagliardo seminorm. We claim that $\psi|_{\Gamma^{\ast\ast}}\in H^{3/2}(\Gamma^{\ast\ast})$. To see this, we first notice that ${\rm div}_{T}\big((\mu^{-1}\nabla\times {\bm u})\times {\bm n}\big)\in H^{-1/2}(\Gamma^{\ast\ast})$ since $\mu^{-1}\nabla\times {\bm u} \in {\bm H}({\rm curl};D^{\ast\ast})$ (see, e.g., \cite[Theorem 3.31]{monk2003finite}). Here ${\rm div}_{T}$ denote the tangential divergence on $\Gamma^{\ast\ast}$. Using the boundary condition of $\psi$ on $\Gamma^{\ast\ast}$, we further have
\begin{equation}\label{wa:5-5-3-4}
{\rm div}_{T}(\lambda \nabla_{T} \psi) = -{\rm i}\textit{\textomega}^{-1} {\rm div}_{T}\big((\mu^{-1}\nabla\times {\bm u})\times {\bm n}\big)\in H^{-1/2}(\Gamma^{\ast\ast}) \quad \text{in}\;\; \Gamma^{\ast\ast}.   
\end{equation}
Moreover, since $\nabla_{T}\psi = {\bm 0}$ on $\partial D^{\ast\ast}\setminus \Gamma^{\ast\ast}$ and $D^{\ast\ast}$ has a connected boundary, we deduce that $\psi$ is a constant on $\partial D^{\ast\ast}\setminus \Gamma^{\ast\ast}$. For convenience, we let $\psi = 0$ on $\partial D^{\ast\ast}\setminus \Gamma^{\ast\ast}$. Hence, $\psi = 0$ on $\partial \Gamma^{\ast\ast}$. Combining this boundary condition with \cref{wa:5-5-3-4}, and using classical regularity results for elliptic equations (recalling that $\lambda$ is Lipschitz continuous on $\Gamma$), we have that $\psi|_{\Gamma^{\ast\ast}}\in H^{3/2}(\Gamma^{\ast\ast})\cap H_{0}^{1}({\Gamma}^{\ast\ast})$. Furthermore, setting $H = {\rm diam}(D^{\ast\ast})$, we can use \cref{wa:5-5-3-1}, \cref{wa:5-5-3-4}, and a scaling argument to derive
\begin{equation}\label{wa:5-5-3-5}
\begin{array}{ll}
\Vert \psi\Vert_{H^{3/2}(\Gamma^{\ast\ast})}\leq C\textit{\textomega}^{-1} \big\Vert {\rm div}_{T}\big((\mu^{-1}\nabla\times {\bm u})\times {\bm n}\big)\big\Vert_{H^{-1/2}(\Gamma^{\ast\ast})}\\[3mm]
\leq C\textit{\textomega}^{-1} \big(H^{-1}\Vert \mu^{-1}\nabla\times {\bm u}\Vert_{{\bm L}^{2}(D^{\ast\ast})} + \Vert \nabla\times(\mu^{-1}\nabla\times {\bm u})\Vert_{{\bm L}^{2}(D^{\ast\ast})} \big)\\[3mm]
\leq  C \big(\textit{\textomega}^{-1}H^{-1}\Vert \mu^{-1}\nabla\times {\bm u}\Vert_{{\bm L}^{2}(D^{\ast\ast})} +  \textit{\textomega}\,\Vert\varepsilon {\bm u}\Vert_{{\bm L}^{2}(D^{\ast\ast})}\big),
\end{array}
\end{equation}
where $C$ may depend on the shapes of $D^{\ast\ast}$ and $\Gamma^{\ast\ast}$, but not on their sizes. Since $\omega_i\subset D^{\ast\ast}$, $H$ can be bounded from below. Suppose that $\textit{\textomega} \geq 1$ without loss of generality. It follows from \cref{wa:5-5-3-5} that
\begin{equation}\label{wa:5-5-3-7}
\Vert \psi\Vert_{H^{\frac32}(\Gamma^{\ast\ast})} \leq C\Vert {\bm u}\Vert_{B^{+},D^{\ast\ast}}.    
\end{equation}
To proceed, we define the operator $P:H^{3/2}(\Gamma^{\ast\ast})\cap H_{0}^{1}({\Gamma}^{\ast\ast})\rightarrow H^{1}(D^{\ast\ast})$ such that if $\varphi\in H^{3/2}(\Gamma^{\ast\ast})\cap H_{0}^{1}({\Gamma}^{\ast\ast})$ then $P\varphi=u\in H^{1}(D^{\ast\ast})$ satisfies
\begin{equation}\label{wa:5-5-3-8}
\nabla \cdot(\varepsilon \nabla u) = 0 \;\;{\rm in}\;\, D^{\ast\ast}, \quad u = 0 \;\; {\rm on}\;\, \partial D^{\ast\ast}\setminus \Gamma^{\ast\ast},\quad u = \varphi \;\; {\rm on}\;\, \Gamma^{\ast\ast}. 
\end{equation}
Clearly, $P=P_{3}P_{2}P_{1}$, where $P_{1}=id:H^{3/2}(\Gamma^{\ast\ast})\cap H_{0}^{1}({\Gamma}^{\ast\ast})\subset H_{0}^{1}({\Gamma}^{\ast\ast})$, $P_{2} = id: H^{1}(\partial D^{\ast\ast})\subset H^{1/2}(\partial D^{\ast\ast})$, and $P_{3}:H^{1/2}(\partial D^{\ast\ast})\rightarrow H^{1}(D^{\ast\ast})$ is defined by $P_{3}\varphi = u$ with $\nabla \cdot(\varepsilon \nabla u) = 0$ in $D^{\ast\ast}$ and $u=\varphi$ on $\partial D^{\ast\ast}$. Here we identify $H_{0}^{1}({\Gamma}^{\ast\ast})$ with a subspace of $H^{1}(\partial D^{\ast\ast})$. Let $d_{n}(P)$ be the $n$-width of operator $P$ (resp. $d_{n}(P_1)$ and $d_{n}(P_2)$). Then, it follows from the multiplicativity property \cref{eq:1-3} of $n$-widths that 
\begin{equation}\label{wa:5-5-3-9}
    d_{2n}(P) \leq d_{n}(P_1) d_{n}(P_2) \Vert P_{3}\Vert.
\end{equation}
By means of finite element error estimates (see, e.g., \cite[Chapter 14.3]{brenner2008mathematical}) and recalling that $\Gamma^{\ast\ast}$ and $\partial D^{\ast\ast}$ are polygons, we can easily show that
\begin{equation}\label{wa:5-5-3-10}
d_{n}(P_1)\leq C|\Gamma^{\ast\ast}|^{1/(2d-2)} n^{-1/(2d-2)},\quad d_{n}(P_2)\leq C|\partial D^{\ast\ast}|^{1/(2d-2)} n^{-1/(2d-2)},
\end{equation}
where $C$ is independent of the sizes of $\Gamma^{\ast\ast}$ and $\partial D^{\ast\ast}$. Inserting \cref{wa:5-5-3-10} into \cref{wa:5-5-3-9} leads us to
\begin{equation}\label{wa:5-5-3-11}
d_{2n}(P) \leq C \Vert P_{3}\Vert \,|\partial D^{\ast\ast}|^{1/(d-1)} n^{-1/(d-1)}.  
\end{equation}
Since the function $\psi$ defined by \cref{wa:5-5-3-2} lies in the image of operator $P$, it follows from the definition of the $n$-width $d_{n}(P)$ and \cref{wa:5-5-3-11} that there exists an $n$-dimensional space $\Psi_{n}(D^{\ast\ast})\subset H^{1}(D^{\ast\ast})$ such that 
\begin{equation}\label{wa:5-5-3-12}
\inf_{v\in \Psi_{n}(D^{\ast\ast})} \Vert \nabla(\psi - v)\Vert_{L^{2}(D^{\ast\ast})}\leq C |\partial D^{\ast\ast}|^{1/(d-1)} n^{-1/(d-1)} \Vert \psi \Vert_{H^{3/2}(\Gamma^{\ast\ast})}.
\end{equation}
Here we have used the fact that $\Vert P_{3}\Vert$ is bounded. Let ${\bm \Psi}_{n}(D^{\ast\ast}) = \nabla (\Psi_{n}(D^{\ast\ast})) \subset {\bm L}^{2}(D^{\ast\ast})$. Combining \cref{wa:5-5-3-7} and \cref{wa:5-5-3-12} gives
\begin{equation}\label{wa:5-5-3-13}
\inf_{{\bm v}\in {\bm \Psi}_{n}(D^{\ast\ast})} \Vert \nabla \psi - {\bm v}\Vert_{L^{2}(D^{\ast\ast})}\leq C |\partial D^{\ast\ast}|^{1/(d-1)} n^{-1/(d-1)} \Vert {\bm u}\Vert_{B^{+},D^{\ast\ast}}.
\end{equation}
If $\Gamma^{\ast\ast}$ is a union of faces $\Gamma^{\ast\ast}_{j}$, $j=1,\cdots,J$, then we can define $\psi = \sum_{j=1}^{J} \psi_{j}$, with $\psi_{j}\in H^{1}(D^{\ast\ast})$ ($j=1,\cdots,J)$ satisfying
\begin{equation}\label{face-decomposition}
\begin{array}{cc}
\nabla \cdot(\varepsilon \nabla \psi_j) = 0 \;\;{\rm in}\;\, D^{\ast\ast}, \quad \nabla_{T} \psi_{j} = {\bm 0} \;\; {\rm on}\;\, \partial D^{\ast\ast}\setminus \Gamma_{j}^{\ast\ast},\\[3mm] 
{\rm i}\textit{\textomega} \lambda \nabla_{T} \psi_j = (\mu^{-1}\nabla\times {\bm u})\times {\bm n} \;\; {\rm on}\;\, \Gamma^{\ast\ast}_{j}.
\end{array}
\end{equation}
Clearly, a similar approximation result as \cref{wa:5-5-3-13} holds for each $\psi_{j}$, and hence holds for $\psi$ as well. Let ${\bm S}_{n}(D^{\ast\ast}) = {\bm S}^{0}_{n}(D^{\ast\ast}) + {\bm \Psi}_{n}(D^{\ast\ast}) \subset {\bm L}^{2}(D^{\ast\ast})$. Combining \cref{appro_u0} and \cref{wa:5-5-3-13}, and noting that $\Vert {\bm u}^{0} \Vert_{{\bm H}({\rm curl};{D^{\ast\ast}})} \leq C\Vert {\bm u}\Vert_{B^{+},D^{\ast\ast}}$, we see that for any $a,b\geq 0$,
\begin{equation}\label{wa:5-5-3-14}
\begin{array}{ll}
\displaystyle  \inf_{{\bm v}\in {\bm S}_{n}(D^{\ast\ast})} \big(a\Vert {\bm u}-{\bm v}\Vert_{L^{2}({D}^{\ast}\setminus D)} + b\Vert {\bm u}-{\bm v}\Vert_{L^{2}({D}^{\ast})}\big)  \\[3mm]
  \leq C\Big(a\big|\mathsf{V}({D}^{\ast}\setminus D,\,\delta)\big|^{1/d}n^{-1/d}+b|{D}^{\ast\ast}|^{1/d}n^{-1/d} \;+ \\[3mm]
   \qquad \qquad  (a+b) |\partial D^{\ast\ast}|^{1/(d-1)} n^{-1/(d-1)} \Big)\Vert {\bm u}\Vert_{B^{+},D^{\ast\ast}}.
\end{array}
\end{equation}
Case $b$ of \cref{ass:3-1-2} follows directly from \cref{{wa:5-5-3-14}} by noting that $|\partial D^{\ast\ast}|^{1/(d-1)}\sim |D^{\ast\ast}|^{1/d}$ since $D^{\ast\ast}$ are assumed to be (truncated) nested concentric cubes between $\omega_i$ and $\omega_i^{\ast}$ (nevertheless, this relation is not essential because $|D^{\ast\ast}|$ and $|\partial D^{\ast\ast}|$ are both bounded). Case $ab$, however, is a bit delicate. The estimate \cref{wa:5-5-3-14} differs from \cref{eq:3-2} in the additional term $|\partial D^{\ast\ast}|^{1/(d-1)} n^{-1/(d-1)}$. But, recalling that in the proof of \cref{thm:3-1}, $\big|\mathsf{V}({D}^{\ast}\setminus D,\,\delta)\big|^{1/d} \sim N^{-1/d}\sim n^{-1/(d^{2}-d)}$, the two terms $\big|\mathsf{V}({D}^{\ast}\setminus D,\,\delta)\big|^{1/d} n^{-1/d}$ and $n^{-1/(d-1)}$ are of the same order. Hence, case (iii) in \cref{thm:3-1} still holds true in this case with $\alpha = 1/d$. 
\end{proof}

Combining the weak approximation property proved above and the Caccioppoli inequality \cref{cac:5-5-3-2} leads us to the following exponential local convergence result.
\begin{theorem}[Exponential local convergence]\label{upperbound:5-5-3}
For each $i=1,\cdots,M$, \cref{thm:3-1} holds true with $C_{\rm cac}^{\rm I} = \sqrt{3}\mu^{-1/2}_{\rm min}$, $C_{\rm cac}^{\rm II} = \sqrt{3}\textit{\textomega}\,\varepsilon^{1/2}_{\rm max}$, and $\alpha = 1/d$.
\end{theorem}

We conclude this subsection by noting that the quasi-optimal convergence of the method in this setting goes beyond the abstract framework established in \cref{sec:3-2}, as contrasted with the cases of Helmholtz and elastic wave equations. Indeed, \cref{ass:3-2-2} does not hold true in this case. In general, the proof of the quasi-optimal convergence of FEMs for Maxwell's equations depends heavily on suitable Helmholtz decompositions and an approximate divergence-free property of the error function. The latter is indeed the bottleneck in our case. To fix this problem and to stabilize the method, one can suitably enrich the test space. We will investigate this issue in future work.

\subsection{Fourth-order problems}\label{sec:fourth-order-problems}
So far, we have only considered second-order PDE problems. In this subsection, we apply our theory to fourth-order problems of physical interest. On a plane polygonal domain $\Omega$, we consider the problem of bending of plates with variable thickness: Find $u:\Omega\rightarrow \mathbb{R}$ such that 
\begin{equation}\label{eq:fourth-order problems}
\left\{
\begin{array}{cc}
{\displaystyle \frac{\partial^{2} }{\partial x^{2}}\Big[\mathtt{D}(x,y) \big(\frac{\partial^{2} u}{\partial x^{2}} + \nu \frac{\partial^{2} u}{\partial y^{2}}\big) \Big]+ 2(1-\nu) \frac{\partial^{2} }{\partial x \partial y} \Big[\mathtt{D}(x,y) \frac{\partial^{2} u}{\partial x \partial y}  \Big]}\\[2ex]
{\displaystyle +\frac{\partial^{2} }{\partial y^{2}}\Big[\mathtt{D}(x,y) \big(\frac{\partial^{2} u}{\partial y^{2}} + \nu \frac{\partial^{2} u}{\partial x^{2}}\big) \Big] = f(x,y)\qquad \text{in}\;\;\;\Omega,}\\[2ex]
{\displaystyle u = \frac{\partial u}{\partial {\bm n}}  = 0 \qquad \text{on}\;\;\;\Gamma.}
\end{array}
\right.
\end{equation}
Here $\mathtt{D}(x,y) = \frac{Eh^{3}(x,y)}{12(1-\nu^{2})}$ represents the flexural rigidity of the plate, where $E$ is Young’s modulus, $h(x,y)$ is the thickness, and $\nu\in [0,\frac12]$ is Poisson’s ratio of the plate. We assume that $h(x,y)\in C^{1}(\Omega)$, and that $0<h_{\rm min}\leq h(x,y)\leq h_{\rm max}$ for all $(x,y)\in \Omega$. The model \cref{eq:fourth-order problems} is called the clamped plate model. Other plate models (e.g., the simply-supported model) can also be considered.

For any subdomain $D\subset \Omega$, we define the bilinear form $B_{D}(\cdot,\cdot):H^{2}(D)\times H^{2}(D)\rightarrow \mathbb{R}: \,B_{D}(u,v)=\int_{D} \mathtt{D}(x,y) \big(\Delta u \Delta v - (1-\nu)(u_{xx}v_{yy} + u_{yy}v_{xx} - 2u_{xy}v_{xy})\big) dxdy$. Clearly, $B_{D}(\cdot,\cdot)$ is bounded on $H^{2}(D)\times H^{2}(D)$. Moreover, denoting by $\mathtt{D}_{\rm min}$ the lower bound of $\mathtt{D}(x,y)$, it is easy to show that 
\begin{equation}\label{fourth-order:coercive}
 B_D(u,u)\geq \frac{1}{2}\mathtt{D}_{\rm min}(1-\nu)|u|_{H^{2}(D)}\quad \text{for all}\;\; u\in H^{2}(D).
\end{equation}
Therefore, $\Vert\cdot\Vert_{B,D} := \big(B_{D}(\cdot,\cdot)\big)^{1/2}$ is a seminorm on $H^{2}(D)$ and a norm on $H^{2}_{0}(D)$. Setting $B(\cdot,\cdot) = B_{\Omega}(\cdot,\cdot)$ and $\Vert\cdot\Vert_{B} = \Vert\cdot\Vert_{B,\Omega}$, we see that \cref{ass:2-1-0} holds true with $\mathcal{H}(\Omega)= H_{0}^{2}(\Omega)$, $\Vert \cdot\Vert_{\mathcal{H}(\Omega)} = \Vert\cdot\Vert_{B}$, and $C_{0}=0$. Given $f\in H^{-2}(\Omega)$, the weak formulation of problem \cref{eq:fourth-order problems} is to find $u\in H^{2}_{0}(\Omega)$ such that
\begin{equation}\label{fourth-order:weak-formulation}
B(u,v) = \langle f,v\rangle_{\Omega}\qquad \text{for all}\;\; v\in H_{0}^{2}(\Omega).     
\end{equation}
It is clear that problem \cref{fourth-order:weak-formulation} is uniquely solvable. Define $H^{2}_{\Gamma}(D) = \{u\in H^{2}(\Omega): u={\partial u}/{\partial {\bm n}} = 0 \;\; \text{on}\;\;\partial D\cap \Gamma\}$. Then, \cref{ass:2-1-1} holds true with $\mathcal{H}(D) = H^{2}_{\Gamma}(D)$, $\mathcal{H}_{0}(D) = H^{2}_{0}(D)$, and with the local bilinear form $B_{D}(\cdot,\cdot)$ defined above.

In this setting, we let the partition of unity functions $\chi_{i}$ satisfy $\chi_{i}\in C^{2}(\overline{\omega_i})$ and $\chi_{i} =\chi_{i,x} = \chi_{i,y}= 0$ on $\partial\omega_i\setminus \Gamma$. Then, the abstract partition of unity $\{P_{i}\}_{i=1}^{M}$ in \cref{def:2-1-1} is defined in a usual way by $P_{i}u = \chi_{i}u$. Next we define the generalized harmonic space on $D\subset \Omega$ in this setting:
\begin{equation}
H_{B}(D) = \{u\in H^{2}_{\Gamma}(D): B_{D}(u,v) = 0\quad \text{for all}\;\;v\in  H^{2}_{0}(D)\}.    
\end{equation}
Let $P_{1}$ be the space of linear polynomials, and let $D\subset\Omega$ satisfy the cone condition. If $\partial D\cap \Gamma\neq \emptyset$, then $H_{B}(D)\cap P_{1}=\emptyset$. It is known \cite[Chapter 5.9]{brenner2008mathematical} that in this case, the form $B_{D}(\cdot,\cdot)$ is coercive on $H_{B}(D)$ and thus $(H_{B}(D), B_{D}(\cdot,\cdot))$ is a Hilbert space. If $\partial D\cap \Gamma=\emptyset$, then it follows from \cref{fourth-order:coercive} and the Deny-Lions theorem that $B_{D}(\cdot,\cdot)$ is coercive on $H_{B}(D)/P_{1}$, and that $\big(H_{B}(D)/P_{1},  B_{D}(\cdot,\cdot)\big)$ is a Hilbert space. Hence, \cref{ass:2-2-3} holds true with $\mathcal{K} = \emptyset$ when $\partial D\cap \Gamma\neq \emptyset$, and with $\mathcal{K}=P_{1}$ when $\partial D\cap \Gamma=\emptyset$. Next, we establish the Caccioppoli inequality in this setting. 
\begin{lemma}[Caccioppoli-type inequality]\label{lem:cac-5-6}
Let $D\subset D^{\ast}$ be subdomains of $\Omega$ with $\delta:={\rm dist}(D,\partial D^{\ast}\setminus\partial \Omega)>0$, and let $u\in H_{B}(D^{\ast})$. Then,
\begin{equation}\label{cac:5-6-1}
\Vert u\Vert_{B,D} \leq C\Vert \mathtt{D}\Vert^{1/2}_{W^{1,\infty}(D^{\ast})} \big(\delta^{-2}\Vert u\Vert_{L^{2}(D^{\ast}\setminus D)} + \delta^{-1}\Vert \nabla u\Vert_{{\bm L}^{2}(D^{\ast}\setminus D)} \big),
\end{equation}
where $C$ is a generic constant independent of $\mathtt{D}$ and $\delta$.
\end{lemma}
\begin{proof}
Let $u\in H_{B}(D^{\ast})$, and let $\eta\in C^{\infty}(\overline{D^{\ast}})$ satisfy
\begin{equation}\label{cac-5-6:cut-off}
\eta=1 \;\;\text{in}\;\, D, \;\;\, {\rm dist}({\rm supp}(\eta), \partial D^{\ast}\setminus\partial \Omega)>0, \;\;\,|\eta|_{W^{j,\infty}(D^{\ast})}\leq \delta^{-j}\, (j=0,1,2). 
\end{equation}
Clearly, we have $\eta^{2}u\in H_{0}^{2}(D^{\ast})$, and thus $B_{D^{\ast}}(u,\eta^{2}u)=0$. It follows that 
\begin{equation}\label{cac-fourth-order:decomposition}
\begin{array}{cc}
\displaystyle I_{0}:=\int_{D^{\ast}} \eta^{2}\mathtt{D}(x,y) \big(\Delta u \Delta u - (1-\nu)(u_{xx}u_{yy} + u_{yy}u_{xx} - 2u_{xy}u_{xy})\big) dxdy \\[2ex]
= -[I_{\Delta} -(1-\nu) (I_{xx} + I_{yy} -2I_{xy})],
\end{array} 
\end{equation}
where
\begin{equation}
\begin{array}{cc}
\displaystyle I_{\Delta} = \int_{D^{\ast}} \mathtt{D}(x,y) \big( 2 \eta u\Delta u\Delta \eta + 4\eta \Delta u\nabla u\cdot \nabla \eta + 2u\Delta u |\nabla \eta|^{2} \big)dxdy,\\[2ex]
\displaystyle I_{xx} = \int_{D^{\ast}} \mathtt{D}(x,y) \big( 2 \eta u u_{xx}\eta_{yy} + 4\eta u_{xx} u_{y}\eta_{y} + 2u u_{xx} |\eta_{y}|^{2}\big)dxdy.
\end{array}    
\end{equation}
$I_{yy}$ and $I_{xy}$ are defined similarly as $I_{xx}$. Next we estimate $I_{\Delta}$. Using \cref{cac-5-6:cut-off} yields
\begin{equation}\label{I-delta-1}
\begin{array}{cc}
\displaystyle  \int_{D^{\ast}} \mathtt{D}(x,y) \big( 2 \eta u\Delta u\Delta \eta + 4\eta \Delta u\nabla u\cdot \nabla \eta \big)dxdy \leq \frac{1+\nu}{8} \Vert \mathtt{D}^{1/2}\eta \Delta u\Vert^{2}_{L^{2}(D^{\ast})} \\[2ex]
+\,C\Vert \mathtt{D}\Vert_{L^{\infty}(D^{\ast})} \big(\delta^{-4}\Vert u\Vert^{2}_{L^{2}(D^{\ast}\setminus D)} + \delta^{-2}\Vert \nabla u\Vert^{2}_{{\bm L}^{2}(D^{\ast}\setminus D)}\big).
 \end{array}
\end{equation}
To proceed, using integration by parts and the fact that $u|\nabla \eta|^{2} = 0$ on $\partial D^{\ast}$ gives
\begin{equation}\label{I-delta-2}
\begin{array}{cc}
\displaystyle \int_{D^{\ast}}\mathtt{D}(x,y)u\Delta u |\nabla \eta|^{2} dxdy = - \int_{D^{\ast}} \nabla \mathtt{D} \cdot \nabla u|\nabla \eta|^{2}u \,dxdy \\[2ex] 
\displaystyle - \int_{D^{\ast}}\mathtt{D}(x,y) \big(2u\Delta \eta \nabla \eta\cdot\nabla u + |\nabla u|^{2}|\nabla\eta|^{2}\big) dxdy \\[2ex]
\displaystyle \leq C\Vert \mathtt{D}\Vert_{W^{1,\infty}(D^{\ast})} \big(\delta^{-4}\Vert u\Vert^{2}_{L^{2}(D^{\ast}\setminus D)} + \delta^{-2}\Vert \nabla u\Vert^{2}_{{\bm L}^{2}(D^{\ast}\setminus D)} \big).
\end{array}
\end{equation}
Combining \cref{I-delta-1,I-delta-2} and noting that $(1+\nu)\Vert \mathtt{D}^{1/2}\eta \Delta u\Vert^{2}_{L^{2}(D^{\ast})}\leq 2I_{0}$, we get
\begin{equation}
|I_{\Delta}|\leq \frac{1}{4}I_{0}   + C\Vert \mathtt{D}\Vert_{W^{1,\infty}(D^{\ast})} \big(\delta^{-4}\Vert u\Vert^{2}_{L^{2}(D^{\ast}\setminus D)} + \delta^{-2}\Vert \nabla u\Vert^{2}_{{\bm L}^{2}(D^{\ast}\setminus D)} \big).
\end{equation}
The other terms in \cref{cac-fourth-order:decomposition} can be estimated similarly. Hence, the lemma is proved.
\end{proof}

\textbf{Local eigenproblems}. Using the Caccioppoli inequality \cref{cac:5-6-1} and the Rellich–Kondrachov theorem, we can verify \cref{ass:2-2-2} similarly as before. The local eigenproblems are given by: Find $\lambda_{i}\in \mathbb{R} \cup \{+\infty \}$ and $\phi_{i}\in H_{B}(\omega_{i}^{\ast})$ such that
\begin{equation}\label{eigen:5-6}
{B}_{\omega^{\ast}_i}(\chi_{i}\phi_{i}, \chi_i v) = \lambda_{i} {B}_{\omega^{\ast}_i}(\phi_{i}, v)\quad \forall v\in H_{B}(\omega_{i}^{\ast}).
\end{equation} 

To obtain the exponential local convergence, we need to prove the weak approximation property. Indeed, we have
\begin{lemma}\label{lem:hp-weak-approx}
Let $D\subset D^{\ast}\subset {D}^{\ast\ast}$ be subdomains of $\Omega$ with $\delta:={\rm dist} ({D}^{\ast},\partial D^{\ast\ast}\setminus \partial \Omega)>0$, and let $\mathsf{V}_{\delta}({D}^{\ast}\setminus D):= \big\{{\bm x}\in {D}^{\ast\ast}: {\rm dist}({\bm x}, {D}^{\ast}\setminus D)\leq \delta \big\}$. 
Then, for any integer $m\geq C_{1}\big|\mathsf{V}_{\delta}({D}^{\ast}\setminus D)\big|\delta^{-d}$, there exists an $m$-dimensional space $Q_{m}(D^{\ast\ast})\subset H^{1}(D^{\ast\ast})$ such that for any $u\in H^{2}(D^{\ast\ast})$ and any $a_{1},a_{2}>0$,
    \begin{equation}\label{wa:5-6}
    \begin{array}{cc}
   {\displaystyle \inf_{v\in Q_{m}(D^{\ast\ast})}\Big(a_{1} \big\Vert \nabla (u-v)\big\Vert_{L^{2}(D^{\ast}\setminus D)} + a_{2}\Vert u-v\Vert_{L^{2}(D^{\ast}\setminus D)}\Big)  }\\[2ex]
{\displaystyle \leq C_{2} \Big( a_{1}\big|\mathsf{V}_{\delta}({D}^{\ast}\setminus D)\big|^{1/d} m^{-1/d} + a_{2} \big|\mathsf{V}_{\delta}({D}^{\ast}\setminus D)\big|^{2/d} m^{-2/d}\Big)\,|u|_{H^{2}(D^{\ast\ast})}, }
        \end{array}    
    \end{equation}
where the constants $C_{1}$ and $C_{2}$ depend only on $d$.    
\end{lemma}
\cm{The proof of \cref{lem:hp-weak-approx} is very similar to that of \cref{lem:5-3} (i) -- the only difference is the definition of $Q_{m}({D}^{\ast\ast})$ in \cref{eq:5-22} and the associated error estimate \cref{eq:5-23}. Specifically, noting that $u\in H^{2}(D^{\ast\ast})$, here we define $Q_{m}({D}^{\ast\ast})$ as the span of the classical "hat functions" associated with the mesh $\widetilde{T}_{H}$ (these functions are naturally extended to $D^{\ast\ast}$ as $H^{1}$ functions), and use standard error estimates for finite element nodal interpolants (see \cite[Theorem 4.4.20]{brenner2008mathematical}) in place of \cref{eq:5-23}.}

In what follows, we show how the inequality \cref{cac:5-6-1} and the approximation result \cref{wa:5-6} fit into the abstract framework in \cref{sec:framework-higher-order} for higher-order problems. Let $M=2$, and for any subdomain $D\subset\Omega$, let $\mathcal{L}(D):=H^{1}(D)$, $\mathcal{L}_{1}(D):={\bm L}^{2}(D)$, and $\mathcal{L}_{2}(D):=L^{2}(D)$. Moreover, we define the operators $A_{D,1}: H^{1}(D)\rightarrow {\bm L}^{2}(D)$ by $A_{D,1}(u) = \nabla u$, and $A_{D,2}: H^{1}(D)\rightarrow L^{2}(D)$ by $A_{D,2} = id$. Clearly, inequality \cref{hp:norm-equiv} holds with $c_{0}=1$. With this setup, we see from \cref{lem:cac-5-6} that the generalized Caccioppoli-type inequality \cref{hp:cac-inequality} holds true with $C_{\rm cac}^{\rm I} = C\Vert \mathtt{D}\Vert^{1/2}_{W^{1,\infty}(\Omega)}$, and from \cref{lem:hp-weak-approx} and \cref{fourth-order:coercive} we see that the generalized weak approximation property \cref{hp:weak-approx} holds with $C_{\rm wa}=C\big(\frac{1}{2}\mathtt{D}_{\rm min}(1-\nu)\big)^{-1}$ and $\alpha = 1/d$. Hence, we have
\begin{theorem}[Exponential local convergence]\label{upperbound:5-6}
For each $i=1,\cdots,M$, \cref{thm:3-2} holds true with $\alpha = 1/d$.
\end{theorem}
The quasi-optimal convergence of the method follows from \cref{quasi-optimal-coercive}.


\section{Conclusion and outlook}\label{sec:conclusion}
In this paper, we have formulated an abstract framework for the design, implementation, and analysis of MS-GFEM, and applied it to various multiscale PDEs. A key component of the framework is the abstract local approximation theory. Based on this theory, we established a unified framework for characterizing the low-rank property of multiscale PDEs. Our analysis reveals that MS-GFEM has the same theoretical underpinnings as a variety of efficient numerical techniques that exploit the low-rank property of PDEs. Most importantly, we proved sharpened convergence rates for the method  surpassing previous results. This work provides a comprehensive foundation unifying local basis construction, approximation properties, and connections to related methods for MS-GFEM.


There are several important directions for future research. First, we have proved the upper bound $O(e^{-cn^{1/d}})$ for the local approximation errors of MS-GFEM. A natural question is: What is the optimal local convergence rate of MS-GFEM for practical applications? This question is equivalent to asking the asymptotic decay rate of the Kolmogrov $n$-widths $d_{n}(\omega_i,\omega_i^{\ast})$, or the optimal estimate of the approximate separability of the Green's functions. We conjecture that the optimal local convergence rate is $O(e^{-cn^{1/(d-1)}})$. This is indeed the case for two-dimensional Poisson-type equations with constant coefficients, as shown in \cite{babuska2011optimal} by a direct calculation; see also \cite{borm2013low,borm2016approximation} for such upper bound estimates by using a Green’s representation formula. Also, numerical results in different contexts \cite{ma2022error,chupeng2023wavenumber,faustmann2015h,grasedyck2009domain} support this conjecture. 

The dependence of the obtained error estimates for MS-GFEM on the coefficient contrast of the problems was not investigated in this paper. Indeed, this dependence is hidden in the constant $\Theta$ in \cref{thm:3-1}. While theoretical results may exhibit a pessimistic dependence on the contrast, numerical results showed that the convergence rates are nearly contrast-independent in various settings \cite{ma2022error,bebendorf2003existence}. A rigorous proof of these numerical observations requires a careful analysis of the proof of \cref{thm:3-1}. We will address this issue in a forthcoming work.

We have mainly dealt with the issue of approximation for MS-GFEM in this paper. The stability of the method, however, was analysed within the framework of standard compact perturbation method. This leads to pessimistic stability conditions for convection diffusion problems. The stability issue for highly indefinite problems will be addressed in a separate paper using some special techniques. For example, one can suitably enrich the test and trial space or use a Petrov-Galerkin formulation to guarantee the stability. An alternative is to give up the partition of unity technique and work in the framework of Trefftz methods or discontinuous Galerkin methods.

Finally, we note that an important class of PDEs, saddle point problems, do not fit into our framework formally. Indeed, they can also be solved by MS-GFEM, but there exist two aspects that need a different treatment compared with the current method. First, one needs to appropriately transform the local approximation problems at hand before using the abstract approximation theory in this paper. Second, the local spaces need to be carefully designed such that the inf-sup condition is fulfilled. We refer to \cite{alber2024mixed} for the application of MS-GFEM to the Darcy's equation in mixed formulation.

\appendix
\section{Kolmogorov $n$-widths and their properties \cite{pinkus1985n}}\label{Kolmogrov-nwidth}
Let $\mathbb{H}_{1}$ and $\mathbb{H}_{2}$ be Hilbert spaces, and let $T:\mathbb{H}_{1}\rightarrow\mathbb{H}_{2}$ be a bounded linear operator. For each $n\in\mathbb{N}$, we can define the Kolmogrov $n$-width of the operator $T$ as
\begin{equation}\label{eq:1-1}
d_{n}(T) = \inf_{S(n)\subset \mathbb{H}_{2}} \sup_{u\in \mathbb{H}_{1}}\inf_{v\in S(n)}\frac{\Vert Tu - v\Vert_{\mathbb{H}^{2}}}{\Vert u\Vert_{\mathbb{H}^{1}}}, 
\end{equation}
where the leftmost infimum is taken over all $n$-dimensional subspaces $S(n)$ of $\mathbb{H}_{2}$. It is a classical result in operator theory that $\lim_{n\rightarrow\infty}d_{n}(T) = 0$ if and only if $T$ is compact. Indeed, the sequence $\big( d_{n}(T)\big)_{n\in\mathbb{N}}$ provides a quantification of the degree of compactness of $T$. 

Let $T^{\ast}:\mathbb{H}_{2}\rightarrow \mathbb{H}_{1}$ be the adjoint of $T$, and let $\{\lambda_{j}\}$ and $\{u_{j}\}$ be the eigenvalues and eigenvectors (in nonascending order) of the problem
\begin{equation}\label{eq:1-2}
T^{\ast}Tu = \lambda u.
\end{equation}
$\{\lambda_{j}^{1/2}\}$, $\{u_{j}\}$, and $\{Tu_{j}\}$ are usually referred to as the singular values, the right and left singular vectors of $T$, respectively. If the operator $T$ is compact, we can characterize the $n$-width $d_{n}(T)$ via the singular values and singular vectors of $T$ as follows.
\begin{lemma}\label{lem:1-1}
Let $T:\mathbb{H}_{1}\rightarrow\mathbb{H}_{2}$ be a compact operator, and let $\{\lambda_{j}\}$ and $\{u_{j}\}$ be the eigenvalues and eigenvectors of problem \cref{eq:1-2}. Then, $d_{n}(T) = \lambda^{1/2}_{n+1}$, and the optimal approximation space is given by $S^{\rm opt}(n) = {\rm span}\big\{Tu_{1}, \ldots, Tu_{n}\big\}$.
\end{lemma}

The $n$-width satisfies the following additivity and multiplicativity properties. Let $\mathbb{H}_{1}$, $\mathbb{H}_{2}$, and $\mathbb{H}_{3}$ be Hilbert spaces, and let $T,\,Q: \mathbb{H}_{1}\rightarrow\mathbb{H}_{2}$ and $R: \mathbb{H}_{2}\rightarrow\mathbb{H}_{3}$ be bounded linear operators. Then, 
\begin{align}
d_{m+n}(T+Q)\leq &\, d_{m}(T)+d_{n}(Q),\label{additivity-property}\\[1mm]
 d_{m+n}(RT)\leq d_{m}(R)d_{n}(T),&\quad d_{n}(RT)\leq \Vert R\Vert \,d_{n}(T),\label{eq:1-3}
\end{align}
where $\Vert R\Vert$ denotes the operator norm of $R$. 

\section{Proof of \cref{prop:3-1-1}}\label{sec:appendix-1-2}
\begin{proof}
Let $\widetilde{D}$ with $D\subset \widetilde{D}\subset D^{\ast}$ be an intermediate domain between $D$ and $D^{\ast}$ such that ${\rm dist} ({D},\partial \widetilde{D}\setminus \partial \Omega) = {\rm dist} (\widetilde{D},\partial D^{\ast}\setminus \partial \Omega)>0$. In view of \cref{ass:2-2-3}, we may assume that $\Vert\cdot\Vert_{B^{+}, D}$ is a norm on $\mathcal{H}_{B}(D)$ for any subdomains $D\subset\Omega$; otherwise we can consider the corresponding quotient spaces. Let $\widetilde{R}:\mathcal{H}_{B}(D^{\ast})\rightarrow \mathcal{L}(\widetilde{D})$ be defined by $\widetilde{R}(u) = u|_{\widetilde{D}}$. Then, \cref{ass:3-1-2} implies that $\widetilde{R}$ is compact. Let $\{u_{n}\}_{n=1}^{\infty}$ be a bounded sequence in $\mathcal{H}_{B}(D^{\ast})$. Due to the compactness of operator $\widetilde{R}$, there is a subsequence, still denoted by $\{u_{n}\}$, such that $\{u_{n}|_{\widetilde{D}}\}\subset \mathcal{H}_{B}(\widetilde{D})$ is a Cauchy sequence in $\mathcal{L}(D^{\ast})$. It follows from \cref{ass:3-1-1} that $\{u_{n}|_{D}\}$ is a Cauchy sequence in $\mathcal{H}_{B}(D)$. Since $\mathcal{H}_{B}(D)$ is a closed subspace of $\mathcal{H}(D)$ (\cref{prop:2-1}), the sequence $\{u_{n}|_{D}\}$ converges to an element of $\mathcal{H}_{B}(D)$. Hence, operator $R$ is compact.  
\end{proof}

\section{some useful lemmas}
\begin{lemma}{\cite[Theorem 2.10]{oleinik1992mathematical}}\label{lem:wa-5-3-1}
Suppose that $D\subset\mathbb{R}^{d}$ is a bounded domain of diameter $r$ and is star-shaped with respect to a ball $Q_{r_1}$ of radius $r_{1}$. Then for any ${\bm u}\in {\bm H}^{1}(D)$,
\begin{equation}
  \Vert \nabla {\bm u}\Vert_{L^{2}(D)}^{2}\leq C_{1}\Big(\frac{r}{r_1} \Big)^{d+1} \Vert {\bm \varepsilon} ({\bm u})\Vert^{2}_{L^{2}(D)} +  C_{2}\Big(\frac{r}{r_1} \Big)^{d} \Vert \nabla {\bm u} \Vert^{2}_{L^{2}(Q_{r_1})},
\end{equation}
where $C_{1}$ and $C_{2}$ are constants depending only on $d$.
\end{lemma}

\begin{lemma}[Glazman Lemma \cite{glazman1965direct}]\label{Glazmann Lemma}
Let $\mathcal{H}$ be a Hilbert space, and let $T:\mathcal{H}\rightarrow \mathcal{H}$ be a compact self-adjoint operator. Denote by $N(\lambda, T)$ the number of eigenvalues of $T$ above $\lambda$. Then
\begin{equation}
N(\lambda, T) = {\rm max}\, {\rm dim} \big\{\mathcal{L}\subset \mathcal{H}: (Tu,u)_{\mathcal{H}}>\lambda (u, u)_{\mathcal{H}},\;\,\forall u\in \mathcal{L}\setminus \{ 0\}  \big\}.
\end{equation}
\end{lemma}

\bibliographystyle{siamplain}
\bibliography{references}

\begin{thebibliography}{100}

\bibitem{abdulle2006analysis}
{\sc A.~Abdulle}, {\em Analysis of a heterogeneous multiscale fem for problems
  in elasticity}, Math. Models Methods Appl. Sci., 16 (2006), pp.~615--635.

\bibitem{abdulle2014discontinuous}
{\sc A.~Abdulle and M.~E. Huber}, {\em Discontinuous galerkin finite element
  heterogeneous multiscale method for advection--diffusion problems with
  multiple scales}, Numer. Math., 126 (2014), pp.~589--633.

\bibitem{abdulle2012heterogeneous}
{\sc A.~Abdulle, E.~Weinan, B.~Engquist, and E.~Vanden-Eijnden}, {\em The
  heterogeneous multiscale method}, Acta Numerica, 21 (2012), pp.~1--87.

\bibitem{alber2024mixed}
{\sc C.~Alber, C.~Ma, and R.~Scheichl}, {\em A mixed multiscale spectral
  generalized finite element method}, arXiv preprint arXiv:2403.16714,  (2024).

\bibitem{altmann2021numerical}
{\sc R.~Altmann, P.~Henning, and D.~Peterseim}, {\em Numerical homogenization
  beyond scale separation}, Acta Numerica, 30 (2021), pp.~1--86.

\bibitem{amestoy2015improving}
{\sc P.~Amestoy, C.~Ashcraft, O.~Boiteau, A.~Buttari, J.-Y. L'Excellent, and
  C.~Weisbecker}, {\em Improving multifrontal methods by means of block
  low-rank representations}, SIAM J. Sci. Comput., 37 (2015), pp.~A1451--A1474.

\bibitem{angleitner2023exponential}
{\sc N.~Angleitner, M.~Faustmann, and J.~M. Melenk}, {\em Exponential meshes
  and h-matrices}, Computers \& Mathematics with Applications, 130 (2023),
  pp.~21--40.

\bibitem{araya2013multiscale}
{\sc R.~Araya, C.~Harder, D.~Paredes, and F.~Valentin}, {\em Multiscale
  hybrid-mixed method}, SIAM Journal on Numerical Analysis, 51 (2013),
  pp.~3505--3531.

\bibitem{babuska1972survey}
{\sc I.~Babuska}, {\em Survey lectures on the mathematical foundations of the
  finite element method}, The Mathematical Foundations of the Finite Element
  Method with Applicaions to Partial Differential Equations,  (1972),
  pp.~3--359.

\bibitem{babuvska2014machine}
{\sc I.~Babu{\v{s}}ka, X.~Huang, and R.~Lipton}, {\em Machine computation using
  the exponentially convergent multiscale spectral generalized finite element
  method}, ESAIM: Math. Model. Numer. Anal., 48 (2014), pp.~493--515.

\bibitem{babuska2011optimal}
{\sc I.~Babu{\v{s}}ka and R.~Lipton}, {\em Optimal local approximation spaces
  for generalized finite element methods with application to multiscale
  problems}, Multiscale Model. Simul., 9 (2011), pp.~373--406.

\bibitem{babuvska2020multiscale}
{\sc I.~Babu{\v{s}}ka, R.~Lipton, P.~Sinz, and M.~Stuebner}, {\em
  {Multiscale-Spectral GFEM} and optimal oversampling}, Comput. Methods Appl.
  Mech. Eng., 364 (2020), p.~112960.

\bibitem{babuvska1997partition}
{\sc I.~Babu{\v{s}}ka and J.~M. Melenk}, {\em The partition of unity method},
  Int. J. Numer. Methods Eng., 40 (1997), pp.~727--758.

\bibitem{babuvska2000can}
{\sc I.~Babu{\v{s}}ka and J.~Osborn}, {\em Can a finite element method perform
  arbitrarily badly?}, Mathematics of computation, 69 (2000), pp.~443--462.

\bibitem{bachmayr2017kolmogorov}
{\sc M.~Bachmayr and A.~Cohen}, {\em Kolmogorov widths and low-rank
  approximations of parametric elliptic pdes}, Math. Comp., 86 (2017),
  pp.~701--724.

\bibitem{bachmayr2018parametric}
{\sc M.~Bachmayr, A.~Cohen, and W.~Dahmen}, {\em Parametric pdes: sparse or
  low-rank approximations?}, IMA J. Numer. Anal., 38 (2018), pp.~1661--1708.

\bibitem{bauer2016maxwell}
{\sc S.~Bauer, D.~Pauly, and M.~Schomburg}, {\em The maxwell compactness
  property in bounded weak lipschitz domains with mixed boundary conditions},
  SIAM J. Math. Anal., 48 (2016), pp.~2912--2943.

\bibitem{bebendorf2005efficient}
{\sc M.~Bebendorf}, {\em Efficient inversion of the galerkin matrix of general
  second-order elliptic operators with nonsmooth coefficients}, Math. Comp., 74
  (2005), pp.~1179--1199.

\bibitem{bebendorf2003existence}
{\sc M.~Bebendorf and W.~Hackbusch}, {\em Existence of $\mathcal{H}$-matrix
  approximants to the inverse {FE}-matrix of elliptic operators with
  ${L}^{\infty}$-coefficients}, Numer. Math., 95 (2003), pp.~1--28.

\bibitem{benezech2022scalable}
{\sc J.~B{\'e}n{\'e}zech, L.~Seelinger, P.~Bastian, R.~Butler, T.~Dodwell,
  C.~Ma, and R.~Scheichl}, {\em Scalable multiscale-spectral gfem with an
  application to composite aero-structures}, J. Comput. Phys., 508 (2024),
  p.~113013.

\bibitem{bensoussan2011asymptotic}
{\sc A.~Bensoussan, J.-L. Lions, and G.~Papanicolaou}, {\em Asymptotic analysis
  for periodic structures}, vol.~374, American Mathematical Soc., 2011.

\bibitem{bernstein1912ordre}
{\sc S.~Bernstein}, {\em Sur l'ordre de la meilleure approximation des
  fonctions continues par des polyn{\^o}mes de degr{\'e} donn{\'e}}, vol.~4,
  Hayez, imprimeur des acad{\'e}mies royales, 1912.

\bibitem{birman2007weyl}
{\sc M.~S. Birman and N.~D. Filonov}, {\em Weyl asymptotics of the spectrum of
  the maxwell operator with non-smooth coefficients in lipschitz domains},
  American Mathematical Society Translations: Series 2, 220 (2007), pp.~27--44.

\bibitem{boffi2013mixed}
{\sc D.~Boffi, F.~Brezzi, and M.~Fortin}, {\em Mixed finite element methods and
  applications}, Springer-Verlag, Berlin, 2013.

\bibitem{bonizzoni2022super}
{\sc F.~Bonizzoni, P.~Freese, and D.~Peterseim}, {\em Super-localized
  orthogonal decomposition for convection-dominated diffusion problems}, arXiv
  preprint arXiv:2206.01975,  (2022).

\bibitem{borm2010approximation}
{\sc S.~B{\"o}rm}, {\em Approximation of solution operators of elliptic partial
  differential equations by $\mathcal{H}$- and $\mathcal{H}^{2}$-matrices},
  Numer. Math., 115 (2010), pp.~165--193.

\bibitem{borm2010efficient}
{\sc S.~B{\"o}rm}, {\em Efficient numerical methods for non-local operators:
  H2-matrix compression, algorithms and analysis}, vol.~14, European
  Mathematical Society, 2010.

\bibitem{borm2016approximation}
{\sc S.~B{\"o}rm and S.~Christophersen}, {\em Approximation of integral
  operators by green quadrature and nested cross approximation}, Numer. Math.,
  133 (2016), pp.~409--442.

\bibitem{borm2013low}
{\sc S.~B{\"o}rm and J.~G{\"o}rdes}, {\em Low-rank approximation of integral
  operators by using the green formula and quadrature}, Numerical Algorithms,
  64 (2013), pp.~567--592.

\bibitem{boulle2023elliptic}
{\sc N.~Boull{\'e}, D.~Halikias, and A.~Townsend}, {\em Elliptic pde learning
  is provably data-efficient}, arXiv preprint arXiv:2302.12888,  (2023).

\bibitem{boulle2022learning}
{\sc N.~Boull{\'e}, S.~Kim, T.~Shi, and A.~Townsend}, {\em Learning green's
  functions associated with parabolic partial differential equations}, arXiv
  preprint arXiv:2204.12789,  (2022).

\bibitem{boulle2023learning}
{\sc N.~Boull{\'e} and A.~Townsend}, {\em Learning elliptic partial
  differential equations with randomized linear algebra}, Foundations of
  Computational Mathematics, 23 (2023), pp.~709--739.

\bibitem{braess2007finite}
{\sc D.~Braess}, {\em Finite elements: Theory, fast solvers, and applications
  in solid mechanics}, Cambridge University Press, 2007.

\bibitem{brenner2008mathematical}
{\sc S.~C. Brenner, L.~R. Scott, and L.~R. Scott}, {\em The mathematical theory
  of finite element methods}, vol.~3, Springer, 2008.

\bibitem{brown2023multiscale}
{\sc D.~L. Brown and D.~Gallistl}, {\em Multiscale sub-grid correction method
  for time-harmonic high-frequency elastodynamics with wave number explicit
  bounds}, Computational Methods in Applied Mathematics, 23 (2023), pp.~65--82.

\bibitem{buhr2018randomized}
{\sc A.~Buhr and K.~Smetana}, {\em Randomized local model order reduction},
  SIAM J. Sci. Comput., 40 (2018), pp.~A2120--A2151.

\bibitem{calo2016multiscale}
{\sc V.~M. Calo, E.~T. Chung, Y.~Efendiev, and W.~T. Leung}, {\em Multiscale
  stabilization for convection-dominated diffusion in heterogeneous media},
  Comput. Methods Appl. Mech. Eng., 304 (2016), pp.~359--377.

\bibitem{candes2009fast}
{\sc E.~Candes, L.~Demanet, and L.~Ying}, {\em A fast butterfly algorithm for
  the computation of fourier integral operators}, Multiscale Model. Simul., 7
  (2009), pp.~1727--1750.

\bibitem{cao2002multiscale}
{\sc L.~Cao, J.~Cui, and D.~Zhu}, {\em Multiscale asymptotic analysis and
  numerical simulation for the second order {Helmholtz} equations with rapidly
  oscillating coefficients over general convex domains}, SIAM J. Numer. Anal.,
  40 (2002), pp.~543--577.

\bibitem{cao2010multiscale}
{\sc L.~Cao, Y.~Zhang, W.~Allegretto, and Y.~Lin}, {\em Multiscale asymptotic
  method for maxwell's equations in composite materials}, SIAM J. Numer. Anal.,
  47 (2010), pp.~4257--4289.

\bibitem{chaillat2017theory}
{\sc S.~Chaillat, L.~Desiderio, and P.~Ciarlet}, {\em Theory and implementation
  of h-matrix based iterative and direct solvers for helmholtz and
  elastodynamic oscillatory kernels}, J. Comput. Phys., 351 (2017),
  pp.~165--186.

\bibitem{chandrasekaran2010numerical}
{\sc S.~Chandrasekaran, P.~Dewilde, M.~Gu, and N.~Somasunderam}, {\em On the
  numerical rank of the off-diagonal blocks of {Schur} complements of
  discretized elliptic {PDEs}}, SIAM J. Matrix Anal. Appl., 31 (2010),
  pp.~2261--2290.

\bibitem{chaumont2020multiscale}
{\sc T.~Chaumont-Frelet and F.~Valentin}, {\em A multiscale hybrid-mixed method
  for the {Helmholtz} equation in heterogeneous domains}, SIAM J. Numer. Anal.,
  58 (2020), pp.~1029--1067.

\bibitem{chen2020random}
{\sc K.~Chen, Q.~Li, J.~Lu, and S.~J. Wright}, {\em Random sampling and
  efficient algorithms for multiscale pdes}, SIAM J. Sci. Comput., 42 (2020),
  pp.~A2974--A3005.

\bibitem{chen2021exponential}
{\sc Y.~Chen, T.~Y. Hou, and Y.~Wang}, {\em Exponential convergence for
  multiscale linear elliptic pdes via adaptive edge basis functions},
  Multiscale Model. Simul., 19 (2021), pp.~980--1010.

\bibitem{chen2023exponentially-2}
{\sc Y.~Chen, T.~Y. Hou, and Y.~Wang}, {\em Exponentially convergent multiscale
  finite element method}, Communications on Applied Mathematics and
  Computation,  (2023), pp.~1--17.

\bibitem{chen2023exponentially}
{\sc Y.~Chen, T.~Y. Hou, and Y.~Wang}, {\em Exponentially convergent multiscale
  methods for 2d high frequency heterogeneous helmholtz equations}, Multiscale
  Model. Simul., 21 (2023), pp.~849--883.

\bibitem{chungmultiscale}
{\sc E.~Chung, Y.~Efendiev, and T.~Y. Hou}, {\em Multiscale model reduction:
  multiscale finite element methods and their generalizations}, vol.~212,
  Springer, 2023.

\bibitem{chung2014generalized}
{\sc E.~T. Chung, Y.~Efendiev, and S.~Fu}, {\em Generalized multiscale finite
  element method for elasticity equations}, GEM-International Journal on
  Geomathematics, 5 (2014), pp.~225--254.

\bibitem{chung2014generalized-2}
{\sc E.~T. Chung, Y.~Efendiev, and S.~Fu}, {\em Generalized multiscale finite
  element method for elasticity equations}, GEM-International Journal on
  Geomathematics, 5 (2014), pp.~225--254.

\bibitem{chung2018constraint}
{\sc E.~T. Chung, Y.~Efendiev, and W.~T. Leung}, {\em Constraint energy
  minimizing generalized multiscale finite element method}, Comput. Methods
  Appl. Mech. Eng., 339 (2018), pp.~298--319.

\bibitem{chung2019mixed}
{\sc E.~T. Chung and C.~S. Lee}, {\em A mixed generalized multiscale finite
  element method for planar linear elasticity}, J. Comput. Appl. Math., 348
  (2019), pp.~298--313.

\bibitem{chung2019adaptive}
{\sc E.~T. Chung and Y.~Li}, {\em Adaptive generalized multiscale finite
  element methods for h(curl)-elliptic problems with heterogeneous
  coefficients}, Journal of Computational and Applied Mathematics, 345 (2019),
  pp.~357--373.

\bibitem{ciarlet2017approximation}
{\sc P.~Ciarlet~Jr, S.~Fliss, and C.~Stohrer}, {\em On the approximation of
  electromagnetic fields by edge finite elements. part 2: A heterogeneous
  multiscale method for maxwell’s equations}, Computers \& Mathematics with
  Applications, 73 (2017), pp.~1900--1919.

\bibitem{demlow2011local}
{\sc A.~Demlow, J.~Guzm{\'a}n, and A.~Schatz}, {\em Local energy estimates for
  the finite element method on sharply varying grids}, Math. Comp., 80 (2011),
  pp.~1--9.

\bibitem{ming2005analysis}
{\sc W.~E, P.~Ming, and P.~Zhang}, {\em Analysis of the heterogeneous
  multiscale method for elliptic homogenization problems}, Journal of the
  American Mathematical Society, 18 (2005), pp.~121--156.

\bibitem{efendiev2013generalized}
{\sc Y.~Efendiev, J.~Galvis, and T.~Y. Hou}, {\em Generalized multiscale finite
  element methods (gmsfem)}, J. Comput. Phys., 251 (2013), pp.~116--135.

\bibitem{efendiev2009multiscale}
{\sc Y.~Efendiev and T.~Y. Hou}, {\em Multiscale finite element methods: theory
  and applications}, vol.~4, Springer Science \& Business Media, 2009.

\bibitem{engquist2008asymptotic}
{\sc B.~Engquist and P.~E. Souganidis}, {\em Asymptotic and numerical
  homogenization}, Acta Numerica, 17 (2008), pp.~147--190.

\bibitem{engquist2007fast}
{\sc B.~Engquist and L.~Ying}, {\em Fast directional multilevel algorithms for
  oscillatory kernels}, SIAM J. Sci. Comput., 29 (2007), pp.~1710--1737.

\bibitem{engquist2011sweeping}
{\sc B.~Engquist and L.~Ying}, {\em Sweeping preconditioner for the helmholtz
  equation: hierarchical matrix representation}, Commun. Pure Appl. Math., 64
  (2011), pp.~697--735.

\bibitem{engquist2018approximate}
{\sc B.~Engquist and H.~Zhao}, {\em Approximate separability of the green's
  function of the helmholtz equation in the high frequency limit}, Commun. Pure
  Appl. Math., 71 (2018), pp.~2220--2274.

\bibitem{faustmann2022h}
{\sc M.~Faustmann, J.~M. Melenk, and M.~Parvizi}, {\em $\mathcal{H}$-matrix
  approximability of inverses of {FEM} matrices for the time-harmonic maxwell
  equations}, Adv. Comput. Math., 48 (2022), p.~59.

\bibitem{faustmann2015h}
{\sc M.~Faustmann, J.~M. Melenk, and D.~Praetorius}, {\em $\mathcal{H}$-matrix
  approximability of the inverses of {FEM} matrices}, Numer. Math., 131 (2015),
  pp.~615--642.

\bibitem{freese2024super}
{\sc P.~Freese, M.~Hauck, T.~Keil, and D.~Peterseim}, {\em A super-localized
  generalized finite element method}, Numer. Math., 156 (2024), pp.~205--235.

\bibitem{freese2021super}
{\sc P.~Freese, M.~Hauck, and D.~Peterseim}, {\em Super-localized orthogonal
  decomposition for high-frequency helmholtz problems}, arXiv preprint
  arXiv:2112.11368,  (2021).

\bibitem{fu2019efficient}
{\sc S.~Fu, K.~Gao, R.~L. Gibson, and E.~T. Chung}, {\em An efficient
  high-order multiscale finite element method for frequency-domain elastic wave
  modeling}, Computational Geosciences, 23 (2019), pp.~997--1010.

\bibitem{fu2021wavelet}
{\sc S.~Fu, G.~Li, R.~Craster, and S.~Guenneau}, {\em Wavelet-based edge
  multiscale finite element method for {Helmholtz} problems in perforated
  domains}, Multiscale Model. Simul., 19 (2021), pp.~1684--1709.

\bibitem{gallistl2018numerical}
{\sc D.~Gallistl, P.~Henning, and B.~Verf\"{u}rth}, {\em Numerical
  homogenization of h(curl)-problems}, SIAM J. Numer. Anal., 56 (2018),
  pp.~1570--1596.

\bibitem{glazman1965direct}
{\sc I.~M. Glazman}, {\em Direct methods of qualitative spectral analysis of
  singular differential operators}, Jerusalem, 1965.

\bibitem{grasedyck2008parallel}
{\sc L.~Grasedyck, R.~Kriemann, and S.~Le~Borne}, {\em Parallel black box-lu
  preconditioning for elliptic boundary value problems}, Comput. Vis. Sci., 11
  (2008), pp.~273--291.

\bibitem{grasedyck2009domain}
{\sc L.~Grasedyck, R.~Kriemann, and S.~Le~Borne}, {\em Domain decomposition
  based-lu preconditioning}, Numer. Math., 112 (2009), pp.~565--600.

\bibitem{greengard1987fast}
{\sc L.~Greengard and V.~Rokhlin}, {\em A fast algorithm for particle
  simulations}, J. Comput. Phys., 73 (1987), pp.~325--348.

\bibitem{hackbusch2015hierarchical}
{\sc W.~Hackbusch et~al.}, {\em Hierarchical matrices: algorithms and
  analysis}, vol.~49, Springer, 2015.

\bibitem{harder2016hybrid}
{\sc C.~Harder, A.~L. Madureira, and F.~Valentin}, {\em A hybrid-mixed method
  for elasticity}, ESAIM: Math. Model. Numer. Anal., 50 (2016), pp.~311--336.

\bibitem{hauck2022multi}
{\sc M.~Hauck and D.~Peterseim}, {\em Multi-resolution localized orthogonal
  decomposition for helmholtz problems}, Multiscale Model. Simul., 20 (2022),
  pp.~657--684.

\bibitem{hauck2023super}
{\sc M.~Hauck and D.~Peterseim}, {\em Super-localization of elliptic multiscale
  problems}, Math. Comp., 92 (2023), pp.~981--1003.

\bibitem{henning2016new}
{\sc P.~Henning, M.~Ohlberger, and B.~Verf{\"u}rth}, {\em A new heterogeneous
  multiscale method for time-harmonic maxwell's equations}, SIAM J. Numer.
  Anal., 54 (2016), pp.~3493--3522.

\bibitem{henning2016multiscale}
{\sc P.~Henning and A.~Persson}, {\em A multiscale method for linear elasticity
  reducing poisson locking}, Comput. Methods Appl. Mech. Eng., 310 (2016),
  pp.~156--171.

\bibitem{henning2020computational}
{\sc P.~Henning and A.~Persson}, {\em Computational homogenization of
  time-harmonic maxwell's equations}, SIAM J. Sci. Comput., 42 (2020),
  pp.~B581--B607.

\bibitem{hetmaniuk2010special}
{\sc U.~L. Hetmaniuk and R.~B. Lehoucq}, {\em A special finite element method
  based on component mode synthesis}, ESAIM: Mathematical Modelling and
  Numerical Analysis, 44 (2010), pp.~401--420.

\bibitem{ho2016hierarchical}
{\sc K.~L. Ho and L.~Ying}, {\em Hierarchical interpolative factorization for
  elliptic operators: differential equations}, Commun. Pure Appl. Math., 69
  (2016), pp.~1415--1451.

\bibitem{hou1999convergence}
{\sc T.~Hou, X.-H. Wu, and Z.~Cai}, {\em Convergence of a multiscale finite
  element method for elliptic problems with rapidly oscillating coefficients},
  Math. Comp., 68 (1999), pp.~913--943.

\bibitem{hou1997multiscale}
{\sc T.~Y. Hou and X.-H. Wu}, {\em A multiscale finite element method for
  elliptic problems in composite materials and porous media}, J. Comput. Phys.,
  134 (1997), pp.~169--189.

\bibitem{hughes1995multiscale}
{\sc T.~J. Hughes}, {\em Multiscale phenomena: Green's functions, the
  dirichlet-to-neumann formulation, subgrid scale models, bubbles and the
  origins of stabilized methods}, Comput. Methods Appl. Mech. Eng., 127 (1995),
  pp.~387--401.

\bibitem{hughes2007variational}
{\sc T.~J. Hughes and G.~Sangalli}, {\em Variational multiscale analysis: the
  fine-scale green’s function, projection, optimization, localization, and
  stabilized methods}, SIAM J. Numer. Anal., 45 (2007), pp.~539--557.

\bibitem{jikov2012homogenization}
{\sc V.~V. Jikov, S.~M. Kozlov, and O.~A. Oleinik}, {\em Homogenization of
  {Differential Operators and Integral Functionals}}, Springer Science \&
  Business Media, 2012.

\bibitem{le2014msfem}
{\sc C.~Le~Bris, F.~Legoll, and A.~Lozinski}, {\em Msfem {\`a} la
  crouzeix-raviart for highly oscillatory elliptic problems}, Partial
  Differential Equations: Theory, Control and Approximation: In Honor of the
  Scientific Heritage of Jacques-Louis Lions,  (2014), pp.~265--294.

\bibitem{le2017numerical}
{\sc C.~Le~Bris, F.~Legoll, and F.~Madiot}, {\em A numerical comparison of some
  multiscale finite element approaches for advection-dominated problems in
  heterogeneous media}, ESAIM: Math. Model. Numer. Anal., 51 (2017),
  pp.~851--888.

\bibitem{luo2014fast}
{\sc S.~Luo, J.~Qian, and R.~Burridge}, {\em Fast huygens sweeping methods for
  helmholtz equations in inhomogeneous media in the high frequency regime}, J.
  Comput. Phys., 270 (2014), pp.~378--401.

\bibitem{chupeng2023wavenumber}
{\sc C.~Ma, C.~Alber, and R.~Scheichl}, {\em Wavenumber explicit convergence of
  a multiscale generalized finite element method for heterogeneous helmholtz
  problems}, SIAM J. Numer. Anal., 61 (2023), pp.~1546--1584.

\bibitem{ma2022exponential}
{\sc C.~Ma and J.~M. Melenk}, {\em Exponential convergence of a generalized fem
  for heterogeneous reaction-diffusion equations}, Multiscale Model. Simul., 22
  (2024), pp.~256--282.

\bibitem{ma2022error}
{\sc C.~Ma and R.~Scheichl}, {\em Error estimates for discrete generalized fems
  with locally optimal spectral approximations}, Math. Comp., 91 (2022),
  pp.~2539--2569.

\bibitem{ma2021novel}
{\sc C.~Ma, R.~Scheichl, and T.~Dodwell}, {\em Novel design and analysis of
  generalized finite element methods based on locally optimal spectral
  approximations}, SIAM J. Numer. Anal., 60 (2022), pp.~244--273.

\bibitem{maalqvist2011multiscale}
{\sc A.~M{\aa}lqvist}, {\em Multiscale methods for elliptic problems},
  Multiscale Model. Simul., 9 (2011), pp.~1064--1086.

\bibitem{maalqvist2014localization}
{\sc A.~M{\aa}lqvist and D.~Peterseim}, {\em Localization of elliptic
  multiscale problems}, Math. Comp., 83 (2014), pp.~2583--2603.

\bibitem{maalqvist2020numerical}
{\sc A.~M{\aa}lqvist and D.~Peterseim}, {\em Numerical homogenization by
  localized orthogonal decomposition}, SIAM, 2020.

\bibitem{melenk1995generalized}
{\sc J.~M. Melenk}, {\em On generalized finite element methods}.
\newblock Ph.D. thesis, Department of Mathematics, University of Maryland,
  1995.

\bibitem{ming2006numerical}
{\sc P.~Ming and X.~Yue}, {\em Numerical methods for multiscale elliptic
  problems}, Journal of Computational Physics, 214 (2006), pp.~421--445.

\bibitem{monk2003finite}
{\sc P.~Monk}, {\em Finite element methods for {Maxwell's} equations}, Oxford
  University Press, 2003.

\bibitem{nguyen1973bases}
{\sc T.~V. Nguyen}, {\em Bases communes pour certains espaces de fonctions
  harmoniques}, Bull. Sc. Math, 97 (1973), pp.~33--49.

\bibitem{ohlberger2018new}
{\sc M.~Ohlberger and B.~Verfurth}, {\em A new heterogeneous multiscale method
  for the {Helmholtz} equation with high contrast}, Multiscale Model. Simul.,
  16 (2018), pp.~385--411.

\bibitem{oleinik1992mathematical}
{\sc O.~Oleinik, A.~Shamaev, and G.~Yosifian}, {\em Mathematical Problems in
  Elasticity and Homogenization}, Elsevier, 1992.

\bibitem{owhadi2015bayesian}
{\sc H.~Owhadi}, {\em Bayesian numerical homogenization}, Multiscale Model.
  Simul., 13 (2015), pp.~812--828.

\bibitem{owhadi2019operator}
{\sc H.~Owhadi and C.~Scovel}, {\em Operator-Adapted Wavelets, Fast Solvers,
  and Numerical Homogenization: From a Game Theoretic Approach to Numerical
  Approximation and Algorithm Design}, vol.~35, Cambridge University Press,
  2019.

\bibitem{owhadi2011localized}
{\sc H.~Owhadi and L.~Zhang}, {\em Localized bases for finite-dimensional
  homogenization approximations with nonseparated scales and high contrast},
  Multiscale Model. Simul., 9 (2011), pp.~1373--1398.

\bibitem{owhadi2017gamblets}
{\sc H.~Owhadi and L.~Zhang}, {\em Gamblets for opening the
  complexity-bottleneck of implicit schemes for hyperbolic and parabolic
  odes/pdes with rough coefficients}, J. Comput. Phys., 347 (2017),
  pp.~99--128.

\bibitem{park2004multiscale}
{\sc P.~J. Park and T.~Y. Hou}, {\em Multiscale numerical methods for
  singularly perturbed convection-diffusion equations}, International Journal
  of Computational Methods, 1 (2004), pp.~17--65.

\bibitem{payne1961korn}
{\sc L.~Payne and H.~F. Weinberger}, {\em On korn's inequality}, Archive for
  Rational Mechanics and Analysis, 8 (1961), pp.~89--98.

\bibitem{peterseim2017eliminating}
{\sc D.~Peterseim}, {\em Eliminating the pollution effect in {Helmholtz}
  problems by local subscale correction}, Math. Comp., 86 (2017),
  pp.~1005--1036.

\bibitem{pinkus1985n}
{\sc A.~Pinkus}, {\em n-widths in Approximation Theory}, Springer-Verlag,
  Berlin, 1985.

\bibitem{saranen1982inequality}
{\sc J.~Saranen}, {\em On an inequality of friedrichs}, Mathematica
  Scandinavica,  (1982), pp.~310--322.

\bibitem{schleuss2022optimal}
{\sc J.~Schleu{\ss} and K.~Smetana}, {\em Optimal local approximation spaces
  for parabolic problems}, Multiscale Model. Simul., 20 (2022), pp.~551--582.

\bibitem{skiba2000asymptotics}
{\sc N.~Skiba and V.~Zahariuta}, {\em Asymptotics of kolmogorov diameters for
  some classes of harmonic functions on spheroids}, Journal of Approximation
  Theory, 102 (2000), pp.~175--188.

\bibitem{smetana2016optimal}
{\sc K.~Smetana and A.~T. Patera}, {\em Optimal local approximation spaces for
  component-based static condensation procedures}, SIAM J. Sci. Comput., 38
  (2016), pp.~A3318--A3356.

\bibitem{strouboulis2001generalized}
{\sc T.~Strouboulis, K.~Copps, and I.~Babu{\v{s}}ka}, {\em The generalized
  finite element method}, Comput. Methods Appl. Mech. Eng., 190 (2001),
  pp.~4081--4193.

\bibitem{toselli2004domain}
{\sc A.~Toselli and O.~Widlund}, {\em Domain decomposition methods-algorithms
  and theory}, vol.~34, Springer Science \& Business Media, 2004.

\bibitem{verfurth2019heterogeneous}
{\sc B.~Verf{\"u}rth}, {\em Heterogeneous multiscale method for the maxwell
  equations with high contrast}, ESAIM: Math. Model. Numer. Anal., 53 (2019),
  pp.~35--61.

\bibitem{weinan2003heterognous}
{\sc E.~Weinan and B.~Engquist}, {\em The heterognous multiscale methods},
  Communications in Mathematical Sciences, 1 (2003), pp.~87--132.

\bibitem{xia2015high}
{\sc B.~Xia and V.~H. Hoang}, {\em High-dimensional finite element method for
  multiscale linear elasticity}, IMA J. Numer. Anal., 35 (2015),
  pp.~1277--1314.

\bibitem{xia2013efficient}
{\sc J.~Xia}, {\em Efficient structured multifrontal factorization for general
  large sparse matrices}, SIAM J. Sci. Comput., 35 (2013), pp.~A832--A860.

\bibitem{xia2010fast}
{\sc J.~Xia, S.~Chandrasekaran, M.~Gu, and X.~S. Li}, {\em Fast algorithms for
  hierarchically semiseparable matrices}, Numer. Linear Algebra Appl., 17
  (2010), pp.~953--976.

\bibitem{xia2010superfast}
{\sc J.~Xia, S.~Chandrasekaran, M.~Gu, and X.~S. Li}, {\em Superfast
  multifrontal method for large structured linear systems of equations}, SIAM
  J. Matrix Anal. Appl., 31 (2010), pp.~1382--1411.

\bibitem{zhao2022constraint}
{\sc L.~Zhao and E.~Chung}, {\em Constraint energy minimizing generalized
  multiscale finite element method for convection diffusion equation}, arXiv
  preprint arXiv:2203.16035,  (2022).

\end{thebibliography}
\end{document}